\documentclass[11pt]{article}

\usepackage[english]{babel}
\usepackage{amsmath,amsthm,amssymb}
\usepackage[all]{xy}
\usepackage{setspace}
\usepackage{color}
\definecolor{grau}{rgb}{0.3,0.3,0.3}
\usepackage[colorlinks, linkcolor=grau, citecolor=grau, urlcolor=grau]{hyperref}
\usepackage{breakurl}
\usepackage{bibspacing}
\usepackage[top=1.3in, bottom=1.6in, left=1.3in, right=1.3in]{geometry}
\usepackage{booktabs}
\usepackage[ruled,vlined]{algorithm2e}
\usepackage{enumerate}

\newtheorem{theorem}{Theorem}[section]

\newtheorem{corollary}[theorem]{Corollary}
\newtheorem{lemma}[theorem]{Lemma}

\newtheorem{proposition}[theorem]{Proposition}

\theoremstyle{definition}

\newtheorem{remark}[theorem]{Remark}

\numberwithin{equation}{section}

\title{Integral points on coarse Hilbert moduli schemes} 

\author{  Rafael von K\"anel \ \ and \ \ Arno Kret 
}
\newcommand{\OL}{\mathcal O}

\newcommand{\Pic}{\textup{Pic}}
\newcommand{\uM}{\textup{M}}
\newcommand{\uT}{\textup{T}}
\newcommand{\C}{\mathbb C}

\newcommand{\isomto}{\overset \sim \to}
\newcommand{\GL}{{\textup {GL}}}

\newcommand{\ZZ}{\mathbb Z}
\newcommand{\Z}{\mathbb Z}
\newcommand{\Aut}{\textup{Aut}}

\newcommand{\Q}{\mathbb Q}

\newcommand{\Hom}{\textup{Hom}}

\newcommand{\spec}{\textnormal{Spec}}

\newcommand{\QQ}{\mathbb Q}

\newcommand{\End}{\textup{End}}

\newcommand{\Disc}{{\textup{Disc}}}

\newcommand{\lhk}{\left(}
\newcommand{\rhk}{\right)}
\newcommand{\ces}{(ES)}

\newcommand{\absg}{\underline{A}_g}
\newcommand{\hilbmod}{M}
\newcommand{\abomult}{\underline{M}}
\newcommand{\order}{\Gamma}
\newcommand{\sch}{\textnormal{(Sch)}}
\newcommand{\sets}{\textnormal{(Sets)}}
\newcommand{\SL}{\textnormal{SL}}

\newcommand{\CC}{\mathbb C}

\newcommand{\RR}{\mathbb R}
\newcommand{\cmp}{\mathcal M_{\mathcal P}}

\newcommand{\smp}{M_{\mathcal P}}

\newcommand{\disc}{\textnormal{disc}}
\newcommand{\rad}{\textnormal{rad}}
\newcommand{\cO}{\mathcal O}
\newcommand{\cI}{\mathcal I}
\newcommand{\torm} {{\overline{\mathcal M}_2}}
\newcommand{\minm} {{M_2^*}}
\newcommand{\minmq} {{M^*_{2,\mathbb Q}}}
\def\sl2{\textnormal{SL}_2}

\def\proj{\textnormal{Proj}}
\def\gl2{\textnormal{GL}_2}

\def\mgl2{M_{\textnormal{GL}_2,g}}
\def\ces{(ES)}
\newcommand {\OK}  {{\mathcal O_{K}}}

\newcommand{\DarkGreenempty}[1]{}

\definecolor{darkgreen}{rgb}{0.0, 0.4, 0.26}

\usepackage[continuous, page]{pagenote}
\makepagenote

\begin{document}
\date{}
\maketitle

\newpage

{\scriptsize
\begin{abstract}
We continue our study of integral points on moduli schemes by combining the method of Faltings (Arakelov, Par{\v{s}}in, Szpiro) with modularity results and Masser--W\"ustholz isogeny estimates. In this work we explicitly bound the height and the number of integral points on coarse Hilbert moduli schemes outside the branch locus.     

In the first part we define and study coarse Hilbert moduli schemes with their heights and branch loci. Building on the foundational works of Rapoport, Katz--Mazur, Deligne--Pappas and using the Faltings height, we first develop the basic definitions. Then we prove some geometric results for the height and the branch locus. In particular we relate the branch locus to automorphism groups of the involved moduli stack which allows us to compute the branch locus in some relevant cases. 

In part two we establish the effective Shafarevich conjecture for abelian varieties $A$  over a number field $K$ such that $A_{\bar{K}}$ has CM or $A_{\bar{K}}$ is of $\gl2$-type and isogenous to all its $\textnormal{Aut}(\bar{K}/\QQ)$-conjugates. If $A_{\bar{K}}$ has no CM, then we use isogeny estimates and results of Ribet and Wu to reduce via Weil restriction to the known case $K=\QQ$ proven by one of us via Faltings' method and modularity.  We also use isogeny estimates to reduce the CM case to effectively bounding the height $h(\Phi)$ of simple CM types in terms of the discriminant. To bound $h(\Phi)$ we follow Tsimerman's strategy and we combine the averaged Colmez conjecture with explicit analytic estimates for L-functions.

In the third part we continue our explicit study of the Par{\v{s}}in construction given by the forgetful morphism of Hilbert moduli schemes. We now work out our strategy for arbitrary number fields $K$ and we explicitly bound the number of polarizations and module structures on abelian varieties over $K$ with real multiplications. 

In the last part we illustrate our results by applying them to two classical surfaces first studied by Clebsch (1871) and Klein (1873): We explicitly bound the Weil height and the number of their integral points. Hirzebruch proved that both surfaces are models over $\CC$ of a Hilbert modular surface $Y_\CC$.
To show that they are coarse Hilbert moduli schemes over $\ZZ[\tfrac{1}{30}]$, we go into the construction of Rapoport and Faltings--Chai of the integral minimal compactification of $Y_\CC$ and we make it explicit via Hirzebruch's work. Here we use geometric results for integral Hirzebruch--Zagier divisors of Bruinier-Burgos-K\"uhn and Yang, and we compute the Fourier expansion of certain Eisenstein series via analytic results of Klingen, Siegel and Zagier. To explicitly relate the height to the Weil height, we combine Pazuki's height comparison with formulas for Hilbert theta functions due to G\"otzky, Gundlach and Lauter--Naehrig--Yang. 
\end{abstract}
}

\newpage

\tableofcontents
\vspace{0.5cm}

\newpage

\section{Introduction}

\noindent In \cite{vkkr:hms} we explicitly studied  integral points on Hilbert moduli schemes generalizing the results for moduli schemes of elliptic curves in \cite{rvk:intpointsmodell}. In this paper we continue our  study  and we consider more generally coarse Hilbert moduli schemes. In particular we combine the method of Faltings (Arakelov, Par{\v{s}}in, Szpiro) with modularity results and Masser--W\"ustholz isogeny estimates in order to explicitly bound the height and the number of integral points on coarse Hilbert moduli schemes outside the branch locus. 

\subsection{Integral points on coarse Hilbert moduli schemes}\label{sec:introintegralchms}
To provide some motivation for the explicit study of integral points on coarse Hilbert moduli schemes, we discuss two underlying Diophantine problems. 
\subsubsection{The Clebsch--Klein surfaces}\label{sec:introck}
Let $S$ be a finite set of rational primes and let $\ZZ_S^{\times}$ be the unit group of  $\ZZ_S=\ZZ[1/N_S]$ where $N_S=\prod_{p\in S}p$ with $N_S=1$ if $S$ is empty.  We first consider the cubic equation
\begin{equation}\label{eq:ck}
x_1^3+x_2^3+x_3^3+x_4^3=1=x_1+x_2+x_3+x_4,  \quad x_i\in \ZZ_S^\times.
\end{equation}
Let $\sigma_i$ be the $i$-th elementary symmetric polynomial. It turns out ($\mathsection$\ref{sec:diophaeq}) that solving \eqref{eq:ck} is equivalent to a special case of the following substantially more difficult Diophantine problem: Determine the set $Y(\ZZ_S)$ of $\ZZ_S$-points of the surface $Y/\ZZ$ given by
\begin{equation}\label{def:sigma24surface}
Y=X\setminus Z, \quad X\subset \mathbb P^4_\ZZ: \ \sigma_2=0=\sigma_4.
\end{equation}   
Here $Z\subset X$ is the union of the five $\ZZ$-points defined by the five permutations $e_i$ of $(1,0,\dotsc,0)$. More explicitly, let $\Sigma$ be the set of $x\in\ZZ^5$ with coprime coordinates such that for all rational primes $p\notin S$  the image of $x$ in $(\ZZ/p\ZZ)^5$ does not lie in some line $(\ZZ/p\ZZ )e_i$. Then Lemma~\ref{lem:solutionssigma24} identifies $Y(\ZZ_S)$ with the set of solutions (modulo $\pm1$) of  
\begin{equation}\label{eq:sigma24}
\sigma_2(x)=0=\sigma_4(x), \quad x\in\Sigma. 
\end{equation}
In the 1870s, Clebsch~\cite{clebsch:clebschsurface} and Klein~\cite{klein:cubicsurfaces} initiated the study of geometric properties of the surfaces  \eqref{eq:ck} and \eqref{def:sigma24surface}. These surfaces play a fundamental role in geometry and over the last 150 years many authors showed that they have many striking properties. For example, Hirzebruch~\cite{hirzebruch:ck} identified $
Y_\CC$ with a Hilbert modular surface for $\QQ(\sqrt{5})$.

In this paper we study the Diophantine equations \eqref{eq:ck} and \eqref{eq:sigma24}. They are non-trivial in the sense that they have infinitely many solutions over $\QQ$. Diophantine approximation leads to very strong results for \eqref{eq:ck}:  Evertse--Schlickewei--Schmidt~\cite{evscsc:uniteq} bounded  the number of solutions of \eqref{eq:ck} only in terms of $|S|$, and Corvaja--Zannier~\cite{coza:intpointssurfaces,coza:intpointscertainsurfaces,coza:intpointsdivisibility} established the degeneration of an even larger set of solutions of \eqref{eq:ck}; see \eqref{degeneration}.

An application of our main results for integral points (Theorem~\ref{thm:mainint}) with the surface $Y/\ZZ$ gives explicit finiteness results for \eqref{eq:ck} and \eqref{eq:sigma24}. To discuss these results, we consider again the set $Y(\ZZ_S)$ of $S$-integral points of $Y$ and we denote by $h$ the usual logarithmic Weil height (\cite[p.16]{bogu:diophantinegeometry}). Define $e=10^{12}$ and $c=10^e$. We obtain the following:

\vspace{0.3cm}
\noindent{\bf Corollary.}
\emph{Any $P\in Y(\ZZ_S)$ satisfies $h(P)\leq cN_S^{24}$ and it holds  $|Y(\ZZ_S)|\leq (cN_S)^e$.}
\vspace{0.3cm}

More explicitly, we deduce in Corollaries~\ref{cor:sigma24} and \ref{cor:ck} that any solution $x$ of \eqref{eq:ck} or \eqref{eq:sigma24} satisfies $h(x)\leq cN_S^{24}$ and that  \eqref{eq:ck} and \eqref{eq:sigma24} have at most $(cN_S)^{e}$ solutions. Here the height bound assures that for any given $S$ one can in principle determine all solutions of \eqref{eq:ck} and \eqref{eq:sigma24}. Moreover, if $S$ is small then our height bound is sufficiently strong for completely solving the equations in practice when combined with efficient sieves; we are currently trying to construct such efficient sieves for \eqref{eq:ck} and \eqref{eq:sigma24}. Additional discussions of the Diophantine problems \eqref{eq:ck} and \eqref{eq:sigma24} can be found in Sections~\ref{sec:ckmainresults} and \ref{sec:diophaeq}.


\subsubsection{Coarse Hilbert moduli schemes}\label{sec:chmsintro}
Let $L/\QQ$ be a totally real number field of degree $g$ with ring of integers $\OL$. Before we give the formal definition, we try to briefly explain and motivate the notion of coarse Hilbert moduli schemes. Intuitively one can think of a coarse Hilbert moduli scheme $Y$ as a variety whose geometric points parametrize pairs $(A,\alpha)$, where $A$ is a `polarized' abelian variety of dimension $g$ with $\OL$-multiplication and $\alpha$ lies in a certain set $\mathcal P(A)$ associated to $A$. If here in addition all points of $Y$ parametrize pairs $(A,\alpha)$ then $Y$ is a Hilbert moduli scheme. For instance, for any $n\in \ZZ_{
\geq 1}$ consider the set $\mathcal P(n)(A)=\{(\OL/n\OL)^2\cong A_n\}$ of principal level $n$-structures on $A$. If $n\geq 3$ then there is a Hilbert moduli  scheme $Y(n)$ for $\mathcal P(n)$ such that a connected component of $Y(n)(\CC)$ identifies with $\mathbb H^g/\Gamma(n)$, where the group $\Gamma(n)=\ker\bigl(\SL_2(\OL)\to \SL_2(\OL/n\OL)\bigl)$ acts on $\mathbb H^g$ via the $g$ distinct real embeddings of $L$ and via the action of $\SL_2(\RR)$ on $\mathbb H=\{z\in\CC; \, \textnormal{im}(z)>0\}$. In the cases $n=1$ and $n=2$, there is no Hilbert moduli scheme for $\mathcal P(n)$ but there still exists a coarse Hilbert moduli scheme $Y(n)$. The Diophantine equations \eqref{eq:ck} and \eqref{eq:sigma24} come from the coarse Hilbert moduli scheme $Y(2)$ for $L=\QQ(\sqrt{5})$. 
In fact the literature contains many more interesting examples of varieties which are (birationally equivalent to) coarse Hilbert moduli schemes: For example one can find in \cite{hiva:hmsclass,hiza:hmsclass,vandergeer:hilbertmodular}  many surfaces of general type but also various surfaces which are rational, blown-up K3, or elliptic over $\mathbb P^1$. Moreover, as demonstrated in \cite{rvk:intpointsmodell} for $g=1$, the moduli  formalism allows to explicitly study  a class of Diophantine equations of interest that is substantially more general  than just defining equations of $\mathbb H^g/\Gamma$ for finite index subgroups $\Gamma\subseteq \SL_2(\OL)$. This motivates to work with the following general notions developed in detail in Section \ref{sec:chmsdefs}. 

\paragraph{Definitions.}  We build on the formalism of moduli problems of Katz--Mazur~\cite{kama:moduli} which is very suitable for our purpose.  Let $\mathcal M$ be the Hilbert moduli stack associated to $L$ by Rapoport~\cite{rapoport:hilbertmodular} and Deligne--Pappas~\cite{depa:hilbertmodular}, and let $\mathcal P$ be a presheaf on $\mathcal M$. We call $\mathcal P$ a moduli problem. Further we say that $\mathcal P$ is algebraic if the category $\cmp$ of pairs $(x,\alpha)$ with $x\in \mathcal M$ and $\alpha\in\mathcal P(x)$ is an algebraic stack (\cite{sp}). Let $Y$ be a scheme. We say that $Y$ is a coarse Hilbert moduli scheme of an algebraic moduli problem $\mathcal P$ on $\mathcal M$, and we write $Y=M_{\mathcal P}$, if there exists a coarse moduli space $\pi:\cmp\to Y$ in the usual sense ($\mathsection$\ref{sec:coarsehmsdef}). The branch locus $B_{\mathcal P}\subseteq Y$ of $\mathcal P$ is defined as the complement of the union of all open subschemes $U\subseteq Y$ such that $\pi_U$ is \'etale.  Further we say that  $Y$ is a Hilbert moduli scheme of $\mathcal P$ if it represents $\cmp$. A discussion of the notion of Hilbert moduli schemes can be found in \cite[$\mathsection$3]{vkkr:hms}. If $Y$ is a Hilbert moduli scheme of $\mathcal P$, then $\mathcal P$ is algebraic and $Y$ is a coarse Hilbert moduli scheme of $\mathcal P$ with empty branch locus $B_{\mathcal P}$. Finally, we say that $\mathcal P$ is arithmetic if $\cmp$ is a separated finite type DM stack~(\cite{sp}) and we say that a scheme $Y$ is variety over $\ZZ$ if it is separated finite type over $\ZZ$. 

\subsubsection{Main result for integral points}

We continue our notation. Let $Y$ be a variety over $\ZZ$, let $Z\subseteq Y$ be a closed subscheme and let $\Delta$ be the absolute value of the discriminant of the totally real number field $L$ of degree $g$. Our main result  for integral points (Theorem~\ref{thm:mainint}) is a slightly improved version of the following theorem in which $e=4^{8g}$ and $c=9^{9^{9g}}$.

\vspace{0.3cm}
\noindent{\bf Theorem A.}
\emph{Suppose that $Y$ becomes over a ring $\ZZ[1/\nu]$, $\nu\in \ZZ_{\geq 1}$, a coarse moduli scheme of some arithmetic moduli problem $\mathcal P$ on $\mathcal M$ with branch locus $B_{\mathcal P}\subseteq Z_{\ZZ[1/\nu]}$.
\begin{itemize}
\item[(i)] Then any point $P\in (Y\setminus Z)(\ZZ_S)$ satisfies $h_\phi(P)\leq c(\nu N_S)^{24g},$
\item[(ii)] and the cardinality of $(Y\setminus Z)(\ZZ_S)$ is at most $c|\mathcal P|_{\CC}(\nu N_{S})^{e}\Delta\log(3\Delta)^{2g-1}$.
\end{itemize}}
\vspace{0.2cm}
Here $h_\phi: Y(\bar{\QQ})\to \mathbb R$ denotes the height ($\mathsection$\ref{sec:heightdef}) given by the pullback of the stable Faltings height \cite{faltings:finiteness} along the forgetful morphism $\phi$ induced by forgetting $\mathcal P$-level structures, and  $|\mathcal P|_\CC=\sup_{x\in \mathcal M(\CC)}|\mathcal P(x)|$ denotes the maximal number of $\mathcal P$-level structures over $\CC$ defined in \eqref{def:maxlvl}. We now discuss various aspects of Theorem~A.

\paragraph{Finiteness.}If $|\mathcal P|_\CC<\infty$ then (ii) gives that $(Y\setminus Z)(\ZZ_S)$ is finite. This finiteness result is essentially due to Faltings~\cite{faltings:finiteness} as explained in \eqref{eq:finitenessviafalt}. To this end, for large classes of representable moduli schemes $Y$ of abelian varieties,  Deligne--Szpiro~\cite[$\mathsection$5.3]{szpiro:faltings} and Ullmo~\cite{ullmo:ratpoints} deduced from Faltings~\cite{faltings:finiteness} the finiteness of $S$-integral points  and also of rational points if $Y$ is projective. For example, Ullmo deduced this for subvarieties of adjoint Shimura varieties of abelian type with neat level structure (\cite[Thm 3.2]{ullmo:ratpoints}); for ball quotients see also Dimitrov--Ramakrishnan~\cite{dira:mordellic} using Mordell--Lang~\cite{faltings:ml}. 

\paragraph{Known effective bounds.} For many coarse modular curves $Y$ with $Y(\CC)=\mathbb H/\Gamma$ for some $\Gamma\subset\textnormal{SL}_2(\ZZ)$, Bilu~\cite{bilu:israel,bilu:modular}, Sha~\cite{sha:intmodimrn,sha:intmod2} and Cai~\cite{cai:intx0p,cai:int} used the theory of logarithmic forms to prove strong explicit height bounds for all points with $S$-integral $j$-invariant. Suppose  now that $\mathcal P$ is representable. Then Theorem~A was proven in \cite{vkkr:hms}, and in \cite{rvk:intpointsmodell} when $g=1$. The case $g=1$ contains various classical Diophantine problems \cite[$\mathsection$4.2.6]{vkma:computation}, including the case of Mordell equations obtained in \cite[Cor 7.4]{rvk:intpointsmodell} and the case of $S$-unit equations (see Frey~\cite{frey:ternary}) obtained independently and simultaneously by K.~\cite[Cor 7.2]{rvk:intpointsmodell} and by Murty--Pasten~\cite[Thm 1.1]{mupa:modular}. For all these classical Diophantine problems and many others, effective bounds were first established via  Diophantine approximations or the theory of logarithmic forms; see \cite{bawu:logarithmicforms,evgy:bookuniteq,evgy:bookdiscreq} for an overview. Moreover, the latter methods usually can deal with any number field $K$, while substantially new ideas (see $\mathsection$\ref{sec:introdgeneralizations}) are required to generalize Theorem~A to any $K$.


\paragraph{Branch locus.}For both statements of Theorem~A it is necessary to restrict to $\ZZ_S$-points of $Y$ outside the branch locus $B_{\mathcal P}$. Indeed there are many coarse Hilbert moduli schemes $Y$ of arithmetic $\mathcal P$ with $|\mathcal P|_\CC<\infty$ and $|Y(\ZZ_S)|=\infty$. For example, if $L=\QQ$ and $\mathcal P=\{*\}$ is constant, then the coarse Hilbert moduli scheme $Y=\mathbb A^1_\ZZ$ satisfies $Y(\ZZ_S)=\ZZ_S$. 

Motivated by Theorem~A, we study in Section~\ref{sec:branchauto} the branch locus $B_{\mathcal P}$ of coarse Hilbert moduli schemes. We prove Proposition~\ref{prop:auto} which relates $B_{\mathcal P}$ to jumps of the sizes of the automorphism groups of the underlying stack $\cmp$. This result allows to compute the branch locus in certain cases of interest. For instance, we deduce in Corollary~\ref{cor:y2emptybranchlocus} that the branch locus of $Y(2)$ is empty  and thus Theorem~A holds for all $\ZZ_S$-points of $Y(2)$.

\paragraph{Applications.}
As illustrated in the Corollary, Theorem~A~(i) can provide (after a height comparison) an explicit bound for the Weil height which in principle allows to solve the underlying Diophantine equations. See also the analogous discussion in \cite[$\mathsection$1.1]{vkkr:hms}. 

Theorem~A applies to a class of varieties which is substantially more general than just integral models of complex Hilbert modular varieties $\mathbb H^g/\Gamma$ with $\Gamma\subseteq\SL_2(\OL)$ of finite index. Also, the class of (coarse) Hilbert moduli schemes is stable under covers.  For example, let $X$ be a variety over $\ZZ$ which admits a morphism of schemes 
\begin{equation}\label{eq:finitecover}
f:X\to Y
\end{equation}
to a Hilbert moduli scheme $Y=M_{\mathcal P}$ with height $h_\phi$.  Then $X=M_{\mathcal P'}$ is again a Hilbert moduli scheme with height $h_{\phi'}=f^*h_\phi$,  and it holds that $|\mathcal P'|_{\CC}<\infty$ if $f$ is quasi-finite and $|\mathcal P|_{\CC}<\infty$.   
The construction \eqref{eq:finitecover} was for instance used in \cite[$\mathsection$7.4]{rvk:intpointsmodell} to study cubic Thue equations. Other explicit (relative) curves $X$ are discussed in \cite[$\mathsection$4.2.6]{vkma:computation}, and we hope to find in addition some interesting higher dimensional $X$. 

Theorem~A also applies to $\QQ$-rational points if $X_\QQ$ is proper and thus $X(\QQ)=X(\ZZ_S)$ for some $\ZZ_S$.  In this context, the problem of constructing proper curves $X_\QQ$ with a quasi-finite morphism \eqref{eq:finitecover} was investigated by Alp\"oge in his work \cite{alpoge:thesis,alpoge:modularitymordell} on algorithms for rational points of curves. In particular he observed that for higher dimensional Hilbert modular varieties $Y_\QQ$ there are many such proper curves $X_\QQ$; this was also observed independently by Baldi--Grossi~\cite[$\mathsection$1.2]{bagr:obstructionhmv} in the context of existence of rational points.

\paragraph{Bounds.} We conducted some effort (see $\mathsection$\ref{sec:shapeofbounds}) to obtain bounds in Theorem~A which are polynomial in terms of $\nu N_S$ and to include optimizations which exponentially improved the constant $c$ and the exponents $e$ and $24g$. However, we did not try to fully optimize the constants which we simplified according to Baker's philosophy.  If $\nu N_S$ has reasonable size then our simplified bounds are already strong enough for solving equations when combined with efficient sieves such as for example those in \cite{deweger:lllred,vkma:computation,ghsi:tm}. On the other hand, for large $\nu N_S$ it is often beneficial for solving equations to  work out optimized bounds which might be complicated but can be computed instantly; we leave this for the future. Theorem~A crucially depends on a relevant special case (Theorem~B) of the effective Shafarevich conjecture for abelian varieties over number fields. After discussing Theorem~B in the next section, we give an outline of the proof of Theorem~A in $\mathsection$\ref{sec:proofintro} which contains a discussion  ($\mathsection$\ref{sec:shapeofbounds}) of the shape of our bounds.


\subsection{Effective Shafarevich conjecture}\label{sec:esintro}

Let $K$ be a number field with ring of integers $\OK$,  let $S$ be a nonempty open subscheme of $\spec(\OK)$ and let $g\in \ZZ_{\geq 1}$. Faltings~\cite{faltings:finiteness} proved the Shafarevich conjecture which asserts that the set of isomorphism classes of (polarized) abelian schemes over $S$ of relative dimension $g$ is finite. Let $h_F$ be the stable Faltings height ($\mathsection$\ref{sec:heightdef}) 
introduced by Faltings in \cite{faltings:finiteness}. It is known that effective versions of the Shafarevich conjecture would have many striking Diophantine applications. For example, the following Conjecture $\ces$ implies via Kodaira's construction an effective version of the Mordell conjecture; see  R\'emond~\cite{remond:construction}. 

\vspace{0.3cm}
\noindent{\bf Conjecture~$\ces$.}
\emph{There exists an effective constant $c$, depending only on $K$, $S$ and $g$, such that any abelian scheme $A$  over $S$ of relative dimension $g$ satisfies $h_F(A)\leq c.$}
\vspace{0.3cm}

\noindent The case $g=1$ of this conjecture was proven in \cite{coates:shafarevich,fuvkwu:elliptic,jore:shafarevich} via the theory of logarithmic forms~\cite{bawu:logarithmicforms}. In the case $g\geq 2$ the following is known: (1) when $K=\QQ$ and $A$ is of (product) $\gl2$-type, and (2) when $A_{\bar{K}}$ is a simple non-CM abelian variety of $\gl2$-type with $G_\QQ$-isogenies as in \eqref{def:Gisogintro} below.   Here (1) was established in 2013  by the first named author \cite{rvk:gl2} by combining Faltings' method \cite{faltings:finiteness} with modularity and isogeny estimates, while (2) was obtained  in \cite[Prop 9.9]{rvk:modular} by reducing to (1). Also, building on  \cite{rvk:gl2}, Alp\"oge~\cite{alpoge:modularitymordell} obtained an algorithmic version of Conjecture~$\ces$ in the case when $A$ is of $\gl2$-type and $K$ is totally real of odd degree; see $\mathsection$\ref{sec:introdgeneralizations}.




\paragraph{$\gl2$-type with $G_\QQ$-isogenies.} We now define a class of abelian schemes with the following properties: In the case $K=\QQ$ they are the usual abelian schemes over $S$ of $\gl2$-type (\cite{ribet:gl2}), and in the case $g=1$ they are the elliptic curves $E$ over $S$ with geometric generic fiber $E_{\bar{K}}$ isogenous to all its $G_\QQ$-conjugates; such elliptic curves $E_{\bar{K}}$ were studied among others by Gross~\cite{gross:cmell}, Ribet~\cite{ribet:gl2} and Elkies~\cite{elkies:kcurves} who called them $\QQ$-curves. 

Let $A$ be an abelian scheme over $S$ of relative dimension $g$ and denote by $\End(A)$ the ring of $S$-group scheme morphisms $A\to A$. We say that $A$ is of $\gl2$-type with $G_\QQ$-isogenies if the $\QQ$-algebra $\End^0(A)=\End(A)\otimes_\ZZ\QQ$ contains a number field $F$ of degree $g$ and if for each $\sigma$ in $G_\QQ=\Aut(\bar{K}/\QQ)$ there is an isogeny $\mu_\sigma:\sigma^*A_{\bar{K}}\to A_{\bar{K}}$ such that 
\begin{equation}\label{def:Gisogintro}
\mu_\sigma \circ \sigma^*(f)=f \circ \mu_\sigma, \quad f\in F. 
\end{equation}
Here $\sigma^*$ denotes the pullback via $\sigma$ and we identified $F$ with its image in the endomorphism algebra of  the base change $A_{\bar{K}}$ of $A$ to an algebraic closure $\bar{K}$ of the function field $K$ of $S$.  As explained in $\mathsection$\ref{sec:Gisogdef}, the notion of $\gl2$-type with $G_\QQ$-isogenies generalizes various definitions in the literature such as for example Ribet's notion of $\QQ$-HBAV in \cite{ribet:fieldsofdef}.

More generally, we say that $A$ is of product $\gl2$-type with $G_\QQ$-isogenies  if $A$ is isogenous to a product $\prod A_i$ of abelian schemes $A_i$ over $S$ of $\gl2$-type with $G_\QQ$-isogenies. 

\paragraph{Results.}To measure $K$ and $S$ we use the quantities $d=[K:\QQ]$, $D_K=|\textnormal{Disc}{(K/\QQ)}|$ and $N_S=\prod_{v\in S^c} N_v$, where $N_v$ is the norm of the prime ideal $v$ of $\OK$ and $S^c$ is the complement of $S$. The first part of the next theorem establishes the effective Shafarevich Conjecture~$\ces$ for abelian schemes of product $\gl2$-type with $G_\QQ$-isogenies.

\vspace{0.3cm}
\noindent{\bf Theorem B.}
\emph{
The following statements hold with $l=(d3^{4g^2})!$.
\begin{itemize}
\item[(i)]Let $A$ be an abelian scheme over $S$ of relative dimension $g$. If $A$ is of product $\gl2$-type with $G_\QQ$-isogenies, then  $h_F(A)\leq (4gl)^{144gl}\textnormal{rad}(D_KN_S)^{24g}.$
\item[(ii)] The number of isomorphism classes of abelian schemes over $S$ of relative dimension $g$ and of product $\gl2$-type with $G_\QQ$-isogenies is at most $ (4gl)^{(5g)^7l}\textnormal{rad}(D_KN_S)^{(5g)^7}.$
\end{itemize}
}
Here $\rad(n)=\prod_{p\mid n}n$ denotes the radical of $n\in\ZZ_{\geq 1}$. As explained in $\mathsection$\ref{sec:esproofdiscussion} below, statements (i) and (ii) are essentially \cite[Thm A and Thm B]{rvk:gl2} respectively when restricted to the case $K=\QQ$. In Section~\ref{sec:es} we also obtain other results in the context of Conjecture~$\ces$. In particular we prove Conjecture~$\ces$ in the CM case and we give more precise versions of Theorem~B in certain situations of interest.  Discussions of various aspects of our bounds in (i) and (ii) can be found in $\mathsection$\ref{sec:shapeofbounds} below.

\subsubsection{Ideas of the proofs}\label{sec:esproofdiscussion}
We continue our notation introduced above. Let $A$ be an abelian scheme over an open $S\subseteq \spec(\OK)$ such that $A$ is of $\gl2$-type with $G_\QQ$-isogenies.
\paragraph{Strategy.} Building on \cite{rvk:gl2}, the main idea of the proof of Conjecture~$\ces$ for $A$ is to combine Faltings' method~\cite{faltings:finiteness} with modularity as follows. By the known effective isogeny estimates ($\mathsection$\ref{sec:hfprop}) based on the method of Faltings \cite{faltings:finiteness} or Masser--W\"ustholz~\cite{mawu:periods,mawu:abelianisogenies}, it suffices to show that each isogeny class contains an abelian scheme with effectively bounded height. Then one constructs such an abelian scheme by combining inter alia the Tate conjecture~\cite{faltings:finiteness}, modularity results and effective analytic estimates. We now describe more precisely the steps carried out in the current work.

\paragraph{Key case.}In the first part of the proof we assume that $A_{\bar{K}}$ has no CM. After possibly enlarging $K$ in a controlled way, we can reduce in Proposition~\ref{prop:Gisog} the problem to the case $S\subseteq \spec(\ZZ)$ as follows. On using and generalizing arguments of Ribet~\cite{ribet:gl2,ribet:fieldsofdef}, Pyle~\cite{pyle:gl2} and Wu~\cite{wu:virtualgl2}, we show  that $A$ is isogenous to a power of an abelian scheme $B$ over $S$ of $\gl2$-type with $G_\QQ$-isogenies $\mu_\sigma$ satisfying $\mu_\sigma\circ \sigma^*(f)=f\circ \mu_\sigma$ for all $f\in \End^0(B_{\bar{K}})$. Then we go into the proof of \cite[Prop 9.9]{rvk:modular}: A result of Ribet~\cite{ribet:gl2} and Wu~\cite{wu:virtualgl2}, which relies on a Galois cohomology computation of Tate, gives that $B$ is geometrically a factor of an abelian variety $C_\QQ$ over $\QQ$ of $\gl2$-type. Moreover,   we show that $C_\QQ$ extends here to an abelian scheme $C$ over some explicit open $T\subseteq \spec(\ZZ)$. Now, over $T$ we can apply the case of Conjecture~$\ces$ proven in \cite[Thm A]{rvk:gl2} via the above described strategy using Serre's modularity conjecture~\cite{khwi:serre}. This explicitly bounds $h_F(C)$ and then isogeny estimates allow us to deduce Conjecture~$\ces$ for $A$.

\paragraph{Abelian varieties with CM.}In the second part we prove Conjecture~$\ces$ for abelian varieties with CM. First, we use isogeny estimates and a result of Colmez \cite{colmez:conjecture} to reduce the problem to the case of abelian varieties $A_\Phi$ associated to some CM type $\Phi$ of a CM field $E$. Moreover, after possibly replacing $K$ by a controlled field extension, we can use CM theory to relate the ramification loci of $E/\QQ$ and $K/\QQ$. This allows then to further reduce the CM case of Conjecture~$\ces$ to effectively bounding $h_F(A_\Phi)$ in terms of $D_E$.  The height $h_F(A_\Phi)$ is related to certain L-values associated to $E$ by the  averaged Colmez conjecture \cite{colmez:conjecture}. This conjecture was established independently and simultaneously by Yuan--Zhang~\cite{yuzh:avcolmez} and by  Andreatta--Goren--Howard--Madapusi-Pera~\cite{aghm:avcolmez}. Tsimerman~\cite[Cor 3.3]{tsimerman:ao} deduced from the averaged Colmez conjecture the asymptotic bound $h_F(A_\Phi)\leq D_E^{o_g(1)}$; see also the discussions in $\mathsection$\ref{sec:avcolmez}. On working out explicitly (via Rademacher's arguments~\cite{rademacher:lindeloef}) some of the analytic estimates for L-values used in \cite[Cor 3.3]{tsimerman:ao}, we deduce from the averaged Colmez conjecture a much weaker but fully explicit bound for $h_F(A_\Phi)$ in terms of $D_E$ and $g$ which is sufficient for our purpose.   
 
\paragraph{Number of isomorphism classes.}To prove the bound in Theorem~B~(ii) for the number of isomorphism classes, we follow Faltings~\cite{faltings:finiteness}. After passing to a controlled finite field extension $L$ of $K$, we divide the proof into two parts: (a) finiteness of the number of isogeny classes and (b) finiteness of each isogeny class. Part (a) is done as follows: We apply basic CM theory in the CM case, while in the non-CM case we use the above described arguments to reduce again to the situation $S\subseteq \spec(\ZZ)$ and then we apply \cite[Thm B]{rvk:gl2}. To deal with part (b), we combine as in the proof of \cite[Thm B]{rvk:gl2} the height bound in Theorem A with
the minimal isogeny degree estimate of Masser--W\"ustholz~\cite{mawu:abelianisogenies} based on the theory of logarithmic forms. In fact throughout this work we use the most recent versions, due to Gaudron--R\'emond~\cite{gare:newisogenies}, of the isogeny and endomorphism results of Masser--W\"ustholz~\cite{mawu:periods,mawu:abelianisogenies,mawu:factorization}.  Finally we count via an explicit version of the Hermite--Minkowski theorem the number of possible field extensions $L$ of $K$ and we explicitly control  the number of (unpolarized) twists over $L$ by using a trick (Lemma~\ref{lem:counttwists}) involving Weil restriction and isogeny estimates.

\subsection{Proof of Theorem~A}\label{sec:proofintro}
We continue the notation of $\mathsection$\ref{sec:introintegralchms}. Suppose that $Z\subseteq Y$ and $\mathcal P$ are as in Theorem A, and put $U=Y\setminus Z$. To prove Theorem~A, we use and generalize the strategy of \cite{rvk:intpointsmodell,vkkr:hms} and we apply the method of Faltings~\cite{faltings:finiteness}~(Arakelov, Par{\v{s}}in, Szpiro) as follows.  
\begin{itemize}
\item[(a)] Par{\v{s}}in construction:   We construct a finite map $\phi:U(\ZZ_S)\to \absg(T)$ for some controlled open $T\subset \spec(\OL_K)$ with $K$ a number field, where $\absg$ classifies isomorphism classes of  abelian schemes of relative dimension $g$. Moreover, we compute that $h_\phi$ is the pullback $\phi^*h_F$ of $h_F$ by $\phi$, we show that each abelian scheme in $\phi(U(\ZZ_S))$ is of $\gl2$-type with $G_\QQ$-isogenies and we explicitly bound the fibers of $\phi$. 
\item[(b)] Effective Shafarevich conjecture: As the abelian schemes in $\phi(U(\ZZ_S))$ are all of $\gl2$-type with $G_\QQ$-isogenies by part (a), we can apply Theorem~B. This gives us explicit bounds for the cardinality $|\phi(U(\ZZ_S))|$ and for the height $h_F(A)$ of all $A\in \phi(U(\ZZ_S))$ in terms of $K$, $T$ and $g$ and then in terms of $\nu$, $N_S$ and $g$.  
\end{itemize}
Combining (a) and (b) leads to the estimate for $h_\phi=\phi^*h_F$ in (i) and to the bound for $|U(\ZZ_S)|$ in (ii) since $U(\ZZ_S)=\phi^{-1}\bigl(\phi(U(\ZZ_S))\bigl)$.  We discussed in $\mathsection$\ref{sec:esintro} the main ingredients for Theorem~B used in part (b). Now, we explain the principal ideas of part (a).

\subsubsection{Construction of $\phi$ and basic properties}\label{sec:phiconstructionintro}
We continue our notation. Unfortunately the forgetful morphism of a coarse Hilbert moduli scheme is only defined on geometric points. To obtain a map $\phi:U(\ZZ_S)\to \absg(T)$ with the desired properties for $\ZZ_S$-points, we work over the base $\ZZ_S$ (after possibly enlarging $S$ in a controlled way) and we study in  Section~\ref{sec:coverconstruction} the finite cover $Y'\to Y$ given by
$$Y'=\cmp\times_{\mathcal M}\mathcal M_{\mathcal P(n)}\to \cmp\to^\pi Y. 
$$
We show that $Y'$ has the following key properties: It is a Hilbert moduli scheme of the product presheaf $\mathcal P'=\mathcal P\times \mathcal P(n)$ on $\mathcal M$ for a suitable $n\geq 3$ and the pull back of $Y'\to Y$ via $P\in U(\ZZ_S)$ is a finite \'etale cover of $\spec(\ZZ_S)$ of bounded degree.  This allows us to construct a controlled open $T\subseteq \spec(\OK)$ with $K$ a number field and a map $$\phi:U(\ZZ_S)\to \absg(T)$$ factoring through the forgetful morphism $\phi':Y'(T)\to \absg(T)$ of the Hilbert moduli scheme $Y'$. Moreover, we show that $\phi$ extends on $\bar{\QQ}$-points to the forgetful morphism
$Y(\bar{\QQ})\to  \absg(\bar{\QQ})$ of the coarse moduli scheme $Y=M_{\mathcal P}$.
Then it is a formal consequence of the construction of $\phi$ and geometric properties of the stable Faltings height $h_F$ that 
$$h_\phi=\phi^*h_F.$$ 
Similarly, it is a formal consequence of the construction of $\phi$ that each abelian scheme in $\phi(U(\ZZ_S))$ is of $\gl2$-type with $G_\QQ$-isogenies. Here the key observation is that $U(\ZZ_S)$ is contained in $Y(\QQ)$ which are the points of $Y(\bar{\QQ})\cong \cmp(\bar{\QQ})$ fixed by all $\sigma\in G_\QQ$ and hence there exist $\OL$-compatible isomorphisms $\sigma^*A_{\bar{\QQ}}\isomto A_{\bar{\QQ}}$ where $A=\phi(P)$ and $P\in U(\ZZ_S)$.

Further, we show that to control the fibers of $\phi$ it suffices to suitably bound the degree of the forgetful morphism $\phi':Y'(T)\to \absg(T)$ of the Hilbert moduli scheme $Y'$ of $\mathcal P'$.
\subsubsection{Bounding the degree of the forgetful morphism}Let $Y$ be a Hilbert moduli scheme of an arithmetic moduli problem $\mathcal P$ on $\mathcal M$, where $\mathcal M$ is the Hilbert moduli stack associated to a totally real field $L$ of degree $g$ with ring of integers $\OL$. Let $K$ be a number field and let $T\subseteq\spec(\OK)$ be open. In \cite{vkkr:hms} we developed an approach to bound the degree of the forgetful morphism $\phi:Y(T)\to \absg(T)$. The key steps of this approach (outlined in \cite[$\mathsection$1.2]{vkkr:hms}) go through for arbitrary $K$. However, in some steps substantial new ideas are required to deal with problems not appearing in the case $K=\QQ$ worked out in \cite{vkkr:hms}. To explain this, we recall the decomposition 
$$
\phi:Y(T)\to^{\phi_\alpha} \hilbmod(T)\to^{\phi_\varphi} \abomult(T)\to^{\phi_\iota} \absg(T)
$$
into forgetful morphisms.
Here $\hilbmod(T)$ and $\abomult(T)$ are the sets of isomorphism classes of triples $(A,\iota,\varphi)$  and pairs $(A,\iota)$ respectively, where $[A]$ lies in $\absg(T)$, $\iota:\OL\to \End(A)$ is a ring morphism and $\varphi$ is a polarization ($\mathsection$\ref{sec:hilbmodstackdef}) of $(A,\iota)$. To bound $\deg(\phi)$  it suffices to bound the degrees of $\phi_\alpha$, $\phi_\varphi$ and $\phi_\iota$, where we define $\deg(f)=\sup_{x\in X}|f^{-1}(x)|$ for any map of sets $f:X'\to X$. In \cite[Lem 8.2]{vkkr:hms} we proved  that $\deg(\phi_\alpha)\leq |\mathcal P|_\CC$. 

\paragraph{The map $\phi_\varphi$.}  The fiber of the forgetful morphism $\phi_\varphi:\hilbmod(T)\to\abomult(T)$ over $(A,\iota)$ in $\abomult(T)$ identifies with the set $\textnormal{Pol}(A,\iota)$ of equivalence classes of polarizations on $(A,\iota)$. Thus $\deg(\phi_\varphi)$ is bounded by the following result, proven in \cite[Prop 9.1]{vkkr:hms} when $K=\QQ$.

\vspace{0.3cm}
\noindent{\bf Proposition C.}
\emph{
If $A$ is an abelian scheme over $T$ of relative dimension $g$ with a ring morphism $\iota:\OL\to\End(A)$, then $|\textnormal{Pol}(A,\iota)|\leq 8^{g}$.}
\vspace{0.3cm}

\noindent  We shall obtain in Proposition~\ref{prop:polbound} a more general result with a sharper bound. To prove Proposition~\ref{prop:polbound} in Section~\ref{sec:polarizations}, we use the approach of \cite[$\mathsection$9]{vkkr:hms} where we worked out the case $K=\QQ$. Unfortunately  a crucial Lie algebra argument  breaks down for arbitrary $K$ since Lie$(A)$ can become too large. However, we can still reduce the problem to controlling a quotient $\Gamma^\times/U$ where $U$ is a subgroup of the units $\Gamma^\times$ of a certain order $\Gamma$ in a field $F\supseteq L$ of degree at most $2g$. This reduction (see Lemma~\ref{lem:polembedding})  combines the formal results of \cite[$\mathsection$9]{vkkr:hms} with the computation of the endomorphism algebra of $(A,\iota)$ in Lemma~\ref{lem:divalgcompchar0} which uses $H_1(A(\CC),\QQ)$ instead of Lie$(A)$. If $K\neq \QQ$ then $F/L$ is not always totally real and the discriminant $D_{F/L}$ can be unbounded. Hence additional arguments are required to control $|\Gamma^\times/U|$. In $\mathsection$\ref{sec:polda}, we use algebraic number theory to compute and relate various unit groups in extensions of number fields. This leads to an estimate for $|\Gamma^\times/U|$ in terms of a regulator ratio $\tfrac{R_F}{R_{F/L}R_L}$ which is uniformly bounded  even when $D_{F/L}$ is unbounded.

\paragraph{The map $\phi_\iota$.} The map $\phi_\iota:\abomult(T)\to \absg(T)$ is induced by forgetting the ring morphism $\iota:\OL\to\End(A)$. As the fiber of $\phi_\iota$ over some $A$ in $\absg(T)$ identifies with the set of isomorphism classes of $\OL$-module structures on $A$, Theorem~D\,(i) below gives $$\deg(\phi_\iota)\leq c(d\|h_F\|)^{e}\Delta\log(3\Delta)^{2g-1}.$$ Here $c=c(g)$ and $e=e(g)$ are as in Theorem~D, $d=[K:\QQ]$, $\Delta=|\Disc(L/\QQ)|$, and $\|h_F\|$ is the supremum of the function $\max(1,h_F)$ on $\absg(T)$ which satisfies $\|h_F\|<\infty$ by Faltings~\cite[Thm 6]{faltings:finiteness}. As already discussed in Section~\ref{sec:esintro},  it is currently out of reach to explicitly bound $\|h_F\|$ only in terms of $K$, $T$ and $g$. However, in the situation ($\mathsection$\ref{sec:phiconstructionintro}) appearing in the proof of Theorem~A, we can apply here (see Theorem~D\,(ii) below)  the bound for $h_F$ in Theorem~B since the abelian schemes in $\phi(U(\ZZ_S))$ are all of $\gl2$-type with $G_\QQ$-isogenies. This together with the above results for $\phi_\alpha$ and $\phi_\varphi$ lead to an explicit bound for the degree of the Par{\v{s}}in construction $\phi:U(\ZZ_S)\to \absg(T)$ in $\mathsection$\ref{sec:phiconstructionintro}.

\paragraph{Bounding $\OL$-structures.}Let $\OL$ be the ring of integers of an arbitrary totally real field $L$ of degree $g$ and let $A$ be an abelian scheme over $S$ of relative dimension $g$, where $S$ is any connected Dedekind scheme whose function field is a number field $K$.  In Section~\ref{sec:endo} we continue our study of $\OL$-module structures on $A$ initiated in \cite[$\mathsection$10]{vkkr:hms}. In particular we obtain the following result which holds more generally for any order $\Gamma$ of $L$ by Theorem~\ref{thm:endobound}.

 \vspace{0.3cm}
\noindent{\bf Theorem D.}
\emph{The following statements hold with $d=[K:\QQ]$ and $\Delta=|\Disc(L/\QQ)|$.
\begin{itemize}
\item[(i)] The number $n_A$ of isomorphism classes of $\OL$-module structures on $A$ satisfies  $$n_A\leq c\cdot h(A)^{e}\Delta\log(3\Delta)^{2g-1}$$
for $h(A)=d\max(h_F(A),1)$, $e=(2g)^9$ and $c=(3g)^{(3g)^{11}}$.
\item[(ii)]If $S\subseteq \spec(\OK)$ is open and $A$ is of product $\gl2$-type with $G_\QQ$-isogenies, then $$n_A \leq k\cdot\textnormal{rad}(N_SD_K)^{24ge}\Delta\log(3\Delta)^{2g-1}.$$
Here $N_S$ and $D_K$ are as in Theorem~B, and $k=3^{3^a}$ for $a=7^{4g^2}d^2$.
\end{itemize}
}
\noindent In \cite[Thm 10.1]{vkkr:hms} we proved Theorem~D when $K=\QQ$, and the Jordan--Zassenhaus theorem (JZ) leads to a finiteness result (\cite[Lem 10.13]{vkkr:hms}) which is more general than Theorem~D but not explicit. Theorem~B and (i) directly imply (ii).  To show (i) we use the approach developed in \cite[$\mathsection$10]{vkkr:hms}. It avoids (JZ) and has two parts: A reduction to the case when $R=\End(A)$ has a special form, and a bound for such special $R$. After some modifications, the reduction in \cite{vkkr:hms} works for arbitrary $K$. It consists of several steps and  uses the Masser--W\"ustholz isogeny theorem~\cite{mawu:abelianisogenies} based on transcendence.  This allows us to reduce to the situation when $\OL$ is replaced by an order $\Gamma$ of $L$ and when $R=\uM_n(\OL_D)$ with $\OL_D$ a maximal order of $D=\OL_D\otimes_\ZZ\QQ$ such that either 
\begin{itemize}
\item[(1)] the $\QQ$-algebra $D$ identifies with a subfield $Z$ of $L$ of relative degree $[L:Z]=n$,
\item[(2)] or $D$ is a non-split quaternion algebra and the center of $D$ identifies with a subfield $Z$ of $L$ of relative degree $[L:Z]=2n$,
\item[(3)] or $D$ is a CM field of degree $[D:\QQ]=2g/n$.
\end{itemize}
In \cite{vkkr:hms} we settled the case (1). However (2) and (3) both require substantial new ideas since a crucial dimension argument  breaks down. We refer to $\mathsection$\ref{sec:quatcase} and $\mathsection$\ref{sec:FirstCMcase} for an outline of our strategies to deal with (2) and (3) respectively. Roughly speaking the main steps are as follows: For (2), after decomposing (similar to \cite[Lem 10.4]{vkkr:hms}) into compatible morphisms,  we can reduce (as in \cite[$\mathsection$10.4]{vkkr:hms}) to the problem of suitably bounding the set $X_0/{\sim}$ of isomorphism classes of $R\otimes\Gamma$-lattice structures on $R$. Then we construct an explicit finite map from $X_0/{\sim}$ to the class group of $L$.  For (3), after decomposing again into compatible morphisms, we tensor with $\OL_D$ to double the degree. This allows us to reduce to the situation treated in \cite[Prop 10.3]{vkkr:hms} and then we use Diophantine analysis to bound everything in terms of $D$ and $\Gamma$. In each of the two cases (2) and (3) we obtain bounds in terms of a discriminant of $D$ which we control via   the Masser--W\"ustholz endomorphism estimates~\cite{mawu:factorization} based on transcendence.

\subsubsection{Shape of the bounds}\label{sec:shapeofbounds}

The bounds in Theorems A, B and D depend polynomially on $\nu N_S$ and $\rad(D_KN_S)$. This crucially exploits the polynomial dependence on $N_S$ in \cite[Thm A and B]{rvk:gl2} and the strong shape of the Masser--W\"ustholz isogeny estimates~\cite{mawu:abelianisogenies,mawu:periods} based on transcendence, see also the related discussions in \cite[Rem 1)]{rvk:gl2}. Further, there are quite a few steps in the proofs where we had to modify or completely replace standard finiteness arguments in order to obtain effective polynomial bounds.  For example, while it is conceptually satisfying in $\mathsection$\ref{sec:phiconstructionintro} to work with a single controlled $T\subseteq\spec(\OL_K)$, this requires a huge number field $K$  whose discriminant might not be polynomial in terms of $N_S$. To avoid such huge  fields $K$, we actually work with many distinct $T\subseteq \spec(\OL_{K})$ coming from smaller fields $K$ and then we control the (number of) possible $T$.

We did not try to fully optimize the constants and exponents in our polynomial bounds. However, we conducted some effort to include a few optimizations providing exponential improvements. For example, in terms of $g$ we exponentially improve $c$ and $e$ in Theorem~A by working out stronger versions of the effective Shafarevich conjecture in Theorem~\ref{thm:es} in certain relevant situations (Propositions~\ref{prop:esgl2} and \ref{prop:cm}). Further, to assure that the exponent $24g$ in Theorems~A and B is not exponential in terms of $g$, we crucially exploit that the exponent $24$ in \cite[Thm A]{rvk:gl2} does not depend on the dimension; in fact this was the main motivation for making the effort in \cite[p.18-19]{rvk:gl2} to remove the dimension  in the exponent of \cite[Thm A]{rvk:gl2}. We also emphasize that our bounds directly benefit from the many important refinements of the Masser--W\"ustholz isogeny results obtained by Bost--David~\cite{bost:abelianisogenies}, Viada~\cite{viada:abeliansubvariety} and Gaudron--R\'emond~\cite{gare:periods,gare:isogenies,gare:newisogenies}.

We now compare the bounds in Theorem~A with known results.  The polynomial bounds for Hilbert modular schemes in \cite[Thm A]{vkkr:hms} and for moduli schemes of elliptic curves in \cite[Thm A]{rvk:intpointsmodell} have sharper constants and exponents than Theorem~A. The reasons are that some steps (which blow up the bounds) in the proof of Theorem~A can be omitted for Hilbert moduli schemes and that more tools are available for elliptic curves. For example, the proof of \cite[Thm A]{rvk:intpointsmodell} uses a height-conductor inequality (see Frey~\cite{frey:ternary}) for elliptic curves which was proven independently and simultaneously by K.~\cite{rvk:intpointsmodell} and by Murty--Pasten~\cite{mupa:modular} via Frey's modular degree approach; see K.--Matschke~\cite{vkma:computation} and Pasten~\cite{pasten:shimabc} for some refinements of the height-conductor inequality. Further, it is quite difficult to compare Theorem~A~(i) with the height bound of Bilu, Cai and Sha (see \cite{cai:int}) for congruence modular curves  since the bounds have different shapes: Their bound which in addition holds for any number field is better than Theorem~A~(i) in most situations, except when the level is large since Theorem~A~(i) depends polynomially on $\nu$.

\subsubsection{Conditional generalizations and algorithmic versions}\label{sec:introdgeneralizations}
The above described proof shows that the key for generalizing Theorem~A is the effective Shafarevich Conjecture~$\ces$. An optimal form of Conjecture~$\ces$ would follow from  the generalized Szpiro conjecture (\cite{szpiro:seminaire88,frey:linksulm}) with effective constants, see for example \cite[$\mathsection$4.1]{gakamo:alpbach}. Going beyond (modular) abelian varieties,  Venkatesh developed in \cite{beseve:torsiongrowth,venkatesh:icm,venkatesh:heightconjecture} a far-reaching height conjecture for motives of automorphic forms:  It relates an automorphic height to an arithmetic height, where the arithmetic height is defined in the spirit of Kato's height of motives~\cite{kato:heights} which in turn generalizes the Faltings height.  Venkatesh showed in \cite{venkatesh:heightconjecture} that the height conjecture is closely related to the Bloch--Kato conjecture for the adjoint L-function and the motivic conjecture of \cite{gave:deriveddeform,venkatesh:derivedhecke,prve:asterisque}; in particular these two conjectures imply the height conjecture. 

To generalize Theorem~A to any number field $K$, it suffices by the above described proof to prove the $\gl2$-case of Conjecture~$\ces$. The strategy of \cite{rvk:gl2} for $K=\QQ$, combining Faltings' method with modularity and isogeny estimates, works for any $K$ except the part using a geometric version of modularity (which is currently not available for all $K$). Introducing new ideas, Alp\"oge~\cite{alpoge:modularitymordell}  showed that one can replace in this strategy modularity by potential modularity which allowed him to prove an algorithmic version of the $\gl2$-case of Conjecture~$\ces$ for totally real $K$ of odd degree. Conditional on modularity and/or other standard conjectures, he similarly deduced in \cite{alpoge:thesis} algorithmic versions of other relevant cases of Conjecture~$\ces$, including the fundamental case of principally polarized $A$ jointly with Lawrence (see \cite[$\mathsection$7]{alpoge:thesis}). Unconditionally, he obtained in \cite{alpoge:unpeu} an algorithmic bound for the number of isomorphism classes of $A$ as in Conjecture~$\ces$ of $\gl2$-type by proving `height-free' isogeny estimates.

Finally we mention that (potential) automorphy/modularity is established  in more and more general situations, see \cite{aletal:cmmodularity,neth:ihes1,neth:ihes2,boetal:absurfacesmodular} for recent breakthroughs.

\subsection{Clebsch--Klein surfaces as coarse moduli schemes}\label{sec:introckcoarse}   

In Section~\ref{sec:ck} we study geometric properties of the Clebsch--Klein surfaces. In particular we show that these surfaces are coarse Hilbert moduli schemes. To describe our results in more detail, we use the notation and terminology introduced in Section~\ref{sec:introintegralchms}. Let $Y\subset X$ be the relative surfaces over $\ZZ$ defined in \eqref{def:sigma24surface}. We define the closed subscheme
\begin{equation}\label{def:cubicsurface}
V\subset\mathbb P^4_\ZZ: \sum z_i=0=\sum z_i^3, \quad \textnormal{and}\quad U=V\setminus \cup V_+(z_i)
\end{equation}
with the union taken over the five coordinate functions $z_i$ of $\mathbb P^4_\ZZ=\textnormal{Proj}\bigl(\ZZ[z_i]\bigl)$. Clebsch~\cite{clebsch:clebschsurface} and Klein~\cite{klein:cubicsurfaces} initiated the study of the geometry of the surfaces defined in \eqref{def:sigma24surface} and \eqref{def:cubicsurface}. Building on the work of Hirzebruch~\cite{hirzebruch:ck} in which he proved $Y(\CC)\cong\mathbb H^2/\Gamma(2)$ for $L=\QQ(\sqrt{5})$, we shall obtain in Theorem~\ref{thm:cksurfaces} the following result.

\vspace{0.2cm}
\noindent{\bf Theorem E.}
\emph{The surface $Y/\ZZ$ becomes over $\ZZ[\tfrac{1}{30}]$ a coarse moduli scheme of an arithmetic moduli problem $\mathcal P$ on $\mathcal M$ with the following properties.} 
\begin{itemize}
\item[(i)] The branch locus $B_{\mathcal P}$ is empty and $|\mathcal P|_{\bar{\CC}}\leq 16$.
\item[(ii)] Any $P\in Y(\bar{\QQ})$ satisfies $h(P)\leq 2h_\phi(P)+8^8\log(h_\phi(P)+8)$.
\end{itemize}
\vspace{0.2cm}
 
Here $\mathcal M$ is the Hilbert moduli stack associated to $L=\QQ(\sqrt{5})$ and  $h$ is the logarithmic Weil height on $Y\subset \mathbb P^4_\ZZ$. The moduli problem $\mathcal P$ in Theorem~E is closely related to the moduli problem $\mathcal P(2)$ of principal level 2-structures; see \eqref{def:symplecticlvl} and \eqref{def:morphismckthm}.  

Combining Theorem~E with the open immersion $U\hookrightarrow Y$ over $\ZZ[1/3]$ constructed in Lemma~\ref{lem:openimmersion}, one obtains an analogue of Theorem~E for $U$. In particular, for any algebraically closed field $k$ of characteristic not in $\{2,3,5\}$ we obtain natural moduli interpretations of the points in $Y(k)$ and $U(k)$ in terms of abelian varieties $A$ over $k$ with a ring morphism $\ZZ[\tfrac{1+\sqrt{5}}{2}]\hookrightarrow \End(A)$; the moduli interpretations are given in \eqref{eq:moduliintY} and \eqref{eq:moduliintU}.  

The height bound in Theorem~E~(ii) is linear in terms of $h_\phi$ which is best possible in the sense that for any exponent $e<1$ the bound $h\ll h^{e}_\phi$ can not hold on $Y(\bar{\QQ})$.  Furthermore the factor $2$ in front of $h_\phi$ might be optimal, but the constants $8^8$ and $8$ can be improved up to a certain extent; see the discussions at the end of Section~\ref{sec:heightcomparison}.


\paragraph{Moduli interpretation of $X$.} Define $B=\ZZ[\tfrac{1}{30}]$ and let $M_2^*$ be the minimal compactification of Faltings--Chai~\cite{fach:deg,chai:hilbmod} of the coarse moduli scheme $M_2$ of the stack $\mathcal M_2=\cmp$ over $B$ constructed by Rapoport~\cite[Thm 1.22]{rapoport:hilbertmodular}; see Section~\ref{sec:moduliinterpretation}. In course of the proof of Theorem~E, we obtain in Proposition~\ref{prop:moduliinter} the following result. 

\vspace{0.2cm}
\noindent{\bf Proposition F.}
\emph{There is an isomorphism $\minm\isomto X_B$ which restricts to $M_2\isomto Y_B$.} 
\vspace{0.2cm}

\noindent Explicitly the isomorphism $\minm\isomto X_B$ is given by the five Eisenstein series $E_i$ of $\mathsection$\ref{sec:eisenstein} which correspond to the five cusps of $M_2$. Hirzebruch~\cite{hirzebruch:ck} essentially proved Proposition~F over $\QQ$. Indeed his arguments  over $\CC$, which were modified to work over $\QQ$ by Shepherd-Barron--Taylor~\cite{sbta:ck}, show that the $E_i$ define isomorphisms over $\QQ$: 
\begin{equation}\label{eq:introgenfiberiso}
M_{2,\QQ}^*\isomto X_\QQ \quad \textnormal{and}\quad M_{2,\QQ}\isomto Y_\QQ.
\end{equation}
This shows that $Y$ becomes a coarse Hilbert moduli scheme over $\QQ$, and hence over some open $T\subseteq\spec(\ZZ)$. Unfortunately we have here no control over  $T$ since $M_2$ is only defined abstractly. Also, unlike in the case of relative curves, for relative surfaces an isomorphism over the generic fiber does not automatically extend to an isomorphism over the smooth locus. In fact the five $E_i$ can not define a morphism over $\ZZ[\tfrac{1}{2\Delta}]$ for $\Delta=5$ the discriminant of $L$, since one needs to invert the prime $3$ by Lemma~\ref{lem:eisensteinfourierexpansion} and the $q$-expansion principle.

\subsubsection{Principal ideas of Proposition F and Theorem~E~(i)}\label{sec:outlineproofofmoduliint}

To extend the isomorphisms \eqref{eq:introgenfiberiso} over an explicit open $T\subseteq \spec(\ZZ)$, it seems hopeless to directly carry through Hirzebruch's arguments~\cite{hirzebruch:ck} over $T$. The main problem is that his arguments rely on results for algebraic surfaces whose arithmetic analogues over $T$ are not available. Thus we use a different strategy which combines Hirzebruch's arguments with the construction of Faltings--Chai~\cite{fach:deg,chai:hilbmod} of the minimal compactification $M_2^*$.  A key point in the latter construction is the Stein factorization
\begin{equation}\label{eq:introsteinfact}
\torm\to^{\bar{\pi}} \minm\to \mathbb P_{B}^n
\end{equation}
of the morphism $\torm\to \mathbb P^n_B$ induced by the fact (\cite[Thm~(iii)]{chai:hilbmod}) that a positive power of the Hodge bundle $\omega$ on $\torm$ is generated by its global sections, where $\torm$ is a smooth toroidal compactification of $\mathcal M_2$. More precisely, as the Eisenstein series $E_i$ have Fourier expansions over $B$ by Lemma~\ref{lem:eisensteinfourierexpansion}, they correspond to global sections $s_i$ of $\omega^{\otimes 2}$ via the $q$-expansion principle \cite{rapoport:hilbertmodular} and we obtain in Proposition~\ref{prop:fingen} the following result.

\vspace{0.2cm}
\noindent{\bf Proposition G.}
\emph{The invertible sheaf $\omega^{\otimes 2}$ on $\torm$ is globally generated by the five $s_i$.}
\vspace{0.2cm}

\noindent An application of \eqref{eq:introsteinfact} with the morphism $\torm\to \mathbb P^4_B$ induced by Proposition~G gives $\torm\to^{\bar{\pi}}\minm\to^f \mathbb P^4_B$ for some finite morphism $f$ which extends the isomorphism in \eqref{eq:introgenfiberiso}. Then we use  Zariski's main theorem to conclude that $f:\minm\to X_B$ is an isomorphism which restricts to $M_2\isomto Y_B$. In particular $Y$ becomes over $B$ a coarse Hilbert moduli scheme  with initial morphism $\pi:\mathcal M_2\to Y_B$ given by $f\bar{\pi}|_{\mathcal M_2}$. Our study of automorphism groups in Section~\ref{sec:branchauto} implies that the branch locus $B_{\mathcal P}$ of $\pi$ is empty and then we deduce Theorem~E~(i). We next describe the ideas used in the proof of Proposition~G.

\subsubsection{Outline of the proof of Proposition G}\label{sec:outlinefingenproof}

 On using fundamental results of Rapoport~\cite{rapoport:hilbertmodular} and Faltings--Chai~\cite{fach:deg,chai:hilbmod} and also the $q$-expansion of $E_i$ computed in Lemma~\ref{lem:eisensteinfourierexpansion}, we first reduce Proposition~G to the analogous problem for $\pi_*\omega^{\otimes 2}$ on $M_2$ for $\omega=\omega|_{\mathcal M_2}$. Here $\pi_*\omega^{\otimes 2}$ is an invertible sheaf on $M_2$ by Lemma~\ref{lem:descendingomega} which we deduce from a result of Olsson~\cite{olsson:gtorsors} after computing the automorphism groups of $\mathcal M_2$. Thus we are reduced to show that
\begin{equation}\label{eq:introddivisorintersection}
\cap D_i=\emptyset
\end{equation}
for $D_i$ the divisor on $M_2$ defined by the global section $s_i$ of $\omega^{\otimes 2}$ corresponding to $E_i$. To prove \eqref{eq:introddivisorintersection} we use a strategy which is inspired by Hirzebruch's approach \cite{hirzebruch:ck} over $\CC$: We  first compute $D_i$ in terms of certain modular curves $gC\subset M_2$ and then we exploit the moduli interpretation of $gC$. We now explain the key steps of the strategy in more detail.

\paragraph{Eisenstein series $E_i$.}In $\mathsection$\ref{sec:eisenstein} we consider the Eisenstein series $E_i$ of weight two for the principal congruence subgroup $\Gamma(2)\subset \sl2(\OL)$ which vanishes at all cusps except the cusp $i$, where we label the five cusps of $M_2$ by $0,\dotsc,4$ and where $\OL$ is the ring of integers of $L=\QQ(\sqrt{5})$. The Fourier expansion of $E_i$ at the cusp $j$ is of the form  
\begin{equation}\label{eq:introfourierexpansion}
\sum a_\nu e^{2\pi i\textnormal{Tr}(\nu\tau)}, \quad a_\nu\in \ZZ, \quad \gcd(a_\nu)=1, \quad a_0=\begin{cases}
3  & \textnormal{if } j=i,\\
0 & \textnormal{if } j\neq i.
\end{cases}
\end{equation}
We shall prove this in Lemma~\ref{lem:eisensteinfourierexpansion} by computing explicitly the quantities in Klingen's formulas~\cite{klingen:eisensteinfourier} for the Fourier coefficients of $E_i$ which he obtained via analytic methods. Here the computation of the constant term $a_0$ relies on the result $\zeta_{L}(-1)=\tfrac{1}{30}$ of Siegel~\cite{siegel:zetavalues} and  Zagier~\cite{zagier:zetavalues}, where $\zeta_{L}$ is the Dedekind zeta function of $L=\QQ(\sqrt{5})$.

\paragraph{Modular curves inside $M_2$.} Hirzebruch, Zagier and van der Geer  studied the modular curve $F_1=T_1$ on complex Hilbert modular varieties (\cite{hirzebruch:hmsenseignement,hiza:hmsintersection,geza:hms13}).   Moreover, Bruinier--Burgos--K\"uhn~\cite{brbuku:hilbmod} and Yang~\cite{yang:hmsurfaceamerican} studied an integral model of the Hirzebruch--Zagier divisor $F_1$ on $\mathbb H^2/\sl2(\mathcal O)$. Building on their works, we define in $\mathsection$\ref{sec:modularcurves} modular curves inside $M_2$ and we prove some geometric properties. The modular curve $C\subset M_2$ is defined as the image of the proper morphism $\pi\phi:\mathcal E_2\to M_2$, where $$\phi:\mathcal E_2\to \mathcal M_2, \quad (E,\alpha)\mapsto (E,\alpha)\otimes \OL,$$
is defined in \eqref{def:stackimmersion} for $\mathcal E_2$ the moduli stack over $B$ of elliptic curves with symplectic level 2-structure. More generally, as $G=\sl2(\OL/2\OL)$ acts on $\mathcal M_2$, we can define analogously a modular curve $gC\subset M_2$ for each $g\in G$. We show the following: The curves $gC\subset M_2$ are disjoint in the sense that $gC\cap g'C$ is empty if $gC\neq g'C$,  and any connected component of the curve $F_1\subset \mathbb H^2/\Gamma(2)$ is given by $gC(\CC)$ for some $g\in G$.   In particular the curves $gC$ are integral models of the connected components of $F_1\subset \mathbb H^2/\Gamma(2)$.

\paragraph{Computing the divisor $D_i$.}To compute the divisor $D_i$ on $M_2$ associated to the global section $s_i$ of $\pi_*\omega^{\otimes 2}$ corresponding to $E_i$, we first show in Lemma~\ref{lem:divhorizontal} that each prime divisor of $D_i$ is the closure of its generic fiber. For this we use results of Rapoport~\cite{rapoport:hilbertmodular} and Faltings--Chai~\cite{fach:deg,chai:hilbmod} and we work with the normal scheme $\minm$ which is proper over $B$ with irreducible fibers: In particular we combine the Fourier expansion \eqref{eq:introfourierexpansion} with the $q$-expansion principle to show that $s_i$ can not vanish on any vertical divisor of $\minm$. Then we compute in Lemma~\ref{lem:divisorintersection} the closure in $M_2$ of the generic fiber $D_{i,\QQ}$ of $D_i$: This closure, and hence $D_i$, is given by four disjoint $gC$. Here we combine our geometric results for $gC$  with  Hirzebruch's computation~\cite{hirzebruch:ck} of $D_{i,\QQ}$ in terms of the connected components of $F_1\subset \mathbb H^2/\Gamma(2)$. Moreover, we can determine which four $gC$ appear inside $D_i$ via the intersection behaviour (\cite{hirzebruch:ck}) with the cusp resolutions in the minimal desingularization of $M_2^*(\CC)$. Then the disjointness of the $gC\subset M_2$ leads to $\cap D_i=\emptyset$.

\subsubsection{Ideas of the proof of the height bound in Theorem~E~(ii)}

To bound the logarithmic Weil height $h$ on $Y(\bar{\QQ})$ in terms of the height $h_\phi$, we may and do work over $\bar{\QQ}$ and we write $Y$ for $Y_{\bar{\QQ}}$ and $\mathbb P^4$ for $\mathbb P^4_{\bar{\QQ}}$.  The proof consists of two parts.

\paragraph{Relate $h$ to $h_\theta$.}In the first part of the proof, we relate $h$ to a Theta height on $Y$ which we pull back from the coarse moduli space $A_{2,2}$ over $\bar{\QQ}$ of principally polarized abelian surfaces with symplectic level 2-structure. More precisely, Lemma~\ref{lem:thetaheight} provides
\begin{equation}\label{eq:introthetaheightbound}
h(P)\leq h_\theta(P), \quad P\in Y(\bar{\QQ}),
\end{equation}
 where $h_\theta$ is the pull back along the forgetful morphism $Y\to A_{2,2}$ of the theta height on $A_{2,2}$ given by the fourth powers of the 10 even Theta nullwerte $\theta_m$. To prove \eqref{eq:introthetaheightbound} we identify  the five coordinate functions on $Y\subset \mathbb P^4$ with the five Eisenstein series $E_i$ and we compute $E_i$ in terms of G\"otzky's~\cite{gotzki:theta} classical Theta functions $\theta_{ij}$ as follows:
\begin{equation}\label{eq:introeisensteinthetacomp}
E_i^4=\pm \prod_{j\neq i}\theta_{ij}^4.
\end{equation}
Here we crucially exploit Hirzebruch's work~\cite{hirzebruch:ck} and we use a transformation formula for $\theta_{ij}^8$ due to G\"otzky~\cite{gotzki:theta}. Then we apply results for Hilbert theta functions obtained by Gundlach~\cite{gundlach:theta} and Lauter--Naehrig--Yang~\cite{lanaya:theta} to show that $\theta_{ij}$ is (up to $\pm 1$) the pullback along $Y\to A_{2,2}$ of some even $\theta_m$. This together with \eqref{eq:introeisensteinthetacomp} leads to \eqref{eq:introthetaheightbound}.

\paragraph{Relate $h_\theta$ to $h_\phi$.}In the second part of the proof, we relate the Theta height $h_\theta$ on $Y$ to some Theta height $h_\Theta$ on the coarse moduli space $A_2$ over $\bar{\QQ}$ of principally polarized abelian varieties.  Faltings' work~\cite{faltings:finiteness} allows to compare $h_\Theta$ with the stable Faltings height $h_F$ on $A_2(\bar{\QQ})$.  Moreover, based on ideas of Bost--David, Pazuki~\cite{pazuki:heights} obtained an explicit comparison for $h_F$ and $h_\Theta$ on $A_2(\bar{\QQ})$. This allows us to explicitly relate $h_\Theta$, and thus $h_\theta$, to the height $h_\phi$ and then we deduce from \eqref{eq:introthetaheightbound} the bound for $h$  in Theorem~E~(ii). 

\subsection{Acknowledgements}

We would like to thank Peter Sarnak, Umberto Zannier and Shouwu Zhang for useful comments, and we are grateful to Shijie Fan for many discussions related to this paper.

This paper is the second in a series of papers containing the results obtained in our long
term project on integral points of certain higher dimensional varieties. While working on
this project over the last years, we were supported by the IH\'ES, IAS Princeton, MSRI,
MPIM Bonn, Princeton University, University of Amsterdam, and IAS Tsinghua university. We would like to thank all of these institutions for their support. Further we were
supported by an EPDI fellowship, NSF (No. 300.154591), Veni grant, NSFC(No. 12050410238) and  Tsinghua University Dushi Program(No. 53120300122).

This version of the paper is identical with the version from 2021, up to the new Section~\ref{sec:ck}, some additional discussions and references, and an updated introduction.


\newpage

\section{Conventions and notation}

Unless mentioned otherwise, we shall use throughout this work the following conventions and notation. Let $M$ be a set. We denote by $\lvert M\rvert$ the number of distinct elements of $M$ and we write $N^c$ for the complement in $M$ of a subset $N\subseteq M$.  Let $f:M'\to M$ be a map of sets. We define the set-theoretic degree of $f$ by $\deg(f) = \sup_{m \in M} |f^{-1}(m)|$, and we say that $f$ is finite if for each $m\in M$ the fiber $f^{-1}(m)$ of $f$ over $m$ is finite. If $G$ is a group acting on $M$, then we denote by $M^G=\{m\in M; gm=m \textnormal{ for all }g\in G\}$ the set of fixed points. Further, by $\log$ we mean the principal value of the natural logarithm and the product taken over the empty set is defined as $1$.

\paragraph{Algebra.}  Let $K$ be a number field. We denote by $\OK$ the ring of integers of $K$, we identify a finite place $v$ of $K$ with a nonzero prime ideal of $\OK$ and vice versa, and  we write $N_v$ for the cardinality of the residue field of $v$. The absolute value of the discriminant of $K/\QQ$ will be denoted by $D_K$, and we write $N(I)$ for the absolute norm $|\OK/I|$ of an ideal $I\subseteq \OK$.  We say that a number field $L\supseteq K$ is unramified over a subset $T\subseteq \spec(\OK)$ if the field extension $L/K$ is unramified at each nonzero $v\in T$. Let $n\in \ZZ_{\geq 1}$. For any unique factorization domain $A$ and $x\in A^n$, we write $\gcd(x_i)$ for $\gcd(x_1,\dotsc,x_n)$.  The radical $\textnormal{rad}(m)$ of a nonzero $m\in\ZZ$ is defined as the product $\prod_{p\mid m} p$ of all the rational primes $p$ which divide $m$.  We denote by $R^\times$ the group of units of a (not necessarily commutative) ring $R$ and we write $\uM_n(R)$ for the endomorphism ring of the free $R$-module $R^n$; the unit group of $\uM_n(R)$ will be denoted by $\textnormal{GL}_n(R)$  and $\textnormal{SL}_n(R)\subseteq \textnormal{GL}_n(R)$ is the kernel of a determinant morphism. 
For any field $k$ we denote by $\bar{k}$ an algebraic closure of $k$. 

\paragraph{Schemes.} We denote by $\sets$ and $\sch$ the categories of sets and schemes respectively. Let $\mathcal C$ be a category. We shall often omit the subscript $\mathcal C$ of $\Hom_{\mathcal C}$  when it is clear from the context in which category we are working. A contravariant functor from $\mathcal C$ to $\sets$ is called a presheaf on $\mathcal C$.    
Let $S$ be a scheme. We write $\Hom_S$ for $\Hom_{\mathcal C}$  when $\mathcal C$ is the category of $S$-schemes. If $T$ and $Y$ are $S$-schemes, then we define $Y(T)=\Hom_S(T,Y)$  and we write $Y_T=Y\times_S T$ for the base change of $Y$ from $S$ to $T$. We shall often identify an affine scheme $S=\spec(R)$ with the ring $R$. For example, if $T=\spec(R)$ is  affine then we write $Y_R$ for $Y_T$ and $Y(R)$ for $Y(T)$.  A variety $Y$ over $S$ is an $S$-scheme $Y$ whose structure morphism $Y\to S$ is separated and of finite type. 
If $S$ is integral then we denote by $k(S)$ the function field of $S$. Following \cite{bolura:neronmodels}, we say that $S$ is a Dedekind scheme if $S$ is a normal noetherian scheme of dimension 0 or 1. Unless mentioned otherwise, we equip a closed subset of a scheme with the unique reduced closed subscheme structure. 
If $G$ is a group then we denote by $G_S$ the constant group scheme over $S$.


\paragraph{Abelian schemes and orders}Let $A$ and $B$ be abelian schemes over an arbitrary scheme $S$. We denote by $\Hom(A,B)$ the set of $S$-group scheme morphisms $A\to B$, and we write $\Hom^0(A,B)$ for $\Hom(A,B)\otimes_\ZZ\QQ$ and $\End^0(A)$ for $\End(A)\otimes_\ZZ\QQ$. 

Let $\mathcal O$ be a not necessarily commutative ring. We shall freely use the results, notation and terminology in \cite[$\mathsection$2]{vkkr:hms} on abelian schemes, $\OL$-abelian schemes and orders. A ring morphism $\iota:\OL\to \End(A)$ is  called an $\OL$-module structure on $A$, and two $\OL$-module structures $\iota,\iota'$ on $A$ are isomorphic if there exists an automorphism $\tau$ of $A$ such that $\iota'(x)=\tau\iota(x)\tau^{-1}$ for all $x\in \OL$.  Further $(A,\iota)$ denotes an object in the category of $\OL$-abelian schemes and $A\otimes_\OL I$ denotes an abelian scheme which represents the tensor product  of $(A,\iota)$ with a finite projective right $\OL$-module $I$.  In situations where the specific choice of a ring morphism $\iota:\OL\to \End(A)$ is not relevant, we usually omit $\iota$ from the notation and we simply write $A$ for an $\OL$-abelian scheme $(A,\iota)$. 

Further $A^\vee=\textnormal{Pic}^0(A)$ denotes the dual of $A$, $\textnormal{Hom}(A,A^\vee)^{\textnormal{sym}}$ is the set of symmetric morphisms and  $\textnormal{Pol}(A)$ denotes the set of polarizations of $A$. For any nonzero $n\in \ZZ$ we denote by $A_n$  the kernel of  the morphism $[n]:A\to A$ defined by multiplication with $n$. If $A$ is an abelian scheme over $S$ of relative dimension $g$, then we say that $A$ has CM (resp. that $A$ is of $\gl2$-type) if the $\QQ$-algebra $\End^0(A)$ contains a commutative semisimple  $\QQ$-algebra of rank $2g$ (resp. a number field of degree $g$). 

\paragraph{Stacks.}

Unless mentioned otherwise, we shall use the terminology and definitions of the Stacks project~\cite{sp}. Let $S$ be a scheme.  By an algebraic stack over $S$ we mean an algebraic stack over $(\textnormal{Sch}/S)_{\textnormal{fppf}}$ in the sense of \cite{sp}.  We usually identify schemes over $S$ and algebraic spaces over $S$ with the associated algebraic stacks over $S$.   For any category $\mathcal X$ fibered in groupoids over $(\textnormal{Sch}/S)$, we denote by $[\mathcal X]$ the presheaf on $(\textnormal{Sch}/S)$ sending an $S$-scheme $T$ to the set $[\mathcal X(T)]$ of isomorphism classes of the objects in the fiber $\mathcal X(T)$ of $\mathcal X$ over $T$. Throughout we shall freely use well-known standard results for algebraic stacks or DM stacks which all can be conveniently found in \cite{sp}.

\newpage

\section{Coarse Hilbert moduli schemes}\label{sec:chmsdefs}

In this section we introduce some terminology. After recalling the definition of Hilbert moduli stacks, we define and discuss coarse Hilbert moduli schemes and their branch loci. We also construct a height on coarse Hilbert moduli schemes. 

\subsection{Hilbert moduli stacks}\label{sec:hilbmodstackdef}

Throughout this work we shall use the notation and terminology of \cite[$\mathsection$3.1]{vkkr:hms} on Hilbert moduli stacks. These stacks were constructed by Rapoport~\cite{rapoport:hilbertmodular} and Deligne--Pappas~\cite{depa:hilbertmodular}. For later use we now briefly recall their definition in \cite[$\mathsection$2]{depa:hilbertmodular}. Let $L$ be a totally real number field of degree $g$ with ring of integers $\OL$.  For any nonzero finitely generated $\OL$-submodule $I$ of $L$ with a positivity notion $I_+$, one obtains an algebraic stack $\mathcal M^I$ whose objects  over an arbitrary scheme $S$ are  triples $(A,\iota,\varphi)$ given by:   \begin{itemize}
\item[(i)] An abelian scheme $A$ over $S$ of relative dimension $g$.
\item[(ii)] A ring morphism $\iota:\OL\to \End(A)$.
\item[(iii)] A morphism $\varphi:I\to \Hom_\OL(A,A^\vee)^\textnormal{sym}$ of $\OL$-modules such that the induced morphism $I\otimes_{\mathcal O}A \to A^\vee$ is an isomorphism and such that $\varphi(I_+)\subseteq \textnormal{Pol}(A)$.
\end{itemize}
Here $\Hom_\OL(A,A^\vee)^\textnormal{sym}$ denotes the $\OL$-module of symmetric morphisms $(A,\iota)\to(A^\vee,\iota^\vee)$, where $\iota^\vee:\OL\to \End(A^\vee)$ is the ring morphism obtained by sending $x\in \OL$ to the endomorphism $\textnormal{Pic}^0(\iota(x))$ of $A^\vee$. A morphism $\varphi$ as in (iii) is called an $I$-polarization of $(A,\iota)$. We say that an algebraic stack   is a Hilbert moduli stack  if it is isomorphic to $\mathcal M^I$ associated to some $L$, $I$ and $I_+$ as above.  Further, if $I=\mathfrak d^{-1}$ is the $\ZZ$-dual of $\OL$ and $I_+$ is the (standard) positivity notion $(+,\dotsc,+)$ on $I$, then we call an $I$-polarization simply a polarization and we call $\mathcal M^I$ the Hilbert moduli stack associated to $L$. 
In what follows, we usually identify a Hilbert moduli stack $\mathcal M$ with some $\mathcal M^I$ via an isomorphism $\mathcal M\cong \mathcal M^I$ and we shall freely use the geometric results on Hilbert moduli stacks discussed in \cite[$\mathsection$3.1]{vkkr:hms}.

\subsection{Moduli problems}

Katz--Mazur developed in \cite{kama:moduli} the moduli problem formalism for $g=1$ using Mumford's language of relative representability~\cite{mumford:picardmoduli}. Building on \cite{kama:moduli} we now develop the moduli problem formalism for all $g\geq 1$ using the language of algebraic stacks. This allows us to freely apply in our proofs the theory of algebraic stacks which was substantially extended and improved over the last decades by many authors \cite{sp}.

Let $\mathcal M$ be a Hilbert moduli stack and let $\mathcal P$ be a presheaf on $\mathcal M$. We call $\mathcal P$ a moduli problem on $\mathcal M$. Further, for each object $x$ of $\mathcal M$, the elements of the set $\mathcal P(x)$ are called $\mathcal P$-level structures on $x$. The category $\cmp$, whose objects are pairs $(x,\alpha)$ with $x$ an object of $\mathcal M$ and $\alpha\in\mathcal P(x)$, is fibered in groupoids over $\mathcal M$ via the forgetful morphism  $$\cmp\to\mathcal M, \quad (x,\alpha)\mapsto x.$$
Then $\cmp$ is fibered in groupoids over $\sch$ via the structure morphism of $\mathcal M$. We say that the moduli problem $\mathcal P$ is algebraic if  $\cmp$ is an algebraic stack. For example any representable $\mathcal P$ is algebraic, and $\mathcal P$ is algebraic if the forgetful morphism $\cmp\to \mathcal M$ is representable in algebraic spaces.  
Further we say that the moduli problem $\mathcal P$ is arithmetic if $\cmp$ is a DM stack which is separated and of finite type over $\ZZ$.


Let $S$ be a scheme. As in \cite{rvk:intpointsmodell,vkkr:hms} we use the quantity $|\mathcal P|_S$ in order to control the number of distinct $\mathcal P$-level structures on any object in $\mathcal M(S)$. It is defined by
\begin{equation}\label{def:maxlvl}
|\mathcal P|_S=\sup |\mathcal P(x)|
\end{equation}
with the supremum taken over all objects $x$ of $\mathcal M(S)$. 
We say that $\mathcal P$ is finite over $\mathcal M(S)$ if  $\mathcal P(x)$ is finite for all $x\in \mathcal M(S)$. In particular  $\mathcal P$ is finite over $\mathcal M(S)$ if $|\mathcal P|_S<\infty$.

\subsection{Coarse Hilbert moduli schemes}\label{sec:coarsehmsdef}

Let $\mathcal M$ be a Hilbert moduli stack, let $Y$ be a scheme and let $\mathcal P$ be an algebraic moduli problem on $\mathcal M$. We say that $Y$ is a Hilbert moduli scheme of $\mathcal P$ if there exists an object in $\mathcal M(Y)$ which represents the presheaf $\mathcal P$. Then $Y$ is a Hilbert moduli scheme of $\mathcal P$ if and only if there exists an equivalence $\cmp\isomto Y$. More generally,  we say that $Y$ is a coarse Hilbert moduli scheme of $\mathcal P$, and we write $Y=\smp$, if there is a morphism
\begin{equation*}
\pi:\cmp\to Y
\end{equation*}
with the following properties: It induces a bijection $[\cmp(k)]\isomto Y(k)$ for each algebraically closed field $k$, and it is initial in the sense that 
any morphism from  $\cmp$ to an algebraic space factors uniquely through $\pi$.
In other words, $Y$ is a coarse moduli scheme of $\mathcal P$ if and only if there is a coarse moduli space $\cmp\to Y$ for $\cmp$ in the usual sense (\cite[$\mathsection $11]{olsson:stacks}). As $\pi$ is initial, it is unique up to unique isomorphism. In Section~\ref{sec:cms} we further discuss some geometric properties of $\pi$ in the case when $\mathcal P$ is arithmetic.

\paragraph{Terminology.} Unless mentioned otherwise, by a (coarse) moduli scheme we mean a (coarse) Hilbert moduli scheme and by an initial morphism $\cmp\to Y$ we mean a morphism which is initial among all morphisms to an algebraic space. Further, we simply say that $Y=M_{\mathcal P}$ is a (coarse) moduli scheme when $Y$ is a (coarse) moduli scheme of an algebraic moduli problem $\mathcal P$ on $\mathcal M$.
When working with (coarse) moduli schemes it is usually important to specify the involved moduli problem $\mathcal P$ on $\mathcal M$, since $Y$ can be a (coarse) moduli scheme of moduli problems on $\mathcal M$ which are geometrically very different. However, we often do not need to specify the morphism $\pi:\cmp\to Y$ since  it is unique up to isomorphism.

\paragraph{Branch locus.}Let $Y$ be a coarse moduli scheme of an algebraic moduli problem $\mathcal P$ on $\mathcal M$. We define the branch locus $B_{\mathcal P}$ of $\mathcal P$ as the complement in $Y$ of the union of all open subschemes $U\subseteq Y$ such that $\pi_U$ is \'etale where $\pi:\cmp\to Y$ is an initial morphism.   For example $B_{\mathcal P}$ is empty when $Y$ is a moduli scheme of $\mathcal P$, and for given $\mathcal P$ the branch locus $B_{\mathcal P}\subseteq Y$ does not depend on the specific choice of an initial morphism $\pi$.

\subsection{Height on coarse moduli schemes}\label{sec:heightdef} 
Let $\mathcal M$ be a Hilbert moduli stack and let $S$ be a connected Dedekind scheme whose function field $k$ is algebraic over $\QQ$. We now generalize to coarse moduli schemes the height which was defined in \cite[(3.3)]{rvk:intpointsmodell} and \cite[(3.4)]{vkkr:hms} in the case of moduli schemes. 

\paragraph{Faltings height.} For any abelian scheme $A$ over $S$ we denote by 
$
h_F(A)
$
the stable Faltings height $h_F$ of the generic fiber of $A$ introduced by Faltings~\cite{faltings:finiteness}. Here we use Faltings' original normalization~\cite[p.354]{faltings:finiteness} of the involved metric; see for instance \cite[$\mathsection$2.1]{rvk:gl2}. In fact for each $g\in\ZZ_{\geq 1}$ we obtain a height function $$h_F:\absg(S)\to \mathbb R$$ since isomorphic abelian schemes over $S$ have the same stable Faltings height. Here $\absg$ denotes the presheaf on $\sch$ which sends a scheme $S$ to the set of isomorphism classes of abelian schemes over $S$ of relative dimension $g$. This underlined $\absg$ should not be confused with the usual $A_g=A_{g,1}$ which classifies principally polarized abelian schemes $(A,\psi)$.  

\paragraph{Height on coarse moduli schemes.}Suppose that $\mathcal M\cong\mathcal M^I$ is associated to some $(L,I,I_+)$ and write $g=[L:\QQ]$. Let $Y$ be a coarse moduli scheme of some algebraic moduli problem $\mathcal P$ on $\mathcal M$. Now, we  define the height function
\begin{equation}\label{def:cheight}
h_\phi: Y(S)\to \mathbb R
\end{equation}
as the pullback of $h_F:\absg(\bar{k})\to\RR$ along the map $\phi_{\bar{k}}:Y(S)\to \absg(\bar{k})$. Here $\phi_{\bar{k}}$ is the composition of the following three maps: The map $Y(S)\to Y(\bar{k})$ induced by choosing an algebraic closure $\bar{k}$ of $k$, the bijection $\pi^{-1}:Y(\bar{k})\isomto [\cmp(\bar{k})]$ induced by an initial morphism $\pi:\cmp\to Y$, and the forgetful map defined by $[(x,\alpha)]\mapsto [x]\mapsto [A]$ where we identify via an equivalence $\mathcal M\isomto\mathcal M^I$ an object $x$ of $\mathcal M$ with its image $(A,\iota,\varphi)$ in $\mathcal M^I$. 

On using that $\phi_{\bar{k}}$ is induced by morphisms of presheaves and that the stable Faltings height $h_F$ is invariant under geometric isomorphisms, we deduce the following:
\begin{itemize}
\item[(i)] 
The function $h_\phi:Y(S)\to \RR$ is compatible with any dominant base change $S'\to S$, where $S'$ is a connected Dedekind scheme whose function field is algebraic over $\QQ$.
\item[(ii)] If $Y=M_{\mathcal P}$ is a moduli scheme, then there is an equivalence $\pi:\cmp\isomto Y$ such that $h_\phi$ in \eqref{def:cheight} coincides with the height $h_\phi$ on $Y(S)$ defined in \cite[(3.4)]{vkkr:hms}.  
\end{itemize}
Further, if $S$ and $T$ are nonempty open subschemes of $\spec(\ZZ)$, and if $Y$ is a scheme such that $Y_T$ is a coarse moduli scheme of some algebraic moduli problem $\mathcal P$ on $\mathcal M$, then we define the height $h_\phi$ on $Y(S)$ as follows:  For any $P\in Y(S)$ we put $$h_{\phi}(P)=h_{\phi}(P'),$$ where $P'$ is the canonical image of $P$ in $Y(U)\cong Y_T(U)$ for $U=S\cap T$ and  $h_\phi:Y_T(U)\to \RR$ is the height defined in \eqref{def:cheight} with respect to an initial morphism $\pi:\cmp\to Y_T$.  We point out that the height $h_\phi$ on $Y(S)$ depends on the  (unique up to isomorphism) initial morphism $\pi$ which in turn depends on the involved moduli problem $\mathcal P$.

\subsection{Coarse moduli schemes over $\ZZ[1/n]$}\label{sec:overt}

Let $\mathcal M$ be a Hilbert moduli stack and let $T\subseteq \spec(\ZZ)$ be a nonempty open subscheme. Sometimes we shall work over the open substack $\mathcal M_T$ of $\mathcal M$: On replacing in the above definitions and constructions $\mathcal M$ and $\sch$ by $\mathcal M_T$ and $(\textnormal{Sch}/T)$ respectively, we directly obtain the notion of an algebraic or arithmetic  moduli problem on $\mathcal M_T$ and the notion of a (coarse) moduli scheme $Y$ over $T$ of an algebraic moduli problem $\mathcal P$ on $\mathcal M_T$. Here $\pi:(\mathcal M_T)_{\mathcal P}\to Y$ is initial among all morphisms to an algebraic space over $T$.  

For any algebraic moduli problem $\mathcal P$ on $\mathcal M_{T}$ and any $T$-scheme $S$, we define analogously as in \eqref{def:maxlvl} the quantity $|\mathcal P|_S$.  Let $Y$ be a coarse moduli scheme  over $T$ of an algebraic moduli problem $\mathcal P$ on $\mathcal M_T$ with an initial morphism $\pi:(\mathcal M_T)_{\mathcal P}\to Y$. Then we define analogously as in Sections~\ref{sec:coarsehmsdef} and \ref{sec:heightdef} the height $h_\phi$ and the branch locus $B_{\mathcal P}$. 

While it is usually difficult to `correctly' extend the coarse moduli scheme $Y$ over $T$ to a coarse moduli scheme over $\ZZ$, one can always naively extend $Y$ over $\ZZ$ via \eqref{eq:naiveextension}. On the other hand, if $U\subseteq T$ is nonempty open then the projection $Y_U\to Y$ is flat and $Y_U$ is a coarse moduli scheme over $U$; see the discussion surrounding \eqref{eq:heightrestriction}.

\paragraph{Principal level structures.}To discuss an example we suppose that $\mathcal M\cong M^I$ is associated to some $(L,I,I_+)$ and we take $n\in \ZZ_{\geq 1}$. Now, we consider the moduli problem $\mathcal P(n)$  of principal level $n$-structures:  For any scheme $S$ the presheaf $\mathcal P(n)$ sends an object $x=(A,\iota,\varphi)$ of $\mathcal M(S)\cong \mathcal M^I(S)$ to the set of isomorphisms $(\OL/n\OL)_S^2\isomto A_n$ which are compatible with the $\OL$-actions, where the ring of integers $\OL$ of $L$ acts on $(\OL/n\OL)_S^2$ via multiplication and on $A_n$ via $\iota$. We also denote by $\mathcal P(n)$ the restriction of $\mathcal P(n)$ to the open substack $\mathcal M_T\subseteq \mathcal M$ for $T=\spec(\ZZ[1/n])$.  The forgetful morphism $(\mathcal M_T)_{\mathcal P(n)}\to \mathcal M_T$ is finite \'etale and $\mathcal M_T$ is a finite type separated DM stack over $T$. Thus $\mathcal P(n)$ is an arithmetic moduli problem on $\mathcal M_T$ and there exists a coarse moduli scheme $Y(n)$ over $T$ of $\mathcal P(n)$.   
\newpage

\section{Integral points on coarse moduli schemes: Main results}\label{sec:results}

\noindent To state our results, we use the terminology introduced above. Let $Y$ be a variety over $\ZZ$, let $Z\subseteq Y$ be a closed subscheme, let $S\subseteq\spec(\ZZ)$ be a nonempty open subscheme and let $\mathcal M$ be a Hilbert moduli stack. Suppose that $\mathcal M\cong \mathcal M^I$ is associated to some $(L,I,I_+)$. We write $g=[L:\QQ]$ and $\Delta=|\textnormal{Disc}(L/\QQ)|$. For any nonempty open subscheme $U\subseteq \spec(\ZZ)$ we denote by $N_U=\prod p$  the product of all rational primes $p$ not in $U$.

\begin{theorem}\label{thm:mainint}   Suppose that there exists a nonempty open $T\subseteq\spec(\ZZ)$ such that $Y_T$ is a coarse moduli scheme of some arithmetic moduli problem $\mathcal P$ on $\mathcal M$ with $B_\mathcal P\subseteq Z_T$. 
\begin{itemize}
\item[(i)] If $U=T\cap S$ then any point $P\in (Y\setminus Z)(S)$ satisfies $h_\phi(P)\leq c_1N_U^{e_1}$.
\item[(ii)] The cardinality of $(Y\setminus Z)(S)$ is at most $c_2|\mathcal P|_{\bar{\QQ}}N_{U}^{e_2}\Delta\log(3\Delta)^{2g-1}$.
\end{itemize}
\end{theorem}
Here $e_1=\max(24,5g)$ and  $c_1=7^{7^{7g}}$,  while $e_2=6\cdot 3^{8g}$ and $c_2=9^{9^{9g}}$.
Notice that $N_U=\textnormal{rad}(N_T N_S)$, and assertion (ii) holds with $|\mathcal P|_\CC$ in place of $|\mathcal P|_{\bar{\QQ}}$.  
Then Theorem~A in the introduction follows by applying Theorem~\ref{thm:mainint} with $I$ the inverse different of the ring of integers of $L$ and $I_+$ the standard positivity notion. To see this take $T=\spec(\ZZ[1/\nu])$ with $\nu$ as in Theorem~A and consider the complement of $S\subseteq \spec(\ZZ)$ which is a finite set of rational primes.  
The proof of Theorem~\ref{thm:mainint}  can be found in Section~\ref{sec:proofmainresults}, and we refer to the introduction for an outline of the main ideas used in the proof.

\paragraph{The variety $Y(2)$.}We next give a first application of Theorem~\ref{thm:mainint}. Consider the coarse moduli scheme $Y(2)$ over $\ZZ[1/2]$ of the moduli problem $\mathcal P(2)$ of principal level 2-structures on $\mathcal M_{\ZZ[1/2]}$. We shall deduce in Section~\ref{sec:proofcorollaries} the following result.  
\begin{corollary}\label{cor:y2}
Let $Y$ be a variety over $\ZZ$. Suppose that there exists a nonempty open $T\subseteq \spec(\ZZ[1/2])$ such that $Y_T\cong Y(2)_T$ and put $U=S\cap T$. Then the following holds.
\begin{itemize}
\item[(i)]  Any point $P\in Y(S)$ satisfies $h_\phi(P)\leq c_1N_U^{e_1}.$
\item[(ii)] The cardinality of  
$Y(S)$ is at most $c_2N_U^{e_2}\Delta\log(3\Delta)^{2g-1}.$
\end{itemize}
\end{corollary}

In applications of Corollary~\ref{cor:y2}, such as for example the Corollary stated in the introduction, one usually chooses $T$ such that $Y(2)$ is smooth over $T$. Corollary~\ref{cor:y2emptybranchlocus} gives that $Y(2)$ is smooth over $\spec(\ZZ[\tfrac{1}{2\Delta}])$, and  if we take $T=\spec(\ZZ[\tfrac{1}{2\Delta}])$ then (the proof of) Corollary~\ref{cor:y2} provides that each point $P\in Y(S)$ satisfies $$h_\phi(P)\leq c_1(\Delta N_S)^{e_1} \quad \textnormal{and}\quad |Y(S)|\leq c_2(\Delta N_S)^{e_2}.$$
We now consider the coarse Siegel moduli space $A_{2,2}$ of principally polarized abelian surfaces with symplectic level two structure.  For certain $S$, Le Fourn~\cite{lefourn:A2rungeparis,lefourn:A2runge,lefourn:tubularbaker} and Box--Le Fourn~\cite{bofo:siegelmodulisintegral} proved strong explicit height bounds for the $S$-integral points of $A_{2,2}$ outside the divisor $Z$ of products of elliptic curves.  Their results hold for arbitrary number fields, and their proofs use and refine Levin's approaches via Runge's method \cite{levin:intpointsrunge,levin:extendrunge} or via the theory of logarithmic forms \cite{bawu:logarithmicforms,levin:intpointslogforms}. There are many coarse Hilbert moduli schemes $Y$ which admit a finite morphism $Y\to A_{2,2}$. 
However, usually the image of $Y\to  A_{2,2}$ intersects the divisor $Z$ and one can not directly deduce Diophantine results for $Y$ via the variety $A_{2,2}\setminus Z$. For example, for certain $Y$ related to $Y(2)$, forgetting $\iota:\OL\to \End(A)$ induces a canonical finite morphism $Y\to A_{2,2}$ whose image intersects the divisor $Z$.  On the other hand, Corollary~\ref{cor:y2} gives no information for $S$-integral points of $A_{2,2}\setminus Z$ defining abelian varieties $A$ with $\End(A)=\ZZ$.

\newpage

\section{Moduli problems and coarse moduli spaces}\label{sec:cms}

Let $\mathcal M$ be a Hilbert moduli stack. Throughout this section we work over an arbitrary nonempty open subscheme $B\subseteq \spec(\ZZ)$, and   we write throughout $\mathcal M$ for $\mathcal M_B$ to simplify notation. We now begin to prove some basic properties of moduli problems on $\mathcal M$ and their coarse moduli spaces over $B$ which shall be used in our proofs below.

\paragraph{Products.} Let $\mathcal P$ and $\mathcal Q$ be moduli problems on $\mathcal M$, and let $\mathcal P\times\mathcal Q$ be the product presheaf on $\mathcal M$. For several computations we shall need to explicitly identify $\cmp\times_{\mathcal M}\mathcal M_{\mathcal Q}$ with $\mathcal M_{\mathcal P\times\mathcal Q}$.  On using the description of the fiber product in \cite[0040]{sp}, we see that one obtains an equivalence of categories fibered in groupoids over $(\textnormal{Sch}/B)$
\begin{equation}\label{eq:productequivalence}
\cmp\times_{\mathcal M}\mathcal M_{\mathcal Q}\isomto \mathcal M_{\mathcal P\times\mathcal Q}
\end{equation}
by sending an object $(U,(x,\alpha),(y,\beta),f)$ to the object $(x,\alpha,\mathcal Q(f)\beta)$ and by sending a morphism $(a,b)$ to the morphism determined by $a$.
It turns out that one obtains an inverse $\mathcal M_{\mathcal P\times\mathcal Q}\isomto\cmp\times_{\mathcal M}\mathcal M_{\mathcal Q}$ to \eqref{eq:productequivalence} by sending an object $(x,\alpha,\beta)$ to the object $(U,(x,\alpha),(x,\beta),\textnormal{id})$ and by sending a morphism $f$ to the morphism determined by $f$. This inverse and \eqref{eq:productequivalence} are both morphisms of categories over $\cmp$, where $\cmp\times_{\mathcal M}\mathcal M_{\mathcal Q}\to \cmp$ is the first projection and  $\mathcal M_{\mathcal P\times \mathcal Q}\to \cmp$ is the forgetful morphism $(x,\alpha,\beta)\mapsto (x,\alpha)$.

\paragraph{Group action on a moduli problem.}Let $\mathcal P$ be a moduli problem on $\mathcal M$ and let $G$ be a group. An action of $G$ on $\mathcal P$ is a morphism of presheaves $G\times \mathcal P\to \mathcal P$ such that $G\times \mathcal P(x)\to \mathcal P(x)$ defines a $G$-action on $\mathcal P(x)$ for each $x\in \mathcal M$. Here we view $G$ as a constant presheaf on $\mathcal M$. 
Suppose that $G$ acts on $\mathcal P$. Then for each $g\in G$ one obtains an automorphism $\tau_g$ of $\cmp$ by  sending an object $(x,\alpha)$ to the object $(x,g\alpha)$ and by sending a morphism $(x,\alpha)\to (x',\alpha')$ to the morphism $(x,g\alpha)\to (x',g\alpha')$  which lies over the same morphism $x\to x'$. Furthermore, the automorphism $\tau_g$ of $\cmp$ is compatible with the forgetful morphism $\cmp\to \mathcal M$ and we observe that
\begin{equation}\label{def:gactionmp}
G\to \Aut(\cmp/\mathcal M), \quad g\mapsto \tau_g,
\end{equation}
is compatible with the group structures on $G$ and $\Aut(\cmp/\mathcal M)$ 
where $\Aut(\cdot)$ is taken in the (2,1)-category of categories fibered in groupoids over $\mathcal M$. 
We now consider the case when $\mathcal M=\mathcal M^I$ is associated to $(L,I,I_+)$,  $\mathcal P=\mathcal P(n)$ is the moduli problem of principal level $n$-structures for some $n\in\ZZ_{\geq 1}$, and $G=\GL_2(\OL/n\OL)$ for $\OL$ the ring of integers  of $L$. In this case we obtain an action of $G$ on $\mathcal P$ by defining
\begin{equation}\label{def:gactionpn}
G\times \mathcal P(x)\to\mathcal P(x), \quad (g,\alpha)\mapsto \alpha g^{-1}.
\end{equation}
Here $x=(A,\iota,\varphi)$ lies in $\mathcal M(S)$ for some $B$-scheme $S$ and $\alpha g^{-1}$ denotes the composition of $\alpha:(\OL/n\OL)^2_S\isomto A_n$ with the automorphism of $(\OL/n\OL)^2_S$ defined by $g^{-1}\in \GL_2(\OL/n\OL)$. In what follows, the action of $\GL_2(\OL/n\OL)$ on $\mathcal P(n)$ always refers to \eqref{def:gactionpn}.

\paragraph{Some geometric properties.} Let $Y$ be a coarse moduli scheme over $B$ of an arithmetic moduli problem $\mathcal P$ on $\mathcal M$. Suppose that $\pi: \cmp\to Y$ is an initial morphism. Then $\pi$ has the following properties:
It is a universal homeomorphism, it is proper and quasi-finite, and its formation is compatible with flat base change $Y'\to Y$ of algebraic spaces over $B$.
See for example Conrad's approach to the Keel--Mori theorem~\cite{kemo:coarse} via stacks~\cite{conrad:coarse}. 
To this end, we notice that the separated finite type DM stack $\cmp$ over $B$ satisfies all assumptions made in \cite[Thm 1.1]{conrad:coarse}: The diagonal of $\cmp$ over $B$ is finite since the separated stack $\cmp$ over $B$ is quasi-DM, 
and therefore the inertia stack of $\cmp$ is finite over $\cmp$. 
(Here we used that a morphism of algebraic spaces is finite if it is proper and locally quasi-finite: This holds for schemes and hence for algebraic spaces by \cite[03XX]{sp}.)
Furthermore, as $\cmp$ is separated finite type over $B$ and $\pi$ is surjective,  it follows from \cite[Thm 1.1]{conrad:coarse} that $Y$ is a variety over $B$.


\paragraph{Galois action on algebraic points.} Let $Y$ be a coarse moduli scheme over $B$ of an arithmetic moduli problem $\mathcal P$ on $\mathcal M$. Then $Y$ is automatically a variety over $B$. The absolute Galois group $G=\textnormal{Aut}(\bar{\QQ}/\QQ)$ acts on the $\bar{\QQ}$-points of $Y$ by precomposition. 
This induces an action of $G$ on the set of isomorphism classes $[\cmp(\bar{\QQ})]$ via transport of structure involving the bijection $[\cmp(\bar{\QQ})]\isomto Y(\bar{\QQ})$ defined by an initial morphism $\pi:\cmp\to Y$. Then, on using the canonical identification $Y(\QQ)\isomto Y(\bar{\QQ})^G$, we obtain 
  \begin{equation}\label{eq:ginvarianceisoclasses}
Y(\QQ)\isomto [\cmp(\bar{\QQ})]^{G}.
\end{equation}
Further it turns out that $[\cmp(\bar{\QQ})]^G$ is the set of isomorphism classes of $(x,\alpha)$ in $\cmp(\bar{\QQ})$ such that for each $\sigma\in G$ the object $(x,\alpha)$ is isomorphic to its pullback  $\sigma^*(x,\alpha)$ by the automorphism $\sigma$ of $\bar{\QQ}$. Indeed this follows from a formal computation using that the bijection $[\cmp(\bar{\QQ})]\isomto Y(\bar{\QQ})$ is induced by $\pi$ which is a morphism of categories over $(\textnormal{Sch}/B)$ and that the pullback of an object is unique up to isomorphism. 

\paragraph{Compatibility with base change.} The formation of the height of coarse moduli schemes is compatible with base change to a nonempty open $T\subseteq B$. More precisely let $Y$ be a coarse moduli scheme over $B$ of an arithmetic moduli problem $\mathcal P$ on $\mathcal M$ with an initial morphism $\pi:\cmp\to Y$ and associated height $h_\phi$, and denote by $\mathcal P'$ the restriction of $\mathcal P$ to the open substack $\mathcal M_T$ of $\mathcal M$.  Then $Y_T$ is a coarse moduli scheme over $T$ of $\mathcal P'$ with initial morphism $\pi_T: (\mathcal M_T)_{\mathcal P'}\to Y_T$ and associated height $h_{\phi'}$ such that   
\begin{equation}\label{eq:heightrestriction}
h_{\phi'}=\iota^*h_\phi.
\end{equation}
  
\noindent Here $\iota:Y_T\hookrightarrow Y$ is the projection and we identified $(\mathcal M_T)_{\mathcal P'}$ with $\cmp\times_{B} T$ by sending an object $(x,\alpha)$ to the object $(U,(x,\alpha),U\to T,\textnormal{id})$ and by sending a morphism $f$ to the morphism $(f,p(f))$ where $U=p(x,\alpha)$ and $p:\cmp\to B$. We now prove the statements in \eqref{eq:heightrestriction}. The formation of coarse moduli spaces commutes with flat base change and $\iota$ is an open immersion. Thus the base change $\pi_T$ of $\pi$ is indeed an initial morphism. Then we see that \eqref{eq:heightrestriction} follows from the (functorial) definitions of $h_\phi$ and $h_{\phi'}$, since the above described identification $(\mathcal M_T)_{\mathcal P'}\isomto\cmp\times_{B} T$ is an isomorphism of categories over $\mathcal M$.

\paragraph{Extension.}Let $T\subseteq B$ be nonempty open  and let $Y$ be a coarse moduli scheme over $T$ of an algebraic moduli problem $\mathcal P$ on $\mathcal M_T$ with initial morphism $\pi:(\mathcal M_T)_{\mathcal P}\to Y$. One can always naively extend $Y$ to a coarse moduli scheme over $B$ as follows:  Let $\mathcal P'$ be the presheaf on $\mathcal M$ which sends an object $x$ to the set $\mathcal P(x)$ if $x$ lies in the open substack $\mathcal M_T\subseteq \mathcal M$ and to the empty set otherwise.  Then the composition
\begin{equation}\label{eq:naiveextension}
\pi':\mathcal M_{\mathcal P'}\isomto (\mathcal M_T)_{\mathcal P}\to^\pi Y
\end{equation}
is initial among all morphisms to an algebraic space $Z$ over $B$, since $\pi$ is initial over $T$ and since any morphism $\mathcal M_{\mathcal P'}\to Z$ factors uniquely through the open immersion $Z_T\hookrightarrow Z$.   Here  $\mathcal M_{\mathcal P'}\isomto (\mathcal M_T)_{\mathcal P}$ is the isomorphism of categories over $(\textnormal{Sch}/B)$ given by the identity functor.  In particular $\mathcal P'$ is algebraic, and $Y$ is a coarse moduli scheme of $\mathcal P'$ with initial morphism $\pi'$ and branch locus $B_{\mathcal P'}=B_{\mathcal P}$.  Furthermore $\mathcal P'$ is arithmetic if $\mathcal P$ is arithmetic.

\subsection{Construction of covers}\label{sec:coverconstruction}

We continue our notation and terminology. In particular we continue to work over an arbitrary nonempty open subscheme $B\subseteq \spec(\ZZ)$ and we write $\mathcal M=\mathcal M_B$. 

Let $Y$ be a coarse moduli scheme over $B$ of an arithmetic moduli problem $\mathcal P$ on  $\mathcal M$. 
To construct various covers,  let $\pi:\cmp\to Y$ be an initial morphism and take a $B$-scheme $Z$ with a surjective morphism $Z\to \mathcal M$. Now, we consider the algebraic stack over $B$ 
$$Y'=\cmp\times_{\mathcal M}Z, \quad \textnormal{ and } \quad Y'\to\cmp\to^\pi Y$$ where $Y'\to\cmp$ is the projection. The composition $Y'\to Y$ is surjective and $Y'$ is an algebraic space over $B$  as we shall see in the proof of the following lemma. 

\begin{lemma}\label{lem:y'rep}If $Z\to \mathcal M$ is finite, then $Y'\to Y$ is finite surjective and $Y'$ is a scheme.
\end{lemma}
\begin{proof}
We first observe that the morphism $Y'\to Y$ is always surjective. Indeed it factors as $Y'\to \cmp\to^\pi Y$ where the projection $Y'\to \cmp$  is a base change of the surjective $Z\to \mathcal M$ and the initial morphism $\pi:\cmp \to Y$ is a homeomorphism. 

The forgetful morphism $\cmp\to \mathcal M$ might not be representable in algebraic spaces. To prove that the algebraic stack $Y'$ is an algebraic space, it suffices by \cite[04SZ]{sp} to show that any object of $Y'$ has trivial automorphism group. 
This is a direct computation: Let $U$ be a scheme,  let $y=(U,(x,\alpha),z,f)$ be an object of the fiber product $Y'=\cmp\times_{\mathcal M} Z$ and let $g=(a,b)$ be in $\Aut_{Y'(U)}(y)$. We obtain that $b=1$ since any object of $Z(U)$ has trivial automorphism group. This together with $fa'=b'f$ implies that $a'=f^{-1}b'f=1$, where $a'$ and $b'$ are the images of $a$ and $b$ under $\cmp\to \mathcal M$ and $Z\to \mathcal M$ respectively. It follows that $a=1$ since any automorphism of $(x,\alpha)$ is determined by its underlying automorphism of $x$, and thus $g=(a,b)=1$. We conclude that $Y'$ is an algebraic space. 

Now, we assume that the morphism $Z\to \mathcal M$ is finite. Then its base change $Y'\to \cmp$ is finite. 
Further, the moduli problem $\mathcal P$ is arithmetic by assumption and therefore the initial morphism $\pi:\cmp\to Y$ is proper and quasi-finite. Thus the composition $Y'\to \cmp\to^\pi Y$ is separated and locally quasi-finite, and then \cite[03XX]{sp} implies that the algebraic space $Y'$ is a scheme since $Y$ is a scheme. This completes the proof.\end{proof}

\paragraph{Finite \'etale cover of $Y\setminus B_{\mathcal P}$.} We now consider the open subscheme $U=Y\setminus B_{\mathcal P}$ of $Y$ where $B_{\mathcal P}$ is the branch locus of $\mathcal P$. Let $S$ be a connected Dedekind scheme with function field $k$. For any $P\in U(S)$ we obtain  cartesian squares
$$
\xymatrix@R=4em@C=4em{
S\times_Y Y' \ar[r] \ar[d] & Y'_U \ar@{^{(}->}[r] \ar[d] & Y'\ar[d] \\
S \ar[r]^P &  U \ar@{^{(}->}[r] &  Y.} 
$$
If the surjective morphism $Z\to \mathcal M$ is finite, then we get here a diagram of noetherian schemes with all vertical arrows finite surjective. This follows from Lemma~\ref{lem:y'rep} and our assumption that $\mathcal P$ is arithmetic which assures that $Y=M_{\mathcal P}$ is noetherian.

\begin{lemma}\label{lem:etcover}
Suppose that the surjective morphism $Z\to \mathcal M$ is finite \'etale, and let $T$ be a connected component of $S\times_{Y} Y'$. Then the following statements hold.
\begin{itemize}
\item[(i)] The morphisms $Y'_U\to U$ and  $T\to  S$ are finite \'etale.
\item[(ii)] The scheme $T$ is a connected Dedekind scheme.
\item[(iii)] If in addition $Z$ is a moduli scheme over $B$ of some moduli problem $\mathcal Q$ on $\mathcal M$ and $Z\to \mathcal M$ is up to an equivalence $Z\isomto \mathcal M_{\mathcal Q}$ the forgetful morphism, then $$[k(T):k]\leq |\mathcal Q|_{\bar{k}}.$$
\end{itemize}

\end{lemma}
\begin{proof}
The surjective morphism $Z\to \mathcal M$ is finite by assumption. Thus, as already explained, the above displayed cartesian squares form a diagram of noetherian schemes with all vertical arrows finite surjective. In particular the scheme $S'=S\times_Y Y'$ is locally noetherian and hence the flat closed immersion $T\hookrightarrow S'$, induced by the canonical scheme structure on the connected component $T$ of $S'$, is an open immersion.

We now show (i). The above discussion shows that the morphisms $Y'_U\to U$ and $T\hookrightarrow S'\to S$ are all finite. To see that they are \'etale, we denote by $f$ the morphism $Y'\to Y$ and we let $U'\subseteq Y'$ be the \'etale locus of $f$. As $Z\to\mathcal M$ is finite \'etale by assumption, we can apply \eqref{eq:branchlocuscompgeneral} below which gives that $B_{\mathcal P}$ is precisely the branch locus $f(Y'\setminus U')$ of $f$. This together with $U=Y\setminus B_{\mathcal P}$ shows that  $f^{-1}(U)\subseteq U'$ 
and hence $Y'_U\to U$ is \'etale. 
Thus the base change $S'\to S$ of $Y'_U\to U$ is also \'etale and then the composition $T\hookrightarrow S'\to S$ is \'etale since $T\hookrightarrow S'$ is an open immersion. This proves (i).

We next show (ii). An application of (i) gives that $T$ is finite \'etale, and thus smooth of relative dimension zero, over the regular noetherian scheme $S$ of dimension one. This implies that $T$ is again a regular noetherian scheme of dimension one. In other words the connected component $T$ is a connected Dedekind scheme as claimed in (ii).

To show (iii) we relate the degree of various morphisms to the number of points in the fiber $Y'_{\bar{s}}$ of the morphism $Y'_U\to U$ over the geometric point $u:\spec(\bar{k})\to S\to^P U$. For this purpose we use that $Y'_U\to U$ and $T\to S$ are finite \'etale by (i) and that the (locally constant) degree of a finite and locally free morphism is stable under arbitrary base change. 
Now, the degree $[k(T):k]$ of the function field $k(T)$ of $T$ over the function field $k=k(S)$ of $S$ coincides with the degree of $T\to S$ 
which is at most the degree of $S'\to S$. 
Let $n$ be the degree of $Y'_U\to U$ over the connected component $V$ of $U$ containing $u$.  The morphism $P:S\to U$ factors as $S\to V\hookrightarrow U$, since $S$ is connected and $V$ is open in the noetherian $U\subseteq Y$. 
Then, on using that the degree is stable under base change, we obtain that the degree of $S'\to S$ equals $n$  and that $n$ coincides with the degree $[Y'_{u}:\bar{k}]$ of the projection  $Y'_{u}\to \spec(\bar{k})$. 
Combining everything leads to
\begin{equation}\label{eq:degreeboundintermsoflvlstructures}
[k(T):k]\leq n=[Y'_{u}:\bar{k}].
\end{equation}
We now assume in addition that $Z$ is a moduli scheme over $B$ of some moduli problem $\mathcal Q$ on $\mathcal M$ and that $Z\to \mathcal M$ is the composition of an equivalence $Z\isomto \mathcal M_{\mathcal Q}$ with the forgetful morphism $\mathcal M_{\mathcal Q}\to \mathcal M$. To bound $[Y'_{u}:\bar{k}]$ in terms of $\mathcal Q$,  we use that the degree  $[Y'_{u}:\bar{k}]$ coincides with the cardinality of the set of sections of the finite \'etale morphism $Y'_{u}\to\spec(\bar{k})$. This set of sections can be computed as follows. By construction the equivalence $\mathcal M_{\mathcal P\times \mathcal Q}\isomto Y'$ defined right after \eqref{eq:productequivalence} is a morphism of categories over $\cmp$. Then, after recalling that $Y'\to Y$ factors as $Y'\to \cmp \to^\pi Y$, we see that taking isomorphism classes of objects over $\bar{k}$ gives a commutative diagram 
\begin{equation}\label{diag:fibercomp}
\xymatrix{
[\mathcal M_{\mathcal P\times\mathcal Q}(\bar{k})] \ar[r]^{\ \ \ \, \sim} \ar[d] & Y'(\bar{k}) \ar[d]\\
[\cmp(\bar{k})] \ar[r]^{ \ \, \sim} & Y(\bar{k}). 
}
\end{equation}
Now, the set of sections of $Y'_{u}\to \spec(\bar{k})$ identifies with the fiber of $Y'(\bar{k})\to Y(\bar{k})$ over $u\in Y(\bar{k})$ 
which in turn identifies via the above diagram with the set of isomorphism classes $[(x,\alpha,\beta)]$ in $[\mathcal M_{\mathcal P\times\mathcal Q}(\bar{k})]$ such that $\pi([(x,\alpha)])=u$. The number of these isomorphism classes is at most
$|\mathcal Q(x)|$ since they are generated by $(x,\alpha,\beta)$ with $\beta\in \mathcal Q(x)$. (Here the number of these isomorphism classes can be strictly smaller than $|\mathcal Q(x)|$; for example this is the case when $(x,\alpha)$ has a non-trivial automorphism.)  
Finally, on putting everything together, we conclude that the degree $[Y'_{u}:\bar{k}]$ is at most $\sup_{x\in \mathcal M(\bar{k})}|\mathcal Q(x)|=|\mathcal Q|_{\bar{k}}$ and then  \eqref{eq:degreeboundintermsoflvlstructures} implies (iii). This completes the proof of the lemma. \end{proof}

\paragraph{Branch locus.}Suppose that the surjective morphism $Z\to \mathcal M$ is finite \'etale. Then we can compute the branch locus $B_{\mathcal P}$ of $\mathcal P$ in terms of the branch locus of $Y'\to Y$ which is defined as  $B_f=f(Y'\setminus U')$, where $f$ denotes the morphism $Y'\to Y$ and  $U'$ is the union of all open subspaces of $Y'$ over which $f$ is \'etale. It turns out that
\begin{equation}\label{eq:branchlocuscompgeneral}
B_{\mathcal P}=B_f.
\end{equation}
To prove this equality, we denote by $U_{\mathcal P}$ and $U$ the complement in $Y$ of $B_{\mathcal P}$ and $B_f$ respectively. The base change $Y'\to \cmp$ of $Z\to \mathcal M$ is again \'etale surjective, and Lemma~\ref{lem:y'rep} gives that $f$ is a finite surjective morphism of schemes. In particular  $U\subseteq Y$ is open and $f_{U}$ is \'etale since $f^{-1}(U)\subseteq U'$. 
Therefore, as $f_U$ factors as $Y'_U\to (\cmp)_U\to^{\pi_U} U$ with $Y'_U\to (\cmp)_U$ \'etale surjective, we obtain that $\pi_{U}$ is \'etale and thus $U\subseteq U_{\mathcal P}$. 
To prove the converse inclusion, we take an open $V\subseteq Y$ such that $\pi_V$ is \'etale. Thus $f_V$, which is the composition of $\pi_V$ with the \'etale $Y'_V\to (\cmp)_V$, is also \'etale and hence $f^{-1}(V)\subseteq U'$. 
Then, on writing  $V=Y\setminus f(Y'\setminus f^{-1}(V))$ which exploits that $f$ is surjective,  we deduce that $V\subseteq Y\setminus f(Y'\setminus U')=U$ and thus $U_{\mathcal P}\subseteq U$. This completes the proof of \eqref{eq:branchlocuscompgeneral}.

Further, we remark here that the formation of the branch locus is compatible with the base change to a nonempty open subscheme $T\subseteq B$: The branch locus $B_{\mathcal P'}$ of the coarse moduli scheme $Y_T$ over $T$ of the restriction $\mathcal P'$ of $\mathcal P$ to $\mathcal M_T$ is given by 
\begin{equation}\label{eq:branchlocusrestriction}
B_{\mathcal P'}=(B_{\mathcal P})_T.
\end{equation} 
This follows for example from the arguments of the above proof of \eqref{eq:branchlocuscompgeneral}  after recalling from \eqref{eq:heightrestriction} that the flat base change $\pi_T$ of $\pi$ is an initial morphism $(\mathcal M_T)_{\mathcal P'}\to Y_T$.

\paragraph{Height on $Y'$.} Suppose that $Z$ is a moduli scheme over $B$ of some moduli problem $\mathcal Q$ on $\mathcal M$, and assume that the surjective morphism $Z\to \mathcal M$ is finite and is the composition of an equivalence $Z\isomto \mathcal M_{\mathcal Q}$ with the forgetful morphism $\mathcal M_{\mathcal Q}\to \mathcal M$. Then Lemma~\ref{lem:y'rep} and \eqref{eq:productequivalence} give that $Y'$ is a moduli scheme over $B$ of the product moduli problem $\mathcal P\times \mathcal Q$ on $\mathcal M$. Now, let $h_{\phi'}$ be the height on $Y'$ defined in \eqref{def:cheight} with respect to the inverse of the equivalence $Y'\isomto \mathcal M_{\mathcal P\times \mathcal Q}$ described right after \eqref{eq:productequivalence}. 
A formal computation shows
\begin{equation}\label{eq:heightcomp}
h_{\phi'}=f^*h_\phi
\end{equation}
where $f$ denotes the morphism $Y'\to Y$ and $h_\phi$ is the height on  $Y$ defined in \eqref{def:cheight} with respect to $\pi:\cmp\to Y$. Here we used that the equivalence $Y'\isomto\mathcal M_{\mathcal P\times \mathcal Q}$ given in \eqref{eq:productequivalence} is by construction a morphism of categories over $\cmp$. 

\subsection{Automorphism groups of objects of Hilbert moduli stacks}

Let $\mathcal M^I$ be the Hilbert moduli stack associated to some $(L,I,I_+)$ and let $S$ be a connected scheme. We denote by $\mathcal C$ the category of $\OL$-abelian schemes over $S$ for $\OL$ the ring of integers of the totally real field $L$. For any $\OL$-abelian scheme $B$ over $S$ we write $\End^0_{\mathcal C}(B)=\End_{\mathcal C}(B)\otimes_\ZZ\QQ$. The goal of this section is to prove the following lemma.

\begin{lemma}\label{lem:autofield}
Let $\sigma$ be an automorphism of an object  $(A,\iota,\varphi)$ of $\mathcal M^I(S)$, and denote by $L(\sigma)$ the $\QQ$-subalgebra of $\End^0_{\mathcal C}(A)$ generated by $\sigma$ and $\iota(\OL)$. Then $L(\sigma)$ is a field which is isomorphic either to $L$ or to a CM field with maximal totally real subfield $L$.  
\end{lemma}
At the end of this section, we shall deduce this result from the next Lemma~\ref{lem:divalg}. We remark that in the case when $S$ has a point whose residue field is of characteristic zero, a refinement of Lemma~\ref{lem:divalg}~(ii) will be given in Lemma~\ref{lem:divalgcompchar0} below.

\begin{lemma}\label{lem:divalg}
Suppose that $A$ and $A'$ are $\OL$-abelian schemes over $S$ which are both of relative dimension $g=[L:\QQ]$. Then the following statements hold. 
\begin{itemize}
\item[(i)] Any nonzero element of $\Hom_{\mathcal C}(A,A')$ is an isogeny.
\item[(ii)] The $\QQ$-algebra $\End^0_{\mathcal C}(A)$ is finite and a division algebra.
\end{itemize}

\end{lemma}
In the case when $S$ is in addition of finite type over $\ZZ$, Lemma~\ref{lem:divalg}~(i) can be found in van der Geer's book~\cite[X.1.6]{vandergeer:hilbertmodular}. His arguments  rely on a trick of Drinfeld and they are completely different to our proof which exploits that $L$ has a real embedding.

\begin{proof}[Proof of Lemma~\ref{lem:divalg}]
 To prove (i) we take a nonzero $f\in \Hom_{\mathcal C}(A,A')$. For each point $s$ of the connected scheme $S$, the base change $f_s:A_s\to A'_s$ is a morphism of $\OL$-abelian varieties of dimension $g$ which is again nonzero by rigidity.  Thus, to show that $f$ is an isogeny, we may and do assume that $S$ is the spectrum of a field $k$. Then the image $B=\textnormal{Im}(f)$ is an abelian subvariety of $A$. As $f$ is compatible with the $\OL$-multiplications, we obtain a ring morphism $\OL\to \End(B)$ which induces $L\hookrightarrow\End^0(B)$.  
 Next we use the following basic result 
 (\cite[1.3.1]{lang:cm}): (1) Let $C$ be an abelian variety over $k$  and let $F$ be a number field contained in $\End^0(C)$. If $[F:\QQ]>\dim(C)$ then $F$ is totally imaginary with $[F:\QQ]=2\dim(C)$.    Applying this with $C=B$ and $F=L$ gives that $\dim(A')=g=[L:\QQ]$ is at most $\dim(B)$ since $L$ is totally real. Thus the abelian subvariety  $B$ of $A'$ has the same dimension as $A'$. This implies that $A'=B=\textnormal{Im}(f)$. As $A$ and $A'$ have the same dimension, it follows that $f$ is an isogeny. This proves (i).

We now show (ii). To reduce to the case when $S$ is the spectrum of a field, we take $s\in S$ and we denote by $\mathcal C_s$ the category of $\OL$-abelian varieties over $k(s)$. Base change from our connected scheme $S$ to $k(s)$  induces a ring morphism $\End^0_{\mathcal C}(A) \to \End^0_{\mathcal C_s}(A_{s})$ which is injective by rigidity.  Thus the $\QQ$-algebras $\End^0_{\mathcal C}(A)$ and $\End^0_{\mathcal C_s}(A_{s})$ are finite, since they embed into the finite $\QQ$-algebra $\End^0(A_s)$. Next we observe: (2) If $R$ is a not necessarily commutative $\QQ$-algebra which is finite and a division algebra, then any $\QQ$-subalgebra $R'$ of $R$ is a division algebra. (Indeed multiplication with a nonzero $r\in R'$ defines a $\QQ$-module morphism $m_r:R'\to R'$ which is injective since $R$ is a division algebra and hence $m_r$ is surjective since $R'$ is finite over $\QQ$.)   An application of (2) with $R'$ the image of $\End^0_{\mathcal C}(A)$ in $R=\End^0_{\mathcal C_s}(A_{s})$ shows that we may and do assume that $S$ is the spectrum of a field $k$. Then any nonzero $f \in \End_{\mathcal C}(A)$ is an isogeny of abelian varieties over $k$ by (i). Hence there is an isogeny $f':A\to A$ such that $ff'=[d]=f'f$ for $d$ the degree of $f$.  It follows that $f'$ lies in $\End_{\mathcal C}(A)$ and $g=f'\otimes d^{-1}$ is an inverse of $f\otimes 1$ in $\End^0_{\mathcal C}(A)$. As any element of $\End^0_{\mathcal C}(A)$ is of the form $f\otimes q$ with $f\in \End_{\mathcal C}(A)$ and $q\in \QQ$, we deduce that $\End^0_{\mathcal C}(A)$ is a division algebra as claimed in (ii). This completes the proof.
\end{proof}

\begin{proof}[Proof of Lemma~\ref{lem:autofield}]
Let $\sigma$ be an automorphism of an object $(A,\iota,\varphi)$ of $\mathcal M^I(S)$ and recall that  $L(\sigma)$ denotes the $\QQ$-subalgebra of $\End^0_{\mathcal C}(A)$ generated by $\sigma$ and $\iota(\OL)$. We first show that $L(\sigma)$ is a number field.  Lemma~\ref{lem:divalg}~(ii) gives that the $\QQ$-algebra  $\End^0_{\mathcal C}(A)$ is finite and a division algebra.  Thus its $\QQ$-subalgebra $L(\sigma)$ is also a division algebra by (2) from the proof of Lemma~\ref{lem:divalg}~(ii). Furthermore, on using that  $\iota(\OL)$ is contained in the center of $\End_{\mathcal C}(A)$, we see that  $L(\sigma)$  is commutative. Hence $L(\sigma)$ is a number field. 

To compute the field $L(\sigma)$, we write $g=[L:\QQ]$ and we take $s\in S$. Now we  apply (1) from the proof of Lemma~\ref{lem:divalg}~(i) with $C=A_s$ and $F$ the image of $L(\sigma)\subseteq \End_{\mathcal C}^0(A)\subseteq \End^0(A)$ inside $\End^0(A_s)$. This shows that $L(\sigma)$ either has degree at most $\dim(A_s)=g$ or $L(\sigma)$ is totally imaginary of degree $2g$.  Therefore, as $L(\sigma)$ contains the totally real field $L'=\iota(\OL)\otimes_\ZZ\QQ$ of degree $[L':\QQ]=g$, we find that either $L(\sigma)=L'$ or $L(\sigma)$ is a CM field with maximal totally real subfield $L'$. Then we deduce Lemma~\ref{lem:autofield}.   
\end{proof}

\subsection{The branch locus and automorphism groups}\label{sec:branchauto}

We now use the above covers to study automorphism groups of points in the branch locus of coarse moduli schemes. Let $\mathcal M$ be a Hilbert moduli stack.  We continue to work over an arbitrary nonempty open subscheme $B\subseteq \spec(\ZZ)$ and we write again $\mathcal M=\mathcal M_B$. 

Let $Y$ be a coarse moduli scheme over $B$ of an arithmetic moduli problem $\mathcal P$ on $\mathcal M$, and let $\pi:\cmp\to Y$ be an initial morphism.   We now introduce some terminology. For any point $y\in Y$, we denote by $\Aut(y)$ the automorphism group algebraic space (\cite[0DTR]{sp}) of the point $\pi^{-1}(y)$ of $\cmp$  and the separable rank of $\Aut(y)$ is denoted by $$|\Aut(y)|.$$ This rank coincides with the number of geometric points of $\Aut(y)$ which is a finite group scheme over a field since the inertia stack of $\cmp$ is finite over $\cmp$. In particular $|\Aut(y)|$ is an invariant of the point $\pi^{-1}(y)$ of $\cmp$. Further, a geometric automorphism group of $\cmp$ is the automorphism group of an object defined by a geometric point of $\cmp$.

The main goal of this subsection is to prove the following result which says that the points in the branch locus $B_{\mathcal P}\subseteq Y$ of $\mathcal P$ have `large' automorphism groups if a global congruence condition on the geometric automorphism groups is satisfied.

\begin{proposition}\label{prop:auto}
Suppose that the forgetful morphism $\cmp\to \mathcal M$ is finite and that $\cmp$ is reduced. Further assume that each geometric automorphism group of $\cmp$ has even order. Then any point  $y\in B_{\mathcal P}$ satisfies $|\Aut(y)|\geq 4$.
\end{proposition}  
The first two assumptions are usually satisfied in situations of interest in arithmetic such as for example when $\cmp\to \mathcal M$ is finite \'etale,  while the assumption on the geometric automorphism groups is more restrictive. 
The proof given below shows that the assumptions can be relaxed up to a certain extent, see Remark~\ref{rem:assumptionsautogroupprop}.


The information provided by Proposition~\ref{prop:auto} can be useful to compute the branch locus $B_{\mathcal P}$. To illustrate this, we consider the coarse moduli scheme $Y(2)$ over $\ZZ[1/2]$ of the moduli problem $\mathcal P(2)$ of principal level 2-structures on $\mathcal M_{\ZZ[1/2]}$. Suppose that $\mathcal M\cong \mathcal M^I$ is associated to $(L,I,I_+)$ and write $\Delta$ for the discriminant of $L/\QQ$.

\begin{corollary}\label{cor:y2emptybranchlocus}
 The branch locus of $Y(2)$ is empty, and $Y(2)$ is smooth over $\ZZ[\tfrac{1}{2\Delta}]$.
\end{corollary}
\begin{proof}
We now take $B=\spec(\ZZ[1/2])$ and $\mathcal P=\mathcal P(2)$. Then $\cmp$ satisfies the first assumption in Proposition~\ref{prop:auto}. Indeed the forgetful morphism $\cmp\to \mathcal M$ is finite \'etale, which implies that $\cmp$ is reduced since $\mathcal M$ is reduced.  We now compute the geometric automorphism groups of $\cmp$. Let $\sigma$ be an automorphism of an object $(x,\alpha)\in \cmp(k)$ where $\spec(k)\to B$ is a geometric point. After recalling that $\mathcal M\cong \mathcal M^I$ is associated to $(L,I,I_+)$, we identify $x$ with a triple $(A,\iota,\varphi)\in \mathcal M^I(k)$. Then $\sigma$ is induced by an automorphism of the abelian variety $A$ over $k$ whose restriction $\sigma_2$ to $A_2$ satisfies $\sigma_2\alpha=\alpha$ as morphisms $(\OL/2\OL)_k^2\to A_2$,  where $\OL$ is the ring of integers of $L$. It follows that $\sigma_2$ is the identity since $\alpha$ is an isomorphism. Thus Serre's lemma \cite[Lem 4.7.1]{grra:neronmodels} shows that either $\sigma$ is the identity or $\sigma$ has order two. Hence Lemma~\ref{lem:autofield} assures that $\sigma$ is given by $[\pm 1]$. Thus the automorphism group of $(x,\alpha)\in\cmp(k)$ has order two, since $[-1]$ defines an automorphism of $(x,\alpha)$. Now, as all geometric automorphism groups of $\cmp$ have (even) order two, we can apply Proposition~\ref{prop:auto} which implies that $B_{\mathcal P}$ is empty.

We now deduce the second part. The stack $\cmp$ has an \'etale scheme cover,  the initial morphism  $\pi:\cmp\to Y(2)$ is an \'etale cover since $B_{\mathcal P}$ is empty, and $\cmp$ is smooth over $\ZZ[\tfrac{1}{2\Delta}]$ since the forgetful morphism $\cmp\to \mathcal M$ is \'etale and since $\mathcal M$  identifies over $\ZZ[\tfrac{1}{2\Delta}]$ with the smooth stack of Rapoport~\cite[Thm 1.20]{rapoport:hilbertmodular}.  As being smooth is \'etale local on source and target, we then deduce that $Y(2)$ is smooth over $\ZZ[\tfrac{1}{2\Delta}]$.  
\end{proof}
In the remaining of this section we prove Proposition~\ref{prop:auto}. The main ideas are as follows: After identifying $B_{\mathcal P}$ with the branch locus of a cover $Y'\to Y$ where $Y'=\cmp\times_{\mathcal M}\mathcal M_{\mathcal P(n)}$ with $n\in\ZZ_{\geq 3}$, we show that $Y'\to Y=Y'/G$ is a categorical quotient for the action of $G=\GL_2(\OL/n\OL)$ on $Y'$. To understand the \'etale locus of $Y'\to Y'/G$, we then compute in \eqref{eq:isodecompgroup} the stabilizer groups in terms of automorphism groups of $\cmp$ and we exploit our assumption on the geometric automorphism groups, see also Remark~\ref{rem:assumptionsautogroupprop}.   

\begin{proof}[Proof of Proposition~\ref{prop:auto}]
Suppose that the forgetful morphism $\cmp\to \mathcal M$ is finite, that $\cmp$ is reduced and that each geometric automorphism group of $\cmp$ has even order. Then Proposition~\ref{prop:auto} is equivalent to the statement that any point $y\in Y$ with $|\Aut(y)|\leq 3$ does not lie in the branch locus $B_{\mathcal P}\subseteq Y$ of $\mathcal P$. We now prove this statement.  
\paragraph{1.}In the first step we view $B_{\mathcal P}$ as the branch locus of a suitable scheme cover of $Y$ by applying the results of Section~\ref{sec:coverconstruction}. To get an \'etale scheme cover of $\mathcal M$ with a simple moduli interpretation, we now reduce to the situation in which $B$ is a scheme over $\ZZ[1/n]$ for some $n\in\ZZ_{\geq 3}$. This is possible since all involved constructions are compatible with localization on the base. The details of this reduction are as follows. There exist $n\in\ZZ_{\geq 3}$ and a point $y_T\in Y_{T}$ which becomes $y\in Y$ via the projection $Y_T\hookrightarrow Y$ for $T=B\times_\ZZ \ZZ[1/n]$. Further the discussion surrounding \eqref{eq:heightrestriction} shows that $Y_{T}$ is a coarse moduli scheme over $T$ of the arithmetic moduli problem $\mathcal P'$ on $\mathcal M_T$ given by the restriction  of $\mathcal P$ to the open substack $\mathcal M_{T}$ of $\mathcal M$. The flat base change $\pi_{T}:(\mathcal M_T)_{\mathcal P'}\to Y_T$ of $\pi$ is again an initial morphism, where we identified $(\mathcal M_T)_{\mathcal P'}$ with $\cmp\times_B T$. Hence we deduce that $|\Aut(y_T)|=|\Aut(y)|$ and that $(B_{\mathcal P})_T\subseteq B_{\mathcal P'}$ which assures that the point $y\in Y$ does not lie in $B_{\mathcal P}$ if the point $y_T\in Y_T$ does not lie in $B_{\mathcal P'}$. Therefore we may and do replace $B$ by the nonempty open subscheme $T\subseteq \spec(\ZZ)$ which is a scheme over $\ZZ[1/n]$ for some $n\in \ZZ_{\geq 3}$.

Now,  as   $\mathcal M$ is a stack over a $\ZZ[1/n]$-scheme $B$, the forgetful morphism $\mathcal M_{\mathcal Q}\to \mathcal M$ is a finite \'etale cover, where $\mathcal Q$ is the moduli problem of principal level $n$-structures on $\mathcal M$. Further the algebraic stack $\mathcal M_{\mathcal Q}$ is a scheme since $n\geq 3$. Then $Y'=\cmp\times_{\mathcal M}\mathcal M_{\mathcal Q}$ is a scheme and we obtain a finite surjective morphism of schemes $$f:Y'\to Y$$
which factors as $Y'\to \cmp\to^\pi Y$ with $Y'\to \cmp$ the projection. Indeed this follows from an application of Lemma~\ref{lem:y'rep} with the scheme $Z=\mathcal M_{\mathcal Q}$ and the forgetful morphism $Z\to \mathcal M$. Furthermore, as the forgetful morphism $\mathcal M_{\mathcal Q}\to \mathcal M$ is a finite \'etale cover, the branch locus $B_f\subseteq Y$ of the morphism $f$ satisfies
$B_{\mathcal P}=B_f$ by \eqref{eq:branchlocuscompgeneral}.

\paragraph{2.}In the next step we view the morphism $f:Y'\to Y$ as a categorical quotient. The scheme $Y'$ is quasi-projective over $B$, since $Z=\mathcal M_{\mathcal Q}$ is quasi-projective over $B$ and since $\cmp\to  \mathcal M$ is finite by assumption.
Further the projection $Y'\to \cmp$ is a finite locally free cover since $\mathcal M_{\mathcal Q}\to \mathcal M$ is a finite \'etale cover, and the fiber product $R=Y'\times_{\cmp}Y'$ is a scheme since it is a base change of $\mathcal M_{\mathcal Q}\to\mathcal M$ which is representable in schemes. Then it follows from \cite[Thm 3.1]{conrad:coarse} and the universal property of $\pi$ that $f:Y'\to Y$ is a categorical quotient in schemes over $B$ of the pre-equivalence relation $$R\to Y'\times_{B}Y'$$ defined by the groupoid in algebraic spaces $(Y',R,s,t,c)$ over $B$, where $s,t: R\to Y'$ correspond to the two projections and $c:R\times_{Y'}R\to R$ corresponds to the projection which comes from the canonical equivalence $R\times_{Y'}R\cong Y'\times_{\cmp} Y'\times_{\cmp} Y'$.

To describe the categorical quotient $f$ more explicitly, we may and do assume that the Hilbert moduli stack $\mathcal M=\mathcal M^I$ is associated to some $(L,I,I_+)$.  
Let $\OL$ be the ring of integers of $L$, and let $G$ be the constant group scheme over $B$ defined by the finite group $\GL_2(\OL/n\OL)$ which acts on $\mathcal M_{\mathcal P\times\mathcal Q}$ via \eqref{def:gactionmp} and \eqref{def:gactionpn}. Then we obtain an action of $G$ on the $B$-scheme $Y'$ by transport of structure using the equivalence $Y'\isomto \mathcal M_{\mathcal P\times\mathcal Q}$ in \eqref{eq:productequivalence} and the inverse defined right after \eqref{eq:productequivalence}. We now show that $f$ is a categorical quotient in schemes over $B$ for the induced morphism

  \begin{equation}\label{eq:gcategoricalquot}
G\times_B Y'\to Y'\times_B Y'.
\end{equation}
As $f$ is a categorical quotient in schemes over $B$ for $R\to Y'\times_B Y'$, it suffices  to prove for any morphism of $B$-schemes $Y'\to X$ that it is $R$-invariant if and only if it is $G$-invariant. It follows from (1) below that this is equivalent to the statement that for any geometric point $\spec(k)\to B$ the map $ Y'(k)\to X(k)$ is constant on weak $R$-orbits if and only if it is constant on weak $G$-orbits.  But this statement holds since the weak $R$-orbits are precisely the weak $G$-orbits by (2) below. Hence we conclude \eqref{eq:gcategoricalquot}. 

(1) Let $Y'\to X$ be a morphism of schemes over $B$. Then $Y'\to X$ is $R$-invariant (resp. $G$-invariant) if and only if for any geometric point $\spec(k)\to B$ the map $ Y'(k)\to X(k)$ is constant on weak $R$-orbits (resp.  weak $G$-orbits). This follows for example from the results in \cite[048M]{sp} which  require that the schemes $R$ and $G\times_B Y'$ are reduced with $R\rightrightarrows Y'$ and $G\times_B Y'\rightrightarrows Y'$ locally of finite type. These conditions are satisfied. Indeed, on using that $Y'\to \cmp$ is a base change of the finite \'etale $\mathcal M_{\mathcal Q}\to \mathcal M$ and that $G$ is finite \'etale over $B$, we see that $R$ and $G\times_B Y'$ are both finite  \'etale over $Y'$ which is finite \'etale over $\cmp$ and  by assumption $\cmp$ is reduced and of finite type over $B$.

(2) For any geometric point $\spec(k)\to B$, two $k$-points $u,u'\in Y'(k)$ are weakly $R$-equivalent if and only if they are weakly $G$-equivalent. This can be verified by a direct computation. Let $\phi:Y'\isomto \mathcal M_{\mathcal P\times\mathcal Q}$ be the equivalence in \eqref{eq:productequivalence}, let $\psi$ be the inverse defined right after \eqref{eq:productequivalence}, and write $\phi(u)=(x,\alpha,\beta)$ and $\phi(u')=(x',\alpha',\beta')$. The following computations crucially exploit  that $\phi$ and $\psi$ are morphisms of categories over $\cmp$. We first suppose that $u$ is weakly $R$-equivalent to $u'$. Then there exists $r\in R(k)$ with $s(r)=u$ and $t(r)=u'$. Hence $r=(u,u',\tau)$ for some isomorphism $\tau:(x,\alpha)\isomto (x',\alpha')$ in $\cmp(k)$ which induces an isomorphism $(x,\alpha,\beta^*)\isomto (x',\alpha',\beta')$ where $\beta^*\in \mathcal Q(x)$.  Then we compute that $g\cdot u=\psi(x,\alpha,\beta^*)=u'$ where $g\in G(k)$ is defined by $(\beta^*)^{-1}\beta$. Hence $u$ is weakly $G$-equivalent to $u'$ as desired. To prove the converse, we now suppose that $u$ is weakly $G$-equivalent to $u'$. Then there exists $g\in G(k)$ such that $g\cdot u=u'$ and hence $(x,\alpha,g\cdot\beta)$ is isomorphic to $(x',\alpha',\beta')$ in $\mathcal M_{\mathcal P\times\mathcal Q}(k)$. Thus there is an isomorphism $\tau:(x,\alpha)\isomto (x',\alpha')$ in $\cmp(k)$ such that $r=(u,u',\tau)\in R(k)$ satisfies $s(r)=u$ and $t(r)=u'$, which means that $u$ is weakly $R$-equivalent to $u'$ as desired. This completes the proof of (2). 

\paragraph{3.} Now we are ready to show that the morphism $f$ is \'etale over any point $y\in Y$ with $|\Aut(y)|\leq 3$. For this purpose we may assume that $f$ is the geometric quotient morphism $q:Y'\to Y'/G$ which sends a geometric point to its $G$-orbit coming from the action of  $G$ on the quasi-projective $B$-scheme $Y'$. Indeed \eqref{eq:gcategoricalquot} combined with the universal property of categorical quotients gives a unique isomorphism $\chi:Y\isomto Y'/G$ of schemes over $B$ such that $q=\chi f$. In particular $q$ is \'etale over $\chi(y)$ if and only if $f$ is \'etale over $y$, and the separable rank $|\Aut(\chi(y))|$ defined with respect to the initial morphism $\chi \pi:\cmp\to Y'/G$ satisfies $|\Aut(y)|=|\Aut(\chi(y))|$. In what follows we assume that $f=q$.

Let $y\in Y$ be a point with $|\Aut(y)|\leq 3$. We consider a point $y'\in Y'$ with $f(y')=y$ and we denote by $D\subseteq \GL_2(\OL/n\OL)$ the stabilizer group of $y'$ with respect to the action of $\GL_2(\OL/n\OL)$ on $Y'$. Then $f:Y'\to Y'/G$ is \'etale at $y'$ if and only if the geometric quotient morphism $Y'\to Y'/D$ is \'etale at $y'$. In particular $f$ is \'etale at $y'$ in the situation when $D$ acts trivially on all points of $Y'$. We now show that we are in this situation. To understand the action of $D$ on $Y'$, we shall construct in step 4. below a group $\Aut(y)(k)$ with the following properties: Its cardinality equals $|\Aut(y)|$, it has at most one element $\sigma$ of order two and there exist an isomorphism of groups
  \begin{equation}\label{eq:isodecompgroup}
\Aut(y)(k)\isomto D
\end{equation}
which sends this $\sigma$ to $-1$.  
Then we see that $|\Aut(y)(k)|=|\Aut(y)|=2$, since by assumption each geometric automorphism group of $\cmp$ has even order and  $|\Aut(y)|\leq 3$. 
Thus the above mentioned properties of $\Aut(y)(k)$ imply that this group is generated by $\sigma$. Hence we obtain that $D=\{1,-1\}$ since the isomorphism \eqref{eq:isodecompgroup} sends $\sigma$ to $-1$. To understand the action of $-1\in D$ on the points of $Y'$, we consider a geometric point $\spec(k')\to B$. Let $(x,\alpha)$ be in  $\cmp(k')$. The forgetful morphism $\cmp\to \mathcal M$ identifies $\Aut(x,\alpha)$ with a subgroup of $\Aut(x)$. Further, as the inertia stack $\mathcal I_{\cmp}\to \cmp$ is finite, the group $\Aut(x,\alpha)$ is finite. Hence Cauchy's theorem and our assumption on the geometric automorphism groups of $\cmp$ imply that $\Aut(x,\alpha)$ contains an element of order two. Then it follows from Lemma~\ref{lem:autofield} and $\Aut(x,\alpha)\subseteq \Aut(x)$ that this element is given by multiplication with $[-1]$.  Thus for each object $(x,\alpha,\beta)$ of $\mathcal M_{\mathcal P\times \mathcal Q}(k')$ there exists an isomorphism $(x,\alpha,\beta)\isomto (x,\alpha,-1\cdot \beta)$ defined by the element in $\Aut(x,\alpha)$ given by $[-1]$. This implies that $-1$, and hence the whole group $D=\{1,-1\}$, acts trivially on all points of $Y'$.  We conclude that $f$ is \'etale at each $y'\in Y'$ with $f(y')=y$, that is $f$ is \'etale over $y$. This shows that $y$ does not lie in $B_f$ and then the equality $B_f=B_{\mathcal P}$ obtained in step 1. assures that our point $y$ does not lie in $B_{\mathcal P}$ as desired.

\paragraph{4.}It remains to construct the group $\Aut(y)(k)$ in \eqref{eq:isodecompgroup} with all the desired properties.  For this purpose we use again the equivalence $\phi:Y'\isomto \mathcal M_{\mathcal P\times\mathcal Q}$. Let $y':\spec(k)\to Y'$ be a geometric point over $y'\in Y'$ and denote by $\phi(y')=(x,\alpha,\beta)$ its image in $\mathcal M_{\mathcal P\times\mathcal Q}(k)$. Then $(x,\alpha)$ lies in $\cmp(k)$ and we consider the group algebraic space over $k$ given by $$\Aut(y)=\textnormal{Isom}_{\cmp}((x,\alpha),(x,\alpha)).$$ 
Notice that the group $\Aut(y)(k)$ identifies with $\Aut(x,\alpha)$. As $f(y')=y$ and  $\phi$ is a morphism over $\cmp$, we find that $\pi(x,\alpha)=y$. This implies that $|\Aut(y)|=|\Aut(y)(k)|$ since the separable rank $|\Aut(y)|$ is an invariant of the point $\pi^{-1}(y)$ of $\cmp$ and since $k$ is an algebraically closed field. The forgetful morphism $\cmp\to \mathcal M$ identifies $\Aut(y)(k)=\Aut(x,\alpha)$ with a subgroup of $\Aut(x)$, and it follows from Lemma~\ref{lem:autofield} that any element of order two in $\Aut(x)$ is given by multiplication with $[-1]:A\to A$ where $x=(A,\iota,\varphi)$. Thus $\Aut(y)(k)$ has at most one element of order two. 
Next, we construct an isomorphism $\Aut(x,\alpha)\isomto D$. For each $\sigma\in \Aut(x,\alpha)$ we denote by $\sigma_n$ the restriction to $A_n$ of the automorphism of $A$ underlying $\sigma$. Then we obtain a group morphism
$$\Aut(x,\alpha)\isomto D, \quad \sigma\mapsto  \beta^{-1}\sigma_n\beta.$$
Indeed $g=\beta^{-1}\sigma_n\beta$ lies in the decomposition group $D$ of $y'$ since $\sigma$ defines an isomorphism $(x,\alpha,\beta)\isomto (x,\alpha,g\cdot\beta)$ and hence $g\cdot y'=y'$, where as before we identify an automorphism of the constant group scheme $(\OL/n\OL)^2_k$ over the algebraically closed field $k$ with the corresponding element in $G(k)=\GL_2(\OL/n\OL)$. To see that the displayed morphism is injective, we take $\lambda\in I_+$ where $I_+$ are the totally positive elements of the $\OL$-submodule $I$ of $L$ which appears in the definition of $\mathcal M$. Now we consider the automorphism group $\Aut(A,\varphi(\lambda))$ of the polarized abelian variety $(A,\varphi(\lambda))$ over $k$.  The forgetful morphisms induce an embedding $\Aut(x,\alpha)\hookrightarrow \Aut(A,\varphi(\lambda))$, and the restriction morphism $\Aut(A,\varphi(\lambda))\to \Aut(A_n)$ is injective since $n\geq 3$. Therefore the composition $\Aut(x,\alpha)\to \Aut(A_n)$ given by $\sigma\mapsto \sigma_n$ is injective, which implies that the displayed morphism is injective. On the other hand, the displayed morphism is surjective since for any $g\in D$ there exists an isomorphism $(x,\alpha,\beta)\isomto (x,\alpha,g\cdot \beta)$ which defines an automorphism $\sigma$ of $(x,\alpha)$ with $g=\beta^{-1}\sigma_n\beta$. Further, we notice that $\beta^{-1}[-1]\beta=-1$ and we recall that $\Aut(y)(k)=\Aut(x,\alpha)$. Therefore we see that the displayed group morphism, and hence the group $\Aut(y)(k)$, has all the desired properties. This completes the proof of Proposition~\ref{prop:auto}. \end{proof}

We now briefly discuss where we used the assumptions appearing in Proposition~\ref{prop:auto}.

\begin{remark}\label{rem:assumptionsautogroupprop}
The assumption that $\cmp\to \mathcal M$ is finite assures in 2. that the scheme $Y'=\cmp\times_{\mathcal M}\mathcal M_{\mathcal Q}$ is quasi-projective over $B$ which we use to view $Y'\to Y$ as a geometric quotient $Y'\to Y'/G$ induced by the action of $G=\GL_2(\OL/n\OL)$. The assumption that $\cmp$ is reduced assures that $Y'$ is reduced which we use to relate $R$-invariance with $G$-invariance, see (1). Finally, the assumption on the geometric automorphism groups assures that certain decomposition groups $D\subseteq G$ act trivially on all points of $Y'$ and hence $Y'\to Y'/D$ is \'etale. This assumption can be relaxed by studying the \'etale locus of $Y'\to Y'/D$ in terms of the automorphism groups of $\cmp$; such a study could be the subject of a future work.
\end{remark}

\newpage

\section{Morphims of abelian schemes over Dedekind schemes}\label{sec:dedmorphisms}

Let $S$ be a connected Dedekind scheme. In this section we introduce some notation and terminology, and we collect basic results for morphisms of abelian schemes over $S$ which we shall use freely in what follows. We refer to \cite[$\mathsection$5]{vkkr:hms} for detailed proofs or precise references for those statements in this section which are not further explained.



\paragraph{Basic properties.}Let $A$, $A'$ and $A''$ be abelian schemes over $S$. Composition induces a pairing $\circ$ from $\Hom^0(A,A')\times\Hom^0(A',A'')$ to $\Hom^0(A,A'')$. Further  $f\mapsto f\otimes 1$ embeds the finitely generated free abelian group $\Hom(A,A')$  into $\Hom^0(A,A')$. In what follows we shall often identify a morphism $f:A\to A'$ with its image $f\otimes 1$ in $\Hom^0(A,A')$. For each nonzero $n\in\ZZ$ and any $f\in\Hom(A,A')$ we denote by $f/n$ the element $f\otimes n^{-1}$ of $\Hom^0(A,A')$. Let $k$ be the function field of the connected Dedekind scheme $S$. Base change from $S$ to the spectrum of $k$ induces canonical isomorphisms 
\begin{equation}\label{eq:dedequivalence}
\Hom(A,A')\isomto \Hom(A_k,A_{k}') \quad \textnormal{and} \quad \End(A)\isomto \End(A_k). 
\end{equation}
A direct computation shows that any element in $\Hom^0(A,A')$ takes the form $f\otimes q$ with $f\in \Hom(A,A')$ and $q\in \QQ$. In particular, we see that the center of the  $\QQ$-algebra $\End^0(A)$ identifies with $Z\otimes_\ZZ\QQ$ where $Z$ denotes the center of the ring $\End(A)$.  We say that $A$ is simple if $A_k$ is simple. Further we say that $A'$ is a factor of $A$ if there exists a surjective morphism $A\to A'$ of abelian schemes over $S$.  The proof of Poincar\'e's reducibility theorem together with \eqref{eq:dedequivalence}  shows the following: The abelian scheme  $A$ is simple if and only if $\End^0(A)$ is a division algebra,  and
   $A'$ is a factor of $A$ if and only if there exists an abelian scheme $B$ over $S$ such that $A'\times_S B$ is isogenous to $A$.  We say that $A$ is isotypic if $A$ is isogenous to a power of a simple abelian scheme over $S$.

\paragraph{Isogenies.} Let $\varphi \colon A\to A'$ be a morphism of abelian schemes over $S$. Then $\varphi$ is an isogeny  if and only if $\varphi_k:A_k\to A'_k$ is an isogeny of abelian varieties over $k$. Suppose now that $\varphi$ is an isogeny. Then $\varphi$ is flat and the rank $d=\deg(\varphi)$ of its kernel $\ker(\varphi)$ is constant on the connected scheme $S$. Consider the `inverse' isogeny (\cite[(5.1)]{vkkr:hms}) 
\begin{equation}\label{def:inverseisog}
\varphi':A'\to A, \quad \textnormal{ and } \quad \varphi^* \colon \End^0(A') \isomto \End^0(A), \quad f \mapsto \varphi'/d \circ f \circ \varphi.
\end{equation}
Here $\varphi^*$ is an isomorphism of $\QQ$-algebras  and  \cite[Lem 5.1]{vkkr:hms} gives that $\varphi'\varphi = [d]$ and $\varphi\varphi' = [d]$ where $[n]$ denotes the multiplication by $n\in\ZZ$ morphism. In particular any isogeny $A\to A'$ becomes invertible inside $\Hom^0(A,A')$.
If $A'$ is simple and if $A$ and $A'$ are of the same relative dimension over $S$, then any nonzero element $\varphi$ of $\Hom(A,A')$ is an isogeny. Indeed  $\varphi$ corresponds via \eqref{eq:dedequivalence} to a nonzero $\varphi_k:A_k\to A_k'$ which is an isogeny since $A_k'$ is simple and $\dim(A_k)=\dim(A'_k)$.

\paragraph{Conjugate abelian variety.}Let $A$ be an abelian variety defined over an arbitrary field $k$, let $\sigma$ be an automorphism of $k$ and write $\sigma^*A$ for the base change of $A$ via $\sigma:k\to k$. 
The abelian variety $\sigma^*A$ is called the conjugate of $A$ by $\sigma$. We denote by
$$\sigma^*:\Hom(A,A')\isomto \Hom(\sigma^*A,\sigma^*A')$$
the isomorphism of abelian groups induced by base change via $\sigma$, 
where $A'$ is an abelian variety over $k$. We also denote by $\sigma^*$ the isomorphism of $\QQ$-vector spaces given by $\sigma^*\otimes_\ZZ\QQ$. 
Base change properties of (group) schemes imply the following: The map $\sigma^*$ is a morphism of rings if $A=A'$, and $\varphi:A\to A'$ is an isogeny if and only if $\sigma^*(\varphi)$ is an isogeny.  

\newpage

\section{Effective Shafarevich conjecture}\label{sec:es}

Let $K$ be a number field. In this section we prove new cases of the effective Shafarevich conjecture for abelian varieties over $K$. In particular, we establish this conjecture for abelian varieties of product $\gl2$-type with $G_\QQ$-isogenies and for CM abelian varieties. 

Let $g\geq 1$ be a rational integer, let $S$ be a nonempty open subscheme of $\spec(\OK)$, and let $A$ be an abelian scheme over $S$ of relative dimension $g$. We recall that $A$ is of product $\gl2$-type with $G_\QQ$-isogenies if $A$ is isogenous to a product $\prod A_i$ of abelian schemes $A_i$ over $S$ of $\gl2$-type with $G_\QQ$-isogenies  as in \eqref{def:Gisogintro}. A discussion of the notion of an abelian scheme over $S$ of $\gl2$-type with $G_\QQ$-isogenies can be found in $\mathsection$\ref{sec:Gisogdef} below.

\paragraph{Height bound.}As before, we use the quantities $d=[K:\QQ]$, $D_K$ and $N_S=\prod_{v\in S^c} N_v$. We define $l=(d3^{4g^2})!$ and we denote by $h_F$ the stable Faltings height ($\mathsection$\ref{sec:heightdef}) of an abelian scheme over $S$. The main goal of this section is to prove the following result.

\begin{theorem}\label{thm:es}Let $A$ be an abelian scheme over $S$ of relative dimension $g$. 
\begin{itemize}
\item[(i)]If $A$ is of product $\gl2$-type with $G_\QQ$-isogenies then $$h_F(A)\leq (4gl)^{144gl}\textnormal{rad}(N_SD_K)^{24g}.$$
\item[(ii)]If $A_{\bar{K}}$ has CM  then $h_F(A)\leq (3g)^{(5g)^2}\rad(N_SD_K)^{5g}$.
\end{itemize}
\end{theorem}
The proof shows that one can replace in (i) the exponent $24g$  by $\max(24,5g)$, while Proposition~\ref{prop:esgl2} and \eqref{eq:esgl2nocm} provide sharper versions of (i) in relevant special situations. 
Further Proposition~\ref{prop:cm} and Lemma~\ref{lem:hphibound} contain more precise versions of (ii)  in certain cases of interest, see also the discussions in $\mathsection$\ref{sec:avcolmez} for known asymptotic versions of (ii). 

\paragraph{Number of classes.} Another goal of this section is to prove Theorem~B~(ii)  on the number of isomorphism classes of abelian schemes over $S$ of relative dimension $g$ which are of product $\gl2$-type with $G_\QQ$-isogenies. In course of the proof given in Section~\ref{sec:numberisoclass}, we shall obtain various intermediate results which improve Theorem~B~(ii) in certain cases of interest; see for example Proposition~\ref{prop:esnumber} and Lemma~\ref{lem:numbisogclasses}. 

\paragraph{Outline of the section.} In Section~\ref{sec:Gisogenies} we study abelian schemes of $\gl2$-type with $G_\QQ$-isogenies. After motivating the definition ($\mathsection$\ref{sec:Gisogdef}), we discuss and prove Proposition~\ref{prop:Gisog} on the factors of such abelian schemes. This proposition involves a technical condition $(*)$ which is further studied in $\mathsection$\ref{sec:cond*}. We also collect in this section some basic results which are used throughout this work; in particular we consider abelian schemes of $\gl2$-type (Lemma~\ref{lem:gl2endostructure}) and isogenies between conjugate abelian varieties ($\mathsection$\ref{sec:centerisogenies}). 

In Section~\ref{sec:cmav} we begin to bound the height of abelian varieties with CM. We discuss the averaged Colmez conjecture in $\mathsection$\ref{sec:avcolmez} and  we work out in $\mathsection$\ref{sec:analyticest}  explicit analytic estimates for certain L-values attached to CM fields. Then we prove Lemma~\ref{lem:cmram} on the ramification of endomorphism algebras of simple abelian varieties with CM. 

In Sections~\ref{sec:esproofs} and \ref{sec:numberisoclass} we put everything together. After reviewing known results for the Faltings height and discussing isogeny estimates ($\mathsection$\ref{sec:hfprop}),  we prove Theorem~B~(ii) and Theorem~\ref{thm:es} by using the  strategy and the ideas explained in $\mathsection$\ref{sec:esproofdiscussion}.

\subsection{Abelian schemes of $\gl2$-type with $G$-isogenies}\label{sec:Gisogenies}
  
Let $k\subseteq K$ be fields which are algebraic over $\QQ$ and write $G=\Aut(\bar{k}/k)$. In this section we discuss a class of abelian schemes which we call of $\gl2$-type with $G$-isogenies. After motivating our definition, we show in Proposition~\ref{prop:Gisog} that the study of this class of abelian schemes over certain base schemes $S$ with $k(S)=K$ can be reduced to the study of simple abelian schemes of $\gl2$-type over certain base schemes $T$ with $k(T)=k$.

\subsubsection{The notion of abelian schemes of $\gl2$-type with $G$-isogenies}\label{sec:Gisogdef}
Let $k\subseteq K$ and $G$ be as above. Choose $\bar{k}$ with $K\subseteq \bar{k}$ 
and let $S$ be a connected Dedekind scheme with function field $K$. We say that an abelian scheme $A$ over $S$ of relative dimension $g$ is of $\gl2$-type with $G$-isogenies if the $\QQ$-algebra $\End^0(A)$ contains a number field $F$ of degree $g$ and if for each $\sigma\in G$ there is an isogeny $\mu_\sigma:\sigma^*A_{\bar{k}}\to A_{\bar{k}}$ such that 
\begin{equation}\label{def:Gisog}
\mu_\sigma \circ \sigma^*(f)=f \circ \mu_\sigma, \quad f\in F. 
\end{equation}
Here we identified $F$ with its image in $\End^0(A_{\bar{k}})$. More generally, we say that an abelian scheme $A$ over $S$ is of product $\gl2$-type with $G$-isogenies if $A$ is isogenous to a product $\prod A_i$ such that each $A_i$ is an abelian scheme over $S$ of $\gl2$-type with $G$-isogenies.  If $A$ is an abelian scheme over $S$ of $\gl2$-type with $G$-isogenies  then any abelian scheme $A'$ over $S$ which is isogenous to $A$  is again of $\gl2$-type with $G$-isogenies. Indeed one can take here for example $F'=\varphi\circ F\circ\varphi'/d$ and the isogenies $\mu'_\sigma=\varphi_{\bar{k}} \mu_\sigma \sigma^*(\varphi'_{\bar{k}})$ where $\varphi:A\to A'$ is an isogeny of degree $d$ and $\varphi':A'\to A$ is an inverse as in \eqref{def:inverseisog}.

\paragraph{Related notions.} The notion of abelian schemes over $S$ of $\gl2$-type with $G$-isogenies generalizes various definitions in the literature. To explain this, we take $K=\bar{k}$ and we assume that $A$ is an abelian variety over $K$ of $\gl2$-type. In the case $\dim(A)=1$, we observe that $A$ is of $\gl2$-type with $G$-isogenies if and only if $A$ is a $k$-curve\footnote{A $k$-curve is an elliptic curve $E$ over $K$ which is isogenous to $\sigma^*E$ for all $\sigma\in G$.} in the sense of Gross~\cite{gross:cmell}, Ribet~\cite{ribet:gl2} and Elkies~\cite{elkies:kcurves}.  Further, Ribet~\cite{ribet:fieldsofdef} introduced the notion of `$k$-HBAV' which was generalized by Wu~\cite[Def 1.9]{wu:virtualgl2}. She defined $A$ to be $k$-virtual if for each $\sigma\in G$ there is an isogeny $\mu_\sigma:\sigma^*A\to A$ such that 
\begin{equation}\label{def:kvirt}
\mu_\sigma \circ \sigma^*(f)=f \circ \mu_\sigma, \quad f\in \End^0(A). 
\end{equation}
On comparing \eqref{def:kvirt} with \eqref{def:Gisog}, we see that the notion of $\gl2$-type with $G$-isogenies generalizes the definition of $k$-virtual abelian varieties $A$ over $K$ and thus the notion  of `$k$-HBAV'. The additional generality is useful in certain situations of interest in arithmetic such as for example in $\mathsection$\ref{sec:proofspecialmain} where we can directly verify our weaker condition~\eqref{def:Gisog}. However, if $A$ has no CM then our condition \eqref{def:Gisog} is in fact equivalent to the stronger $k$-virtual condition \eqref{def:kvirt}. To prove this claim, we suppose that $A$ has no CM and that there exist $F\subseteq \End^0(A)$ and isogenies $\mu_\sigma$ as in \eqref{def:Gisog}. Then Lemma~\ref{lem:gl2endostructure} gives that $A$ is isotypic and that $F$ contains the center of $\End^0(A)$. Thus the isogenies $\mu_\sigma$ are center compatible, and hence $A$ is $k$-virtual by Lemma~\ref{lem:isotypfullcomp}. This proves the claim.

\subsubsection{Factors of abelian schemes of $\gl2$-type with $G$-isogenies}\label{sec:factorsGisogenies} We continue our notation. Further, we now assume that $k\subseteq K$ are number fields and we suppose that $S\subseteq \spec(\OK)$ is open with complement $S^c$. Then we define  
\begin{equation}\label{def:UTGisog}
U=S\setminus p^{-1}(T^c) \quad \textnormal{and}\quad T=p(U)
\end{equation}
where $p:\spec(\OK)\to\spec(\mathcal O_k)$ is the projection and $T^c$ is the union of $p(S^c)$ with the branch locus of $p$. 
 Weil restriction reduces the study of some abelian schemes over $S$ of $\gl2$-type with $G$-isogenies to the study of abelian schemes over $T$ of $\gl2$-type.

\begin{proposition}\label{prop:Gisog}
Let $A$ be an abelian scheme over $S$ with no CM. Suppose that $A$ is of $\gl2$-type with  $G$-isogenies and assume $(*)$. Then there is a simple abelian scheme $C$ over $T$ of $\gl2$-type such that $A_{U}$ is isogenous to a power of a simple factor of  $C_U$.  
\end{proposition}
Here condition $(*)$ is as follows: The field extension $K/k$ is normal, and the endomorphisms of $A_{\bar{k}}$ and the isogenies $\mu_\sigma$ in \eqref{def:Gisog} are all defined over  $K\subset\bar{k}$. In fact $(*)$ can always be satisfied ($\mathsection$\ref{sec:cond*}) when passing to a controlled field extension of $K$. The proof of Proposition~\ref{prop:Gisog} shows in addition that any simple factor $B$ of $A$ satisfies
\begin{equation}\label{eq:reldimweilres}
\dim(C)\leq \dim(B)[K:k]
\end{equation}
for $\dim$ the relative dimension. Further we observe that $T$ is open in $\spec(\mathcal O_k)$ with complement $T^c$,
and Dedekind's discriminant theorem shows that $T^c$ is the set of finite places $v$ of $k$ such that $v$ lies under $S^c$ or $v$ divides the discriminant ideal of $K/k$. 

\subsubsection{Proof of Proposition~\ref{prop:Gisog}}\label{sec:proofgisogprop}

Let $k\subseteq K\subseteq \bar{k}$, $S$ and $G$ be as in Section~\ref{sec:Gisogdef}, and let $A$ be an abelian scheme over $S$. Throughout this subsection we assume that $A$ is of $\gl2$-type with $G$-isogenies, that $A$ has no CM and that $(*)$ holds. In a first step we show the following result.
\begin{lemma}\label{lem:Gisog1}
There is an isogeny from $A$ to a power of an abelian scheme $B$ over $S$ of $\gl2$-type such that $B_{\bar{k}}$ is a simple $k$-virtual abelian variety of $\gl2$-type with no CM.
\end{lemma}
In particular  $B_{\bar{k}}$ in Lemma~\ref{lem:Gisog1} satisfies all assumptions in Wu~\cite[Prop 1.12]{wu:virtualgl2}. Hence there is a simple abelian variety $C_k$ over $k$ of $\gl2$-type such that $B_{\bar{k}}$ is a factor of $C_{\bar{k}}$. Moreover \cite[$\mathsection$9.4]{rvk:modular}, which crucially relies on the arguments of Ribet~\cite[$\mathsection$6]{ribet:gl2} and Wu~\cite[Prop 1.12]{wu:virtualgl2}, leads to the following result.

\begin{lemma}\label{lem:Gisog2}
Suppose that $K$ is a number field with $S\subseteq\spec(\OK)$ open, and let $B$ be the abelian scheme of Lemma~\ref{lem:Gisog1}. Then there is a simple abelian scheme $C$ over $T$ of $\gl2$-type such that $B_U$ is a simple factor of $C_U$, where $U$ and $T$ are as in \eqref{def:UTGisog}.
\end{lemma}

Now, on combining this result with Lemma~\ref{lem:Gisog1} we deduce Proposition~\ref{prop:Gisog}. The remaining of this section is devoted to the proofs of the above two lemmas.

\begin{proof}[Proof of Lemma~\ref{lem:Gisog1}]
By assumption $S$ is a connected Dedekind scheme and $A$ is an abelian scheme over $S$ of $\gl2$-type with no CM. Hence \eqref{eq:prooflemGisog1} and Lemma~\ref{lem:gl2endostructure} give a simple abelian scheme $B$ over $S$ of $\gl2$-type and $n\in \ZZ_{\geq 1}$ such that $A$ is isogenous to $B^n$. 

We next prove that $B_{\bar{k}}$ is a simple abelian variety over $\bar{k}$ of $\gl2$-type with no CM. Firstly $B_{\bar{k}}$ is of $\gl2$-type, since it is a base change of  $B$ which is of $\gl2$-type. Further, all endomorphisms of $A_{\bar{k}}$ are defined over $K$ by $(*)$. Therefore \eqref{eq:dedequivalence} gives 
\begin{equation}\label{eq:prooflemGisog1}
\End(A_{\bar{k}})\cong \End(A).
\end{equation}
This shows that $A_{\bar{k}}$ has no CM since $A$ has no CM. Then on using that $A_{\bar{k}}$ is isogenous to $B_{\bar{k}}^n$, we deduce that the abelian varieties $B_{\bar{k}}^n$ and $B_{\bar{k}}$ both have no CM. 
Poincar\'e reducibility shows that  any factor of $A$ also has property \eqref{eq:prooflemGisog1}. 
Thus $\End^0(B_{\bar{k}})\cong\End^0(B)$ is a division algebra since $B$ is a simple factor of $A$. This implies that $B_{\bar{k}}$ is simple.

To show that $B_{\bar{k}}$ is $k$-virtual, we now use our assumption that $A$ is of $\gl2$-type with $G$-isogenies. Hence there is a number field $F$ inside $\End^0(A_{\bar{k}})$ of degree $\dim(A_{\bar{k}})$ and for each $\sigma\in G$ there is an isogeny $\mu_\sigma:\sigma^*A_{\bar{k}}\to A_{\bar{k}}$ satisfying \eqref{def:Gisog}.
As $A_{\bar{k}}$ is an abelian variety over $\bar{k}$ of $\gl2$-type with no CM, an application of Lemma~\ref{lem:gl2endostructure} with $A_{\bar{k}}$ gives that the center $Z^0$ of $\End^0(A_{\bar{k}})$ is contained in $F$. Then, on using that $\Hom(\sigma^*A_{\bar{k}},A_{\bar{k}})$ is torsion-free and that $Z^0=Z\otimes_\ZZ \QQ$ for $Z$ the center of $\End(A_{\bar{k}})$, we see that \eqref{def:Gisog} implies $$\mu_\sigma\sigma^*(z)=z\mu_\sigma,\quad z\in Z.$$ In other words the isogeny $\mu_\sigma:\sigma^*A_{\bar{k}}\to A_{\bar{k}}$ is center compatible in the sense of $\mathsection$\ref{sec:centerisogenies}. Now, an application of Lemma~\ref{lem:centcompsimple} with the isotypic abelian variety $A_{\bar{k}}$ and the automorphism $\sigma$ of $\bar{k}$ gives a center compatible isogeny  $\sigma^*B_{\bar{k}}\to B_{\bar{k}}$.  Thus an application of Lemma~\ref{lem:fullcomp} with the simple abelian variety $B_{\bar{k}}$ provides an isogeny $\sigma^*B_{\bar{k}}\to B_{\bar{k}}$ which (is fully compatible and hence) satisfies \eqref{def:kvirt}. 
Therefore we conclude that $B_{\bar{k}}$ is a simple $k$-virtual abelian variety of $\gl2$-type with no CM. This shows that $B$ has the desired properties.
\end{proof}

In the following proof, we assume in addition that $k\subseteq K$ are number fields and that $S\subseteq \spec(\OK)$ is open. Further, we let $p:U\to T$ be as in Section~\ref{sec:factorsGisogenies}. 

\begin{proof}[Proof of Lemma~\ref{lem:Gisog2}]Let $B$ be the abelian scheme over $S$ of Lemma~\ref{lem:Gisog1}. Then $B_{\bar{k}}$ is a simple $k$-virtual abelian variety of $\gl2$-type with no CM. In particular, the $k$-virtual condition assures that  for each $\sigma\in G$ there is an isogeny $\mu_\sigma:\sigma^*B_{\bar{k}}\to B_{\bar{k}}$. 

To show that each $\mu_\sigma$ is defined over $K$, we use condition $(*)$. It assures that $K/k$ is normal and the restriction of $\sigma\in G$ to the subfield $K\subset \bar{k}$ defines an automorphism of $K/k$ which we also denote by $\sigma$. Then $\sigma^*B_K$ is an abelian variety over $K$. It is simple, since it is geometrically isogenous via $\mu_\sigma$ to the simple $B_{\bar{k}}$. 
Moreover, on using that $A_K$ is isogenous to $\sigma^*A_K$ by $(*)$ and that $A$ is isogenous to $B^n$, we see that $\sigma^*B_K$ is a simple factor of $A_K$. Hence $\sigma^*B_K$ has to be isogenous to the simple factor $B_K$, and 
then $\Hom^0(\sigma^*B_K,B_K)\cong \End^0(B_{K})$ identifies with $\Hom^0(\sigma^*B_{\bar{k}},B_{\bar{k}})\cong \End^0(B_{\bar{k}})$ since $B_K$ has property  \eqref{eq:prooflemGisog1}. 
We conclude that each isogeny $\mu_\sigma$ is defined over $K$.    

We now go into Wu's proof of \cite[Prop 1.12]{wu:virtualgl2}. The above statements show that $B_{\bar{k}}$ satisfies all assumptions of \cite[Prop 1.12]{wu:virtualgl2} and that our $K$ has all the properties of the field $F$ appearing in Wu's proof (\cite[$\mathsection$1.2]{wu:virtualgl2arxiv}).   Then Wu's proof, which uses and generalizes Ribet's arguments in \cite[$\mathsection$6]{ribet:gl2} for $\QQ$-curves, shows in addition 
that the Weil restriction $D_k=\textnormal{Res}_{K/k}(B_K)$ of $B_K$ has a simple factor $C_k$ of $\gl2$-type. 

To prove that $C_k$ extends to an abelian scheme $C$ over $T$, we use our assumption that $K$ is a number field. Hence Milne~\cite[Prop 1]{milne:arithmetic} gives that $D_k$ has good reduction at a finite place $v$ of $k$ if  $v$ does not divide the discriminant ideal of $K/k$ and  if $A_K$ has good reduction at all points in $p^{-1}(v)$. Further the generic fiber $A_K$ of the abelian scheme $A$ over $S$ has good reduction at each point in $S$, and Dedekind's discriminant theorem shows that $v$ divides the discriminant ideal of $K/k$ if and only if $v$ lies in the branch locus $B_p$ of $p$. We deduce that $D_k$ extends to an abelian scheme over  $T=(p(S^c)\cup B_p)^c$, 
and then also the factor $C_k$ of $D_k$ extends to an abelian scheme $C$ over $T$. Furthermore \eqref{eq:dedequivalence} gives that $C$ is simple and of $\gl2$-type, since $C_k$ has these properties.

We next show that $B_U$ is a factor of $C_U$. The generic fiber $C_K$ of $C_U$  identifies with the abelian variety $C_k\times_k K$ over $K$, since the composition of $\spec(K)\to U$ with $p:U\to T$ factors as $\spec(K)\to\spec(k)\to T$. Further, as $C_k$ is a factor of $D_k$, we obtain a surjective morphism  $D_k\to C_k$ of abelian varieties over $k$ which induces a surjective morphism
$$
\prod \sigma^*B_K\isomto D_k\times_k K\to C_k\times_k K\isomto C_K
$$
of abelian varieties over $K$
with the product taken over all $\sigma\in \textnormal{Aut}(K/k)$. Here we used that the field extension $K/k$ is Galois by $(*)$ which assures (\cite[p.5]{weil:alggroups}) that the abelian variety $D_k\times_k K$ over $K$ is isomorphic to $\prod \sigma^*B_K$. 
We showed above that $B_K$ is isogenous to each $\sigma^*B_K$. Therefore, on exploiting that the displayed morphism is surjective and hence nonzero, we see that $B_K$ has to be a simple factor of $C_K$. This implies that $B_U$ is a simple factor of $C_U$. We conclude that $C$ has the desired properties.\end{proof}
We conclude this subsection by collecting basic results for abelian schemes of $\gl2$-type which are essentially due to Ribet~\cite{ribet:semistable,ribet:gl2} and Wu~\cite{wu:virtualgl2}. In the above proofs we used (i) and (ii) of the following lemma, while (i)\,-\,(iii) will be used in later sections.

\begin{lemma}\label{lem:gl2endostructure}
Let $S$ be a connected Dedekind scheme whose function field $k$ embeds into $\CC$, and let $A$ be an abelian scheme over $S$ of $\gl2$-type such that $A$ has no CM. 
\begin{itemize}
\item[(i)] The abelian scheme $A$ is  isogenous to a power $B^n$ of a simple abelian scheme $B$ over $S$ of $\gl2$-type for some $n\in \ZZ_{\geq 1}$. In particular the $\QQ$-algebra $\End^0(A)$ is isomorphic to $\textnormal{M}_n(D)$ where $D=\End^0(B)$ is a division algebra.
\item[(ii)] Let $F$ be a subfield of the $\QQ$-algebra $\End^0(A)$ such that $[F:\QQ]$ equals the relative dimension of $A$ over $S$. Then $F$ is the centralizer of $F$ in $\End^0(A)$. 
\item[(iii)] Either $D$ identifies with a subfield $Z$ of $F$ with $[F:Z]=n$ or $D$ is a quaternion algebra whose center identifies with a subfield $Z$ of $F$ with $[F:Z]=2n$.   
\end{itemize}
\end{lemma}
\begin{proof}
To prove (i) we reduce to the case $S=\spec(k)$ which was treated in Wu~\cite[Prop 1.5~(i)]{wu:virtualgl2} when $k$ is an arbitrary field of characteristic zero. Notice that Wu stated the result under the assumption that $A_{\bar{k}}$ has no CM. However, it turns out that her proof can be copied word by word when making the weaker assumption that the abelian variety over $k$ has no CM. Now, the reduction goes as follows:  As $\End^0(A)\isomto \End^0(A_k)$ by \eqref{eq:dedequivalence} we obtain that $A_k$ is of $\gl2$-type with no CM. Then an application of \cite[Prop 1.5]{wu:virtualgl2} with $A_k$ gives an isogeny $A_k\to B_k^n$ where $B_k$ is a simple abelian variety of $\gl2$-type and $n\in \ZZ_{\geq 1}$. In particular $B_k$ is a factor of $A_k$ which extends to the abelian scheme $A$ over $S$. Thus $B_k$ extends to an abelian scheme $B$ over $S$, and then the isogeny $A_k\to B^n_k$ extends to a morphism $A\to B^n$ of abelian schemes over $S$ which is again an isogeny; see for example \cite[$\mathsection$5.1]{vkkr:hms}. This together with $\End^0(B)\cong \End^0(B_k)$ implies (i).

To show (ii) and (iii) we may and do assume that $S=\spec(k)$ since $\End^0(A)\cong \End^0(A_{k})$ and $D\cong \End^0(B_k)$. Our $A$ is isotypic by (i) and $A$ has no CM. Thus no factor of $A$ has CM.    Then we deduce (ii) and (iii) by using precisely the same arguments as Ribet in his proof of \cite[Prop 5.2]{ribet:gl2}; these arguments exploit a trick (attributed by Ribet to Tunnel)  which exploits that the field $k$ embeds into $\CC$ and hence $[D:\QQ]$ divides $2\dim(B)$ since the division algebra $D$ acts $\QQ$-linearly on the $2\dim(B)$-dimensional homology $H_1(B(\CC),\QQ)$. This completes the proof of Lemma~\ref{lem:gl2endostructure}.
\end{proof}

  \subsubsection{Condition $(*)$ and controlled field extensions}\label{sec:cond*}
In this section we investigate the technical condition $(*)$  appearing in Proposition~\ref{prop:Gisog}. For this purpose we let $k\subseteq K\subset \bar{k}$, $G$, $S$ and $T$ be as in Section~\ref{sec:factorsGisogenies}.

Let $A$ be an abelian scheme over $S$ of relative dimension $g$ such that $A_{\bar{k}}$ is of $\gl2$-type with $G$-isogenies. We say that condition $(*)$ holds over a number field $L\supseteq K$  if $L/k$ is normal and if the endomorphisms of $A_{\bar{k}}$ and the isogenies $\mu_\sigma$ in \eqref{def:Gisog} are all defined over $L\subset \bar{k}$. On using Silverberg's work \cite{silverberg:fieldsofdef} and the criterion of N\'eron--Ogg--Shafarevich~\cite{seta:goodreduction}, we obtain that condition $(*)$ holds over a controlled field $L$.

\begin{lemma}\label{lem:cond*}
 Condition $(*)$ holds over a field $L\supseteq K$ such that $L$ is contained in the normal closure of $K(A_n)/k$  for each $n\in\ZZ_{\geq 3}$ and such that $L/k$ is unramified over $T$. 
\end{lemma}
Here and in what follows, for any $n\in \ZZ_{\geq 1}$, we denote by $K(A_n)$ the field of definition of the $n$-torsion points of $A(\bar{k})$. The degree of $K(A_n)$ over $K$ is at most $n^{4g^2}$, and Lemma~\ref{lem:cond*} combined with Dedekind's discriminant theorem allows to explicitly control the discriminant of $L/k$ in terms of $g$, the discriminant of $K/k$, and $N_S=\prod_{v\in S^c}N_v$.
\begin{proof}[Proof of Lemma~\ref{lem:cond*}]
Let $n\in \ZZ_{\geq 1}$ and denote by $K^{\textnormal{cl}}(A_n)$ the normal closure of $K(A_n)$ over $k$. 
To show that $L=\cap_{n\geq 3}K^{\textnormal{cl}}(A_n)$ has the desired properties, we take $\sigma\in  G$. The field $L$ is normal over $k$. 
Thus pulling back $A_L$ via the restriction of $\sigma$ to $L$ gives an abelian variety $\sigma^*A_L$ over $L$ which identifies over $\bar{k}$ with $\sigma^*A_{\bar{k}}$. 
We denote by $L(\sigma^*A_n)$ and $L(A_n)$  the field of definition of the $n$-torsion points of $\sigma^*A_L(\bar{k})$ and  $A_L(\bar{k})$ respectively. On using that $\sigma$ is invertible and that $L(\sigma^*A_n)$ and $L(A_n)$ are normal over $L$, we deduce 
 $$L(\sigma^*A_n)=L(A_n).$$ 
Then an application of Silverberg~\cite[Thm 4.2]{silverberg:fieldsofdef} with $\sigma^*A_L$ and $A_L$ implies that any morphism in $\Hom(\sigma^*A_{\bar{k}},A_{\bar{k}})$ is defined over the subfield $\cap_{n\geq 3}L(A_n)$ of $\bar{k}$. Furthermore we obtain that $L=\cap_{n\geq 3}L(A_n)$, since it holds that $L(A_n)=K(A_n)L$ 
and that $K(A_n)L\subseteq K^{\textnormal{cl}}(A_n)$ for each $n\geq 3$. 
Thus we deduce that condition $(*)$ holds over $L$. 

 To prove that $L/k$ is unramified over $T$, we take a rational prime $p\geq 3$ and for any scheme $S'$ we write $S'[\tfrac{1}{p}]=S'\times_\ZZ \ZZ[\tfrac{1}{p}]$. 
The abelian variety $A_K$ has good reduction at each closed point of $S$, since $A_K$ extends to the abelian scheme $A$ over $S$. Thus $K(A_p)/K$ is unramified over $S[\tfrac{1}{p}]$ 
by the criterion of N\'eron--Ogg--Shafarevich \cite[Thm 1]{seta:goodreduction}. Then the construction of $T$ in \eqref{def:UTGisog} shows that $K(A_p)/k$ is unramified over $T[\tfrac{1}{p}]$ 
and hence its normal closure $K^{\textnormal{cl}}(A_p)/k$ is also unramified over $T[\tfrac{1}{p}]$. 
Therefore, for any rational prime $q\geq 3$ with $q\neq p$, we deduce that the subfield $L$ of $K^{\textnormal{cl}}(A_p)\cap K^{\textnormal{cl}}(A_q)$ is unramified over $T=T[\tfrac{1}{p}]\cup T[\tfrac{1}{q}]$. Finally, as the field $L=\cap_{n\geq 3}K^{\textnormal{cl}}(A_n)$ is contained in $K^{\textnormal{cl}}(A_n)$ for each $n\geq 3$, we conclude that $L/k$ has all the desired properties.
\end{proof}


\subsubsection{Center compatible isogenies between conjugate abelian varieties}\label{sec:centerisogenies}

In this section we collect basic lemmas which we used in $\mathsection$\ref{sec:Gisogdef} and $\mathsection$\ref{sec:proofgisogprop} above. These lemmas allow in situations where isogenies of conjugate abelian varieties are center compatible to reduce to the case where the isogenies are compatible with the whole endomorphism ring. The results should be all well-known. However, we are not aware of suitable references and for the sake of completeness we will thus give the proofs which use and generalize arguments of Ribet~\cite{ribet:fieldsofdef} and Pyle~\cite{pyle:gl2}.

Let $A$ be an abelian variety defined over an arbitrary field $k$, let $\sigma$ be an automorphism of the field $k$ and write $\sigma^*A$ for the base change of $A$ via $\sigma:k\to k$.

\paragraph{Center compatible isogenies.}We say that an isogeny $\mu:\sigma^*A\to A$ is center compatible if any $z$ in the center of $\End(A)$ satisfies $\mu\sigma^*(z)=z\mu$. The existence of such an isogeny does not necessarily guarantee a center compatible isogeny $\sigma^*B\to B$ for the simple factors $B$ of $A$.  
However, in the isotypic case we obtain the following result. 

\begin{lemma}\label{lem:centcompsimple}
Suppose that $A$ is an isotypic abelian variety over $k$. If there exists a center compatible isogeny $\sigma^*A\to A$, then there exists a center compatible isogeny $\sigma^*B\to B$. 
\end{lemma}
It is crucial for our applications in $\mathsection$\ref{sec:proofgisogprop} that the statement of Lemma~\ref{lem:centcompsimple} does not depend on how  $\sigma$ acts on an isogeny $A\to B^n$ or on $\End(A)$. Luckily, these dependences `cancel out' as we shall see in the proof given at the end of this section.

\paragraph{Fully compatible isogenies.}We say that an isogeny $\mu:\sigma^*A\to A$ is fully compatible if any $f\in \End(A)$ satisfies
$\mu\sigma^*(f)=f\mu.$ In the isotypic case, one can always modify center compatible isogenies by an inner automorphism (obtained from the Skolem--Noether theorem) to get fully compatible isogenies. We first consider the simple case. 

\begin{lemma}\label{lem:fullcomp}
Let $B$ be a simple abelian variety over $k$. If there exists a center compatible isogeny $\sigma^*B\to B$, then there exists a fully compatible isogeny $\sigma^*B\to B$.
\end{lemma}
\begin{proof}
We may and do assume that $B\neq 0$. To modifiy a center compatible isogeny $\tau:\sigma^*B\to B$, we use a trick which we learned from Ribet~\cite{ribet:gl2} and Pyle~\cite{pyle:gl2}. In fact their trick works exactly the same in our setting. The details are as follows:

Our $B$ is simple. Thus $R=\End^0(B)$ is a central simple algebra over its center $F$. 
Further  $r\mapsto(\sigma^{-1})^*(\tau^{-1}\circ r\circ\tau)$ defines an $F$-algebra morphism $\phi:R\to R$, since  $\tau$ is center compatible, $F=Z\otimes_\ZZ\QQ$ for $Z$ the center of $\End(B)$, $(\sigma^{-1})^*$ is a ring morphism and $\circ$ is a pairing. 
Hence Skolem--Noether gives $s\in R$ such that $\phi(r)=s\circ r\circ s^{-1}$, and as $B$ is simple we can take any $\mu:\sigma^*B\to B$ with $\mu\otimes q=\tau\circ \sigma^*(s)$ for some $q\in \QQ$. 
\end{proof}

We now use this result to strengthen the statement of Lemma~\ref{lem:centcompsimple} in the sense that we can moreover construct an isogeny $\sigma^*A\to A$ which is fully compatible.  

\begin{lemma}\label{lem:isotypfullcomp}
Let $A$ be an isotypic abelian variety over $k$. If there exists a center compatible isogeny $\sigma^*A\to A$, then there exists a fully compatible isogeny $\sigma^*A\to A$.
\end{lemma}
\begin{proof}
We may and do assume that $A\neq 0$ is isotypic. Thus there is an isogeny $\varphi:A\to B^n$  with $n\in\ZZ_{\geq 1}$ and $B$ a simple abelian variety over $k$. Let $\varphi':B^n\to A$ be its inverse isogeny as in \eqref{def:inverseisog}. Further there is a center compatible isogeny $\sigma^*A\to A$ by assumption. Hence Lemma~\ref{lem:centcompsimple} provides a center compatible isogeny $\sigma^*B\to B$, and then Lemma~\ref{lem:fullcomp} gives a fully compatible isogeny $\mu:\sigma^*B\to B$ since $B$ is simple. After identifying $\sigma^*B^n$ with $(\sigma^*B)^n$, we see that $\mu$ defines a fully compatible isogeny $\mu^n:\sigma^*B^n\to B^n$. 
Now we consider
$$
\tau:\sigma^*A\to^{\sigma^*(\varphi)} \sigma^*B^n\to^{\mu^n} B^n\to^{\varphi'} A.
$$
Then $\tau$ is an isogeny, since it is a composition of isogenies. Further  $\varphi'\varphi$ and $\varphi\varphi'$ are both multiplication with $d=\deg(\varphi)$, and $[d]$ commutes with any morphism of group schemes. Thus we deduce that $\tau\sigma^*(f)[d]=[d]f\tau=f\tau[d]$ for all $f\in\End(A)$, since $\sigma^*$ is a functor and since $\mu^n$ is fully compatible. Then, on using that $\Hom(\sigma^*A,A)$ is torsion-free, we conclude that the isogeny $\tau:\sigma^*A\to A$ is fully compatible as desired.
\end{proof}

It remains to verify the statement of Lemma~\ref{lem:centcompsimple}. The following proof consists of checking various compatibilities (using standard arguments) in order to assure that a certain explicitly constructed isogeny $\tau:\sigma^*B\to B$ is indeed center compatible. 

\begin{proof}[Proof of Lemma~\ref{lem:centcompsimple}]
We may and do assume that $A$ is nonzero. By assumption there exists a center compatible isogeny $\mu:\sigma^*A\to A$ and there is an isogeny $\varphi:A\to B^n$ where $B$ is a simple abelian variety over $k$ and $n\in\ZZ_{\geq 1}$. Let $\varphi':B^n\to A$ be the `inverse isogeny' of $\varphi$ as in \eqref{def:inverseisog} and let $\iota:B\to B^n$ be the diagonal embedding. Now, we define
$$
t:\sigma^*B\to^{\sigma^*(\iota)}\sigma^*B^n\to^{\sigma^*(\varphi')}\sigma^*A\to^{\mu}A\to^\varphi B^n.
$$
The base change $\sigma^*(\iota)$ of $\iota$ is injective, since  $\iota$ is universally injective. 
Then, on using that $\sigma^*(\varphi')$, $\mu$ and $\varphi$ are all isogenies, we deduce that $t$ has finite kernel and thus is nonzero since $\sigma^*B\neq 0$; here we used  our assumption $A\neq 0$ to assure that $B\neq 0$. Hence there is a projection $p:B^n\to B$ such that $t$ composed with $p$ is a nonzero morphism  
$$
\tau:\sigma^*B\to B
$$
of abelian varieties over $k$. Moreover $\tau\neq 0$ is surjective since $B$ is simple. Then, on using that $\sigma^*B$ and $B$ have the same dimension, we conclude that $\tau$ is an isogeny. 

To verify that $\tau$ is center compatible, we take $z$ in the center $Z$ of $\End(B)$. We view $\End(B)$ as a subring of $\End(B^n)$ via the diagonal embedding $\iota$ and we observe that this identifies $Z$ with the center of $\End(B^n)$. A direct computation gives the identities 
\begin{equation}\label{eq:centcomp}
\varphi\circ \varphi^*(z)=z\circ \varphi  \quad \textnormal{ and } \quad \varphi' \circ z=\varphi^*(z)\circ\varphi'
\end{equation} 
inside $\Hom^0(A,B^n)$ and inside $\Hom^0(B^n,A)$ respectively.  Here we used the definition of the $\QQ$-algebra isomorphism $\varphi^*:\End^0(B^n)\isomto \End^0(A)$  in \eqref{def:inverseisog}. Pulling back by $\sigma$ is a  functor. 
Hence  \eqref{eq:centcomp} shows that $\sigma^*(\varphi')\circ\sigma^*(z)$ equals $\sigma^*(\varphi^*(z))\circ\sigma^*(\varphi')$, and $\sigma^*(\iota)$ commutes with $\sigma^*(z)$ since  $\iota z=z\iota$. 
Therefore, on using that $\mu$ is center compatible and that $\varphi^*(z)$ lies in the center of $\End^0(A)$, 
we deduce from \eqref{eq:centcomp} the identity  $$t\circ \sigma^*(z)=z\circ t$$ inside $\Hom^0(\sigma^*B,B^n)$. 
This identity then also holds inside the torsion-free abelian group $\Hom(\sigma^*B,B^n)$, and we observe that $pz=zp$. 
It follows that $\tau \sigma^*(z)=z\tau$ for all $z\in Z$, which means that $\tau$ is center compatible. This completes the proof.
\end{proof}

\subsection{Abelian varieties with CM}\label{sec:cmav}

In this section we prove explicit upper bounds for the Faltings height of certain abelian varieties with CM. For this purpose we first recall the averaged Colmez conjecture in $\mathsection$\ref{sec:avcolmez}. Then we work out in $\mathsection$\ref{sec:analyticest}  explicit analytic estimates for certain L-values and we control in $\mathsection$\ref{sec:cmram} the ramification of CM endomorphism algebras.

\paragraph{CM types and abelian varieties.} We now introduce some terminology which will be used throughout this section. A CM field is a number field which has no real embeddings but which is
 quadratic over a totally real subfield. Let $E$ be a CM field and let $F\subset E$ be the maximal totally real subfield.   A CM type $\Phi$ of $E$ is a subset $\Phi\subset \Hom(E,\mathbb C)$ such that $\Phi\sqcup \Phi \tau=\Hom(E,\mathbb C)$ where
$\tau\in \Aut(E/F)$ is not trivial.   We say that a CM type of $E$ is simple if it is not the full preimage under restriction of a CM type of a proper CM subfield of $E$. For any $\varphi:E\to\CC$, we write $\CC_\varphi$ for $\CC$ viewed as an $E$-algebra via $\varphi$. 

Let $A$ be an abelian variety defined over a subfield $k$ of $\CC$. Suppose that $2\dim(A)=[E:\QQ]$ and assume that there exists a ring morphism $\iota:E\to \End^0(A)$. Then $\iota$ induces an action of $E\otimes_\QQ\CC\cong\prod_{\varphi}\CC_\varphi $ on the Lie algebra $\textnormal{Lie}(A_\CC)$ of $A_\CC$.  This action decomposes  $\textnormal{Lie}(A_\CC)$ into a corresponding product of eigenspaces $\textnormal{Lie}(A_\CC)_\varphi$ over $\CC$. 
Let $\Phi_{A}$ be the set of $\varphi\in \Hom(E,\CC)$ such that the eigenspace $\textnormal{Lie}(A_\CC)_\varphi$ has dimension one. Then $\Phi_A$ is a CM type of $E$ which is simple if and only if $A_\CC$ is a simple abelian variety; see for example \cite[$\mathsection$1.5]{chcooo:cmbook}.   We call $\Phi_A$ the CM type of $E$ associated to $(A,\iota)$.

For any CM type $\Phi$ of $E$, we denote by $A_\Phi$ the usual complex abelian variety associated to $\Phi$ which is defined for example in \cite[1.5.3]{chcooo:cmbook}. The complex abelian variety $A_\Phi$ has the following properties: There exists a ring morphism $m:\mathcal O_E\hookrightarrow \End(A_\Phi)$ such that $\Phi$ is the CM type of $E$ associated to $(A_\Phi,\iota)$ for $\iota=m\otimes \QQ$ 
and  $A_\Phi(\CC)\cong\CC^g/\Phi(\mathcal O_E).$ Here $g=\dim(A_\Phi)$ and $\Phi(\mathcal O_E)$ denotes the lattice in $\mathbb C^g$ given by the image of $\mathcal O_E$ under the embedding $E\hookrightarrow \CC^g$
defined by the $g$ distinct embeddings $E\hookrightarrow \CC$ in $\Phi$. 

\subsubsection{Averaged Colmez conjecture}\label{sec:avcolmez}

To state the averaged Colmez conjecture,  let $E$ be a CM field with maximal totally real subfield $F\subset E$  and let $\Phi$ be a CM type of $E$. We write $g=[F:\QQ]$.

\paragraph{The height $h(\Phi)$.}Let $A_\Phi$ be the complex abelian variety associated to $\Phi$. Any complex abelian variety with CM is defined over a number field and then we define 
$h(\Phi)=h_F(A_\Phi).$
Here the invariance of the stable Faltings height $h_F$ under geometric isomorphisms assures that $h(\Phi)$ does not depend on the choice of a model of $A_\Phi$ over a number field.  Moreover, this invariance of $h_F$ together with Colmez'~\cite[Thm 0.3~(ii)]{colmez:conjecture}  implies that 
\begin{equation}\label{eq:colmezheight}
h(\Phi)=h_F(A)
\end{equation}
for any abelian variety $A$ of dimension $g$ which is defined over a number field $K\hookrightarrow \CC$ and which has the following property: There is a ring morphism $j:\mathcal O_E\to \End(A_\CC)$ such that $\Phi$ is the CM type of $E$ associated to $(A_\CC,j\otimes \QQ)$. Yuan--Zhang obtained in \cite[$\mathsection$2]{yuzh:avcolmez} an alternative proof of \eqref{eq:colmezheight} as a by-product of their height decomposition. 

\paragraph{The formula.} Define $\mathfrak c=D_E/D_F$ and let $L'(s)$ be the derivative of the L-function $L(s)$    of the quadratic character of $\mathbb A_F^\times$ corresponding to $E/F$; here $L(s)$ is without the infinite part.  Now we can state the averaged Colmez conjecture (\cite{colmez:conjecture}) which in the case $g=1$ is a reformulation of the classical Lerch--Chowla--Selberg formula:
\begin{equation}\label{eq:avcolmez}
\tfrac{1}{2^g}\sum_\Phi h(\Phi)=-\tfrac{1}{2}\tfrac{L'(0)}{L(0)}-\tfrac{1}{4}\log \mathfrak c-\tfrac{g}{2}\log \pi
\end{equation}
with the sum taken over all distinct CM types $\Phi$ of $E$. Here the term $-\tfrac{g}{2}\log \pi$  takes into account that we use Faltings' normalization of the height $h_F$.   Colmez~\cite{colmez:conjecture}, Obus~\cite{obus:colmez} and Yang~\cite{yang:hmsurfaceamerican,yang:hmsurfaceasian} proved important special cases of the exact version \cite[Conj 0.4]{colmez:conjecture} of \eqref{eq:avcolmez}. In general  \eqref{eq:avcolmez} was established by Yuan--Zhang~\cite{yuzh:avcolmez} using the method of Yuan--Zhang--Zhang~\cite{yuzhzh:gzbook} and independently by Andreatta--Goren--Howard--Madapusi-Pera~\cite{aghm:avcolmez} using the method of Yang~\cite{yang:hmsurfaceamerican,yang:hmsurfaceasian}.

\paragraph{An explicit upper bound.}Colmez' conjectures give useful bounds for the height $h(\Phi)$ of any CM type $\Phi$ of $E$. For example Colmez~\cite[$\mathsection$3]{colmez:conjecturecompositio} bounded $h(\Phi)$ conditional on the exact version of \eqref{eq:avcolmez} and other conjectures. Further Colmez~\cite[p.365]{colmez:conjecturecompositio}  for $g=1$ and Tsimerman~\cite[Cor 3.3]{tsimerman:ao} in general  deduced from \eqref{eq:avcolmez} the asymptotic bound $$h(\Phi)\leq D_E^{o_g(1)}.$$ 
The proof uses analytic estimates for $L'(0)/L(0)$ via Brauer--Siegel and  the lower bound $h(\Phi)\geq c_g$ which is contained in Faltings' work  (\cite[p.356]{faltings:finiteness}, \cite[1.22]{deligne:faltings}). In fact one can take here $c_g=-\tfrac{g}{2}\log(2\pi^2)$ by Bost's lower bound \eqref{eq:hflower}, 
 and  the analytic estimates (except Brauer--Siegel) applied in \cite[Cor 3.3]{tsimerman:ao} can be worked out explicitly by using Rademacher's arguments in \cite{rademacher:lindeloef}. This leads to the following result. 

\begin{lemma}\label{lem:hphibound}
Let $h_E$ and $h_F$ be the class number of $E$ and $F$ respectively. If $\epsilon>0$ is a real number with $\epsilon\leq 1/2$, then any CM type $\Phi$ of $E$ satisfies $$\tfrac{1}{2^g}h(\Phi)\leq a\tfrac{h_F}{h_E}\mathfrak c^{1/2}\max\bigl((\tfrac{2\pi}{2+\epsilon})^g,\mathfrak c\bigl)^\epsilon+ \tfrac{1}{4}\log \mathfrak c-\tfrac{\gamma}{2}g.$$
Here $a=\tfrac{wQ}{2\epsilon}\zeta(1+\epsilon)\tfrac{(2+\epsilon)^{g\epsilon}}{(2\pi)^{g(\epsilon+1)}}$ where $\zeta(s)$ is the Riemann zeta function, $\gamma$ is the Euler--Mascheroni constant, $w$ is the cardinality of the group $\mu$ of roots of unity in $E$ and $Q=[\mathcal O_E^\times:\mu\OL_F^\times]\in\{1,2\}$ is the Hasse unit index of $E/F$.
\end{lemma}
\begin{proof}
On combining $h(\Phi)\geq -\tfrac{g}{2}\log(2\pi^2)$ in \eqref{eq:hflower} with the explicit analytic estimates worked out in (the proof of) Lemma~\ref{lem:analyticestimates}, we see that \eqref{eq:avcolmez} implies Lemma~\ref{lem:hphibound}.
\end{proof}

  Class field theory gives that $h_F$ divides $h_E$ 
and then Lemma~\ref{lem:hphibound} provides an explicit upper bound for $h(\Phi)$ in terms of $\epsilon$, $D_E$ and $g$. This bound is sufficient for our purpose, but it is asymptotically substantially weaker than Tsimerman's bound $h(\Phi)\leq D_E^{o_g(1)}$. 

\subsubsection{Explicit analytic estimates}\label{sec:analyticest}

We now work out the explicit analytic estimate which we used in Lemma~\ref{lem:hphibound} above. As before, $\zeta(s)$ is the Riemann zeta function and  $\gamma$ is the Euler--Mascheroni constant.  

\begin{lemma}\label{lem:analyticestimates}Let $\epsilon>0$ be a real number with $\epsilon\leq 1/2$, let $E$ be a CM field with maximal totally real subfield $F$, and let $L'(s)$ be the derivative of the L-function $L(s)$ of the  character of $\mathbb A_F^\times$ corresponding to $E/F$. Put $g=[F:\QQ]$ and $\mathfrak c=D_E/D_F$. 
\begin{itemize}
\item[(i)] If the real part $\sigma$ of $s\in\CC$ satisfies $-\epsilon\leq \sigma\leq 1+\epsilon$ then $$|L(s)|\leq \zeta(1+\epsilon)\left(\tfrac{\mathfrak c}{(2\pi)^g}|1+s|^g\right)^{\tfrac{1+\epsilon-\sigma}{2}}.$$
\item[(ii)] Let $h_E$ and $h_F$ be the class number of $E$ and $F$ respectively, let $w$ be the number of roots of unity in $E$ and let  $Q$ be the Hasse unit index of $E/F$. It holds $$ \left|\tfrac{L'(0)}{L(0)}\right|\leq 2a \tfrac{h_F}{h_E}\mathfrak c^{1/2}\max\bigl((\tfrac{2\pi}{2+\epsilon})^g,\mathfrak c\bigl)^\epsilon+\log \mathfrak c+b,$$
where $a=\tfrac{wQ}{2\epsilon}\zeta(1+\epsilon)\tfrac{(2+\epsilon)^{g\epsilon}}{(2\pi)^{g(\epsilon+1)}}$ and $b=g\log (2\pi e^{\gamma})$.
\end{itemize} 
\end{lemma}
  
Here  (ii) is a standard consequence of (i) and the analytic class number formula, while the bound in (i) is up to a factor $\zeta(1+\epsilon)^{g-1}$ due to Rademacher~\cite[Thm 5]{rademacher:lindeloef}. In fact we shall obtain (i) by applying his arguments with our entire Artin L-function $L(s)$.

\begin{proof}[Proof of Lemma~\ref{lem:analyticestimates}]
Let $\epsilon>0$ be a real number with $\epsilon\leq 1/2$. We first introduce some notation. For each $s\in \CC$ we write $s=\sigma+it$ with $\sigma,t\in\RR$ and we denote by $S(a,b)$ the strip $a\leq \sigma\leq b$ for $a,b\in \RR$ with $a<b$. Let $\chi:G\to \CC$ be the non-trivial irreducible character of $G=\Aut(E/F)\cong \ZZ/2\ZZ$. In particular $\chi$ has degree one with image $\chi(G)=\{-1,1\}$, and $L(s)$ coincides with the Artin L-function of $\chi$. 

We now prove (i).  In the case $\sigma\geq 1+\epsilon$,  it holds $\log L(s)=\sum_p\sum_{n\geq 1} \tfrac{1}{n}\chi(p^n)p^{-ns}$ with $\chi(p^n)\in\ZZ$ satisfying $|\chi(p^n)|\leq 1$ and hence $|\log L(s)|\leq \log \zeta(\sigma)$ which implies
  \begin{equation}\label{eq:zetabound}
|L(s)|\leq \zeta(1+\epsilon).
\end{equation}
Let $\Gamma(s)$ be the Gamma function. The field $E$ is CM with maximal totally real subfield $F$.   Thus the conductor $\mathfrak c$ of the non-trivial character $\chi$ of $G$ is given by $\mathfrak c=D_E/D_F$ and the completed Artin L-function $\Lambda(s)$ takes the following form 
  \begin{equation*}
\Lambda(s)=\mathfrak c^{s/2} L_\infty(s) L(s), \quad L_\infty(s)=\pi^{-g(s+1)/2}\Gamma(\tfrac{s+1}{2})^g.
\end{equation*}
As the character $\chi$ is fixed by complex conjugation, the functional equation becomes $\Lambda(s)=W(\chi)\Lambda(1-s)$ for some $W(\chi)\in \CC$ with absolute value 1. Therefore we obtain
\begin{equation*}
L(s)=W(\chi) \left(\frac{\mathfrak c}{\pi^g}\right)^{1/2-s}\left(\frac{\Gamma(1-\tfrac{s}{2})}{\Gamma(\tfrac{s+1}{2})}\right)^gL(1-s).
\end{equation*}
Next, an application of Rademacher's \cite[Lem 2]{rademacher:lindeloef} with $Q=1$ gives for all $s$ in the strip $S(-\tfrac{1}{2},\tfrac{1}{2})$ that $|\Gamma(1-\tfrac{s}{2})/\Gamma(\tfrac{s+1}{2})|$ is at most $\left(\tfrac{1}{2}|1+s|\right)^{1/2-\sigma}$. Then, as $0<\epsilon\leq 1/2$ by assumption, the above displayed expression for $L(s)$ together with \eqref{eq:zetabound} shows 
\begin{equation*}
|L(-\epsilon+it)|\leq A|1-\epsilon+it|^{(1/2+\epsilon)g}, \quad A=\bigl(\tfrac{\mathfrak c}{(2\pi)^g}\bigl)^{1/2+\epsilon}\zeta(1+\epsilon).
\end{equation*}
Further \eqref{eq:zetabound} provides $|L(1+\epsilon+it)|\leq \zeta(1+\epsilon)$, and the Artin $L$-function $L(s)$ of our non-trivial $\chi$ is entire since $G$ is abelian. Then an application of Rademacher's version \cite[Thm 2]{rademacher:lindeloef} of the Phragm\'en--Lindel\"of theorem  with $Q=1$, $\alpha=(1/2+\epsilon)g$ and $\beta=0$ gives for all $s$ in the strip $S(-\epsilon,1+\epsilon)$ an upper bound for $|L(s)|$ as claimed in (i).

We now prove (ii). For this purpose we take the radius $r=\epsilon$ and  the entire function $f(z)=L(z)$ in the basic inequality $|f'(1)|r\leq \sup_{|z-1|=r}|f(z)|$ which follows from Cauchy's integral formula.   This together with (i) leads to 
  $$|L'(1)|\leq \tfrac{1}{\epsilon} \zeta(1+\epsilon)\max\bigl(1,\tfrac{\mathfrak c}{(2\pi)^g}(2+\epsilon)^{g}\bigl)^{\epsilon}.$$
We next express $L(1)$ in terms of class numbers. As $G\cong\ZZ/2\ZZ$, the Artin L-function $L(s)$ of the non-trivial $\chi$ is the quotient $L(s)=\zeta_E(s)/\zeta_F(s)$ of the Dedekind zeta functions of $E$ and $F$. Then the analytic class number formula and \cite[Prop 4.16]{washington:cyclotomic} imply  $$L(1)=\tfrac{(2\pi)^{g}}{wQ}\tfrac{h_E}{h_F}\mathfrak c^{-1/2}.$$  To relate the values $L'(0)/L(0)$ and $L'(1)/L(1)$, we differentiate the functional equation $\Lambda(s)=W(\chi) \Lambda(1-s)$ and we use that the Digamma function $\psi(z)=\Gamma'(z)/\Gamma(z)$ has special values  $\psi(1/2)=-2\log 2-\gamma$ and $\psi(1)=-\gamma$  respectively. Then we obtain   $$-\tfrac{L'(0)}{L(0)}=\tfrac{L'(1)}{L(1)}+\log \mathfrak c-g\log (2\pi e^{\gamma}).$$
Finally, on combining the above displayed results we deduce an upper bound for $|L'(0)/L(0)|$ as claimed in (ii). This completes the proof of Lemma~\ref{lem:analyticestimates}.
\end{proof}

\subsubsection{Ramification of CM endomorphism algebras}\label{sec:cmram}

The endomorphism algebra of a simple complex abelian variety with CM is a CM field whose ramification over $\QQ$ can be controlled by CM theory. The goal of this subsection is to prove the following result in which $N_S=\prod_{v\in S^c}N_v$ is as above.

\begin{lemma}\label{lem:cmram}
Let $K$ be a number field, let $A$ be an abelian scheme over a nonempty open subscheme $S\subseteq\spec(\OK)$ and let $B$ be a simple factor of $A_{\bar{K}}$ such that $B$ has CM. 
\begin{itemize}
\item[(i)] The CM field $E=\End^0(B)$ is unramified at each rational prime $p\nmid D_KN_S$.
\item[(ii)] If $B$ and all its endomorphisms are defined over $K\subset \bar{K}$, then the CM field $E$ is unramified at each rational prime $p\nmid D_K$.
\end{itemize}
\end{lemma}
\begin{proof}
Let $p$ be a rational prime. To show that (ii) implies (i), we assume that (ii) holds. Let $L$ be the field of definition of the endomorphisms of $A_{\bar{K}}$. Poincar\'e's reducibility theorem implies that $B$ is isogenous to an abelian subvariety $B'$ of $A_{\bar{K}}$ and that $B'$  and all its endomorphisms are defined over $L$. Further $B'$ is a simple factor of $A_{\bar{K}}$ and $B'$ has CM. Then an application of (ii) with $B'$ gives that $E\cong \End^0(B')$ is unramified at $p$ if $p\nmid D_L$. 
Further \cite[Thm~4.1]{silverberg:fieldsofdef} implies that $L/K$ is unramified over $S$, since the generic fiber of the abelian scheme $A$ over $S$ is an abelian variety over $K$ with good reduction over $S$. Thus $\textnormal{rad}(D_L)$ divides $D_KN_S$ by Dedekind's discriminant theorem and therefore we conclude that $E$ is unramified at $p$ if $p\nmid N_SD_K$ as claimed in (i).

We now prove (ii). By assumption $B$ is a simple abelian variety with CM. Hence the CM type $\Phi$ of $E$ associated to $(B,\textnormal{id})$ is simple. The reflex field $E^*$ of $(E, \Phi)$ is defined by $E^*=\QQ(\sum_{\varphi\in \Phi}\varphi(e);e\in E)$, and \cite[Thm 3.1.1]{lang:cm} gives $E^*\hookrightarrow K$  since $B$ and all its endomorphisms are defined over $K$ by assumption. Thus the Galois closure $(E^*)^{\textnormal{cl}}$ of $E^*$ is  unramified at $p$ if $p\nmid D_K$.   Further, an application of \cite[Thm 1.5.2]{lang:cm} with our simple $\Phi$ identifies $E$ with the reflex field of $(E^*,\Phi^*)$  for some CM type $\Phi^*$ of $E^*$ and thus $E\hookrightarrow (E^*)^{\textnormal{cl}}$. Hence
 $E$ is unramified at $p$ if $p\nmid D_K$ as claimed in (ii).\end{proof}


\subsection{Proof of the effective Shafarevich conjecture}\label{sec:esproofs}

Let $K$ be a number field, let  $S$ be a nonempty open subscheme of $\spec(\OK)$ and let $g\geq 1$ be a rational integer. Following the strategy of proof  outlined in the introduction, we prove in this section the effective Shafarevich conjecture (see Theorem~\ref{thm:es}) for abelian varieties over $K$  which are of $\gl2$-type with $G_\QQ$-isogenies or which have CM.

\paragraph{No CM case.}In the first part of the proof of Theorem~\ref{thm:es}, we consider abelian varieties which have no CM. In Section~\ref{sec:proofesgl2} we prove  the following result by reducing the problem via Proposition~\ref{prop:Gisog} to the known case $K=\QQ$ stated in \eqref{eq:esgl2}. 

\begin{proposition}\label{prop:esgl2}
Let $A$ be an abelian scheme over $S$ of relative dimension $g$ with no CM. Suppose that $A$ is of $\gl2$-type with $G_\QQ$-isogenies and assume $(*)$. Then    $$h_F(A)\leq (4gd)^{144gd}\textnormal{rad}(D_KN_S)^{24}.$$ 
\end{proposition}
Here $h_F$ is the stable Faltings height ($\mathsection$\ref{sec:heightdef}), while $D_K$, $d=[K:\QQ]$ and $N_S=\prod_{v\in S^c}N_v$ are as before. Further, we recall that condition $(*)$ means that $K/\QQ$ is normal and that all endomorphisms of $A_{\bar{K}}$ and all isogenies $\mu_\sigma$ in \eqref{def:Gisog} are defined over $K\subset \bar{K}=\bar{\QQ}$. 

\paragraph{CM case.}In the second part of the proof of Theorem~\ref{thm:es}, we deal with the general CM case by reducing to the situation  $A=A_\Phi$ which was treated in Lemma~\ref{lem:hphibound}. In course of the proof ($\mathsection$\ref{sec:proofpropcm}) we shall obtain the following bound which only depends on $K$ and $g$.
\begin{proposition}\label{prop:cm}
Let $A$ be an abelian scheme over $S$ of relative dimension $g$ such that $A_{\bar{K}}$ has CM. If all endomorphisms of $A_{\bar{K}}$ are defined over $K\subset \bar{K}$, then $$h_F(A)\leq (3g)^{(5g)^2}\rad(D_K)^{5g}.$$ 
\end{proposition}

In the remaining of this section, we first review some known results and then we give the proofs of  Proposition~\ref{prop:esgl2}, Proposition~\ref{prop:cm} and Theorem~\ref{thm:es}. 

\subsubsection{Properties of the Faltings height and isogeny estimates}\label{sec:hfprop}

Let $K$ be a number field and let $A$ be an abelian variety over $K$ of dimension $g$. In this section we review properties of the stable Faltings height $h_F$ of $A$ used in our proofs. In particular, we collect effective bounds for the variation of $h_F$ under isogenies. 

We start with the following basic property: Any abelian variety $A'$ over $K$ satisfies $h_F(A\times_K A')=h_F(A)+h_F(A')$. Further, Faltings' work \cite{faltings:finiteness} gives that $h_F(A)\geq c_g$ for some constant $c_g$ depending only $g$; see \cite[1.22]{deligne:faltings}. In fact Bost~\cite{bost:lowerbound} obtained
\begin{equation}\label{eq:hflower}
h_F(A)\geq -\tfrac{1}{2}\log(2\pi^2)g.
\end{equation}
See for example \cite[Cor 8.4]{gare:periods} and notice that $h_F(A)=h_B(A)-\frac{g}{2}\log \pi$ where
$h_B$ denotes the height which appears in the statement of \cite[Cor 8.4]{gare:periods}. 

\paragraph{Variation under isogenies.} We now suppose that $A'$ is an abelian variety over $K$ which is isogenous to $A$. Faltings' formula \cite[Lem 5]{faltings:finiteness} implies that
\begin{equation}\label{eq:isogvardeg}
|h_F(A)-h_F(A')|\leq \tfrac{1}{2}\log \deg(\varphi)
\end{equation}
for any isogeny $\varphi:A\to A'$. Moreover, Faltings~\cite{faltings:finiteness} bounded $|h_F(A)-h_F(A')|$ in terms of $A/K$ alone. On using Faltings' arguments and some refinements of Par{\v{s}}in--Zarhin, Raynaud~\cite[Thm 4.4.9]{raynaud:abelianisogenies} worked out an effective bound for the variation under the assumption that $A$ has everywhere semi-stable reduction, that is
\begin{equation}\label{eq:isogvarfalt}
|h_F(A)-h_F(A')|\leq c_A
\end{equation}
with $c_A$ an effective constant (simplified in \cite[Prop 6.10]{rvk:gl2}) depending only on $g$, $D_K$, $[K:\QQ]$ and the bad reduction places of $A$. In fact the semi-stable assumption is not required here (\cite[Lem 3.1]{rvk:gl2}). The above bounds for the variation all do not use Diophantine approximation or transcendence techniques.
Masser--W\"ustholz \cite{mawu:periods,mawu:abelianisogenies,mawu:factorization} obtained bounds via a completely different approach using their isogeny estimates based on transcendence theory.  These isogeny estimates are fully effective (\cite[p.121]{bost:abelianisogenies}), and completely explicit constants were obtained by Gaudron--R\'emond~\cite{gare:isogenies}: Their most recent version~\cite[Thm 1.9]{gare:newisogenies} gives isogenies $A\to A'$ and $A'\to A$ of degree at most $\kappa(A)$ and this combined with  (\ref{eq:isogvardeg}) proves  
\begin{equation}\label{eq:isogvartransc}
\lvert h_F(A)-h_F(A')\rvert\leq \tfrac{1}{2}\log \kappa(A), \quad \kappa(A)=\bigl((7g)^{8g^2}d\max(h_F(A),\log d,1)\bigl)^{4g^2}
\end{equation}
for $d=[K:\QQ]$; one can replace here $\kappa(A)$ by $\kappa(A)^{1/2}$ if $\End(A)$ identifies with a maximal order of $\End^0(A)$.  We remark that  \eqref{eq:isogvarfalt} improves (\ref{eq:isogvartransc}) in some cases, and vice versa in other cases. However, in the situations considered in this paper, it turns out that \eqref{eq:isogvartransc} leads to exponentially better bounds for the same reasons as in \cite[Rem 5.2]{rvk:gl2}.

\subsubsection{Proof of Proposition~\ref{prop:esgl2}}\label{sec:proofesgl2}

We continue our notation. Let $K$ be a number field and let $S\subseteq \spec(\OK)$ be a nonempty open subscheme. As before we write $d=[K:\QQ]$ and $N_S=\prod_{v\in S^c}N_v$. Let $A$ be an abelian scheme over $S$ of relative dimension $g$ which has no CM. We suppose that $A$ is of $\gl2$-type with $G_\QQ$-isogenies and we  assume that condition $(*)$ is satisfied. 

\begin{proof}[Proof of Proposition~\ref{prop:esgl2}]
Let $U\subseteq S$ and $T\subseteq \spec(\ZZ)$ be the open subschemes defined in \eqref{def:UTGisog} with $k=\QQ$. Our $A$ satisfies all the assumptions in Proposition~\ref{prop:Gisog}. Thus this proposition gives a simple abelian scheme $C$ over $T$ of $\gl2$-type such that $A_U$ is isogenous to a power $B^m$ of a simple factor $B$ of $C_U$. As $B$ is a factor of $C_U$ and $U$ is a connected Dedekind scheme, Poincar\'e's reducibility theorem  and \eqref{eq:dedequivalence} provide an abelian scheme $D$ over $U$ such that $B\times_U D$ is isogenous to $C_U$.   Now, on using isogenies
$$A_U\to B^{m} \quad \textnormal{and} \quad B\times_U D\to C_U,$$ 
we can estimate the heights as in \cite[$\mathsection$5.2.2]{rvk:gl2}.  For any abelian scheme $X$ over an open $S'\subseteq S$, we denote by $v_X$ the maximal variation of $h_F$ in the isogeny class of $X$; that is $v_X=\sup \lvert h_F(X)-h_F(X')\rvert$ with the supremum taken over all abelian schemes $X'$ over $S'$ which are isogenous to $X$. The abelian scheme $A'=B^{m}$ over $U$ satisfies
\begin{equation}\label{eq:hfapespf}
h_F(A')=m h_F(B) \ \  \textnormal{ and } \ \ h_F(A)\leq v_{A'}+h_F(A')
\end{equation}
since $A_U$ is isogenous to $A'$. The stable height $h_F$ is invariant under the base change $U\to T$, 
and $B\times_U D$ is isogenous to $C_U$. Thus we obtain that $h_F(B)$ is at most $v_{C_U}+h_F(C)-h_F(D)$ and then we see that the lower bound for $h_F(D)$ in \eqref{eq:hflower} implies 
\begin{equation}\label{eq:hfbespf}
h_F(B)\leq v_{C_U}+h_F(C)+\tfrac{g}{2}\log(2\pi^2).
\end{equation}
Here we used that the relative dimension of $D$ over $U$ is at most $g$. The relative dimension of $B$ over $U$ is $g/m$. Thus the relative dimension $n$ of $C$ over $T$ satisfies $n\leq dg/m$ by \eqref{eq:reldimweilres}, and applying \cite[Thm A]{rvk:gl2} with the abelian scheme $C$ over $T$ of $\gl2$-type gives 
\begin{equation}\label{eq:esgl2}
h_F(C)\leq (3n)^{144n}N_T^{24}.
\end{equation}
As $N_T=\textnormal{rad}(D_KN_S)$  we then
 see that \eqref{eq:isogvartransc}, \eqref{eq:hfapespf} and \eqref{eq:hfbespf} lead to an upper bound for $h_F(A)$ as claimed in Proposition~\ref{prop:esgl2}. We now include some details of our computations described above. For example, one can use the inequalities
$$v_{C_U}\leq \tfrac{1}{4}c, \quad h_F(B)\leq \tfrac{3}{2}c, \quad h_F(A')\leq \tfrac{3}{2}gc \ \textnormal{ and } \ v_{A'}\leq \tfrac{1}{2}c$$
as intermediate steps to deduce that $h_F(A)\leq 2gc$ where $c=(3gd)^{144gd}N_T^{24}$.\end{proof}  

\subsubsection{Proof of Proposition~\ref{prop:cm} and Theorem~\ref{thm:es}}\label{sec:proofpropcm} 
 We continue our notation. Let $K$ be a number field of degree $d$ over $\QQ$, let $S\subseteq \spec(\OK)$ be a nonempty open subscheme with geometric generic point $\bar{s}=\spec(\bar{K})$
   and write again $N_S=\prod_{v\in S^c}N_v$. Let $A$ be an abelian scheme over $S$ of relative dimension $g$.

\begin{proof}[Proof of Proposition~\ref{prop:cm}]
 As in the statement of Proposition~\ref{prop:cm}, we assume that $A_{\bar{s}}$ has CM and that
 all endomorphisms of $A_{\bar{s}}$ are defined over $K\subset \bar{K}$.  Then \eqref{eq:dedequivalence} gives that $\End(A)\isomto \End(A_{\bar{s}})$. We claim that there exist simple factors $A_i$ of $A$ of relative dimension $g_i$,  CM fields $E_i$ of degree $2g_i$ with $\mathcal O_{E_i}\cong \End(A_i)$ and an isogeny 
\begin{equation}\label{eq:cmpropproof}
A'=\prod A_i^{e_i}\to A, \quad e_i\in\ZZ_{\geq 1}.
\end{equation}
To prove this claim, let $A_i$ be a simple factor of $A$ and write $g_i$ for the relative dimension of $A_i$ over $S$. On using that $\End(A)\isomto \End(A_{\bar{s}})$, we see that $\End(A_i)\isomto \End(A_{i,\bar{s}})$.    Then CM theory gives that  $A_i$ has CM since $A_{\bar{s}}$ has CM by assumption.   As $K$ has characteristic zero, this implies  that the $\QQ$-algebra $\End^0(A_i)$ is a CM field $E_i$ of degree $2g_i$.
   The order $\mathcal O=\End(A_i)$ of $E$ is contained in the maximal order of $E$ which is $\mathcal O'=\mathcal O_{E_i}$. Hence an application of \cite[Lem 5.3]{vkkr:hms}, which is based on \cite[$\mathsection$1.7.4]{chcooo:cmbook},   shows that after possibly replacing $A_i$ by an isogenous abelian scheme over $S$ we may and do assume that $\mathcal O_{E_i}\cong\End(A_i)$. Then the claim follows from  Poincar\'e's reducibility theorem and \eqref{eq:dedequivalence}.

To estimate the heights, we choose an embedding $K\hookrightarrow \CC$ and we let $\Phi_i$ be the CM type of $E_i$ associated to $(A_{i,\CC},j\otimes \QQ)$ where $j$ is the composition of the ring morphisms $\mathcal O_{E_i}\isomto\End(A_i)\isomto \End(A_{i,\CC})$.   An application of \eqref{eq:colmezheight} with $A_{i,K}$ gives that $h_F(A_i)=h_F(\Phi_i)$. This implies that  $h_F(A')=\sum e_ih(\Phi_i)$ and then \eqref{eq:cmpropproof} shows 
\begin{equation*}
h_F(A)\leq v_{A'}+\sum e_i h(\Phi_i), \quad g=\sum e_ig_i.
\end{equation*}
To bound $h(\Phi_i)$ in terms of $D_K$ and $g$, we first control $D_{E_i}$. As $\End(A_i)\cong \End(A_{i,\bar{s}})$ we see that $A_{i,\bar{s}}$ has the following properties: It has CM, it is a simple factor of $A_{\bar{s}}$, it is defined over $K\subset \bar{K}$ and  all its endomorphisms are defined over $K\subset \bar{K}$. Then an application of Lemma~\ref{lem:cmram}~(ii) with $B=A_{i,\bar{s}}$ gives that the CM field $E_i\cong \End^0(A_{i,\bar{s}})$ of degree $2g_i$ is unramified at each rational prime $p\nmid D_K$.   Hence \cite[Lem 6.2]{rvk:szpiro}, which is a direct consequence of Dedekind's discriminant theorem, implies $$D_{E_i}\leq \textnormal{rad}(D_K)^{5g}(2g)^{24g^2}.$$
Here we used the effective version of the prime number theorem in \cite[(3.12) and (3.16)]{rosc:formulas} which gives that the number $n$ of rational prime divisors of $D_K$ satisfies $n^n<\textnormal{rad}(D_K)^{3/2}$ in the case when $n> 12$.  Next, we recall that Lemma~\ref{lem:hphibound} and the subsequent discussion give an explicit upper bound for $h(\Phi_i)$ in terms of $\epsilon$, $D_{E_i}$ and $g_i$. We take $\epsilon=1/2$ in this bound and then the above displayed results together with the bound for $v_{A'}$ in \eqref{eq:isogvartransc} and $d\leq 3\max(1,\log D_K)$ lead to an estimate for $h_F(A)$ as claimed in Proposition~\ref{prop:cm}.  

We now include some details of the explicit computations. For example, to compute the final bound one can use the following inequalities as intermediate steps:
$$\zeta(\tfrac{3}{2})\leq 3, \quad a_i\leq \tfrac{w_i}{2}\leq 4g_i^2, \quad h_F(A_i)=h(\Phi_i)\leq (3g)^{24g^2}\rad(D_K)^{5g}.$$
Here $\zeta(s)$ is the Riemann zeta function and $a_i,\omega_i$ are as in Lemma~\ref{lem:hphibound} with $E=E_i$; a proof of the bound $w_i\leq 8g_i^2$ can be found in the discussion right after \eqref{eq:rootsbound}.  \end{proof}

\begin{proof}[Proof of Theorem~\ref{thm:es}~(ii)]
Assume that $A_{\bar{s}}$ has CM and let $L$ be the field of definition of the endomorphisms of $A_{\bar{s}}$. The height $h_F$ is invariant under the base change from $S$ to $S'=S\times_{\mathcal O_K}\mathcal O_L$.    Then, on replacing in the proof of Proposition~\ref{prop:cm} the application of part (ii) of Lemma~\ref{lem:cmram}  by part (i), we obtain the claimed bound for $h_F(A)=h_F(A_{S'})$.\end{proof}

\begin{proof}[Proof of Theorem~\ref{thm:es}~(i)]
Assume first that $A$ is of $\gl2$-type with $G_\QQ$-isogenies such that $A_{\bar{s}}$ has no CM. Then let $L\supseteq K$ be the number field constructed in the proof of Lemma~\ref{lem:cond*}  and denote by $S'$ the preimage of $S$ under the projection  $\spec(\mathcal O_{L})\to \spec(\OK)$. 

We now verify that $A_{S'}$ satisfies all assumptions in Proposition~\ref{prop:esgl2}.  Lemma~\ref{lem:cond*} provides that $(*)$ holds.  In particular all endomorphisms of $A_{\bar{s}}$ are defined over $L$ and then base change induces an isomorphism $\End(A_{S'})\isomto \End(A_{\bar{s}})$ by \eqref{eq:dedequivalence}. This shows that $A_{S'}$ has no CM and that $A_{S'}$ is of $\gl2$-type with $G_\QQ$-isogenies, 
since $A_{\bar{s}}$ has these two properties by assumption. Hence all assumptions in Proposition~\ref{prop:esgl2} are satisfied.

We next control $L$ and $S'$. Lemma~\ref{lem:cond*} gives that the extension $L/\QQ$ is unramified over $T=\spec(\ZZ[1/n])$ for $n=N_SD_K$ 
and that the field $L$ is contained in the Galois closure of $K(A_3)/\QQ$. Thus Dedekind's discriminant theorem and basic Galois theory lead to  \begin{equation}\label{eq:proofofeffshaf}
\textnormal{rad}(N_{S'}D_L)\mid \textnormal{rad}(N_SD_K) \quad \textnormal{and} \quad [L:\QQ]\leq l=(d3^{4g^2})!.
\end{equation}
It holds that $h_F(A)=h_F(A_{S'})$ since the generic fiber of $A_{S'}$ identifies with $A_K\times_{K}L$. 
Hence (the computation used in the proof) of Proposition~\ref{prop:esgl2} provides a  bound for $h_F(A)=h_F(A_{S'})$ in terms of $L$ and $S'$ which together with \eqref{eq:proofofeffshaf} implies
\begin{equation}\label{eq:esgl2nocm}
h_F(A)\leq 2g(3gl)^{144gl}\textnormal{rad}(N_SD_K)^{24}.
\end{equation}
This completes the proof of Theorem~\ref{thm:es}~(i) in the case when the abelian scheme $A$ over $S$ is of $\gl2$-type with $G_\QQ$-isogenies such that $A_{\bar{s}}$ has no CM.

More generally, we now assume  that $A$ is of product $\gl2$-type with $G_\QQ$-isogenies. Then $A$ is isogenous to a product $A'=\prod A_i$ of abelian schemes $A_i$ over $S$ of $\gl2$-type with $G_\QQ$-isogenies. It holds that $h_F(A)\leq v_{A'}+\sum h_F(A_i)$ and the relative dimensions $g_i$ of $A_i$ over $S$ satisfy $g=\sum g_i$. Hence, on combining the above established case and Theorem~\ref{thm:es}~(ii) with the bound for $v_{A'}$ provided by \eqref{eq:isogvartransc}, we deduce Theorem~\ref{thm:es}~(i) in general.  \end{proof}

\subsection{Bounding the number of isomorphism classes}\label{sec:numberisoclass}

Let $K$ be a number field, let $S\subseteq \spec(\OK)$ be nonempty open and let $g\in \ZZ_{\geq 1}$. In this section we use the strategy outlined in the introduction to  bound  the number of isomorphism classes of abelian schemes over $S$ of product $\gl2$-type with $G_\QQ$-isogenies. 

Before we give in $\mathsection$\ref{sec:thmesnumberproof} a proof of Theorem~B~(ii) in full generality, we consider a relevant special situation where we obtain a better bound which is used in Theorem~\ref{thm:mainint}.
\begin{proposition}\label{prop:esnumber} 
Suppose that $K/\QQ$ is normal and let $n\in \ZZ_{\geq 3}$. Then the number of isomorphism classes of abelian schemes $A$ over $S$ of relative dimension $g$ such that $A$ is of  $\gl2$-type with $G_\QQ$-isogenies and $K=K(A_n)$ is at most $
(9gd)^{(4g)^5d} \rad(D_KN_S)^{(4g)^4}.$
\end{proposition}
Here $d=[K:\QQ]$, $D_K$ and $N_S=\prod_{v\in S^c} N_v$ are as before.  We now prove some lemmas which we shall use in the proof of Proposition~\ref{prop:esnumber} given below. To control the size of each isogeny class, we apply minimal isogeny degree estimates for abelian schemes $A$ over $S$ which are based on transcendence ($\mathsection$\ref{sec:hfprop}) and we let $\kappa=\kappa(A_K)$ be as in \eqref{eq:isogvartransc}.
\begin{lemma}\label{lem:isoclassesnumber}
Let $A$ be an abelian scheme over $S$ of relative dimension $g$. Up to isomorphisms, there exist at most $\kappa^{2g}$ abelian schemes over $S$ which are isogenous to $A$.
\end{lemma}
\begin{proof}
The statement follows by using the same arguments as in step 3. of the proof of \cite[Thm 6.6]{rvk:gl2} where the exponent $2g$ comes from Masser--W\"ustholz~\cite[Lem 6.1]{mawu:abelianisogenies}. The details are as follows: If $A'$ is an abelian scheme over $S$ which is isogenous to $A$, then there exists an isogeny $\varphi:A\to A'$ of degree at most $\kappa$ by Gaudron--R\'emond~\cite[Thm 1.9]{gare:newisogenies} and \eqref{eq:dedequivalence}. As $A'$ represents the fppf quotient $A/\ker(\varphi)$,   
the number of such $A'$ up to isomorphisms is thus at most the number of subgroups of $A(\bar{K})^{\textnormal{tor}}$ of cardinality at most $\kappa$. Then \cite[Lem 6.1]{mawu:abelianisogenies} and $A(\bar{K})^{\textnormal{tor}}\cong (\QQ/\ZZ)^{2g}$ imply Lemma~\ref{lem:isoclassesnumber}.   \end{proof}
Next, we bound the number of isogeny classes. If $A$ is an abelian scheme over $S$ of $\gl2$-type with $G_\QQ$-isogenies satisfying condition $(*)$ as in  $\mathsection$\ref{sec:factorsGisogenies}, then any abelian scheme over $S$ which is isogenous to $A$ is again of $\gl2$-type with $G_\QQ$-isogenies satisfying $(*)$.   Proposition~\ref{prop:Gisog} and the proof of \cite[Thm 6.6]{rvk:gl2} lead to the following result.  

\begin{lemma}\label{lem:numbisogclasses}
Let $\mathcal C$ be the set of isogeny classes of abelian schemes $A$ over $S$ of relative dimension $g$ such that $A$ is of $\gl2$-type with $G_\QQ$-isogenies satisfying  $(*)$ and such that $A$ has no CM. It holds $|\mathcal C|\leq 9^{9dg}\textnormal{rad}(D_KN_S)^2$.
\end{lemma}
\begin{proof}
Let $U\subseteq S$ and $T\subseteq \spec(\ZZ)$ be the open subschemes defined in \eqref{def:UTGisog}. For each isogeny class $[A]$ in $\mathcal C$, Proposition \ref{prop:Gisog}  provides  a simple abelian scheme $C$ over $T$ of $\gl2$-type such that $A_{U}$ is isogenous to a power of a simple factor $B$ of  $C_U$. Moreover inequality \eqref{eq:reldimweilres} assures that the abelian scheme $C$ over $T$ has relative dimension at most $dg$. Therefore we see that $[A]\mapsto [B]$ defines an injective map   $$\mathcal C\hookrightarrow \mathcal D$$ into the set $\mathcal D$ of isogeny classes of the simple factors of $C_{U}$ where $C$ is a simple abelian scheme over $T$  of relative dimension at most $dg$ which is of $\gl2$-type. For any $n\in \ZZ_{\geq 1}$  the arguments in steps 2. and 4. of the proof of \cite[Thm 6.6]{rvk:gl2} show the following: There are at most $\tfrac{1}{24}\nu$ distinct isogeny classes of simple abelian schemes over $T$ of relative dimension $n$ which are of $\gl2$-type, where $\nu\in \ZZ$ satisfies $\nu\leq N_T^2(2n+1)^{6\rho}$ for $\rho$ the number of rational primes at most $2n+1$. This implies that $|\mathcal D|\leq 9^{9dg}N_T^2$ since the explicit prime number theorem in \cite[(3.6)]{rosc:formulas} gives $\sum_{n\leq dg}n(2n+1)^{6\rho}\leq 9^{9dg}$.   Then the inclusion $\mathcal C\hookrightarrow \mathcal D$ together with $N_T=\textnormal{rad}(D_KN_S)$ proves Lemma~\ref{lem:numbisogclasses}.   \end{proof}

Next we deal with abelian varieties which have CM. To explicitly bound their isogeny classes, we reduce to the geometric situation over $\bar{K}$ and then we apply  CM theory which reduces the problem to count CM fields with bounded degree and discriminant.

\begin{lemma}\label{lem:numbisogclassescm}
If $n\in \ZZ_{\geq 3}$ then up to isogenies there exist at most $(2g)^{(4g)^3}\rad(D_K)^{10g^2}$ distinct abelian varieties $A$ over $K$ of dimension $g$ such that $A_{\bar{K}}$ has CM and $K=K(A_n)$.
\end{lemma}
\begin{proof}
We take $n\in \ZZ_{\geq 3}$. Let $A$ be an abelian variety over $K$ of dimension $g$ such that  $K=K(A_n)$ and $A_{\bar{s}}$ has CM where $\bar{s}=\spec(\bar{K})$. Silverberg~\cite[Thm 2.4]{silverberg:fieldsofdef} gives that all endomorphisms of $A_{\bar{s}}$ are defined over $K\subset \bar{K}$ since $K=K(A_n)$.  Hence we can apply the arguments of the proof of Proposition~\ref{prop:cm}. They show that $A_{\bar{s}}$ is isogenous to a product $\prod A_i^{e_i}$ with $e_i\in \ZZ_{\geq 1}$ and $A_i$ simple pairwise non-isogenous abelian varieties  over $\bar{K}$ of dimension $g_i$ such that $E_i=\End^0(A_{i})$ is a CM field of degree $2g_i$ with $D_{E_i}\leq D$ for $D=\rad(D_K)^{5g}(2g)^{24g^2}$.  The number of isomorphism classes of CM fields $E_i$ of degree $2g_i$ and discriminant $D_{E_i}\leq D$ is bounded by the explicit version \eqref{eq:countnf} of Hermite's theorem, and each CM field $E_i$ has precisely $2^{g_i}$ distinct CM types; in this proof the CM types are $\bar{K}$-valued in the sense of \cite[p.66]{chcooo:cmbook}.  Further CM theory (\cite[1.5.4.1]{chcooo:cmbook}) gives that the abelian variety $A_{i}$ over $\bar{K}$ is determined up to an $E_i$-compatible isogeny by the isomorphism class of the CM type associated in \cite[p.66]{chcooo:cmbook} to $A_i$, where two CM types $\Phi$ and $\Phi'$, of CM fields $E$ and $E'$ respectively, are isomorphic if there exists an isomorphism $\sigma:E\isomto E'$ of fields with $\Phi=\Phi'\sigma$.  On combining these observations, we find that up to isogenies there exist at most $m=g2^{(2g+1)^2}(2D)^{2g}$ distinct abelian varieties $A_i$ over $\bar{K}$ as above. As $A_{\bar{s}}$ is isogenous to $\prod A_i^{e_i}$ and $g=\sum e_ig_i$, a combinatorial argument (\cite[p.21]{rvk:gl2}) shows that the number of isogeny classes of $A_{\bar{s}}$ with $A$ as above is bounded by $m^g$. Then the number of isogeny classes of $A$ as above is at most $m^g$ since  $K=K(A_n)$. Indeed, if $A$ and $B$ are abelian varieties over $K$ with $K=K(A_n)=K(B_n)$ then Silverberg~\cite[Thm 2.4]{silverberg:fieldsofdef} gives that $\Hom(A,B)\isomto \Hom(A_{\bar{s}},B_{\bar{s}})$ and hence $A$ is isogenous to $B$ if and only if $A_{\bar{s}}$ is isogenous to $B_{\bar{s}}$.  
This completes the proof of the lemma.  \end{proof}

We point out that Orr--Skorobogatov~\cite[Thm A]{orsk:cmfiniteness} obtained a uniform finiteness result for $\bar{\QQ}$-isomorphism classes of abelian varieties of CM type of given dimension which can be defined over a number field of given degree; their arguments use results of Tsimerman~\cite{tsimerman:ao} and Masser--W\"ustholz~\cite{mawu:abelianisogenies} and also Zarhin's quaternion trick. As a (polynomial) bound in terms of $\rad(D_K)$ is sufficient for our purpose, we shall use Lemma~\ref{lem:numbisogclassescm} which has the advantage of providing a dependence on $g$ and $K$.

We are now ready to prove Proposition~\ref{prop:esnumber} and we combine the above lemmas with the explicit height bounds obtained in Propositions~\ref{prop:esgl2} and \ref{prop:cm}.

\begin{proof}[Proof of Proposition~\ref{prop:esnumber}]
Let $A$ be an abelian scheme over $S$ of relative dimension $g$ which is of $\gl2$-type with $G_\QQ$-isogenies. Suppose  that $K/\QQ$ is normal and that $K=K(A_n)$ for some $n\in \ZZ_{\geq 3}$. Then Lemma~\ref{lem:cond*} assures that condition $(*)$ in $\mathsection$\ref{sec:factorsGisogenies} is satisfied.    In particular all endomorphisms of $A_{\bar{K}}$ are defined over $K\subset \bar{K}$. 
Now, we combine Lemma~\ref{lem:isoclassesnumber} with the upper bounds for $h_F(A)$ in Propositions~\ref{prop:esgl2} and \ref{prop:cm}. This shows that the number of isomorphism classes of abelian schemes over $S$ which are isogenous to our given $A$ is at most $(7gd)^{608g^5d} \rad(D_KN_S)^{96g^4}.$
On combining this with the estimates for the number of isogeny classes obtained in Lemmas~\ref{lem:numbisogclasses} and \ref{lem:numbisogclassescm}, we deduce an upper bound for the number of isomorphism classes as claimed in Proposition~\ref{prop:esnumber}.     
\end{proof}
\subsubsection{Proof of Theorem~B~(ii)}\label{sec:thmesnumberproof}

In this section we prove Theorem~B~(ii) by reducing to the situation treated in Proposition~\ref{prop:esnumber}. The reduction uses two results: An explicit bound for the number of twists of an abelian variety over a number field (Lemma~\ref{lem:counttwists}), and an explicit estimate \eqref{eq:countnf} for the number of number fields of given degree and bounded discriminant.

\paragraph{Twists.}Let $A$ be an abelian variety of dimension $g\geq 1$ defined over number field $K$. To bound the number of twists of the unpolarized $A$ over a finite $L/K$, it seems difficult to extend the standard approach via Galois cohomology which works in the polarized case; see Remark~\ref{rem:twists}. Instead we exploit that each twist  embeds into the Weil restriction $R_{L/K}(A_L)$ and then  we apply the isogeny estimates \eqref{eq:isogvartransc} based on transcendence.  This leads to the following result which holds moreover for twists up to isogenies.

\begin{lemma}\label{lem:counttwists}
Let $K\subseteq L$ be a finite field extension of relative degree $n$. There exist at most $(ng)^g$ isogeny classes, and at most $(ng)^g\kappa(A)^{2g}$ isomorphism classes, of abelian varieties $B$ over $K$ such that $B_L$ is isogenous to $A_L$ where $\kappa(A)$ is as in \eqref{eq:isogvartransc}.
\end{lemma}
\begin{proof}
We define $C=R(A_L)$ where $R=R_{L/K}$ denotes the Weil restriction from $L$ to $K$.  Let $B$ be an abelian variety over $K$ with an isogeny $B_L\to A_L$. Its image $R(B_L)\to C$ under the functor $R$ is a morphism  of abelian varieties over $K$.   If $S$ is a $K$-scheme then $ B(S\times_K L)\cong R(B_L)(S)$.   On using that $\bar{K}\otimes_K L\cong \bar{K}^n$, we see that $R(B_L)\to C$ is quasi-finite and hence an isogeny.  Further, the projection $S\times_K L\to S$ induces a morphism  $$
B\hookrightarrow R(B_L)\to C
$$ 
of abelian varieties over $K$. Here $B\hookrightarrow R(B_L)$ is a closed immersion  since the identity of $B(\bar{K})$ factors through $B(\bar{K})\to B(\bar{K}\otimes_K L)$ via the decomposition $\bar{K}\to \bar{K}\otimes_K L\to \prod_{\varphi} \bar{K}\to \bar{K}$ of the identity of $\bar{K}$, where $\prod_\varphi$ ranges over all $K$-morphisms $\varphi:L\to \bar{K}$.    Now,  we can apply the arguments of steps 2 and 3 of the proof of \cite[Thm 6.6]{rvk:gl2} to bound the number of isomorphism classes of $B$ as above. We first control the number of isogeny classes. Let $C_1,\dotsc, C_m$ be pairwise non-isogenous simple abelian varieties over $K$ such that $C$ is isogenous to $\prod C_i^{n_i}$ with $n_i\in \ZZ_{\geq 1}$. As $B\hookrightarrow R(B_L)\to C$ has finite kernel, Poincar\'e's reducibility theorem gives $e_i\in \ZZ_{\geq 0}$ such that $B$ is isogenous to $\prod C_i^{e_i}$. Here $\prod C_i^{e_i}$ has dimension $\sum e_ig_i=\dim(B)$ where $g_i=\dim(C_i)$, and $\dim(B)=g$ since $B$ is geometrically isomorphic to $A$. A combinatorial argument (\cite[p.21]{rvk:gl2}) shows that the number of $e\in \ZZ_{\geq 0}^m$ with $\sum e_i g_i=g$ is at most $\tbinom{a}{g}$ for $a=g+m-1$. Thus the number of distinct $\prod C_i^{e_i}$, and hence the number of isogeny classes of $B$ as above, is at most $$\tbinom{a}{g}\leq m^g.$$ Let $[B_0]$ be such an isogeny class. We now bound the number of isomorphism classes of abelian schemes $B$ as above which are inside $[B_0]$. There exists an isogeny $B_0\to B$ of degree at most $\kappa(B)$ by \eqref{eq:isogvartransc}, and $\kappa(B)$ equals $\kappa(A)$ since the geometrically isomorphic $A$ and $B$ have the same stable Faltings height $h_F$. It follows that $[B_0]$ contains at most $\kappa(A)^{2g}$ distinct isomorphism classes of $B$ as above. Here we used the arguments of step~3 in \cite[p.21]{rvk:gl2} which we applied already in the proof of Lemma~\ref{lem:isoclassesnumber} and which rely on the counting result \cite[Lem 6.1]{mawu:abelianisogenies} of Masser--W\"ustholz. Then we deduce that the number of isomorphism classes of abelian schemes $B$ as above is at most $\tbinom{a}{g}\kappa(A)^{2g}$ where $\tbinom{a}{g}\leq m^g$. This proves the lemma since $m\leq \dim(C)$ and $\dim(C)=ng$.  \end{proof}

\begin{remark}\label{rem:twists}(i) The above proof shows that one can refine  Lemma~\ref{lem:counttwists} as follows. One can replace $(ng)^g$ by the smaller quantity $\tbinom{a}{g}$, where $a=g+m-1$ and $m$ is the number of isogeny classes of simple factors of the Weil restriction $R(A_L)$. Also, if the minimal isogeny degree of abelian varieties of dimension $g$ over $K$ would be bounded by a constant depending only on $g$ and $K$, then one could replace $\kappa(A)$ by this constant. For example,  one can replace $\kappa(A)^{2g}$ by the number 8 if $K=\QQ$ and $g=1$. Indeed this follows by going into the step where we bound the number of isomorphism classes in $[B_0]$ and by replacing therein \eqref{eq:isogvartransc} by  Kenku's refinement \cite[Thm 2]{kenku:ellisogenies} of Mazur's~\cite[Thm 5]{mazur:qisogenies}.

(ii) Let $(A,\lambda)$ be a polarized abelian variety over $K$, let $L$ be a finite Galois extension of $K$ and write $G=\Aut(L/K)$. The set of isomorphism classes  of polarized abelian varieties $(A',\lambda')$ over $K$ such that $(A'_L,\lambda'_L)$ is isomorphic to $(A_L,\lambda_L)$  identifies with the first Galois cohomology group $H^1\bigl(G,\Aut(A_L,\lambda_L)\bigl)$.  
It should be possible to explicitly bound the cardinality of the finite group $\Aut(A_L,\lambda_L)$ and hence of $H^1\bigl(G,\Aut(A_L,\lambda_L)\bigl)$. However, the  automorphism group of an unpolarized abelian variety  over $K$ can be infinite and quite complicated. Thus we used a different approach in Lemma~\ref{lem:counttwists}. 
\end{remark}

\paragraph{Number fields.}Let $d$ and $D$ be in $\ZZ_{\geq 1}$. Hermite proved that the number $n$ of isomorphism classes of number fields $K$ of degree $[K:\QQ]=d$ and discriminant $D_{K}\leq D$ is finite. On making Minkowski's proof of this result explicit, one obtains for example
\begin{equation}\label{eq:countnf}
n\leq 2^{(d+1)^2}D^{d}.
\end{equation}
Stronger bounds in terms of $D$ are known, see for instance Schmidt~\cite{schmidt:countnf}, Ellenberg--Venkatesh~\cite{elve:countnf} and the recent results of Lemke-Oliver--Thorne~\cite{leth:countnf}, Bhargava--Shankar--Wang~\cite{bhshwa:countnf} and Anderson et al. \cite{anetal:countnf}. As a polynomial bound in terms of $D$ is sufficient for our purpose, we shall use \eqref{eq:countnf} which has the advantage of providing also a dependence on $d$. To compute the bound \eqref{eq:countnf}, we go into the proof of \cite[Thm 3.2.13]{neukirch:ant}. We find that $\textnormal{vol}(X)=\tfrac{4}{\pi}(2\pi)^scD^{1/2}$ and $c> \tfrac{\pi}{4}(\tfrac{2}{\pi})^s$.  Hence, if $K$ is totally complex, then it has an integral primitive element $\alpha$ such that each coefficient $a$ of its  minimal polynomial $\prod_\sigma(x-\sigma(\alpha))$  satisfies $|a|=|\sum \prod \sigma(\alpha)|\leq  2^dD$.  If $K$ has a real embedding $\tau_0$, then we take $X=\{(z_\tau)\in K_{\RR}; \, |z_{\tau_0}|< cD^{1/2},|z_\tau|< 1 \textnormal{ for }  \tau\neq \tau_0\}$ with $\textnormal{vol}(X)=(2\pi)^s2^rcD^{1/2}$ and $c>(\tfrac{2}{\pi})^s$  
which leads again to $|a|\leq 2^dD$.  This implies \eqref{eq:countnf}.

We are now ready to prove Theorem~B~(ii). Recall that therein $K$ is a number field of degree $d$,  $S\subseteq \spec(\OK)$ is a nonempty open subscheme, $N_S=\prod_{v\in S^c}N_v$ and $g\in \ZZ_{\geq 1}$.

\begin{proof}[Proof of Theorem~B~(ii)]
Let $A$ be an abelian scheme over $S$ of relative dimension $g$. To reduce to the situation treated in Proposition~\ref{prop:esnumber}, we make a base change to a controlled scheme $T$ which is defined as follows: Choose $\bar{\QQ}$ with $K\subset \bar{\QQ}$, let $L\subset \bar{\QQ}$ be the normal closure of $K(A_3)/\QQ$ and let $p:\spec(\OL_L)\to \spec(\OK)$ be the projection. Now we define $$T=p^{-1}(S).$$   It is a connected Dedekind scheme with function field $L$ of degree $[L:\QQ]\leq l$ where $l=(d3^{4g^2})!$. As in the proof of Lemma~\ref{lem:cond*}, we see that the criterion of N\'eron--Ogg--Shafarevich~\cite[Thm 1]{seta:goodreduction} implies that $K(A_3)$, and hence its normal closure $L$, is unramified over each rational prime $p\nmid 3N_SD_K$.  Thus $\textnormal{rad}(N_TD_L)$ divides $\rad(3N_SD_K)$ and \cite[Lem 6.2]{rvk:szpiro} leads to $D_{L}\leq D$ where  $D=\textnormal{rad}(3N_
SD_K)^{3l}l^{6l^2}$.
Here we used the effective version of the prime number theorem in \cite[(3.12) and (3.16)]{rosc:formulas} which gives that the number $s$ of rational prime divisors of $D_L$ satisfies $s^s<\textnormal{rad}(D_L)^{2}$ if $s> 12$.  

Let $F(S)$ be the set of isomorphism classes of abelian schemes over $S$ of relative dimension $g$ which are of $\gl2$-type with $G_\QQ$-isogenies, and let $F'(S)$ be the set of all $[A]\in F(S)$ with $K(A_3)=K$. The notion of $\gl2$-type with $G_\QQ$-isogenies is stable under the base change $T\to S$ for any $T$ as above.  Hence $[A]\mapsto [A_{T}]$ defines a map
$$
F(S)\to \cup_{T\in \mathcal T}F'(T).$$
Here $\mathcal T$ is the set of all schemes $T=p^{-1}(S)$ where $p:\spec(\OL_L)\to \spec(\OK)$ is the projection for some number field $K\subseteq L\subset \bar{\QQ}$ with  the following properties: $L/\QQ$ is normal, $[L:\QQ]\leq l$, $D_L\leq D$, and $\rad(N_TD_L)$ divides $\rad(3N_SD_K)$. 

To bound $|F(S)|$ we now control $|\mathcal T|$ and the degree of the map $F(S)\to \cup_{T\in \mathcal T} F'(T)$. It follows from \eqref{eq:dedequivalence} that any two abelian schemes over a connected Dedekind scheme are isomorphic if and only if their generic fibers are isomorphic. Thus Lemma~\ref{lem:counttwists} gives that the fibers of the map $F(S)\to \cup_{T\in \mathcal T} F'(T)$ have cardinality at most
$$(lg)^g\sup\kappa([A])^{2g}$$
with the supremum taken over all $[A]\in F(S)$. Here $\kappa([A])=\kappa(A_K)$ and $\kappa(A_K)$ is defined in \eqref{eq:isogvartransc}. Further, the construction of $\mathcal T$ shows that $|\mathcal T|$ is bounded by the number of subfields $L\subset\bar{\QQ}$ with $[L:\QQ]\leq l$ and $D_L\leq D$. Therefore \eqref{eq:countnf} implies  $$|\mathcal T|\leq l^22^{(l+1)^2}D^l.$$ 
If $T\in \mathcal T$ then its function field $L=k(T)$ is normal over $\QQ$ and the definition of $F'(T)$ assures that each $[A]\in F'(T)$ satisfies $L=L(A_3)$. Therefore we can apply Proposition~\ref{prop:esnumber} which bounds $|F'(T)|$ for all $T\in \mathcal T$, while Theorem~\ref{thm:es} bounds  $\sup\kappa([A])$. On combining these bounds with the above displayed results, we deduce
\begin{equation}\label{eq:esgl2numberbound}
|F(S)|\leq  (3gl)^{(6g)^5l^3} \rad(N_SD_K)^{4l^2}.
\end{equation}
We now deal with the general case and we consider $A$ with the following properties $(\dagger)$: $A$ is an abelian scheme over $S$ of relative dimension $g$ and $A$ is of product $\gl2$-type with $G_\QQ$-isogenies. Let $m$ be the number of isogeny classes of abelian schemes over $S$ of relative dimension at most $g$ which are of $\gl2$-type with $G_\QQ$-isogenies. Suppose that $A_1,\dotsc,A_m$ generate these isogeny classes,  
and write $g_i$ for the relative dimension of $A_i$ over $S$.  Then for any $A$ with $(\dagger)$ there exists $e\in \ZZ^{m}_{\geq 0}$ with $g=\sum e_ig_i$ such that $A$  is isogenous to $\prod A_i^{e_i}$. Thus a combinatorial argument, which we used already in  the proof of Lemma~\ref{lem:counttwists}, shows that the number of isogeny classes of abelian schemes $A$ with $(\dagger)$ is at most $m^g$.  
Then Lemma~\ref{lem:isoclassesnumber} gives that the number of isomorphism classes of abelian schemes $A$ with $(\dagger)$ is at most $\sup\kappa(A)^{2g}m^g$ with the supremum taken over all $A$ with $(\dagger)$ for $\kappa(A)=\kappa(A_K)$ as in \eqref{eq:isogvartransc}. Theorem~\ref{thm:es} bounds $\sup \kappa(A)$, and \eqref{eq:esgl2numberbound} leads to $$m\leq \sum c_k\leq \tfrac{g}{2}(g+1)c_g.$$ Here the sum is taken over all $k\in \ZZ_{\geq 1}$ with $k\leq g$, and $c_k$ denotes the expression obtained by replacing $g$ by $k$ in the upper bound of \eqref{eq:esgl2numberbound}. We combine the bounds and then we deduce the estimate for the number of isomorphism classes claimed in Theorem~B~(ii).  
\end{proof}

\newpage

\section{Polarizations}\label{sec:polarizations}

Let $L$ be a totally real number field of degree $g$ over $\QQ$ with ring of integers $\OL$, and let $I$ be a nonzero finitely generated $\OL$-submodule of $L$  with a positivity notion $I_+$. In this section we continue our study in \cite[$\mathsection$9]{vkkr:hms} of equivalence classes of $I$-polarizations on $\OL$-abelian schemes of relative dimension $g$. In particular, we explicitly bound the number of such equivalence classes by using and generalizing the arguments of \cite[$\mathsection$9]{vkkr:hms}. 

Let $S$ be a connected Dedekind scheme with function field of characteristic zero, let $\mathcal C$ be the category of $\OL$-abelian schemes over $S$ and let $(A,\iota)$ be an $\OL$-abelian scheme over $S$ of relative dimension $g$. We denote by $\Phi$ the set of $I$-polarizations on $A$.

\paragraph{Equivalence relation.}Recall from Section~\ref{sec:hilbmodstackdef} that an $I$-polarization  on $A$ is an $\OL$-module morphism from $I$ to the $\OL$-module $\Hom_\OL(A,A^\vee)^{\textnormal{sym}}$, which maps $I_+$ to polarizations  and which  induces an isomorphism $I \otimes_{\OL} A \isomto A^\vee$. We say that two $I$-polarizations  $\varphi,\varphi'\in\Phi$ are equivalent, and we write $\varphi\sim\varphi'$,  if there exists an automorphism $\sigma$ of the $\OL$-abelian scheme $A$ such that $\varphi(\lambda)=\sigma^\vee\varphi'(\lambda)\sigma$ for all $\lambda\in I$.

\paragraph{Equivalence classes.}To state our bound for the equivalence classes, we denote by $w$ and $w'$ the number of roots of unity in $\End^0_{\mathcal C}(A)$ and $\End_{\mathcal C}(A)$ respectively.

\begin{proposition}\label{prop:polbound}
The number of equivalence classes satisfies $$|\Phi/{\sim}|\leq \tfrac{w^2}{w'}\cdot 2^{g-1}.$$
\end{proposition}
The number $w$ is uniformly bounded: For example \eqref{eq:rootsbound} gives that $w\leq 4g^2$ and $\omega\ll_\epsilon g^{1+\epsilon}$ for all real $\epsilon>0$, while Lemma~\ref{lem:polembedding} shows that $w=w'=2$ if the abelian scheme $A$ over $S$ has no CM. In particular we obtain that $w=w'=2$ if the function field of $S$ embeds into $\RR$,  and hence \cite[Prop 9.1]{vkkr:hms} is a special case of Proposition~\ref{prop:polbound}. 
An important feature of Proposition~\ref{prop:polbound} is that the bound $\tfrac{w^2}{w'}\cdot 2^{g-1}$ can be controlled only in terms of the relative dimension $g$ of  $A$. This is crucial for deducing the next result which explicitly controls the degree over $S$ of the natural forgetful morphism $\phi_\varphi \colon \hilbmod \to \abomult$ defined in \cite[(6.2)]{vkkr:hms} with respect to the Hilbert moduli stack $\mathcal M^I$ associated to $(L,I,I_+)$.


\begin{corollary}\label{cor:forgetpol}
The degree of $\phi_\varphi(S)$ is at most $c\cdot 2^g$ for some constant $c\leq 4g^4$.
\end{corollary}
\begin{proof}
We go into the proof of \cite[Cor 9.2]{vkkr:hms} and we replace therein the application of \cite[Prop 9.1]{vkkr:hms} by Proposition~\ref{prop:polbound}. This proves the corollary. 
\end{proof}

The proof of Proposition~\ref{prop:polbound} will be given in the remaining of this section. The strategy is as follows: After applying the formal result \cite[Lem 9.3]{vkkr:hms} which computes the set $\Phi$ in terms of certain isomorphisms $A\isomto A^\vee\otimes_\OL I^{-1}$, we use and generalize the arguments of \cite[Lem 9.4]{vkkr:hms} to construct in Section~\ref{sec:polinjection} an embedding $$\Phi/{\sim}  \hookrightarrow \Gamma^\times/U.$$ Here $\Gamma$ is an order of a number field $F\supseteq L$ which comes from Lemma~\ref{lem:polembedding} and $U$ denotes a certain subgroup of the units $\Gamma^\times$ of $\Gamma$. Then in Section~\ref{sec:polda} we use algebraic number theory to relate the cardinality of $\Gamma^\times/U$ to a ratio $R_F/(R_{F/L}R_L)$ of regulators which can be computed by using classical results from CM theory. 


\subsection{Equivalence classes of polarizations and quotients of unit groups}\label{sec:polinjection}

We continue our notation introduced at the beginning of Section~\ref{sec:polarizations}. The goal of the current section is to view equivalence classes of $I$-polarizations on $\OL$-abelian schemes inside finite quotients of certain unit groups of orders of number fields containing $L$. 

To give a precise statement,  we recall that $\Phi$ denotes the set of $I$-polarizations on an arbitrary $\OL$-abelian scheme $(A,\iota)$ over $S$ of relative dimension $g=[L:\QQ]$ where $S$ is a connected Dedekind scheme with function field of characteristic zero.  Notice that $D=\End^0_{\mathcal C}(A)$ becomes an $L$-algebra via $\iota$, since $\iota(\OL)\subseteq \End_{\mathcal C}(A)$ and $L=\OL\otimes_\ZZ\QQ$.


\begin{lemma}\label{lem:polembedding}
The following statements hold.
\begin{itemize}
\item[(i)] Suppose that $\dim_L(D)=1$. Then the ring $\End_\mathcal C(A)$ identifies with an order $\Gamma$ of $L$  and there exists an embedding $\Phi/{\sim}  \hookrightarrow \Gamma^\times/(\Gamma^\times)^2.$
\item[(ii)] Suppose that $\dim_L(D)\neq 1$. Then $\End_\mathcal C(A)$ identifies with an order $\Gamma\supset \OL$ of a CM field $F$ with maximal totally real subfield $L$ and there exists an embedding $$\Phi/{\sim}  \hookrightarrow \Gamma^\times/N(\Gamma^\times).$$
Here the image $N(\Gamma^\times)$ of $\Gamma^\times$ under the field norm $N:F\to L$ is a subgroup of $\Gamma^{\times}$.
\end{itemize}
\end{lemma}

The assumption that the function field of $S$ has characteristic zero is crucial for (ii) which relies on the computation of the $L$-algebra $D$ in Lemma~\ref{lem:divalgcompchar0} below.

In what follows in this section we write $\otimes=\otimes_\OL$ for simplicity. Our proof of Lemma~\ref{lem:polembedding} uses the formal result in \cite[Lem 9.3]{vkkr:hms} which computes $\Phi/{\sim}$ in terms of equivalence classes of certain isomorphisms $A\isomto A^\vee\otimes I^{-1}$. To state \cite[Lem 9.3]{vkkr:hms}, we need to introduce some notation. Firstly, we observe that multiplication defines an isomorphism $\mu:I^{-1}\otimes I\isomto \OL$ of $\OL$-modules since $I$ is invertible. Secondly, in the category of $\OL$-modules it holds: If $X,Y,Z$ are $\OL$-modules, then mapping any morphism $\varphi:X\to \Hom(Y,Z)$  to the morphism $X\otimes Y\to Z$ defined by $x\otimes y\mapsto \varphi(x)(y)$ gives a bijection $\Hom(X,\Hom(Y,Z))\isomto \Hom(X\otimes Y,Z)$ which is functorial. 
Therefore, on using the isomorphism $I^{-1}\otimes I\otimes A\isomto A$ induced by $\mu$ and tensoring with $I^{-1}$, we obtain a bijection 
$$\tau:\Hom(I,\Hom_{\mathcal C}(A,A^\vee))\isomto \Hom_{\mathcal C}(A,I^{-1}\otimes A^\vee).$$
Further, any $\lambda\in I$ defines a morphism $m_\lambda: I^{-1}\otimes A^{\vee}\to  A^{\vee}$ of $\OL$-abelian schemes over $S$ as follows: If $T$ is an $S$-scheme, then $m_\lambda(T)$ sends $j\otimes x$ in $I^{-1}\otimes A^{\vee}(T)$ to the element $j\lambda \otimes x$ in $\OL\otimes A^{\vee}(T)\cong A^{\vee}(T)$. 
We denote by $\Psi$ the set of isomorphisms $\psi: A\isomto I^{-1}\otimes A^\vee$ of $\OL$-abelian schemes over $S$ such that for each $\lambda\in I$ the morphism $m_\lambda \psi:A\to A^\vee$ is symmetric and moreover a polarization if $\lambda\in I_+$.  
Now \cite[Lem 9.3]{vkkr:hms} says 
\begin{equation}\label{eq:formalpolcomp}
\tau(\Phi)=\Psi.
\end{equation} 
On using the canonical bijection $\tau:\Phi\isomto \Psi$ and the equivalence relation on $\Phi$, we obtain an equivalence relation on $\Psi$ by transport of structure. We write again $\psi\sim\psi'$ if two isomorphisms $\psi,\psi'\in\Psi$ are equivalent. We are now ready to prove Lemma~\ref{lem:polembedding}.

\begin{proof}[Proof of Lemma~\ref{lem:polembedding}]
To show (i) we suppose that $\dim_L(D)=1$. Then we are in the case when $\End^0_{\mathcal C}(A)\cong L$, and this case was treated in our proof of \cite[Lem 9.4]{vkkr:hms}. Therefore we deduce (i) by using exactly the same arguments as in the proof of \cite[(9.2)]{vkkr:hms}.

We next show (ii). Assume that  $\dim_L(D)\neq 1$ and notice that our abelian scheme $A$ over $S$ satisfies all the assumptions made in Lemma~\ref{lem:divalgcompchar0}. Therefore Lemma~\ref{lem:divalgcompchar0} gives an $L$-algebra isomorphism $D\isomto F$ from $D$ to a CM field $F$ with maximal totally real subfield $L$. Further, $\End_{\mathcal C}(A)$ is a subring of $\End(A)$ which is a finitely generated torsion-free abelian goup. Thus $\End_{\mathcal C}(A)$ identifies with an order of $D$. Let $\Gamma\subset F$ be the image of this order under the $L$-algebra isomorphism $D\isomto F$. Then $\Gamma$ is an order of $F$, and $\Gamma$ contains $\OL$ since $\iota(\OL)\subseteq\End_{\mathcal C}(A)$.  This proves the first part of (ii). 

It remains to construct the embedding of $\Phi/{\sim}$ appearing in the second part of (ii). As the equivalence relation on $\Psi$ is defined by transport of structure via \eqref{eq:formalpolcomp}, we see that $\tau:\Phi\isomto \Psi$ induces a bijection from $\Phi/{\sim}$ to the set $ \Psi/{\sim}$. To embed $\Psi$ into $\Gamma^\times$,  we may and do assume that there exists an isomorphism  $\psi_0:A\isomto I^{-1}\otimes A^\vee$ which lies in $\Psi$. We observe that $\psi \mapsto \psi_0^{-1}  \psi$ defines an embedding $\Psi\hookrightarrow \Aut_{\mathcal C}(A)$, and the isomorphism $D\isomto F$ identifies the image of $\Aut_\mathcal C(A)\hookrightarrow D$ with $\Gamma^\times$. Therefore the composition of the maps $\Psi\hookrightarrow \Aut_{\mathcal C}(A)\hookrightarrow D\isomto F$ gives an embedding
$$
j \colon \Psi \hookrightarrow \Gamma^\times.$$ 
Further, the image $N(\Gamma^{\times})$ of $\Gamma^\times$ under the field norm $N:F^\times\to L^\times$ is a subgroup of $\Gamma^\times$, since the order $\Gamma$ is contained in the maximal order $\OL_F$ of $F$ and since $N(\OL_F^\times)$ is contained in $\OL^\times\subseteq \Gamma^\times$. We now claim that $j:\Psi\hookrightarrow \Gamma^\times$ induces an embedding 
\begin{equation}\label{eq:InclusionModSqaures}
\Psi/{\sim} \hookrightarrow \Gamma^\times / N(\Gamma^{\times}).
\end{equation} 
To prove this claim, we suppose that $\psi,\psi'\in\Psi$ and we denote by $\varphi,\varphi'\in \Phi$ their preimages under $\tau$ respectively. Then $\psi\sim\psi'$  if and only if there exists  $\sigma \in \Aut_\mathcal C(A)$ with $
\varphi' =\sigma^\vee \varphi \sigma
$. Moreover, on applying the formal arguments of the proof of \cite[Lem 9.3]{vkkr:hms}, we compute that $
\varphi' =\sigma^\vee \varphi \sigma
$ if and only if  
$
\psi' = (I^{-1}\otimes\sigma^\vee) \psi \sigma
$.
The latter equality is equivalent to 
$
\psi_0^{-1} \psi' = \psi_0^{-1} ( I^{-1}\otimes \sigma^\vee) \psi_0 \psi_0^{-1}  \psi  \sigma,
$
which in turn is equivalent to
$
(\psi_0^{-1}\psi') = \sigma^\star (\psi_0^{-1}\psi) \sigma$. Here for any $f \in \End_{\mathcal C}(A)$ the involution $^\star$ is defined by $f^\star = \psi_0^{-1}  (I^{-1}\otimes f^\vee)  \psi_0.$ Thus we obtain that $\psi\sim \psi'$ if and only if there exists $\sigma\in \Aut_{\mathcal C}(A)$ with 
$$
(\psi_0^{-1}\psi') = \sigma^\star (\psi_0^{-1}\psi) \sigma.$$
To further compute this in terms of $\Gamma^\times$, let $\lambda\in I_+$ and denote by $^*$ the usual Rosati involution on $\End^0(A)$ associated to the polarization $\varphi(\lambda)$ of $A$, where $\varphi=\tau^{-1}(\psi_0)$ lies in $\Phi$ by \eqref{eq:formalpolcomp}. Then our $^\star$ coincides on $\Gamma$ with the restriction of the Rosati involution $^*$ to $D=\End_\mathcal C^0(A)$. Indeed, on using that any $f$ in $\End_\mathcal C(A)$ is a morphism of $\OL$-abelian schemes, we compute $f^\vee m_\lambda=m_\lambda (I^{-1}\otimes f^\vee)$ and then we deduce that $f^*=\varphi(\lambda)^{-1}f^\vee \varphi(\lambda)=f^\star$ inside $\End_\mathcal C^0(A)$ since $m_\lambda\psi_0=\varphi(\lambda)$ by \cite[(9.1)]{vkkr:hms}.  Moreover, as $\End^0(A)\isomto \End^0(A_{k})$ for $k$ the function field of the connected Dedekind scheme $S$, it follows for example from \cite[p.5]{lang:cm} that the positive involution
$^*$ on $D\subseteq\End^0(A)$ identifies via the $L$-algebra isomorphism $D\isomto F$ with the non-trivial automorphism $c$ of $F/L$.
Then the above computations and $^*= \, ^\star$ on $D$ show that $\psi\sim\psi'$ if and only if there exists $\gamma\in \Gamma^\times$ with $$j(\psi') = c(\gamma)j(\psi)\gamma=N(\gamma) j(\psi),$$ where the last equality exploits that $\Gamma$ is commutative. Hence our claim in \eqref{eq:InclusionModSqaures} follows and thus all statements of (ii) are proven. This completes the proof of Lemma~\ref{lem:polembedding}.\end{proof}

The above proof relies on description of the $L$-algebra $D=\End_{\mathcal C}^0(A)$ in the following result which further refines Lemma~\ref{lem:divalg}~(ii) in the case when the base is a connected scheme which has a point whose residue field is of characteristic zero.   
\begin{lemma}\label{lem:divalgcompchar0}
Let $L$ be a totally real number field with ring of integers $\OL$, let $S$ be a connected Dedekind scheme with $k(S)$ of characteristic zero, and let $A$ be an $\OL$-abelian scheme over $S$ of relative dimension $g=[L:\QQ]$. The $L$-algebra $D=\End_{\mathcal C}^0(A)$ either has dimension one or is isomorphic to a CM field with maximal totally real subfield $L$.
\end{lemma}
\begin{proof}
We first treat the case when the function field $k=k(S)$ embeds into $\CC$. Let $A_\CC$ be the base change of $A_k$ to $\CC$ via an embedding $k\hookrightarrow \CC$ and consider the $\QQ$-vector space $V$ given by its first homology $H_1(A_\CC(\CC),\QQ)$. We obtain a ring morphism
$\rho:D\to \End(V)$ by composing $D\subseteq \End^0(A)\to \End^0(A_\CC)$ with the morphism $\End^0(A_\CC)\to \End(V)$ induced by the functor $H_1(\,\ ,\QQ)$.  By assumption $A=(A,\iota)$ is an $\OL$-abelian scheme over $S$. Composing $\rho$ with $\iota\otimes_\ZZ\QQ:L\to D$ gives $V$ an $L$-vector space structure. The $L$-dimension of $V$ is two, since its $\QQ$-dimension is $2g$ and $[L:\QQ]=g$. 
As the image of $L\to D$ lies in the center of $D$,  
we then see that $\rho$ induces a morphism 
\begin{equation}\label{eq:dembedding}
D\to \uM_2(L)
\end{equation}
of $L$-algebras.  
This morphism is injective since $D$ is a division algebra by Lemma~\ref{lem:divalg}.  
This implies that $\dim_L D\leq 4$. We claim that $\dim_L D$ can not be $3$ or $4$. Indeed $\dim_L D=4$ would imply that  $D\isomto \uM_{2}(L)$ which is not possible since $\uM_{2}(L)$ is not a division algebra. Further, the division algebra $D$ is a central simple algebra over its center $Z$. Thus $\dim_Z D$ is a square,  while \eqref{eq:dembedding} and $\dim_L D\leq 4$ assure that $\dim_L Z$ lies in $\{1,2,4\}$.  Hence we deduce that $\dim_L D=3$ is not possible, proving our claim. Suppose now that  $\dim_L D=2$. Then the division $L$-algebra $D$ is commutative.  
In particular $D$ is a number field of degree $2\dim(A_\CC)$ which embeds into $\End^0(A_\CC)$. Thus $D$ is totally imaginary.  As $[L:\QQ]=g$ and $\dim_L D=2$, we then conclude that $D$ is $L$-isomorphic to a CM field with maximal totally real subfield $L$.  
This proves Lemma~\ref{lem:divalgcompchar0} in the case when $k$ embeds into $\CC$. 

The general situation can be reduced to the above case as follows. As $k$ is of characteristic zero, we may and do choose a subfield $k'\subseteq k$ with the following properties: It is finitely generated over $\QQ$, $A_k$ descends to an abelian variety $A_{k'}$ over $k'$ and there is a ring morphism $\End^0(A_k)\to \End^0(A_{\CC})$ where now $A_\CC$ is the base change of $A_{k'}$ to $\CC$ via an embedding $k'\hookrightarrow \CC$. Thus the arguments above \eqref{eq:dembedding} provide again a morphism $D\to \uM_2(L)$ of $L$-algebras and then the arguments  below \eqref{eq:dembedding} imply Lemma~\ref{lem:divalgcompchar0} in general.   
\end{proof}

In the next sections we bound the cardinality of $\Gamma^\times/(\Gamma^\times)^2$ and $\Gamma^\times/N(\Gamma^\times)$.

\subsection{Quotients of unit groups}\label{sec:polda}

In this section we prove two lemmas controlling the cardinality of quotient groups of the form $\Gamma^\times/N(\Gamma^\times)$. For this purpose we use algebraic number theory in order to relate and compute various quotients of unit groups in extensions of number fields.

We introduce some notation. Let $F$ be a number field and let $L\subseteq F$ be a subfield. We denote by $N:F^\times\to L^\times$ the group morphism given by the field norm, we write $R_{F/L}$ for the relative regulator (\cite[$\mathsection$3]{cofr:regulatorprocams}) and we put $R_F=R_{F/\QQ}$. Further for any ring $R$ we denote by $\mu(R)$ the group of roots of unity in $R$ and we  write $w_R=|\mu(R)|$. Our first goal is to prove the following lemma on quotients of unit groups in general number fields.  

\begin{lemma}\label{lem:diophantineanalysis}
The following statements hold. 
\begin{itemize}
\item[(i)] If $\Gamma$ is an order of $F$ with $\OL_L\subseteq \Gamma$, then $N(\Gamma^\times)$ is a subgroup of $\Gamma^\times$ and 
$$|\ker(N)\cap \Gamma^\times|\cdot|\Gamma^\times/N(\Gamma^\times)|\leq|\ker(N)\cap \OL_F^\times|\cdot|\mathcal O_F^\times/N(\mathcal O_F^\times)|.
$$ 
\item[(ii)] The cardinality of the quotient $\mathcal O_F^\times/N(\mathcal O_F^\times)$ satisfies 
$$|\mathcal O_F^\times/N(\mathcal O_F^\times)|\leq w_F\cdot |\mathcal O_F^\times/\mu(F)\mathcal O_L^\times|\cdot\tfrac{R_F}{R_{F/L}R_L}.$$
\end{itemize}
\end{lemma}
\begin{proof}
To show (i) we assume  that the order $\Gamma$ contains $\OL_L$. Any order of $F$ is contained in the maximal order $\OL_F$ of $F$,  
and $N$ sends $\OL_F^\times$ to $\OL_L^\times$ since the norm takes the form $x\mapsto \prod \sigma(x)$ with the product taken over all $\sigma\in \Hom_L(F,\bar{F})$. Thus our assumption that $\Gamma$ contains $\OL_L$ implies that the restriction of $N$ to $\Gamma^\times$ is a group morphism $N:\Gamma^\times\to \Gamma^\times$. In particular $N(\Gamma^\times)$ is a subgroup of $\Gamma^\times$ and  we obtain a diagram 
  $$
\xymatrix{
1 \ar[r]  & \ker(N)\cap \Gamma^\times\ar[r] \ar[d] & \Gamma^\times \ar[r]^N \ar[d] & \Gamma^\times \ar[r] \ar[d] & \Gamma^\times/N(\Gamma^\times) \ar[r]\ar[d] & 1 \\
1 \ar[r] &  \ker(N)\cap \OL_F^\times \ar[r] &  \OL_F^\times \ar[r]^N &  \OL_F^\times \ar[r] & \OL_F^\times/N(\OL_F^\times) \ar[r] & 1} 
$$
where $\Gamma^\times\to \mathcal O_F^\times$ is the inclusion. This is a commutative diagram of abelian groups, with exact rows and with the first three vertical morphisms injective. Therefore an application of the snake lemma shows that the kernel of $\Gamma^\times/N(\Gamma^\times)\to \OL_F^\times/N(\OL_F^\times) $ has cardinality at most $|\ker(N)\cap \OL_F^\times|/|\ker(N)\cap \Gamma^\times|$ when $\ker(N)\cap \OL_F^\times$ is finite. This implies the inequality in (i), which also holds when $\ker(N)\cap \OL_F^\times$ is infinite since $\infty\leq \infty$.   
 
We next prove (ii). Costa--Friedman observed that $|\OL_L^\times/\mu(L)N(\OL_F^\times)|$ can be written as the quotient of regulators $R_F/(R_{F/L}R_L)$, see \cite[(3.1)]{cofr:regulatorprocams}. Then, on using that the index of subgroups is multiplicative and that $N(\OL_F^\times)$ is a subgroup of $\OL_L^\times$, we  compute
$$|\OL_F^\times/N(\OL_F^\times)|=|\OL_F^\times/\OL_L^\times|\cdot |\OL_L^\times/N(\OL_F^\times)|=|\OL_F^\times/\mu(F)\OL_L^\times|\cdot n \cdot \tfrac{R_F}{R_{F/L}R_L} \cdot m.$$  
Here $n=|\mu(F)\OL_L^\times/\OL_L^\times|$  and  $m=|\mu(L)N(\OL_F^\times)/N(\OL_F^\times)|$. The indices $n$ and $m$ are both finite, and they satisfy $n\leq w_F/w_L$ and $m\leq w_L$. Therefore the displayed formula implies the upper bound for $|\OL_F^\times/N(\OL_F^\times)|$ stated in (ii). This completes the proof. 
\end{proof}
  
We now apply the above lemma in the case of CM fields. In this case CM theory allows to compute the quantities appearing in Lemma~\ref{lem:diophantineanalysis} leading to the following result.

\begin{lemma}\label{lem:quotientgroupbound}
Suppose that $F$ is a CM field with maximal totally real subfield $L$ and write $g=[L:\QQ]$. Then any order $\Gamma$ of $F$ with $\OL_L\subseteq \Gamma$ satisfies
$$|\Gamma^\times/N(\Gamma^\times)|\leq \tfrac{w_F^2}{w_\Gamma}\cdot 2^{g-1}.$$
\end{lemma}

\begin{proof}
We assume that $F$, $L$ and $\Gamma$ are as in the statement. Then the claimed upper bound for $|\Gamma^\times/N(\Gamma^\times)|$ is in fact contained in the statements of Lemma~\ref{lem:diophantineanalysis}. To see this we now compute the quantities appearing in Lemma~\ref{lem:diophantineanalysis}. 
We first show that $$\ker(N)\cap \OL_F^\times=\mu(F) \quad \textnormal{and}\quad\mu(\Gamma)= \ker(N)\cap \Gamma^\times.$$ For this purpose we may and do assume that the CM field $F\subset \CC$. Then the non-trivial automorphism of the quadratic field extension $F/L$ is given by complex conjugation $\bar{} \, $ and we obtain that $\ker(N)\cap \OL_F^\times=\{|x|=1\}$ where we write $\{|x|=1\}$ for $\{x\in \OL_F^\times; \ x\bar{x}=1\}.$ Further $\bar{} \, $ commutes with any embedding $F\to \CC$ since $F$ is a CM field. Thus, if $x\in \OL_F^\times$ satisfies $x\bar{x}=1$, then for each embedding $\sigma:F\to \CC$ the complex absolute value of $\sigma(x)$ is 1  and hence $x$ has to be a root of unity. This shows that $\{|x|=1\}$ is contained in $\mu(F)$ which in turn is contained in $\{|x|=1\}$. Thus we obtain that $\ker(N)\cap \OL_F^\times=\{|x|=1\}=\mu(F)$, which implies that $\ker(N)\cap \Gamma^\times=\mu(\Gamma)$ since $\Gamma\subseteq \OL_F$.  
Next, we compute that
$$|\mathcal O_F^\times/\mu(F)\mathcal O_L^\times|\cdot\tfrac{R_F}{R_{F/L}R_L}=2^{g-1}.$$
Indeed this is a consequence of $R_{F/L}=1$ and the standard formula for the regulator ratio $R_F/R_L$ which can be found for example in \cite[Prop 4.16]{washington:cyclotomic}. Here $R_{F/L}=1$  follows directly from the definition of the relative regulator (\cite[$\mathsection$3]{cofr:regulatorprocams}), since $\{|x|=1\}=\mu(F)$ assures that the relative units of $F/L$ have (free) rank zero. Now, we insert the displayed computations  into the inequalities of Lemma~\ref{lem:diophantineanalysis} and then we deduce Lemma~\ref{lem:quotientgroupbound}.\end{proof}

\subsection{Proof of Proposition~\ref{prop:polbound}}\label{sec:polpropproof}
In this section we prove Proposition~\ref{prop:polbound} by combining the results collected in previous sections. We continue the notation introduced at the beginning of Section~\ref{sec:polarizations}.

\begin{proof}[Proof of Proposition~\ref{prop:polbound}]
Recall that the set $\Phi$ in Proposition~\ref{prop:polbound}  denotes the set of $I$-polarizations on an arbitrary $\OL$-abelian scheme $(A,\iota)$ over $S$ of relative dimension $g=[L:\QQ]$, where $S$ is a connected Dedekind scheme with function field of characteristic zero. Hence we can apply Lemma~\ref{lem:polembedding} and we consider again the $L$-algebra $D=\End_{\mathcal C}^0(A)$.

We first assume that $\dim_L(D)=1$. Then Lemma~\ref{lem:polembedding}~(i) gives an order $\Gamma$ of $L$ with $\End_{\mathcal C}(A)\cong \Gamma$ and an embedding $\Phi/{\sim}\hookrightarrow \Gamma^\times/(\Gamma^\times)^2$. In particular it holds $$|\Phi/{\sim}|\leq |\Gamma^\times/(\Gamma^\times)^2|.$$
Further, Dirichlet's unit theorem~\cite[p.81]{neukirch:ant} provides that the free part of $\Gamma^\times$ has rank $g-1$, since $L$ is totally real. Thus $|\Gamma^\times / (\Gamma^{\times})^2|$ is at most $w_{L}\cdot 2^{g-1}$, where we recall that $w_R$ denotes the cardinality of the group $\mu(R)$ of roots of unity in a ring $R$. Notice that $\mu(L)=\{1,-1\}$ are the only roots of unity in the totally real field $L$, the ring $D$ is isomorphic to $L$ since $\dim_L(D)=1$, and the ring $\Gamma\cong\End_{\mathcal C}(A)$ contains $\{1,-1\}$. Hence we obtain that $w=w_L=w_\Gamma=w'$ and it follows that $|\Gamma^\times / (\Gamma^{\times })^2| \leq \tfrac{w^2}{w'}\cdot 2^{g-1}$. Therefore the displayed inequality proves Proposition~\ref{prop:polbound} in the case when $\dim_L(D)=1$.

We now assume that $\dim_L(D)\neq 1$. Then an application of Lemma~\ref{lem:polembedding}~(ii) gives the following: The ring $\End_{\mathcal C}(A)$ identifies with an order $\Gamma\supset \OL$ of a CM field $F$ with maximal totally real subfield $L$ and there exists an embedding $\Phi/{\sim}\hookrightarrow \Gamma^\times/N(\Gamma^\times)$ where we recall that $N:F^\times\to L^\times$ denotes the field norm. It follows that $$|\Phi/{\sim}|\leq |\Gamma^\times/N(\Gamma^\times)|.$$
Next, we observe that the field extension $F/L$ and the order $\Gamma\supset \OL$ of $F$ satisfy all assumptions of Lemma~\ref{lem:quotientgroupbound}. Hence this lemma provides that $|\Gamma^\times/N(\Gamma^\times)|\leq \tfrac{w_F^2}{w_\Gamma}\cdot 2^{g-1}$. Moreover, in this bound we can replace $w_{\Gamma}$ by $w'$ and $w_F$ by $w$, since $\Gamma\cong\End_{\mathcal C}(A)$ and thus $F=\Gamma\otimes_\ZZ\QQ\cong D$. Therefore the displayed inequality proves Proposition~\ref{prop:polbound} in the case when $\dim_L(D)\neq 1$. This completes the proof of Proposition~\ref{prop:polbound}.  
\end{proof}
The above proof shows that the number $w$ appearing in Proposition~\ref{prop:polbound} can be bounded only in terms of the relative dimension $g$ of $A$ over $S$. For example, one obtains an explicit bound and an inexplicit but asymptotically sharper bound for each real $\epsilon>0$,
\begin{equation}\label{eq:rootsbound}
w\leq 4g^2 \quad \textnormal{ and } \quad w\ll_\epsilon g^{1+\epsilon},
\end{equation}
by controlling $w_F$ as follows: Suppose that $\zeta\in F$ generates the group of roots of unity in $F$ and write $n=w_F$. Then we deduce \eqref{eq:rootsbound} by combining $\varphi(n)=[\QQ(\zeta):\QQ]\leq 2g$  
with the elementary bound $n^{1/2}\leq \varphi(n)$ which holds when $n>6$  and the asymptotically sharper bound $n^{1-\epsilon}\ll_\epsilon \varphi(n)$ which follows from \cite[Thm 15]{rosc:formulas} for $\varphi$ the Euler totient function.  
\newpage

\section{Endomorphism structures}\label{sec:endo}

Let $\order$ be an order of an arbitrary number field $L$ of degree $g = [L:\QQ]$, let $S$ be a connected Dedekind scheme whose function field is a number field $K$, and let $A$ be an abelian scheme over $S$ of relative dimension $g$. 
The goal of this section is to show the following result, which was proven in \cite[$\mathsection$10]{vkkr:hms} when $K = \QQ$.  
 

\begin{theorem}\label{thm:endobound} Suppose that $A$ has no CM or that $L$ is totally real. 
\begin{itemize}
\item[(i)] Then the number $n_A$ of isomorphism classes of $\order$-module structures on $A$ satisfies $$n_A \leq c(h(A)|\OL/\Gamma|)^{e}\Delta\log(3\Delta)^{2g-1}$$
where $h(A) = d\max(h_F(A),1)$, $e= (2g)^9$ and $c = (3g)^{(3g)^{11}}$.
\item[(ii)] Suppose that $S$ is an open subscheme of $\spec(\mathcal O_K)$, and define $k=3^{3^a}$ for $a=7^{4g^2}d^2$. If  $A$ is of product $\gl2$-type with $G_\QQ$-isogenies, then $$n_A \leq k|\OL/\Gamma|^e\textnormal{rad}(N_SD_K)^{24ge}\Delta\log(3\Delta)^{2g-1}.$$
\end{itemize} 
\end{theorem}

Here the quantities $d = [K:\QQ]$, $D_K$ and $N_S = \prod_{v\in S^c}N_v$ are as before, while $h_F$ denotes the stable Faltings height ($\mathsection$\ref{sec:heightdef}). Further we write $\OL = \OL_L$ and $\Delta = D_L$. 

In the case when $A$ has no CM, we obtain  in \eqref{eq:extraendoboundcommutative} and \eqref{eq:proofquatbound} more precise bounds which show that Theorem~\ref{thm:endobound} holds in this case with $\Delta\log(3\Delta)^{2g-1}$ replaced by the class number of $L$. If $A$ has CM then a slightly sharper version of (i) can be found in \eqref{eq:extraendothmboundcm}.

Notice that (i) and Theorem~\ref{thm:es} directly imply the bound in (ii) which is uniform in $K$, $g$ and $S$ for all abelian schemes $A$ of product $\gl2$-type with $G_\QQ$-isogenies. This uniformity is crucial for the proof of Theorem~\ref{thm:mainint}~(ii) given in Section~\ref{sec:proofmainresults}.




\subsection{Proof of Theorem~\ref{thm:endobound}}\label{sec:endoproof}
We continue our notation. To prove Theorem~\ref{thm:endobound} we use and generalize the arguments of \cite[$\mathsection$10]{vkkr:hms}. 
As already mentioned,  (ii) follows directly from  (i) and Theorem~\ref{thm:es}. 

We now prove (i). To bound the number $n_A$ of isomorphism classes of $\Gamma$-module structures on $A$, we assume that there exists a ring morphism $\iota:\order\to \End(A)$. 
Hence $A$ is of $\GL_2$-type. In steps 1-4 below we prove Theorem~\ref{thm:endobound} under the assumption that $A$ has no CM and we distinguish two subcases (a) and (b). In step 5 we deal with the CM case. 

\paragraph{1.} Suppose that $A$ has no CM. As $A$ is of $\GL_2$-type, Lemma~\ref{lem:gl2endostructure}~(i) and Yu~\cite[Thm 1.3]{yu:maximalorders} then imply that $A$ is isogenous to a power $A' = B^n$ of a simple abelian scheme $B$ over $S$ such that $\End(B)$ identifies with a maximal order of the division algebra $D = \End^0(B)$. Furthermore, Lemma~\ref{lem:gl2endostructure}~(iii) gives that the $\QQ$-algebra $D$ has the following properties:
\begin{itemize}
\item[(a)] Either $D$ identifies with a subfield $Z$ of $L$ of relative degree $[L:Z] = n$,
\item[(b)] or $D$ is a non-split quaternion algebra and the center $Z(D)$ of $D$ identifies with a subfield $Z$ of $L$ of relative degree $[L:Z] = 2n$. 
\end{itemize}
Let $m$ be the degree of an isogeny $A \to A'$ of minimal degree, that is any isogeny $A\to A'$ has degree at least $m$. 
Write $\mathfrak f$ for the conductor ideal of $\Gamma$. 
On using precisely the same arguments as in steps 2-4 of the proof of \cite[Thm 10.1]{vkkr:hms}, we see in case (a) that  
$$
n_A \leq m^{(2g+1)g}N(\mathfrak f)^{g+1}gh, \quad h=|\Pic(\OL_L)|.
$$
We have $|\OL/\Gamma| \in \mathfrak f$ and hence $N(\mathfrak f) \leq |\OL/\Gamma|^g$.  Then the displayed bound for $n_A$ and the (minimal isogeny degree) estimate $m^2 \leq \kappa(A_K)$ coming from \eqref{eq:isogvartransc} lead to 
\begin{equation}\label{eq:extraendoboundcommutative}
n_A\leq (7g)^{97g^6}|\OL/\Gamma|^{g(g+1)}(h(A)d)^{12g^4}h.
\end{equation} 
Thus the bound for $h$ in terms of $\Delta$ provided by \eqref{eq:lenstraclassnumber} proves statement (i) in case (a). In the following steps 2-4 of the proof we assume that we are in case (b).


\paragraph{2.} 
We are going to compare $n_A$ with the number of isomorphism classes $n_{A'}$ of $\order'$-module structures on $A'$, where $\order' = \ZZ[m\order]$ denotes the subring of $\order$ generated by $m\order$.
We write $R' = \End(A')$. On using precisely the same arguments as in  step 2 of the proof of \cite[Thm 10.1]{vkkr:hms}, we see that $n_A \leq |R'/mR'| n_{A'}$. To determine $|R'/mR'|$, we first recall that $D = \End^0(B)$ and $A' = B^n$. Hence the $\QQ$-algebras $\End^{0}(A')$ and $\uM_n(D)$ are isomorphic. 
Thus the $\QQ$-dimension of $\End^0(A')$ equals $n^2[D:\mathbb Q] = 2ng$, since we are in case (b) 
and since the order $R' = \End(A')$ of $\End^0(A')$ is a free $\ZZ$-module of rank $2ng$. Therefore $|R'/m R'|$ equals $m^{2ng}$. We conclude that  $n_A \leq m^{2ng} n_{A'}$. 

\paragraph{3.} We now bound $n_{A'}$ by applying the main result of $\mathsection$\ref{sec:quatcase} below. As we are in the case (b), the center $Z(D)$ of $D= \End^0(B)$ identifies with a subfield $Z$ of $L$ with $[L:Z] = 2n$. The ring $\End(B)$ identifies with a maximal order $\OL_D$ of $D$ and thus $\End(A') = \End(B^n)$ is isomorphic to $\uM_n(\OL_D)$. Then Proposition~\ref{prop:quatcase} implies 
$$
n_{A'} \leq \tfrac{g}{2n}\delta^{16n^3g^3} N(\disc_Z(D))^{16n^5g^2} h.
$$ 
Here $\delta$ is the index $|\OL/\Gamma'|$ of $\Gamma'$ in $\OL$ and $\disc_Z(D)$ denotes the finite Brauer discriminant (see $\mathsection$\ref{sec:quatcase}) of the quaternion algebra $D$. In fact we applied here Proposition~\ref{prop:quatcase} with the order $\rho(\Gamma')$ of $\rho(L)$ where $\rho:L\to \uM_n(D)$ is the composition of $\iota_\QQ:L\to \End^0(A)$ with $\End^0(A)\cong \End^0(A')\cong \uM_n(D)$; notice that $\rho(L)$ automatically contains $Z(D)$ by the proof of Lemma~\ref{lem:decomposition} below and $[\rho(L):Z(D)] = 2n$ since $Z(D)\cong Z$.  
\paragraph{4.} In this step we bound the index $\delta$ and the discriminant $N(\disc_Z(D))$ and then we complete the proof of (i) in the case when $A$ has no CM. We first relate $N(\disc_Z(D))$ to a covolume of $\End(B_K)$. Write $Z$ for $Z(D)$. 
It holds that $\disc_Z(D)^2 = d(\OL_D/\OL_Z)$ by \eqref{eq:discriminantcomparison}, and $N(d(\OL_D/\OL_Z))$ divides $d(\OL_D/\ZZ)$; note $\OL_Z\subset\OL_D$ (As $\OL_Z$ lies in the center of $D$, we see that $\OL_D\OL_Z$ is an order of $D$.  Hence the maximal order $\OL_D$ of $D$ contains the commutative ring $\OL_Z$ since maximal orders of $D$ are unique up to conjugation.).  Here $d(\cdot)$ denotes the discriminant ideal defined in \cite[p. 218]{reiner:orders} using the reduced trace $\textnormal{tr}$. Notice that $2\textnormal{tr}_{D/\QQ}$ equals the trace $T_{D/\QQ}$ of the representation induced by left multiplication, see \cite[(9.14) and (9.32)]{reiner:orders}.  As $B_K$ is simple and $\End(B_K)$ has $\ZZ$-rank $2g/n$, we find that $nT_{D/\QQ}$ coincides on $\OL_D = \End(B)\cong \End(B_K)$ with the trace $\textnormal{Tr}$ in \cite[Def 2.1]{gare:isogenies}. In the proof of \cite[Prop 2.9]{gare:isogenies}, Gaudron--R\'emond showed that $|\textnormal{det}(\textnormal{Tr}(\varphi_i\varphi_j))| = \textnormal{vol}(\End(B_K))^2$ where $\textnormal{vol}(\cdot)$ is the covolume defined in \cite[p. 2067]{gare:newisogenies} using the Rosati metric and $(\varphi_i)$ is a basis of the free $\ZZ$-module $\End(B_K)$. It follows that 
$$
N(\disc_Z(D)) \leq \textnormal{vol}(\End(B_K)) \leq \kappa(B_K)^{1/4}.
$$
Here $\kappa(B_K)$ is defined in \eqref{eq:isogvartransc} and the upper bound for $\textnormal{vol}(\End(B_K))$ is due to Gaudron--R\'emond~\cite[Thm 1.9\,(3)]{gare:newisogenies}. To estimate the index $\delta$ of $\Gamma' = \ZZ[m\Gamma]$ in $\OL$, we observe that $m\Gamma\subseteq \Gamma'\subseteq \Gamma$ and that $\Gamma/m\Gamma\cong (\ZZ/m\ZZ)^g$. Hence we obtain   
$$
\delta \leq |\OL/m\Gamma| = |\OL/\Gamma|m^{g}.
$$
Now we combine the results of steps 2 and 3 with the displayed bounds for $N(\disc_Z(D))$ and $\delta$ and with the estimate $m^2 \leq \kappa(A_K)$ provided by \eqref{eq:isogvartransc}. Then, after bounding $h_F(B)$ in terms of $h(A)$ and $g$ via \eqref{eq:hflower}, \eqref{eq:isogvartransc} and R\'emond's~\cite[Cor 1.2]{remond:fhprop}, we deduce  
\begin{equation}\label{eq:proofquatbound}
n_A \leq (3g)^{(2g)^{11}}|\OL/\Gamma|^{16g^6}(h(A)d)^{88g^9}h.
\end{equation}
Thus the bound for $h$ in \eqref{eq:lenstraclassnumber} proves (i) in the case when $A$ has no CM. We now include some details of our computations described above. 
To bound $h_F(B)$ in terms of $h(A)$, $g$ and $d$, we used the inequality $h_F(B) \leq h_F(A)+\tfrac{g}{2}\log(2\pi^2)+h_F(B)-h_F(C)$ which follows from \eqref{eq:hflower} and \cite[Cor 1.2]{remond:fhprop}, and then we bounded $h_F(B)-h_F(C)$ via \eqref{eq:isogvartransc} where $C$ is an abelian subvariety of $A_K$ which is isogenous to the factor $B_K$ of $A_K$.
This leads to
\begin{equation}\label{eq:boundetailsextraendo}
h_F(B) \leq \max(2h_F(A)+a,b,d)
\end{equation}
for $a=g\log(2\pi^2)+32g_0^4\log(7g_0)+4g_0^2\log d$, $b=\exp(4g_0^2)$ and $g_0=\dim B_K=g/n$. Here we exploited that any real $x\geq b$ satisfies $4g_0^2\log x\leq x$ and  we distinguished  whether $h_F(B)\geq b$ or not. Further, we bounded the right hand side of \eqref{eq:boundetailsextraendo} by $2abh(A)$.


\paragraph{5.} From now on we assume that $A$ has CM and $L$ is totally real. Then $A_K$ is isotypic since the field $L$ embeds into $\End^0(A_K)$ and $[L:\QQ] = \dim(A_K)$. Indeed, otherwise Wu~\cite[Prop 1.5]{wu:virtualgl2} would imply that the field $L$ embeds into $\End^0(C)$ for some factor $C = A_1^{r_1}$ of $A_K$ with $2\dim(C) = \dim(A_K) = [L:\QQ]$ and this is not possible since $L$ is not totally imaginary (\cite[Thm 3.1\,(ii)]{lang:cm}). It follows that $A$ is isogenous to a power $A' = B^n$ of a simple abelian scheme $B$ over $S$ such that $\End(B)$ identifies with the maximal order $\mathcal O_F$ of a CM field $F$ of degree $[F:\QQ] = 2g/n$.  
In particular the ring $R' = \End(A')$ is isomorphic to $\textnormal{M}_n(\mathcal O_F)$ which is a free $\ZZ$-module of rank $2ng$. Then, as in step 2, 
we have $n_A \leq m^{2ng} n_{A'}$ where $n_{A'}$ and $m$ are defined as before.
Proposition~\ref{prop:BoundCase2} implies that
$$
n_{A'} \leq (2g)^{15g^3}(\delta D_F)^{5g^3} \Delta \log(3\Delta)^{2g-1} .
$$
Here $\delta = |\OL/\Gamma'|$ where the order $\Gamma' = \ZZ[m\Gamma]$ of $L$ is defined as in step 2. To estimate $D_F$ we use that the abelian variety $B_K$ is simple and that the free $\ZZ$-module $\End(B_K)$ has rank $2g/n$. Thus, as in step 4, we find that $nT_{F/\QQ} = \textnormal{Tr}$ on $\OL_F\cong \End(B_K)$ and $$D_F \leq \textnormal{vol}(\End(B_K))^2 \leq \kappa(B_K)^{1/2}.$$
In step 4 we proved that $\delta \leq |\OL/\Gamma|m^g$, and \eqref{eq:isogvartransc} shows that the minimal isogeny degree $m$ satisfies $m^2 \leq \kappa(A_K)$. Further, in \eqref{eq:boundetailsextraendo} we estimated $h_F(B)$ in terms of $h_F(A)$, $g$ and $d$. Now we combine everything and then we see that  (i) also holds in the remaining case when $A$ has CM and $L$ is totally real. Here one can use for example the inequalities
\begin{equation}\label{eq:extraendothmboundcm}
D_F \leq (3g)^{9g^3}d(h(A)d)^{2g^2} \quad  \textnormal{and} \quad n_A\leq (3g)^{(3g)^8}  \delta_\Gamma^{5 g^{3}} (h(A)d)^{35  g^{6}}   \Delta \log(3\Delta)^{2g-1} 
\end{equation}
as intermediate steps. This completes the proof of Theorem~\ref{thm:endobound}.

In the above proof we used the following estimate for the class number: Any number field $K$ of degree $d$ and class number $h_K=|\Pic(\OL_K)|$ satisfies
\begin{equation}\label{eq:lenstraclassnumber}
h_K \leq a_dD_K^{1/2}\bigl(\max(1,\log(D_K)\bigl)^{d-1}, \quad a_d= \tfrac{d^{d-1}}{(d-1)!}.
\end{equation} 
For example, this estimate follows directly from Lenstra's bound~\cite[Thm 6.5]{lenstra:algorithms}.   

\subsection{The quaternion case}\label{sec:quatcase} In this section we prove the bound for the number of conjugation classes of ring morphisms used in the proof of Theorem~\ref{thm:endobound} in the quaternion case (b). 

Let $D$ be a non-split  quaternion algebra with center $Z$, let $n\in \Z_{\geq 1}$, and let $L$ be a number field contained in $\uM_n(D)$ with $Z\subseteq L$ and $[L:Z] = 2n$. Let $\Gamma$ be an order of $L$ and let $\mathcal O_D$ be a maximal order of $D$.
Write $g$ for the degree of $L/\QQ$, let $\delta$ be the index of $\Gamma\subseteq \OL_L$ and denote by $h_L$ the class number of $L$. The finite Brauer discriminant $\textnormal{disc}_Z(D)$ is the product  of all finite $Z$-places $\mathfrak p$ where $D_{\mathfrak p}$ is ramified. 

\begin{proposition}\label{prop:quatcase}
We have 
$$
|\Hom(\Gamma, \uM_n(\cO_D))/\GL_n(\cO_D)| \leq \tfrac{g}{2n}\delta^{16n^3g^3} N(\disc_Z(D))^{16n^5g^2} h_L.
$$
\end{proposition}
Here $\GL_n(\cO_D) = \uM_n(\cO_D)^\times$ acts on the set of ring morphisms $\Hom(\Gamma,\uM_n(\cO_D))$ via conjugation, that is  $\tau \cdot \rho  = (x \mapsto \tau\rho(x) \tau^{-1})$ for $\tau \in \GL_n(\cO_D)$ and $\rho \in \Hom(\Gamma,\uM_n(\cO_D))$. In Proposition~\ref{prop:quatcase}, one can replace $\disc_Z(D)^2$ by the usual discriminant $d(\mathcal O_D/\OL_Z)$ defined for example in \cite[p.218]{reiner:orders}; see \eqref{eq:discriminantcomparison} and notice that $\OL_Z\subset \OL_D$. 

\paragraph{Strategy of proof.} We first decompose
$\Hom(\Gamma, \uM_n(\cO_D))$
 into sets of $\varphi$-compatible morphisms $\Hom_\varphi(\Gamma,\uM_n(\OL_D))$ where $\varphi$ ranges over $\Hom(Z, L)$. 
Let $A = \uM_n(D) \otimes_\varphi L$. 
 We embed each piece $\Hom_\varphi(\Gamma,\uM_n(\OL_D))$ (as in \cite[$\mathsection$10.4]{vkkr:hms}) into the set $X_0/{\sim}$ of isomorphism classes of $\Gamma_A$-module structures on $\uM_n(\OL_D)$ for $\Gamma_A = \uM_n(\OL_D) \otimes_{\varphi^{-1}(\Gamma)} \Gamma$:
$$
\Hom_\varphi(\Gamma,\uM_n(\OL_D))/\GL_n(\cO_D) \overset{\psi_1} \hookrightarrow X_0/{\sim}.
$$
The Jordan--Zassenhaus theorem (JZ) gives that $ X_0/{\sim}$ is finite. However, it seems difficult to make the proof (see \cite[$\mathsection$26]{reiner:orders}) of (JZ) effective for $ X_0/{\sim}$. Hence we use a different approach: 
We first show that $A \cong \uM_{2n}(L)$ is split and hence acts on $L^{2n}$.
We consider a certain set $X$ (resp. $Y$) of 
$\Gamma_A$-sublattices (resp. $\uM_2(\cO_L)$-sublattices) $\Lambda$ of $L^{2n}$ and  use $X,Y$ to construct an explicit finite map to the class group $\cI/L^\times$ of $L$,
$$X_0/{\sim} \overset{\psi_2}\hookrightarrow X/L^\times \overset {\beta} \to Y/L^\times \isomto \cI/L^\times.
$$
If $M\in X_0$ then $\psi_2(M) = g(M)\subset L^{2n}$ for some $A$-isomorphism $g:M\otimes \QQ\isomto L^{2n}$. Such a $g$ always exists by a dimension argument using that $A$ is a simple ring with simple module $L^{2n}$, see the construction of $\psi = \psi_2\psi_1$ in \eqref{def:psiembedding}. The map $\beta:X/L^\times\to Y/L^{\times}$ is induced by $\Lambda\mapsto \uM_{2n}(\OL_L)\Lambda$ and its degree can be controlled (Lemma~\ref{lem:degbetabound}). Finally the bijection $Y/L^\times \isomto \cI/L^\times$ is given by the first projection.

\subsubsection{Proof of Proposition~\ref{prop:quatcase}}
We continue our notation. Let $\varphi \colon Z \to L$ be a morphism. Write $\Hom_\varphi(\Gamma, \uM_n(\cO_D))$ for the set of ring morphisms $\rho \colon \Gamma \to \uM_n(\cO_D)$ such that $\rho_{\Q} \to \uM_n(D)$ is a $Z$-algebra morphism. Here $L$ and $\uM_n(D)$ are $Z$-algebras via $\varphi$ and $Z\subset D\subseteq\uM_n(D)$ respectively. 

 The following decomposition is an analogue of \cite[Lem 10.4]{vkkr:hms}:

\begin{lemma}\label{lem:decomposition}
We have 
$
\Hom(\Gamma, \uM_n(\cO_D)) = \bigcup_{\varphi \in \Hom(Z, L)} \Hom_\varphi(\Gamma, \uM_n(\cO_D)).
$
\end{lemma}
\begin{proof} Let $\rho \in \Hom(\Gamma, \uM_n(\cO_D))$. We claim that $\rho_{\Q}(L) \subset \uM_n(D)$ contains $Z$. Define $K = Z \rho_\QQ(L) \subset \uM_n(D)$. Then $K$ is a commutative  finite \'etale $Z$-algebra. Thus $\dim_Z(K) \leq 2n$  and  
$\dim_{\Q}(K) \leq 2n [Z:\Q] = g$. As $\rho_{\Q}(L) \subseteq K$ is also of dimension $g$ over $\Q$, we obtain $\rho_{\Q}(L) = K$, and hence $Z \subseteq \rho_{\Q}(L)$ as claimed. Thus $\rho_{\Q}$ induces some isomorphism of a subfield $Z' \subset L$ to $Z \subset \uM_n(D)$. Define $\varphi = (\rho_{\Q}|_{Z'})^{-1}$. Then $\rho \in \Hom_{\varphi}(\Gamma, \uM_n(\cO_D))$.
\end{proof}

As $|\Hom(Z, L)| \leq \frac {g}{2n}$, 
the proof of Proposition~\ref{prop:quatcase} is reduced to bounding the set 
\begin{equation}\label{eq:ReduceExtraEndosNonComm}
\Hom_\varphi(\Gamma, \uM_n(\cO_D)) / \GL_n(\cO_D)
\end{equation}
for a fixed morphism $\varphi \colon Z \to L$, where $\GL_n(\cO_D)$ acts by conjugation. In the rest of the argument we assume that  $\Hom_\varphi(\Gamma, \uM_n(\cO_D))$ is nonempty. Define $$\Gamma_Z = \varphi^{-1}(\Gamma), \quad A = \uM_n(D) \otimes_{Z, \varphi} L \quad \textnormal{and} \quad \Gamma_A = \uM_n(\cO_D) \otimes_{\Gamma_Z, \varphi} \Gamma$$ where $\Gamma_Z\subseteq\OL_Z\subset \OL_D$.   
We write $\otimes_{\varphi}$ as shorthand for $\otimes_{\Gamma_Z, \varphi}$ and $\otimes_{Z, \varphi}$. An application  of \cite[Thm 4.12]{jacobson:basicalgebra2} with $n=2$ and $r=n$ gives isomorphisms $\phi \colon D \otimes_{\varphi} L \isomto \uM_2(L)$ and $\uM_n(D) \otimes_{\varphi} L \isomto \uM_{2n}(L)$ of $L$-algebras. Then $\phi(\cO_D \otimes_{\varphi} \Gamma) \subset \uM_2(L)$ is an order of $\uM_2(L)$.  
We replace $\phi$ by a suitable conjugate so that $\phi (\cO_D \otimes_{\varphi} \Gamma) \subseteq \uM_2(\cO_L)$.   Write $\phi' \colon A = \uM_n(D) \otimes_{\varphi} L \isomto \uM_{2n}(L)$ for the isomorphism induced by $\phi$.


The unique (up to isomorphism) simple $A$-module $S$ is given by $L^{2n}$ on which $A$ acts by usual matrix multiplication (using $\phi'$).  Write $X$ for the set of $\Gamma_A$-submodules $\Lambda \subset S$ such that  $\Q\Lambda = S$ and $\Lambda$ is of finite type as $\ZZ$-module. 

We are now going to define a map
\begin{equation}\label{def:psiembedding}
\psi \colon \Hom_\varphi(\Gamma, \uM_n(\cO_D)) \to X / L^\times.
\end{equation}
Let $\rho \in \Hom_\varphi(\Gamma, \uM_n(\cO_D))$. Define the $\Gamma_A$-module $M_\rho = \uM_n(\cO_D)$ given by  $\Gamma_A \times M_\rho \to M_\rho$, $(d \otimes \gamma, m) \mapsto d m \rho(\gamma)$  where on the right hand side $d m \rho(\gamma)$ is the product in $\uM_n(\cO_D)$.
We claim that $M_\rho \otimes \Q \cong S$ as $A$-modules: As $A$ is a simple ring, we have $M_\rho \otimes \Q \cong S^r$ for some integer $r \geq 0$. We compute the dimensions
$$
\dim_{\Q} (M_\rho \otimes \Q) = 4n^2[Z:\QQ] = 2n g = \dim_{\Q}(S),
$$ 
and thus $r = 1$.  This proves the claim. Let $g \colon M_\rho \otimes \Q \isomto S$ be an $A$-isomorphism. We define $\psi(\rho) = g(M_\rho) \in X/L^\times.$  If we had chosen a different  $A$-isomorphism $g' \colon M_\rho \otimes \Q \isomto S$, then $ g'(M_\rho) = h g(M_\rho) $ for $h = g' g^{-1} \in \End_{A}(S)^\times = L^\times$.  In particular the construction does not depend on the choice of $g$ and $\psi$ is well-defined. This completes the definition of $\psi$.

\begin{lemma}
Let $\tau \in \GL_n(\cO_D)$ and $\rho \in \Hom_{\varphi}(\Gamma, \uM_n(\cO_D))$. Then $\psi(\tau \rho \tau^{-1}) = \psi(\rho)$. In particular we obtain a map $\psi \colon \Hom_\varphi(\Gamma, \uM_n(\cO_D))/\GL_n(\cO_D) \to X / L^\times$. 
\end{lemma}
\begin{proof}
The map $k \colon M_\rho \to M_{\tau \rho \tau^{-1}}$ defined by $m \mapsto m\tau^{-1}$ is an $\Gamma_A$-isomorphism.  Choose an isomorphism $g \colon M_\rho \otimes \Q \isomto S$. Let $g' \colon M_{\tau \rho \tau^{-1}} \otimes \Q \isomto S$ be the composition
$gk^{-1}_{\Q}$. By the definition of $\psi$ we have $\psi(\tau \rho \tau^{-1}) = g'(M_{\tau \rho \tau^{-1}}) = g k^{-1}_{\Q}(M_{\tau \rho \tau^{-1}}) = g (M_\rho) = \psi(\rho)$ inside $X/L^\times$, where we used that $k_{\Q}(M_\rho) = M_{\tau \rho \tau^{-1}}$.
\end{proof}

\begin{lemma}
The map $\psi \colon \Hom_\varphi(\Gamma, \uM_n(\cO_D))/\GL_n(\cO_D) \to X / L^\times$ is injective.
\end{lemma}
\begin{proof}
Let $\rho, \rho' \in \Hom_\varphi(\Gamma, \uM_n(\cO_D))$ and assume that $\psi(\rho) = \psi(\rho')$. Then there exist  $A$-isomorphisms $g \colon M_\rho \otimes \Q \isomto S$ and $g' \colon M_{\rho'} \otimes \Q \isomto S$, such that $g(M_\rho) = \lambda g'(M_{\rho'}) \subset S$ for some $\lambda \in L^\times$.  After replacing $g'$ by $\lambda g'$, we obtain $g(M_\rho) = g'(M_{\rho'}) \subset S$.  Define the $\Gamma_A$-isomorphism $h = (g^{\prime, -1} g) \colon M_\rho \isomto M_{\rho'}$.  Then $h$ is a morphism $\uM_n(\cO_D) \to \uM_n(\cO_D)$ such that for all $m \in \uM_n(\cO_D)$ and all $d \otimes \gamma \in \Gamma_A$ we have $h((d \otimes \gamma) m) = (d \otimes \gamma) h(m)$, so
\begin{equation}\label{eq:34r43}
h(d m \rho(\gamma)) = d h(m) \rho'(\gamma). 
\end{equation}
 Taking $\gamma = 1$, we get $h(dm) = d h(m)$ and thus $h(d) = dh(1)$. Define $u = h(1)$, then $h(m) = m h(1) = m u$ for all $m \in \uM_n(\cO_D)$. In particular \eqref{eq:34r43} becomes $d m \rho(\gamma) u = d m u \rho'(\gamma)$. Taking $d = m = 1$ we obtain $\rho(\gamma) u = u \rho'(\gamma)$ for all $\gamma\in \Gamma$, and hence $\rho' = u^{-1} \rho u$ (as $h$ is invertible, $u = h(1)$ is invertible as well). This completes the proof.
\end{proof}

Write $Y$ for the set of $\uM_{2n}(\cO_L)$-submodules $\Lambda$ of $S = L^{2n}$ such that $\Lambda$ is of finite type over $\ZZ$ and such that  $\QQ\Lambda = S$. Consider the mapping $\beta \colon X \to Y$, $\Lambda \mapsto \uM_{2n}(\cO_L)\Lambda$. Recall that $\disc_Z(D)$ is the finite Brauer discriminant and that $\delta = |\OL_L/\Gamma|$.

\begin{lemma}\label{lem:degbetabound}
We have $\textnormal{deg}(\beta) \leq \delta^{16n^3g^3} N(\disc_Z(D))^{16n^5g^2}$.
\end{lemma}
\begin{proof} 
By definition $\Gamma_A = \uM_n(\OL_D)\otimes_\varphi \Gamma$ and $\phi' \colon A = \uM_n(D) \otimes_{\varphi} L \isomto \uM_{2n}(L)$ is induced by the $L$-algebra isomorphism $\phi:D\otimes_\varphi L\isomto \uM_2(L)$ with $\phi(\OL_D\otimes_\varphi \Gamma)\subseteq \uM_2(\OL_L)$. Define $R = \phi'(\Gamma_A)\subseteq \uM_{2n}(\cO_L)$ and $m = |\uM_{2n}(\cO_L)/R| = |\uM_2(\cO_L)/\phi(\cO_D \otimes_{\varphi} \Gamma)|^{n^2}$. We have $m\uM_{2n}(\cO_L) \subseteq R$ and thus $\uM_{2n}(\cO_L) \subseteq m^{-1} R$. If $\Lambda, \Lambda' \in X$ satisfy $\beta(\Lambda) = \beta(\Lambda')$, then  $\Lambda \subseteq \uM_{2n}(\cO_L)\Lambda' \subseteq m^{-1} R \Lambda' = m^{-1} \Lambda'$ and, in the same way, $\Lambda' \subseteq m^{-1} \Lambda$. Hence
$$
\beta^{-1}(\beta(\Lambda))\subseteq \{ \Lambda' \in X\ |\ m\Lambda \subseteq \Lambda' \subseteq m^{-1} \Lambda \} \hookrightarrow \{
\textnormal{subgroups of $\Lambda/m^2\Lambda$} 
\}
$$
for all $\Lambda\in X$, where the injection  is induced by $\Lambda'\mapsto m \Lambda'$ and the subgroups in the image are generated by $\textnormal{rank}_\ZZ(\Lambda') = \dim_\Q(S) = 2ng$ elements. Thus $\deg(\beta)$ is at most $\sup_{\Lambda\in X}|\Lambda/m^2\Lambda|^{2gn} = m^{8n^2g^2},$    where $m \leq \delta^{2gn}N(\disc_Z(D))^{2n^3}$ by Lemma~\ref{lem:Prev} below.  \end{proof}
\begin{lemma}\label{lem:Prev}
We have $|\uM_2(\cO_L)/\phi(\cO_D \otimes_{\varphi} \Gamma)| \leq \delta^{\frac {2g} {n}} N(\disc_Z(D))^{2n}$.
\end{lemma}
\begin{proof} 
Define $R = \cO_{Z,\varphi}$, $\Lambda_0 = \phi(\cO_D \otimes_{R} \varphi(\cO_Z) \Gamma)$ and $\Lambda = \phi(\cO_D \otimes_{R} \cO_L)$. The images of the orders $\cO_D \otimes_{\varphi} \Gamma$ and $\cO_D \otimes_{R} \varphi(\cO_Z) \Gamma$ in $D \otimes_\varphi L$ are the same   and thus  $$
|\uM_2(\cO_L)/\phi(\cO_D \otimes_{\varphi} \Gamma)| = |\uM_2(\cO_L)/\Lambda_0| = |\uM_2(\cO_L)/\Lambda|\cdot|\Lambda/\Lambda_0|.
$$
Tensoring the canonical morphisms $\varphi(\cO_Z) \Gamma \hookrightarrow \cO_L \to \cO_L/\varphi(\cO_Z) \Gamma$ with $\cO_D$ gives   $$
\frac{\cO_D \otimes_{R} \cO_L}{\cO_D \otimes_{R} \varphi(\cO_Z) \Gamma} \isomto \cO_D \otimes_{R} \lhk \cO_L / \varphi(\cO_Z) \Gamma \rhk.
$$
As $\cO_D \otimes_{\Z} \lhk \cO_L / \varphi(\cO_Z) \Gamma \rhk \cong \lhk \cO_L / \varphi(\cO_Z) \Gamma \rhk^{\frac{2g}{n}}$  surjects onto $\cO_D \otimes_{R} \lhk \cO_L / \varphi(\cO_Z) \Gamma \rhk$,  we get  $$
|\Lambda/\Lambda_0| \leq |\cO_L / \varphi(\cO_Z) \Gamma|^{\frac{2g}{n}} \leq \delta^{\frac{2g}{n}}.
$$
We compute $|\uM_2(\cO_L)/\Lambda|$. Write $I$ for the $\OL_L$-fractional ideal $[\uM_2(\cO_L) : \Lambda]_{\cO_L}$ (see e.g. \cite[$\mathsection$9.6]{voight:quatbook}). We have  $I^2 = d(\Lambda/\cO_L)$ by \cite[15.2.11,15.2.15]{voight:quatbook}. Moreover \cite[Ex 15.8]{voight:quatbook} implies that $d(\Lambda/\cO_L) = d(\cO_D/\cO_Z) \cO_L$ since $\Lambda\cong \OL_D\otimes_R\OL_L$, while \cite[15.4.7,15.5.5]{voight:quatbook} gives
\begin{equation}\label{eq:discriminantcomparison}
d(\cO_D/\cO_Z) = \textnormal{disc}_Z(D)^2.
\end{equation}
Thus $I^2 = \disc_{Z}(D)^2 \cO_L$. Taking norms we find  $ |\uM_2(\cO_L)/\Lambda| = N(I) = N(\disc_Z(D))^{[L:Z]}$ where $[L:Z] = 2n$ by assumption. This implies the lemma.
\end{proof}





Write $p_1 \colon L^{2n} \to L$ for the first projection and consider the map $\varsigma \colon Y \to \cI$, $\Lambda \mapsto p_1(\Lambda)$, where $\mathcal I$ is the set of non-zero $\cO_L$-fractional ideals, so that $h_L = |\mathcal I/L^\times|$. 
We use the following version of Morita equivalence:

\begin{lemma}
The map $\varsigma \colon Y \to \cI$ is a bijection with inverse given by $I \mapsto I^{2n}$.  
\end{lemma}
\begin{proof}
If $I \in \mathcal I$ and $\Lambda = I^{2n}$ then $p_1(\Lambda) = I$. Conversely, let $\Lambda\in Y$ and define $I = p_1(\Lambda)$. We claim $I^{2n} = \Lambda$. We first check $\Lambda \subseteq I^{2n}$. Let $v \in \Lambda$. Write $S_{1 i} \in \End(\cO_L^{2n}) = \uM_{2n}(\cO_L)$ for the endomorphism $(x_j) \mapsto (x_{\sigma(j)})$ where $\sigma$ is the permutation $(1 i) \in \mathfrak S_{2n}$. Then $S_{1 i} \Lambda = \Lambda$. In particular $v_i = p_1(S_{1 i} v) \in p_1(\Lambda) = I$ for all $i$. Consequently, $v\in I^{2n}$ as well and indeed $\Lambda \subseteq I^{2n}$. Conversely, let $v \in I^{2n}$. For each $i$ we have $v_i \in I$, and so $v_i = p_1(w_i)$ for some $w_i \in \Lambda$. Let $P_i \in \End(\cO_L^{2n})$ be the endomorphism $(x_1, \ldots, x_{2n}) \mapsto (0, \ldots, 0, x_i, 0, \ldots, 0)$. Define $w_i' = P_i S_{1i} w_i$. Then $w_i' = (0, \ldots, 0, v_i, 0, \ldots, 0) \in \Lambda$, and $v = \sum_{i = 1}^{2n} w_i' \in \Lambda$. This proves $\Lambda = I^{2n}$.
\end{proof}

For each $\lambda \in L^\times$ and $\Lambda \in Y$ we have $\varsigma(\lambda \Lambda) = p_1(\lambda \Lambda) = \lambda p_1(\Lambda) = \lambda \varsigma(\Lambda)$, and in particular we obtain an induced bijection  $\varsigma : Y/L^\times \isomto \cI/L^\times$. In conclusion we have    $$
\Hom_\varphi(\cO_L, \uM_n(\cO_D))/\GL_n(\cO_D) \overset {\psi} \hookrightarrow X/L^\times \overset{\beta}\to Y/L^\times \isomto \cI/L^\times.
$$
So Proposition~\ref{prop:quatcase} follows from \eqref{eq:ReduceExtraEndosNonComm} and the above lemmas.

\subsection{The CM case}\label{sec:FirstCMcase}

In this section we prove the bound for the number of conjugation classes of ring morphisms which we used in the proof of Theorem~\ref{thm:endobound} in the CM case. 

Let $L$ be a totally real field of degree $g$, let $\Gamma$ be an order of $L$ of index $\delta$ in $\OL_L$, let $n\in \ZZ_{\geq 1}$, let $F$ be a CM field of degree $2g/n$, and let $F^+ \subset F$ be the  totally real field of degree $g/n$. The following result is an analogue of Proposition~\ref{prop:quatcase} in the CM case.

\begin{proposition}\label{prop:BoundCase2}
We have 
$
|\Hom(\Gamma, \uM_n(\cO_F))/\GL_n(\cO_F)| \leq c(\delta D_F)^{e}\Delta\log(3\Delta)^{2g-1}.
$
\end{proposition}
Here $c=(2g)^{15g^3}$, $e=5g^3$ and $\Delta=D_L$. In the proof of Proposition~\ref{prop:BoundCase2} we obtain the sharper (but more complicated) bound $\tfrac{2g}{n} N^{g+1} h$ with $N$ and $h$ defined in \eqref{def:cmprecise}. 

\paragraph{Strategy of proof.} 
We first decompose $\Hom(\Gamma, \uM_n(\cO_F))/\GL_n(\cO_F)$  into $\varphi$-compatible morphisms for varying $\varphi \in \Hom(F^+, L)$. Then, after applying  $\otimes_{\varphi^{-1}(\Gamma)}\cO_F$, the degree doubles and we can apply \cite[Prop 10.3]{vkkr:hms}. This gives a bound in terms of the class number of the number field $M_\varphi = L\otimes_{\varphi,F^+}F$ and the conductor of the order $\Gamma\otimes_{\varphi^{-1}(\Gamma)} \cO_F$ of $M_\varphi$. Finally we use Diophantine analysis to bound everything in terms of $F$ and $\Gamma$.

\subsubsection{Proof of Proposition~\ref{prop:BoundCase2}}
We continue our notation. For any $\varphi \in \Hom(F^+, L)$ we define $\Hom_{\varphi}(\Gamma, \uM_n(\cO_F))$ to be the set of morphisms $\rho:\Gamma\to \uM_n(\cO_F)$ such that $\rho_\QQ(\varphi(x)) = x$ for all $x \in F^+$. 
\begin{lemma}\label{lem:DoubleCentralizerAgain} 
We have $\Hom(\Gamma, \uM_n(\cO_F)) = \bigcup_{\varphi\in \Hom(F^+, L)}^{} \Hom_{\varphi}(\Gamma, \uM_n(\cO_F)).$
\end{lemma} 
  
\begin{proof}
As in Lemma~\ref{lem:decomposition}, the statement follows from the following claim:  If $\rho \colon L \to \uM_n(F)$ is a morphism, then $F^+_{} \subseteq \rho(L)$ in $\uM_n(F)$. Let $\alpha \in F$ be a generator of $F/F^+$. Define the subalgebras $K^+ = \rho(L)F^+$ and $K = \rho(L)F$ of $\uM_n(F)$. As $\rho(L)$ is totally real, we have $\alpha \in K$ but $\alpha \notin K^+$  and thus $K^+ \neq K$. Note that $K$ is a semisimple commutative $F$-subalgebra of $\uM_n(F)$.  Thus  $\dim_F(K) \leq n$ and this implies $[K:\Q] \leq n [F:\Q]$ and $[K : L] \leq n [F:\Q]/[L:\Q] = 2$. Note that $\dim_L(K^+) \leq \dim_L(K) \leq 2$, as well. We have $K^+\subset K$ and both algebras have dimension at most 2 over $L$. Since $K^+\neq K$ the dimension of $K^+$ over $L$ must be 1. Thus $K^+ = \rho(L)$, and hence $F^+ \subseteq \rho(L)$ as claimed.
\end{proof}

Let $\varphi\in \Hom(F^+,L)$. Define $\Gamma_\varphi^+ = \varphi^{-1}(\Gamma) \subseteq \cO_{F^+} $ and
$
\Gamma_{\varphi, F} = \Gamma \otimes_{\Gamma_\varphi^+} \cO_F.
$
Sending $\rho:\Gamma\to\uM_n(\cO_F)$ to $\rho_F:\Gamma_{\varphi,F}\to \uM_n(\cO_F)$ with $\rho_F(a \otimes b)=\rho(a)b$ induces an injection
    
  $$
\Hom_\varphi(\Gamma, \uM_n(\cO_F)) / \GL_n(\cO_F) \hookrightarrow \Hom(\Gamma_{\varphi,F}, \uM_n(\cO_F)) / \GL_n(\cO_F), \quad \rho\mapsto \rho_F.
 $$ 
 Define $M_\varphi = L\otimes_{\cO_{F^+}} \cO_F$. Then\DarkGreenempty{We have canonical isomorphisms $\Q\otimes_\ZZ\Gamma_{\varphi,F}\cong L \otimes_{\Gamma_\varphi^+} \cO_F\cong L \otimes_{\Gamma_\varphi^+} \cO_{F^+} \otimes_{\cO_{F^+}} \cO_F$. As $\Gamma_\varphi^+ \subseteq \cO_{F^+}$ is of finite index, [Is order by pull back of orders onenote (Note aug 12 2021, in paper2 / checks] we have the isomorphism [See Note 8 Aug 2021 check isom]
$L \otimes_{\Gamma_\varphi^+} \cO_{F^+} \isomto L,$ $a \otimes b \mapsto a \varphi(b).$ By combining the isomorphisms, we obtain.} 
$\Q\otimes_\ZZ\Gamma_{\varphi,F} \isomto M_\varphi$, and $\Gamma_{\varphi,F}$ identifies with an order of $M_\varphi$. 
Note $M_\varphi$ is a field, since $L$  is totally real while $F$ is totally imaginary.  We have $[M_\varphi : \Q] = 2g$.
 Then \cite[Prop 10.3, (10.4)]{vkkr:hms} give 
$$
|\Hom_{\varphi}(\Gamma, \uM_n(\cO_F)) / \GL_n(\cO_F)| \leq 2N( \mathfrak f_{\varphi,F})^{g+1} h_{M_\varphi}
 $$ 
for $\mathfrak {f}_{\varphi,F}$ the conductor of $\Gamma_{\varphi, F}$. (Here we obtain the factor $2$, since (the proof of) \cite[Prop 10.3]{vkkr:hms} bounds compatible morphisms without the factor $t$ and $\rho_F$ is compatible with a morphism $F\to M_\varphi$ that extends $\varphi \colon F^+ \to L$.)  Lenstra's bound~\cite[Thm 6.5]{lenstra:algorithms} and Lemma~\ref{lem:discboundcmendo} show that $h_{M_\varphi} \leq h$, while Lemma~\ref{lem:conductorbound} provides $N(\Gamma_{\varphi,F}) \leq N$, where  \begin{equation}\label{def:cmprecise}
N = (2\delta g)^{2g^2}\lhk\tfrac{D_F}{D^{1/2}_{F^+}}\rhk^{ng}, \quad
h = d\tfrac{(2g-1+\log d)^{2g-1}}{(2g-1)!}, \quad d = (\tfrac{2}{\pi})^{g}D_L\lhk\tfrac{D_F}{D_{F^+}^{2}}\rhk^{\frac{n}{2}}.
\end{equation}
Combining the above results gives $|\Hom(\Gamma, \uM_n(\cO_F))/\GL_n(\cO_F)| \leq \tfrac{2g}{n} N^{g+1} h$, which leads to the bound claimed in Proposition~\ref{prop:BoundCase2}. It remains to prove Lemmas~\ref{lem:discboundcmendo} and \ref{lem:conductorbound}.
 
\begin{lemma}\label{lem:discboundcmendo}
The discriminant $D_{M_\varphi}$ divides 
$D_L^{2}\lhk\frac{D_F}{D_{F^+}^{2}}\rhk^{n}$. 
\end{lemma}
\begin{proof} 
   Write $\Delta_{k'/k}$ for the discriminant ideal of an extension of fields $k'/k$. From $\Q \subseteq F^+\hookrightarrow M_\varphi$ we get $ D_{M_\varphi} = D_{F^+}^{2n} N(\Delta_{M_\varphi/F^+}) 
$, and  $\Delta_{M_\varphi/F^+}$ divides $\Delta_{L/F^+}^2 \Delta_{F/F^+}^n$ since $M_\varphi$ is the  
composite  
of $L$ and $F$ over $F^+$. As  $N(\Delta_{L/F^+}) = D_{L} D_{F^+}^{-n}$ and $N(\Delta_{F/F^+}) = D_{F} D_{F^+}^{-2},$ we find that $D_{M_\varphi}$ divides $D^{2n}_{F^+} N(\Delta_{L/F^+})^2 N(\Delta_{F/F^+})^n = D_{F^+}^{2n} D_{L}^2 D_{F^+}^{-2n} D_{F}^n D_{F^+}^{-2n}$.       
\end{proof}
\begin{lemma}\label{lem:conductorbound}
We have
$
N(\mathfrak {f}_{\varphi,F}) \leq (2g\delta)^{2g^2}\lhk\tfrac{D_F}{D^{1/2}_{F^+}}\rhk^{ng}.
$
\end{lemma} 
\begin{proof}
It suffices to bound $i = |\cO_{M_\varphi}/\Gamma_{\varphi, F}|$ since $(i) \subseteq \mathfrak {f}_{\varphi,F} = \{a \in \cO_{M_\varphi} \ | \ a \cO_{M_\varphi} \subseteq \Gamma_{\varphi, F}\}$.  Let $\alpha \in \cO_F$ be a generator of $F/F^+$ from Lemma \ref{lem:QuadraticMinkowskiLemma}. 
Write $f \in F^+[x]$ for the minimal polynomial of $\alpha$ over $F^+$. Then $\varphi(f)$ is irreducible over $L$  and \cite[15.2.15]{voight:quatbook} gives  
\begin{equation*}
(\Delta_{\varphi(f)}) = [\cO_{M_\varphi} : \cO_L[\alpha]]^2_{\cO_L}\cdot \Delta_{M_\varphi/L}, 
\end{equation*}
where $\Delta_g$ is the discriminant of a polynomial $g$ and $\OL_L[\alpha]$ is the subring in $\OL_{M_\varphi}$ generated by $1\otimes \alpha$ and $\OL_L$. 
Lemma~\ref{lem:QuadraticMinkowskiLemma} provides
$
N(\Delta_f) \leq 4^{\frac{g}{n}}(\tfrac{g}{n}D_{F^+})^{\frac{3}{2}} N(\Delta_{F/F^+}),
$
and thus   $$
|\cO_{M_\varphi}/\cO_L[\alpha]|^2 \leq N(\Delta_{\varphi(f)}) \leq \lhk 4^{\frac{g}{n}}(\tfrac{g}{n} D_{F^+})^{\frac{3}{2}} N(\Delta_{F/F^+}) \rhk^{[L:\varphi(F^+)]}.
$$
Define $R = \cO_L \otimes_{\Gamma_\varphi^+}\cO_F$. As  $\cO_L[\alpha]\subseteq R \subseteq \cO_{M_\varphi}$, we get
$
|\cO_{M_\varphi}/R| \leq |\cO_{M_\varphi}/\cO_L[\alpha]| 
$ and hence $i \leq |\OL_{M_\varphi}/\OL_L[\alpha]|\cdot |R/\Gamma_{\varphi,F}|$. The index of
$\Gamma_{\varphi,F} = \Gamma \otimes_{\Gamma_\varphi^+} \cO_{F^+}$ inside $R$ 
is bounded by  $
|(\cO_L/\Gamma) \otimes_{\Gamma_\varphi^+} \cO_{F^+}| \leq |(\cO_L/\Gamma) \otimes_{\Z} \cO_{F^+}| = \delta^{g/n}.$ This implies the lemma. 
\end{proof}

The proof of Lemma~\ref{lem:conductorbound} uses a result on generators of quadratic field extensions which can be deduced quite directly from the proof of Minkowski's theorem. We include here the result since we could not find a reference.

\begin{lemma}\label{lem:QuadraticMinkowskiLemma}
Let $k$ be a number field of degree $n$ 
and let $k'/k$ be a quadratic extension with discriminant ideal $\Delta_{k'/k}$. There exists $\alpha \in \OL_{k'}$ such that $k' = k(\alpha)$ and such that
$$
N(\Delta_f) \leq 4^n m_k^3 N(\Delta_{k'/k}) \quad \textnormal{ and } \quad |\cO_{k'}/\cO_k[\alpha]| \leq 2^{n} m_k^{3/2} 
$$
for $\Delta_f\in \OL_k$ the discriminant of the minimal polynomial $f$  of $\alpha$ over $k$, $m_k = \frac {n!}{n^n} \lhk \frac {4} {\pi} \rhk^s D_k^{1/2}$ and $s$ the number of complex places of $k$. In particular $m_k \leq (nD_k)^{1/2}$ if $s = 0$. 
\end{lemma}
\begin{proof} 
Take $\beta\in \mathcal O_{k'}$ with $\beta^2\in \cO_k$ and $k' = k(\beta)$. By \cite[15.2.15]{voight:quatbook} we have  
\begin{equation}\label{eq:index_square_bound}
(4d) = [\cO_{k'}:\cO_k[\beta]]^2_{\cO_k}\cdot\Delta_{k'/k}, \quad d = \beta^2.
\end{equation}
The proof of Minkowski's theorem  gives representatives $J_i\subseteq \OL_k$, $i = 1, \ldots, r$, of the class group of $k$ with $N(J_i) \leq m_k$. Choose $I\subseteq \cO_k$ such that $dI^{-2}$ is a square-free ideal of $\cO_k$. Write $M = dI^{-2}$.  For some $i\in \{1,\dotsc,r\}$ and some $x \in k$ we have $I = J_i (x)\subseteq \cO_k$, so  
$
(dx^{-2}) = J_i^2 M\subseteq \OL_k$ and thus $dx^{-2} \in \cO_k$.  

 We define $\alpha = \beta x^{-1} \in \cO_{k'}$. Thus $k' = k(\alpha)$ and $\alpha^2 = dx^{-2}$. In particular if $\mathfrak p \subset \cO_k$ is a prime ideal such that $\textnormal{ord}_{\mathfrak p}(dx^{-2}) = 1$, then $\mathfrak p$ ramifies in $k'/k$ and $\mathfrak p | \Delta_{k'/k}$.  Write $dI^{-2}$ as a product $AB$ of ideals $A, B \subseteq \cO_k$ where $A$ is a product of prime ideal divisors of $J_i$ and $B$ is coprime to $J_i$. As $AB = M$ is square free, $B$ is square free as well and any prime $\mathfrak p$ dividing $B$ divides $(dx^{-2}) = J_i^2M$ with multiplicity $1$, hence is ramified in $k'/k$. Thus $B|\Delta_{k'/k}$ and $M = AB $ divides $J_i \Delta_{k'/k}$. Write $J_i\Delta_{k'/k} = MC$ with $C \subseteq \cO_k$ an ideal. Using $(dx^{-2}) = J_i^2 M$ we find $J_i^3 \Delta_{k'/k} = C(dx^{-2})$. Thus $(dx^{-2})|J_i^3\Delta_{k'/k}$ and hence $N(dx^{-2}) \leq m_k^3 N(\Delta_{k'/k})$. Then \eqref{eq:index_square_bound} gives $|\cO_{k'}/\cO_{k}[\alpha]|^2 = N(4dx^{-2})/N(\Delta_{k'/k}) \leq 4^{n}m_k^3.$  
\end{proof}

\newpage

\section{Proof of the main results for integral points}\label{sec:proofmainresults}


In this section we combine the results obtained in the previous sections to prove Theorem~\ref{thm:mainint} on integral points on coarse moduli schemes. Let $Y$ be a variety over $\ZZ$, let $S\subseteq \spec(\ZZ)$ be nonempty open, and let $\mathcal M$ be a Hilbert moduli stack. Suppose that $\mathcal M\cong \mathcal M^I$ is associated to some $(L,I,I_+)$, and write $g=[L:\QQ]$ and $\Delta=D_L$. We shall see in Section~\ref{sec:proofmainthm} below that Theorem~\ref{thm:mainint} is a consequence of the following result.

\begin{theorem}\label{thm:specialmain} Let $n\in\ZZ_{\geq 3}$ and suppose that $Y$ is a coarse moduli scheme over $\ZZ[1/n]$ of some arithmetic moduli problem $\mathcal P$ on $\mathcal M_{\ZZ[1/n]}$ with branch locus $B=B_{\mathcal P}$.
\begin{itemize}
\item[(i)] Any point $P\in (Y\setminus B)(S)$ satisfies $h_\phi(P)\leq c_1N_S^{e_1}$.
\item[(ii)] The cardinality of $(Y\setminus B)(S)$ is at most $c_2|\mathcal P|_{\bar{\QQ}}N_S^{e_2}\Delta\log(3\Delta)^{2g-1}$.
\end{itemize}
\end{theorem}

Here $e_1=\max(24,5g)$ and $c_1=(4n)^{(5n)^{4g}}$, while  $e_2=5\cdot n^{8g}$ and $c_2=n^{4^9gn^{12g}}$. Further $|\mathcal P|_{\bar{\QQ}}$ is the maximal number  of $\mathcal P$-level structures  over $\bar{\QQ}$, and $h_\phi$ is the height on $Y(S)$ defined in Section~\ref{sec:overt}. We also recall that for any nonempty open subscheme $T\subseteq\spec(\ZZ)$ we denote by $N_T$ the product of all rational primes $p$ not in $T$. 

\subsection{Proof of Theorem~\ref{thm:specialmain}}\label{sec:proofspecialmain}

We continue the notation introduce above. An outline of the principal ideas of the following proof of Theorem~\ref{thm:specialmain} can be found in the introduction.

\begin{proof}[Proof of Theorem~\ref{thm:specialmain}]
We suppose that $Y$, $\mathcal P$, $B$ and $n$ are as in Theorem~\ref{thm:specialmain}. To prove the theorem, we may and do assume that  $U(S)$ is not empty for $U=Y\setminus B$.  Then we take $P\in U(S)$ and we notice that $S$ is a scheme over $\ZZ[1/n]$ since the coarse moduli scheme $Y$ of $\mathcal P$ on $\mathcal M_{\ZZ[1/n]}$ is a $\ZZ[1/n]$-scheme. In what follows in this proof we work  (in the category of stacks) over $\ZZ[1/n]$, and to simplify notation we write $\mathcal M$ for $\mathcal M_{\ZZ[1/n]}$.  

Furthermore, to prove the statements of Theorem~\ref{thm:specialmain}, we may and do assume  that $\mathcal M$ equals $\mathcal M^I$. Indeed this follows from a formal computation using that $Y$ is a coarse moduli scheme of the arithmetic moduli problem $\mathcal P'=\mathcal P\tau$ on $\mathcal M^I$  where  $\tau:\mathcal M^I\isomto \mathcal M$  is an inverse of the equivalence involved in the definition of the height $h_\phi$.   
\paragraph{1.} In this first step we reduce via the covers constructed  in Section~\ref{sec:coverconstruction} to the situation in which the moduli problem is representable.  Let $\mathcal Q$ be the moduli problem on $\mathcal M$ of principal level $n$-structures. As we work over $\ZZ[1/n]$, the forgetful morphism $\mathcal M_{\mathcal Q}\to \mathcal M$ is finite \'etale and surjective. Furthermore the moduli problem $\mathcal Q$ is representable by a scheme $Z$ since $n\geq 3$ by assumption. Therefore we may apply the constructions and results in Section~\ref{sec:coverconstruction} with $B=\spec(\ZZ[1/n])$, the scheme $Z$ and the morphism $Z\to \mathcal M$ obtained by composing an equivalence $Z\isomto \mathcal M_{\mathcal Q}$ with the forgetful morphism $\mathcal M_{\mathcal Q}\to \mathcal M$. Then Lemmas~\ref{lem:y'rep} and \ref{lem:etcover} together with the diagram displayed right before Lemma~\ref{lem:etcover} give a connected Dedekind scheme $T$ with a finite \'etale morphism $T\to S$ and a morphism $P':T\to Y'$ fitting into a commutative diagram of noetherian schemes  \begin{equation}\label{diag:pp'}
\xymatrix{
T \ar[r]^{P'} \ar[d] & Y' \ar[d]\\
S \ar[r]^P & Y. 
}
\end{equation}
Here we recall that $Y'=\cmp\times_{\mathcal M}\mathcal M_{\mathcal Q}$ and that the morphism $Y'\to Y$ factors as $Y'\to \cmp\to^\pi Y$ where $\pi:\cmp\to Y$ is the initial morphism of the coarse moduli scheme $Y$ of $\mathcal P$ used in the definition of the height $h_\phi$ on $Y(S)$.  Let $h_{\phi'}$ be the height on $Y'(T)$ defined right before \eqref{eq:heightcomp}. Evaluating both functions in \eqref{eq:heightcomp} at $P'\in Y'(T)$ leads to 
\begin{equation}\label{eq:hpp}
h_\phi(P)=h_{\phi'}(P')
\end{equation}
since the height $h_\phi$ is compatible with the dominant base change $T\to S$ and the diagram \eqref{diag:pp'} is commutative.
We write $(x,\alpha,\beta)$ for the image of $P'\in Y'(T)$ under the equivalence $Y'\isomto \mathcal M_{\mathcal P\times \mathcal Q}$  in \eqref{eq:productequivalence}. In particular $x=(A,\iota,\varphi)$ is an object of $\mathcal M(T)$ and $(\alpha,\beta)$ lies in $\mathcal P(x)\times \mathcal Q(x)$. Then it follows from the definition of $h_{\phi'}$ that 
\begin{equation}\label{eq:hphf}
h_{\phi'}(P')=h_F(A)
\end{equation}
and we now proceed to bound the stable Faltings height $h_F(A)$ of $A$. For this purpose we denote by $k$ and $K$ the function fields  of $S$ and $T$ respectively. It holds that $k=\QQ$ since $S$ is a nonempty open subscheme of $\spec(\ZZ)$. Further, on using that $T$ and $S$ are connected Dedekind schemes with $T\to S$ finite, we deduce that $K$ is a number field and that $T$ is an open subscheme of the spectrum of the ring of integers of $K$. Furthermore, the construction of $T$ shows that we may and do assume that $K\subset\bar{\QQ}$.

\paragraph{2.}In the following step we show that the abelian scheme $A$ over $T$ is of $\gl2$-type with $G$-isogenies where $G=\Aut(\bar{k}/k)$. As $x=(A,\iota,\varphi)$ is an object in $\mathcal M(T)$, the relative dimension of $A$ over $T$ equals the degree  $g=[L:\QQ]$  and hence $A$ is of $\gl2$-type since $\iota\otimes_\ZZ \QQ$ embeds $L=\OL\times_\ZZ \QQ$ into $\End^0(A)$ where we denote by $\OL$ the ring of integers of $L$. The reason why $A$ has $G$-isogenies is as follows: The $\OL$-abelian scheme $(A,\iota)$ `comes' from a point in $Y(S)\subseteq Y(k)$ and hence the isomorphism class of $(A,\iota)_{\bar{k}}$ is stable under the action of the Galois group $G$ assuring that there exist isomorphisms $\mu_\sigma:\sigma^*A_{\bar{k}}\isomto A_{\bar{k}}$ with the desired properties. We now work out the details. As $K$ is a finite extension of $k$, we may and do choose $\bar{k}$ with $K\subset \bar{k}$. In particular this defines a morphism $\spec(\bar{k})\to T$ which induces the following commutative diagram:   \begin{equation}\label{diag:ginvariantcheck}
\xymatrix{
Y'(T) \ar[r] \ar[d] & Y'(\bar{k})\ar[r] \ar[d] & Y(\bar{k})\ar[d] \\
[\mathcal M_{\mathcal P\times \mathcal Q}(T)] \ar[r] &  [\mathcal M_{\mathcal P\times \mathcal Q}(\bar{k})] \ar[r] &  [\cmp(\bar{k})].} 
\end{equation} 
Here the commutative diagram on the left hand side comes from the equivalence $Y'\isomto \mathcal M_{\mathcal P\times\mathcal Q}$ in \eqref{eq:productequivalence}. The diagram on the right hand side is commutative, since $Y'\to Y$ factors as $Y'\to\cmp\to^\pi Y$ and since $Y'\isomto \mathcal M_{\mathcal P\times\mathcal Q}$ is a morphism of categories over $\cmp$. Notice that $P'\in Y'(T)$ maps to the class of $(x,\alpha,\beta)$ in $[\mathcal M_{\mathcal P\times\mathcal Q}(T)]$ which in turn maps to $[(x,\alpha)]$ in $[\cmp(\bar{k})]$ where we view $(x,\alpha)$ inside $\cmp(\bar{k})$ via $k(T)=K\subset \bar{k}$. 
We next show that the class $[(x,\alpha)]$  is invariant under the $G$-action on $[\cmp(\bar{k})]$ defined right before \eqref{eq:ginvarianceisoclasses}, that is $$[(x,\alpha)]\in [\cmp(\bar{k})]^G.$$ Denote by  $P'_{\bar{k}}$ and $P_{\bar{k}}$ the images of $P'$ and $P$ under $Y'(T)\to Y'(\bar{k})$ and $Y(S)\to Y(\bar{k})$ respectively. On using that the diagram \eqref{diag:pp'} commutes, we see that  $P'_{\bar{k}}$ maps to the point $P_{\bar{k}}$ in $Y(\bar{k})$. 
In fact $P_{\bar{k}}$ lies in $Y(\bar{k})^G$ since $P$ lies in $Y(S)\subseteq Y(k)$, and \eqref{eq:ginvarianceisoclasses} gives that $Y(\bar{k})^G$ identifies with $[\cmp(\bar{k})]^G$ via $\pi^{-1}$. Hence in view of the commutative diagram \eqref{diag:ginvariantcheck} we see that the image $[(x,\alpha)]$ of $P'$ in $[\cmp(\bar{k})]$ actually lies in $[\cmp(\bar{k})]^G$. In other words $[(x,\alpha)]$ is  indeed invariant under the $G$-action. Then the description of $[\cmp(\bar{k})]^G$ given right after \eqref{eq:ginvarianceisoclasses} shows that for each $\sigma\in G$ there exists an isomorphism $$\mu_\sigma:\sigma^*(x,\alpha)\isomto (x,\alpha)$$ in the category $\cmp(\bar{k})$. In particular $\mu_\sigma$ is an isomorphism between the $\OL$-abelian schemes $(\sigma^*A_{\bar{k}},\sigma^*(\iota_{\bar{k}}))$ and $(A_{\bar{k}},\iota_{\bar{k}})$ over $\bar{k}$, which implies that $\mu_\sigma\circ \sigma^*(l)=l\circ \mu_\sigma$ for each $l$ in $L=\OL\times_\ZZ\QQ$ viewed inside $\End^0(A_{\bar{k}})$ via $\iota_{\bar{k}}\otimes_\ZZ \QQ$. Therefore we conclude that the abelian scheme $A$ over $T$ is of $\gl2$-type with $G$-isogenies as desired.

\paragraph{3.}In this step we apply the effective Shafarevich conjecture for abelian varieties of $\gl2$-type with $G$-isogenies established in Theorem~\ref{thm:es}. To obtain better bounds, we would like to apply here the version of Conjecture (ES) in Proposition~\ref{prop:esgl2}. This requires to check that the following condition $(*)$ is satisfied: The field extension $K/k$ is normal, and the endomorphisms of $A_{\bar{k}}$ and the isomorphisms $\mu_\sigma$ are all defined over  $K\subset\bar{k}$.  

We first verify that $K/k$ is a normal field extension. The reason why $K/k$ is normal is as follows: It is the function field extension induced by $T\to S$ which is essentially a base change of the Galois cover $\mathcal M_{\mathcal P\times \mathcal Q}\to\cmp$ defined by forgetting principal level $n$-structures. In what follows in this paragraph we work out the details. As $\spec(K)\to\spec(k)$ is the generic fiber of $T\to S$, we observe that by Galois theory it suffices to show that the automorphism group $\Aut_S(T)$ of the finite \'etale morphism $T\to S$ acts transitively on  $F_{s}(T)$ where $F_{s}$ is the fundamental functor associated to $s:\spec(\bar{k})\to S$. Recall from (the proof of) Lemma~\ref{lem:etcover} that $T$ is a connected component of the noetherian scheme $S'=S\times_Y Y'$ which is finite \'etale over $S$. Then, on using that any automorphism of a scheme permutes its connected components, we see that we are in fact reduced to show that $\Aut_S(S')$ acts transitively on  $F_{s}(S')$. To this end we consider group homomorphisms   \begin{equation}\label{eq:transitivegroupmorphisms}
\gl2(\OL/n\OL)\to \Aut_Y(Y')\to \Aut_S(S').
\end{equation} 
  The first morphism is obtained as follows: The group $\gl2(\OL/n\OL)$ acts on the product presheaf $\mathcal P\times\mathcal Q$ via \eqref{def:gactionpn} and then it acts on the $Y$-scheme $Y'\to Y$  by \eqref{def:gactionmp} and transport of structure using the equivalence $Y'\isomto \mathcal M_{\mathcal P\times\mathcal Q}$ in \eqref{eq:productequivalence} and the inverse $\mathcal M_{\mathcal P\times\mathcal Q}\isomto Y' $ defined right after \eqref{eq:productequivalence}. The second morphism is induced by the base change $S\to^P U\subseteq Y$. In view of \eqref{eq:transitivegroupmorphisms} and the above reductions, we then see that we are reduced to show that $\gl2(\OL/n\OL)$ acts transitively on $Y'_{u}(\bar{k})$ where we denote by $Y'_{u}(\bar{k})$ the set of $\bar{k}$-rational points of the fiber $Y'_{u}$ of the morphism $Y'\to Y$ over the geometric point $u:\spec(\bar{k})\to^{s} S\to^P U\subseteq Y$. To compute $Y'_{u}(\bar{k})$, we view again our $(x,\alpha)\in\cmp(T)$ inside $\cmp(\bar{k})$ via $k(T)=K\subset \bar{k}$. On using that $Y'\to Y$ factors as $Y'\to\cmp\to^\pi Y$ and \eqref{eq:productequivalence} is a morphism of categories over $\cmp$, we deduce from \eqref{diag:pp'} that $\pi(x,\alpha)=u$. Then the arguments surrounding \eqref{diag:fibercomp} show that the equivalence $Y'\isomto \mathcal M_{\mathcal P\times \mathcal Q}$ in \eqref{eq:productequivalence} induces  \begin{equation}\label{eq:fibercomputation-action}
Y'_{u}(\bar{k})\isomto \{[(x,\alpha,\beta')],\beta'\in \mathcal Q(x)\}.
\end{equation}
Moreover, we observe that  this identification is by construction (via transport of structure) compatible with the $\gl2(\OL/n\OL)$-actions. Now, the key point is that the action of $\gl2(\OL/n\OL)$ on the set $\mathcal Q(x)$ is transitive. Indeed if $\beta$ and $\beta'$ are in $\mathcal Q(x)$, then the automorphism $\beta^{-1}\beta'$ of the constant group scheme $(\OL/n\OL)^2_{\bar{k}}$ identifies with some $g\in \gl2(\OL/n\OL)$ and thus we obtain that $g\cdot \beta'=\beta'g^{-1}=\beta$. It follows that $\gl2(\OL/n\OL)$ acts transitively on the right hand side of \eqref{eq:fibercomputation-action} and hence on $Y'_{u}(\bar{k})$ since the identification in \eqref{eq:fibercomputation-action} is compatible with the $\gl2(\OL/n\OL)$-actions. Then in view of the above reductions we conclude that $K/k$ is indeed normal. Thus the first part of condition $(*)$ is satisfied.

To verify the second part of condition $(*)$, we recall that $(x,\alpha,\beta)$ lies in $\mathcal M_{\mathcal P\times \mathcal Q}(T)$. The existence of the principal level $n$-structure $\beta\in \mathcal Q(x)$ for $x=(A,\iota,\varphi)$ implies that $K$ equals the field of definition $K(A_n)$ of the $n$-torsion points of $A_K$.  Then, on using that $n\geq 3$ and that $K/k$ is a normal extension, we see that the field $L$ appearing in Lemma~\ref{lem:cond*} has to coincide with our $K$. Thus Lemma~\ref{lem:cond*} proves that condition $(*)$ is satisfied. 

Now, we may and do apply Proposition~\ref{prop:esgl2} with  our abelian scheme $A$ over $T$ of relative dimension $g$ which is of $\gl2$-type with $G$-isogenies by step 2.  After recalling that $d=[K:\QQ]$, we see that Propositions~\ref{prop:esgl2} and \ref{prop:cm} lead to the bound
\begin{equation}\label{eq:proofshfbound}
h_F(A)\leq (4gd)^{144g^2d}\rad(N_SD_K)^e, \quad e=\max(24,5g).
\end{equation}
Here we used that the morphism  $T\to S$ is finite and thus integral, which implies that $T$ has to be the spectrum of the integral closure of $\ZZ[1/N_S]$ in $K$ and hence for any finite place $v$ of $K$ with $v\notin T$ the residue characteristic of $v$  divides $N_S$.

\paragraph{4.}In the next step we control the discriminant $D_K$ and degree $d$ in terms of $S$, $g$ and $n$. As $K$ and $k=\QQ$ are the function fields of $T$ and $S$ respectively, Lemma~\ref{lem:etcover}~(iii) gives that $d\leq|\mathcal Q|_{\bar{k}}$. The number of principal level $n$-structures on an $\OL$-abelian scheme over the algebraically closed field $\bar{k}$ is bounded from above by the number of $\OL$-module isomorphisms of $(\OL/n\OL)^2$. Further the cardinality of $\gl2(\OL/n\OL)$ is at most $n^{4g}$ since $\OL$ is a free $\ZZ$-module of rank $g$. Combining these observations leads to
\begin{equation}\label{eq:proofsdegbound}
d\leq |\mathcal Q|_{\bar{k}}\leq n^{4g}.
\end{equation}
To bound the discriminant $D_K$, we use that the finite morphism $T\to S$ is \'etale. This implies that the (corresponding function) field extension $K/\QQ$ is unramified over each rational prime in $S$. Therefore the divisibility statement
\begin{equation}\label{eq:proofsdiscdiv}
\textnormal{rad}(D_K)\mid N_S
\end{equation}
follows from Dedekind's discriminant theorem. We are now ready to prove  the height bound in statement (i), and this will be done in the following step.

\paragraph{5.} To prove statement (i), we combine \eqref{eq:hpp} with \eqref{eq:hphf} and we deduce that our point $P\in U(S)$ satisfies $h_\phi(P)=h_F(A)$. Then the estimate for $h_F(A)$ in \eqref{eq:proofshfbound} together with \eqref{eq:proofsdegbound} and \eqref{eq:proofsdiscdiv} proves the upper bound for $h_{\phi}(P)$ claimed in (i).   

\paragraph{6.} To bound the cardinality of $U(S)$ we reduce again to the situation in which the moduli problem is representable. In steps 1-4 we showed that for each $P\in U(S)$ there exists a morphism $P'\in Y'(T)$ fitting into a commutative diagram \eqref{diag:pp'} where $T$ is the spectrum of the integral closure of $\ZZ[1/N_S]$ in a normal extension $K/\QQ$  with $[K:\QQ]\leq n^{4g}$ and $\textnormal{rad}(D_K)\mid N_S$.  Moreover, we showed that $P'$  lies in the subset $$Y^*(T)\subseteq Y'(T)$$ which consists of those points $Q\in Y'(T)$ such that the equivalence $Y'\isomto \mathcal M_{\mathcal P\times \mathcal Q}$ in \eqref{eq:productequivalence} sends $Q$ to some object $(x,\alpha,\beta)$ of $\mathcal M_{\mathcal P\times \mathcal Q}$ where $x=(A,\iota,\varphi)$ with $A$ an abelian scheme over $T$ of relative dimension $g$ such that $A$ is of $\gl2$-type with $G$-isogenies satisfying condition $(*)$ and $K=K(A_n)$. Then we deduce 
$$|U(S)|\leq \sum_{T\in\mathcal T}|Y^*(T)|$$
where $\mathcal T$ is the set of schemes  $T=\spec(\OK[1/N_S])$ with $K/\QQ$ a normal extension of degree at most $n^{4g}$ and  $\textnormal{rad}(D_K)\mid N_S$.  Here we used that $\OK[1/N_S]$ is the integral closure of $\ZZ[1/N_S]$ in $K$  and we applied the following:  Given $T\to S$, $P':T\to Y'$ and $Y'\to Y$, there is at most one $P:S\to Y$ such that the diagram \eqref{diag:pp'} commutes; indeed the uniqueness of  $P:S\to Y$ follows for example from our assumption that $Y$ is separated since $T\to S$ sends the generic point to the generic point.   As $\rad(D_K)$ divides $N_S$, we see as in the proof of Theorem~B~(ii) that $D_K\leq N_S^{3l}l^{6l^2}$ for $l=n^{4g}$ and then \eqref{eq:countnf} leads to  $$|\mathcal T|\leq l^{7l^3}N_S^{3l^2}.$$
Next, we take $T\in \mathcal T$ and we bound $|Y^*(T)|$. Write $\mathcal P'=\mathcal P\times\mathcal Q$ and $B=\ZZ[1/n]$. The equivalence $Y'\isomto \mathcal M_{\mathcal P'}$ induces an isomorphism of presheaves $Y'\isomto [\mathcal M_{\mathcal P'}] $ on $(\textnormal{Sch}/B)$, 
where $[\mathcal M_{\mathcal P'}]$ sends a $B$-scheme $S$ to the set of isomorphism classes of objects in $\mathcal M_{\mathcal P'}(S)$. Notice that $[\mathcal M_{\mathcal P'}]$ is the presheaf in \cite[Lem 3.1]{vkkr:hms} with $\sch$ replaced by $(\textnormal{Sch}/B)$.  Then, as explained in \cite[$\mathsection$3.2 and $\mathsection$6]{vkkr:hms}, composing  $Y'\isomto [\mathcal M_{\mathcal P'}]$ with the forgetful morphism gives a morphism $\phi':Y'\to \absg$ of presheaves on  $(\textnormal{Sch}/B)$ which factors as 
$$\phi':Y'\to^{\phi_\alpha}\hilbmod\to^{\phi_{\varphi}} \abomult\to^{\phi_\iota} \absg.$$
Here the presheaves $\hilbmod$, $\abomult$, and $\absg$ on $(\textnormal{Sch}/B)$ are as in \cite[$\mathsection$6]{vkkr:hms} and the morphisms $\phi_\alpha$, $\phi_\varphi$ and $\phi_\iota$ are induced by forgetting the $\mathcal P'$-level structure, the $I$-polarization and the $\OL$-module structure respectively. Write $d_\alpha$, $d_\varphi$ and $d_\iota$ for the (set-theoretical) degrees of the maps $\phi_\alpha(T)$, $\phi_\varphi(T)$ and $\phi_\iota(T)$ respectively. The factorization of $\phi'$ shows
\begin{equation}\label{eq:y*bound}
|Y^*(T)|\leq d_\alpha d_\varphi d_\iota\cdot |\phi'(Y^*(T))|.
\end{equation}
To compute the set $\phi'(Y^*(T))$ we take $Q\in Y^*(T)$. Let $(x,\alpha,\beta)$ in $\mathcal M_{\mathcal P'}(T)$ be the image of $Q$ under the equivalence $Y'\isomto \mathcal M_{\mathcal P'}$ and write $x=(A,\iota,\varphi)$. The isomorphism $Y'\isomto [\mathcal M_{\mathcal P'}]$ sends $Q$ to the isomorphism class $[(x,\alpha,\beta)]$ and hence $\phi'(Q)=[A]$.  Thus the definition of the subset $Y^*(T)\subseteq Y'(T)$ assures that the isomorphism classes in $\phi'(Y^*(T))$ are generated by abelian schemes $A$ over $T$ of relative dimension $g$ such that $A$ is of $\gl2$-type with $G_\QQ$-isogenies satisfying $K(A_n)=K$. Moreover, our assumption $T\in \mathcal T$ assures that $K/\QQ$ is normal with $[K:\QQ]\leq n^{4g}=l$ and that $\rad(D_K)$ divides $N_S$. Thus Proposition~\ref{prop:esnumber} gives
$$|\phi'(Y^*(T))|\leq 
(9gl)^{(4g)^5l} N_S^{(4g)^4}.$$
To bound the set-theoretical degree $d_\alpha$ of $\phi_\alpha(T)$ we use that the Hilbert moduli schemes $Z\isomto\mathcal M_{\mathcal Q}$ and $Y'\isomto \mathcal M_{\mathcal P'}$ are both varieties over $\ZZ[1/n]$. Hence, as in the proof of \cite[Cor 4.2]{vkkr:hms}, we see that an application of \cite[Lem 8.2]{vkkr:hms} with the (naive) extension  of the presheaf $\mathcal P'$ to the Hilbert moduli stack over $\ZZ$ gives $d_\alpha\leq |\mathcal P'|_T\leq |\mathcal P'|_{\bar{\QQ}}$. Then the inequalities $|\mathcal P'|_{\bar{\QQ}}\leq |\mathcal Q|_{\bar{\QQ}}|\mathcal P|_{\bar{\QQ}}$ and $|\mathcal Q|_{\bar{\QQ}}\leq l$ show
$$
d_\alpha\leq l|\mathcal P|_{\bar{\QQ}}.
$$
It seems difficult to explicitly bound $d_\iota$ only in terms of $N_S$, $g$ and $L$. However  \eqref{eq:y*bound} holds with $d_\iota$ replaced by  $d'_{\iota}=\sup |\phi_\iota^{-1}([A])|$ with the supremum taken over all $[A]\in \phi'(Y^*(T))$, and $d'_\iota\leq d_\iota$ can be bounded as follows: For each $[A]\in\phi'(Y^*(T))$ we see as in the proof of \cite[Cor 10.2]{vkkr:hms} that $|\phi_\iota^{-1}([A])|$ coincides with the number of isomorphism classes of $\OL$-module structures on $A$. Therefore Theorem~\ref{thm:endobound}~(i) together with the bounds for $h_F(A)$ provided by Propositions~\ref{prop:esgl2} and \ref{prop:cm} leads to
$$d'_\iota\leq  aN_S^{b}   \Delta\log(3\Delta)^{2g-1}$$
for $a=(3g)^{(3g)^{11}}  (4gl)^{37(2g)^{11} l}\leq (4n)^{(4g)^{12} n^{4g}}$ and $b=12(2g)^{10}$. Here we used $T\in \mathcal T$ which assures in particular that we can apply Propositions~\ref{prop:esgl2} and \ref{prop:cm}. Indeed the isomorphism classes in $\phi'(Y^*(T))$ are generated by abelian schemes $A$ over $T$ with $K(A_n)=K$ normal over $\QQ$ since $T\in \mathcal T$, and hence condition $(*)$ is satisfied by Lemma~\ref{lem:cond*}. 

We now put everything together. The above arguments give
$$|U(S)|\leq |\mathcal T|\cdot \sup_{T\in \mathcal T}d_\alpha d_\varphi d'_\iota \cdot |\phi'(Y^*(T))|.$$  
Then, on combining the displayed bounds with the estimate for $d_\varphi$ in Corollary~\ref{cor:forgetpol}, we deduce an upper bound for $|(Y\setminus B)(S)|=|U(S)|$ as claimed in Theorem~\ref{thm:specialmain}~(ii). In the case when $g=1$ and $n=3$ we can take $c_2=2^{9^{9}}$ in Theorem~\ref{thm:specialmain}~(ii). Indeed this follows from the above arguments by computing the constant $a$ instead of estimating it.  \end{proof}

\subsection{Proof of Theorem~\ref{thm:mainint}}\label{sec:proofmainthm}

In this section we deduce our main result for integral points by applying the above established Theorem~\ref{thm:specialmain}. We continue our notation and we recall that $\mathcal M $ is a Hilbert moduli stack, $Y$ is a variety over $\ZZ$ and $S\subseteq \spec(\ZZ)$ is a nonempty open subscheme.
  \begin{proof}[Proof of Theorem~\ref{thm:mainint}]
As in Theorem~\ref{thm:mainint}, we let $Z\subseteq Y$ be a closed subscheme and we suppose that there exists a nonempty open $T\subseteq\spec(\ZZ)$ such that $Y_T$ is a coarse moduli scheme of some arithmetic moduli problem $\mathcal P$ on  $\mathcal M$ with branch locus $B=B_{\mathcal P}$ contained in $Z_T$. To reduce to the situation treated in Theorem~\ref{thm:specialmain}, we take $n\in\ZZ_{\geq 3}$ and we define $S'=S\cap T\cap \spec(\ZZ[1/n])$.  Then Theorem~\ref{thm:mainint} follows from an application of Theorem~\ref{thm:specialmain} with $n=3$ and the open subscheme $Y_{S'}$ of the coarse moduli scheme $Y_T\subseteq Y$, since all involved constructions are compatible with flat base change. 

In what follows in this proof we give the details showing that indeed everything is compatible with the base change $S'\hookrightarrow \spec(\ZZ)$. In a first step we prove that 
\begin{equation}\label{eq:proofsreductionstep}
(Y\setminus Z)(S)\subseteq (Y_{S'}\setminus B_{S'})(S').
\end{equation}
The schemes $S'$ and $S$ are nonempty open subschemes of the irreducible scheme $\spec(\ZZ)$. Thus $S'\subseteq S$ is a dense open subscheme, which implies that $(Y\setminus Z) (S)\subseteq (Y\setminus Z)(S')$ since $Y\setminus Z$ is (an open subscheme of $Y$ which is) separated over $\ZZ$. Further it holds that $(Y\setminus Z)(S')\cong (Y\setminus Z)_{S'}(S')$, since $S'\subseteq S$ are both open subschemes of the terminal object $\spec(\ZZ)$ in $\sch$. Then we deduce \eqref{eq:proofsreductionstep} by using our assumption $B\subseteq Z_T$ which assures that $(Y\setminus Z)_{S'}=Y_{S'}\setminus Z_{S'}$ is an open subscheme of $Y_{S'}\setminus B_{S'}$. 


Next, we check that $Y_{S'}$ satisfies the assumptions of Theorem~\ref{thm:specialmain}.  Notice that $Y_{S'}\cong(Y_T)_{S'}$ and that $S'\subseteq\spec(\ZZ[1/n])$ is open. Then the discussions surrounding \eqref{eq:heightrestriction} and \eqref{eq:naiveextension} imply that $Y_{S'}$ is a coarse moduli scheme over $\ZZ[1/n]$ of some arithmetic moduli problem $\mathcal P'$ on $\mathcal M_{\ZZ[1/n]}$ which is essentially the restriction of $\mathcal P$ to the open substack $\mathcal M_{S'}\subseteq \mathcal M$; see \eqref{def:presheafp'} for explicit constructions.  
In particular the variety $Y_{S'}$ over $\ZZ$, the open $S'\subseteq \spec(\ZZ)$, the arithmetic moduli problem $\mathcal P'$ on $\mathcal M_{\ZZ[1/n]}$ and the integer $n\geq 3$ indeed satisfy all assumptions of Theorem~\ref{thm:specialmain}.  Therefore an application of Theorem~\ref{thm:specialmain} with $Y=Y_{S'}$, $\mathcal P=\mathcal P'$ and $S=S'$ gives for each point $P\in (Y_{S'}\setminus B_{S'})(S')$ that
\begin{equation}\label{eq:proofsfirstbound}
h_{\phi'}(P)\leq k_1N_{S'}^{e_1} \quad \textnormal{ and } \quad |(Y_{S'}\setminus B_{S'})(S')|\leq k_2|\mathcal P'|_{\bar{\QQ}}N_{S'}^{e_2}\Delta\log(3\Delta)^{2g-1}
\end{equation}
  \noindent with the constants $e_1=\max(24,5g)$, $k_1=(4n)^{(5n)^{4g}}$ and $e_2=5\cdot n^{8g}$, $k_2=n^{4^9gn^{12g}}$. Here we used  that $B_{S'}=B_{\mathcal P'}$ which follows from \eqref{def:presheafp'} and \eqref{eq:branchlocusrestriction}, and $h_{\phi'}$ denotes the height on $Y_{S'}$ defined with respect to the initial morphism $\pi'$ in \eqref{def:presheafp'}.  

We claim that $h_{\phi'}$ restricts to the height $h_\phi$ on $Y(S)$ appearing in statement (i). To prove the claim, we consider the inclusion $Y(S)\subseteq Y_T(S')$ obtained by using \eqref{eq:proofsreductionstep} with $Z$ empty and $Y_{S'}\subseteq Y_T$. This inclusion factors as $Y(S)\subseteq Y_T(U)\subseteq Y_T(S')$ where $U=S\cap T$, and the height $h_{\phi}$ on the coarse moduli scheme $Y_T=M_{\mathcal P}$ is stable under the base change $S'\to U$. Thus \eqref{eq:heightrestriction} together with \eqref{def:presheafp'} implies that $h_\phi(P)=h_{\phi'}(P)$ for each point $P\in Y(S)$ viewed inside $Y_{S'}(S')$ using \eqref{eq:proofsreductionstep} with $Z$ empty. This proves our claim.

Now we deduce Theorem~\ref{thm:mainint}: After taking  $n=3$, we see that \eqref{eq:proofsreductionstep} and \eqref{eq:proofsfirstbound} lead to the bound in (i)  and also to the bound in (ii) since \eqref{def:presheafp'} assures that $|\mathcal P'|_{\bar{\QQ}}=|\mathcal P|_{\bar{\QQ}}$. More precisely, we obtain here that Theorem~\ref{thm:mainint} holds with $c_1=2^{7^{7g}}$, $c_2=3^{9^{9g}}$ and $e_2=5\cdot 3^{8g}$.
To verify this in the case when $g=1$, we used that we can take $k_2=2^{9^9}$ in  \eqref{eq:proofsfirstbound} for $g=1$ and $n=3$ as explained at the end of the proof of Theorem~\ref{thm:specialmain}~(ii).

To complete the proof, we give the construction of $\mathcal P'$ and $\pi'$ combining \eqref{eq:heightrestriction} with \eqref{eq:naiveextension}. The presheaf $\mathcal P'$ on $\mathcal M_{\ZZ[1/n]}$ is defined by setting  $\mathcal P'=\mathcal P$ over the open substack $\mathcal M_{S'}\subseteq \mathcal M_{\ZZ[1/n]}$  and by setting  $\mathcal P'(x)=\emptyset$ for all objects $x$ of $\mathcal M_{\ZZ[1/n]}$ not in $\mathcal M_{S'}$. Then we define the initial morphism $\pi':(\mathcal M_{\ZZ[1/n]})_{\mathcal P'}\to Y_{S'}$ as the composition of the morphisms  \begin{equation}\label{def:presheafp'}
(\mathcal M_{\ZZ[1/n]})_{\mathcal P'}\isomto (\mathcal M_{S'})_{\mathcal P'}\to^{\pi_{S'}} Y_{S'}.
\end{equation}
Here $\pi_{S'}:(\mathcal M_{S'})_{\mathcal P'}\to Y_{S'}$ is a base change (see \eqref{eq:heightrestriction}) of the initial morphism $\pi:\cmp\to Y_T$ used in the definition of $h_\phi$, and  $(\mathcal M_{\ZZ[1/n]})_{\mathcal P'}\isomto (\mathcal M_{S'})_{\mathcal P'}$ is the isomorphism of categories over $(\textnormal{Sch}/\ZZ[1/n])$ given by the identity functor where we also write $\mathcal P'$ for the restriction of $\mathcal P'$ to the open substack $\mathcal M_{S'}\subseteq \mathcal M_{\ZZ[1/n]}$. In fact $\mathcal P'$ is automatically an arithmetic moduli problem on $\mathcal M_{\ZZ[1/n]}$, since $(\mathcal M_{S'})_{\mathcal P'}$ identifies with $\cmp\times_\ZZ S'$ as explained right after \eqref{eq:heightrestriction} and $\cmp$ is a separated finite type DM stack over $\ZZ$ by assumption.  
\end{proof}

We now explain that the finiteness statement provided by Theorem~\ref{thm:mainint} is essentially due to Faltings~\cite{faltings:finiteness}. Recall that $Z\subseteq Y$ is closed and  $T\subseteq\spec(\ZZ)$ is  nonemtpy open such that $Y_T$ is a coarse moduli scheme of some arithmetic moduli problem $\mathcal P$ on  $\mathcal M$ with $B_{\mathcal P}\subseteq Z_T$. Suppose that $|\mathcal P|_{\bar{\QQ}}<\infty$. Then the finiteness result 
\begin{equation}\label{eq:finitenessviafalt}
|(Y\setminus Z)(S)|<\infty
\end{equation}
can be proved as follows. As in the proof of Theorem~\ref{thm:mainint}, we reduce to the situation treated in Theorem~\ref{thm:specialmain}. Then the arguments of Theorem~\ref{thm:specialmain} give a finite finite set $\mathcal T$ of schemes $T$, with $T$ an open subscheme of $\spec(\OK)$ for some number field $K$,  such that 
$$|(Y\setminus Z)(S)|\leq  \sum_{T\in \mathcal T}|Y'(T)|.$$  
Here $Y'$ is a moduli scheme of a moduli problem $\mathcal P'$ on $\mathcal M$, which satisfies $|\mathcal P'|_{\bar{\QQ}}<\infty$ since $|\mathcal P|_{\bar{\QQ}}<\infty$ by assumption. Now, we see that \eqref{eq:finitenessviafalt} follows directly from the finiteness of $Y'(T)$ obtained in \cite[Prop 12.1]{vkkr:hms} which ultimately relies on Faltings' finiteness result \cite[Thm~6]{faltings:finiteness}. Furthermore, as the latter two results hold for any number field $F$, we expect that the same strategy allows to deduce the finiteness statement \eqref{eq:finitenessviafalt} more generally for any open subscheme $S\subseteq \spec(\mathcal O_F)$; we leave this for the future.

\subsection{Proof of the corollaries}\label{sec:proofcorollaries}
In this section we prove the corollaries. We continue our notation. Recall that $\mathcal M$ is a Hilbert moduli stack and that $S\subseteq \spec(\ZZ)$ is nonempty open. Let $Y(2)$ be the coarse moduli scheme over $\ZZ[1/2]$ of the moduli problem $\mathcal P(2)$ of principal level 2-structures on $\mathcal M_{\ZZ[1/2]}$. We first deduce Corollary~\ref{cor:y2} from Theorem~\ref{thm:mainint} and Corollary~\ref{cor:y2emptybranchlocus}.
\begin{proof}[Proof of Corollary~\ref{cor:y2}]
Let $Y$ be a variety over $\ZZ$. Suppose that there exists a nonempty open $T\subseteq\spec(\ZZ)$ such that $Y_T\cong Y(2)_T$. We write $\mathcal P$ for $\mathcal P(2)$. As in \eqref{eq:proofsdegbound}, we find that $|\mathcal P|_{\bar{\QQ}}\leq 2^{4g}$. Further, Corollary~\ref{cor:y2emptybranchlocus} gives that the branch locus $B_{\mathcal P}\subseteq Y(2)$ is empty. Thus an application of Theorem~\ref{thm:mainint} with $Y_T$ and $Z=\emptyset$ implies Corollary~\ref{cor:y2}. 

We now give some details explaining how to formally apply here Theorem~\ref{thm:mainint}. After identifying $Y_T$ with the coarse moduli scheme $Y(2)_T$ over $T$ using \eqref{eq:heightrestriction}, we extend $Y_T=Y(2)_T$ to a coarse moduli scheme over $\ZZ$ via $\mathcal P'$ defined  in  \eqref{eq:naiveextension}. Then $Y_T$ is a coarse moduli scheme of the arithmetic moduli problem $\mathcal P'$ on $\mathcal M$ such that on $Y_T(S\cap T)$ the height $h_\phi$ defined with respect to an initial morphism $\pi:(\mathcal M_T)_{\mathcal P}\to Y(2)_T$ equals the height defined with respect to the initial morphism $\pi':\mathcal M_{\mathcal P'}\to Y_T$ in \eqref{eq:naiveextension}. Further it holds $|\mathcal P'|_{\bar{\QQ}}=|\mathcal P|_{\bar{\QQ}}$, while \eqref{eq:branchlocusrestriction} implies $B_{\mathcal P'}= (B_{\mathcal P})_T$ and hence $B_{\mathcal P'}\subset Y_T$ is empty.

Finally, to see that we can indeed take in Corollary~\ref{cor:y2} the same constants as in Theorem~\ref{thm:mainint}, we used that Theorem~\ref{thm:mainint} holds in fact with $c_1=2^{7^{7g}}$, $c_2=3^{9^{9g}}$ and $e_2=5\cdot 3^{8g}$ as explained in its proof. Therefore we can also take the same constants in the remark  below Corollary~\ref{cor:y2} concerning the case $T=\spec(\ZZ[\tfrac{1}{2\Delta}])$.  \end{proof}

\newpage

\section{Integral points on the Clebsch--Klein surfaces}\label{sec:ck}

In this section we study integral points on the Clebsch--Klein surfaces. After discussing various aspects of the motivating Diophantine equations, we show that the Clebsch--Klein surfaces are coarse Hilbert moduli schemes over $\ZZ[\tfrac{1}{30}]$ and we compare the height $h_\phi$ with the Weil height. Then we deduce from Theorem~\ref{thm:mainint}  explicit bounds for the Weil height and the number of integral points on the Clebsch--Klein surfaces.

\subsection{Main results}\label{sec:ckmainresults}

To explain the main results of this section, we continue the notation of $\mathsection$\ref{sec:introck}. Recall that $S$ is an arbitrary finite set of rational primes, that $\ZZ_S=\ZZ[1/N_S]$ for $N_S=\prod_{p\in S}p$ and that $h$ denotes the usual absolute logarithmic Weil height (\cite[p.16]{bogu:diophantinegeometry}).

\subsubsection{Diophantine equations}\label{sec:ckdiopheqresults}
The study of integral points on the Clebsch--Klein surfaces is motivated by the Diophantine equations \eqref{eq:ck} and \eqref{eq:sigma24}. We now discuss various aspects of these equations; see also $\mathsection$\ref{sec:degeneration+rational} for a discussion of two Diophantine results based on completely different methods.

\paragraph{Icosahedron equation.} We first consider \eqref{eq:sigma24}. Let $\Sigma$ be the set of $x\in \ZZ^5$ such that $\gcd(x_i)=1$ and such that for all rational primes $p\notin S$ the image of $x$ in $(\ZZ/p\ZZ)^5$ is not a scalar multiple of a standard basis vector of $(\ZZ/p\ZZ)^5$. Recall that \eqref{eq:sigma24} is given by
\begin{equation}
\sigma_4(x)=0=\sigma_2(x), \quad x\in \Sigma.\tag{\ref{eq:sigma24}}
\end{equation}
Here $\sigma_n$ denotes the $n$-th elementary symmetric polynomial for $n\in \ZZ_{\geq 1}$. 
Notice that \eqref{eq:sigma24} has infinitely many solutions if one removes in the definition of $\Sigma$ any of the two extra assumptions on $x$. Indeed the equations $\sigma_2=0$ and $\sigma_4=0$ are homogeneous, and they define the projective scheme $X\subset\mathbb P^4_\ZZ$ which has infinitely many $\QQ$-points by \eqref{eq:rationalpointslowerbound}. 
 
For any solution $x$ of \eqref{eq:sigma24}, we define the logarithmic Weil height $h(x)=\max_i\log |x_i|.$ A main goal of this section is to deduce from Theorem~\ref{thm:mainint} the following result. 
\begin{corollary}\label{cor:sigma24}
Any solution $x$ of \eqref{eq:sigma24} satisfies $h(x)\leq cN_S^{24}$ and the number of solutions of \eqref{eq:sigma24} is at most $(cN_S)^e$, where $e=10^{12}$ and $c=10^e$.
\end{corollary}
The explicit height bound in Corollary~\ref{cor:sigma24} implies that for any given set $S$ one can in principle determine all solutions of \eqref{eq:sigma24} as follows: For each of the finitely many $x\in \ZZ^5$ with $h(x)\leq cN_S^{24}$, check whether $x$ lies in $\Sigma$ and satisfies \eqref{eq:sigma24}.  Furthermore, for small sets $S$ this allows to completely solve \eqref{eq:sigma24} in practice when combined with an efficient sieve for the $x\in \Sigma$ which satisfy \eqref{eq:sigma24}. Such efficient sieves which can deal with initial height bounds of the size of Corollary~\ref{cor:sigma24} were developed for many Diophantine equations via Diophantine approximation techniques, see for example Baker--Davenport~\cite{bada:diophapp}, de Weger~\cite{deweger:lllred}, K.-Matschke~\cite{vkma:computation}, Gherga--Siksek~\cite{ghsi:tm} and the references therein. However, new ideas are required to construct an efficient sieve for \eqref{eq:sigma24}.


\paragraph{Cubic diagonal equation.}We next discuss equation \eqref{eq:ck}. This equation is contained in \eqref{eq:sigma24} from a geometric point of view, see Lemma~\ref{lem:openimmersion}. If one replaces $\ZZ^\times_S$ by $\QQ^{\times}$, then 
\begin{equation}
x_1^3+x_2^3+x_3^3+x_4^3=1=x_1+x_2+x_3+x_4,  \quad x_i\in \ZZ_S^\times,\tag{\ref{eq:ck}}
\end{equation}
has infinitely many solutions. Indeed the solutions of  \eqref{eq:ck} with $x_i\in \QQ^\times$ correspond to the $\QQ$-points of the variety $U$ over $\ZZ$ defined in  \eqref{def:cubicsurface}  and $U$ has infinitely many $\QQ$-points, see Section~\ref{sec:diophaeq}.  For each solution $x=(x_i)$ of \eqref{eq:ck}, we denote by $h(x)$  the logarithmic Weil height of $x\in \mathbb A_\ZZ^4(\bar{\QQ})$. We deduce from Theorem~\ref{thm:mainint} the  following result.
\begin{corollary}\label{cor:ck}
Any solution $x=(x_i)$ of \eqref{eq:ck} satisfies $h(x)\leq cN_S^{24}$ and the number of solutions of \eqref{eq:ck} is at most $(cN_S)^e$, where $e=10^{12}$ and $c=10^e$.
\end{corollary}
Here again, the explicit height bound implies that for any given $S$ one can in principle determine all solutions of \eqref{eq:ck} and that for small $S$ one can completely solve \eqref{eq:ck} in practice when combining with an efficient sieve. We are currently trying to develop such an efficient sieve for \eqref{eq:ck} via Diophantine approximation techniques. In fact Diophantine approximation techniques give that the number $n$ of solutions of \eqref{eq:ck} satisfies  
$$
n\leq b, \quad b=\exp\bigl(24^{12}(4|S|+1)\bigl).
$$
This is a very special case of the result of Evertse--Schlickewei--Schmidt~\cite{evscsc:uniteq} for the generalized unit equation.  We point out that $b$ does not depend on the size of the primes in $S$ and that $n\leq b$ is substantially stronger than our estimate for $n$ in Corollary~\ref{cor:ck}; indeed $b\leq (3N_S)^{4\cdot 24^{12}}$ and for each real $\epsilon>0$ the prime number theorem gives $b\ll_\epsilon N_S^{\epsilon}$.


\subsubsection{Clebsch--Klein surfaces and coarse Hilbert moduli schemes}\label{sec:cksurfaces+chms}

The solutions of the Diophantine equations \eqref{eq:ck} and \eqref{eq:sigma24} correspond to $\ZZ_S$-points on the Clebsch--Klein surfaces. We next  discuss these surfaces and we state Theorem~\ref{thm:cksurfaces} which gives that these surfaces are coarse Hilbert moduli schemes over $\ZZ[\tfrac{1}{30}]$.

\paragraph{Icosahedron surface.}  To relate the solutions of  \eqref{eq:sigma24} to points on a coarse Hilbert moduli scheme,  let $Z\subset\mathbb P^4_\ZZ$ be the union of (the images of) the five $\ZZ$-points obtained by permuting the coordinates of $(1,0,\dotsc,0)$ and consider the relative surface $Y$ over $\ZZ$:    
\begin{equation}
Y=X\setminus Z, \quad X\subset \mathbb P^4_\ZZ: \, \sigma_2=0=\sigma_4.\tag{\ref{def:sigma24surface}}
\end{equation}
The rational surfaces $Y_\CC\subset X_\CC$ can be obtained via Klein's construction involving the Icosahedron; see \cite{klein:cubicsurfaces,hirzebruch:ck}.  
The solutions of the Diophantine equation \eqref{eq:sigma24} correspond (modulo $\pm 1$) to $\ZZ_S$-points of $Y$ by Lemma~\ref{lem:solutionssigma24}.
Hirzebruch~\cite{hirzebruch:ck} showed that $Y_\CC$ is the Hilbert modular variety of the principal level 2 subgroup of the Hilbert modular group of the field $L=\QQ(\sqrt{5})$. Moreover, building on the work of Hirzebruch~\cite{hirzebruch:ck}, we shall obtain the following theorem which is the main result of this section: Let $\mathcal M$ be the Hilbert moduli stack associated to $L=\QQ(\sqrt{5})$ and define $B=\ZZ[\tfrac{1}{30}]$.
\begin{theorem}\label{thm:cksurfaces}
The scheme $Y_B$ is a coarse moduli scheme of an arithmetic moduli problem $\mathcal P$ on $\mathcal M$ with the following properties. 
\begin{itemize}
\item[(i)] The branch locus $B_{\mathcal P}$ is empty and $|\mathcal P|_{\bar{\QQ}}\leq 16$.
\item[(ii)] Any $P\in Y(\bar{\QQ})$ satisfies $h(P)\leq 2h_\phi(P)+8^8\log(h_\phi(P)+8)$.
\end{itemize}
\end{theorem}  
Here the restriction of $\mathcal P$ to $\mathcal M_B\subset \mathcal M$ is the moduli problem of principal level 2-structures as in \cite[1.21]{rapoport:hilbertmodular}; we shall recall in \eqref{def:symplecticlvl} the definition of $\mathcal P$ on $\mathcal M_B$.  Further $|\mathcal P|_{\bar{\QQ}}$ is the maximal number of $\mathcal P$-level structures over $\bar{\QQ}$ as in \eqref{def:maxlvl}, $h$ is the logarithmic Weil height on $Y\subset \mathbb P^4_\ZZ$, and $h_\phi$ is the height on $Y(\bar{\QQ})\cong Y_B(\bar{\QQ})$ defined in \eqref{def:cheight} with respect to the initial morphism $\pi:\cmp\to Y_B$  given in \eqref{def:morphismckthm}. 

We now discuss some aspects of Theorem~\ref{thm:cksurfaces}. For any algebraically closed field $k$ of characteristic not in $\{2,3,5\}$, it provides a natural moduli interpretation of $Y(k)=Y_B(k)$:
\begin{equation}\label{eq:moduliintY}
\pi^{-1}:Y(k)\isomto [\cmp(k)].
\end{equation}
Concerning the shape of the height bound in (ii): It is linear in $h_\phi$ which is best possible  and the factor $2$ in front of $h_\phi$ might be optimal, but the constants $8^8$ and $8$ can be improved up to a certain extent; see the discussions at the end of Section~\ref{sec:heightcomparison}.

\paragraph{Cubic diagonal surface.}The Icosahedron surface is related to the cubic diagonal surface which was first studied by Clebsch~\cite{clebsch:clebschsurface}. Recall that $V$ is the closed subscheme
\begin{equation}
V\subset\mathbb P^4_\ZZ: \sum z_i=0=\sum z_i^3, \quad \textnormal{and}\quad U=V\setminus \cup V_+(z_i)\tag{\ref{def:cubicsurface}}
\end{equation}
with the union taken over the five coordinate functions $z_i$ of $\mathbb P^4_\ZZ=\textnormal{Proj}\bigl(\ZZ[z_i]\bigl)$. The solutions of \eqref{eq:ck} correspond to $\ZZ_S$-points of $U$ by Lemma~\ref{lem:solutionsck}. Thus the following lemma shows that solving \eqref{eq:ck} is equivalent to solving a special case of  \eqref{eq:sigma24}. 
\begin{lemma}\label{lem:openimmersion}
There exists an open immersion $\varphi:U\hookrightarrow Y$ over $\ZZ[1/3]$ such that
$$\tfrac{1}{4}h(P)\leq h(\varphi(P))\leq 4h(P), \quad P\in U(\QQ).$$
\end{lemma} 
Here $h$ denotes the logarithmic Weil height on $U\subset\mathbb P^4_\ZZ$ and $Y\subset\mathbb P^4_\ZZ$. Now, as the open immersion $\varphi:U\hookrightarrow Y$ is flat,  we can view $U$ as a coarse moduli scheme via Theorem~\ref{thm:cksurfaces} which gives that the open substack $\mathcal M_{\mathcal P'}=U\times_Y\cmp$ of $\cmp$ admits a coarse moduli space $\pi_{U}:\mathcal M_{\mathcal P'}\to U$. In particular we obtain a natural moduli interpretation of $U(k)$:
\begin{equation}\label{eq:moduliintU}
\pi_U^{-1}:U(k)\isomto [\mathcal M_{\mathcal P'}(k)]
\end{equation}
for any algebraically closed field $k$ whose characteristic does  not lie in $\{2,3,5\}$. As $\pi$ is \'etale by Theorem~\ref{thm:cksurfaces}~(i), its base change $\pi_U $ remains \'etale and  hence the branch locus of $\pi_U:\mathcal M_{\mathcal P'}\to U$ is again empty. In the proof of Lemma~\ref{lem:openimmersion} given in Section~\ref{sec:diophaeq}, we shall explicitly construct the open immersion $\varphi:U\hookrightarrow Y$ over $\ZZ[1/3]$.
 
\paragraph{Terminology.}The literature contains various names for (surfaces which are birational to) the surfaces defined in \eqref{def:sigma24surface} and \eqref{def:cubicsurface} such as for example Clebsch diagonal (cubic) surface, Clebsch surface, Klein's icosahedron surface or Clebsch--Klein surface. In what follows we shall refer by Clebsch--Klein surfaces to the surfaces defined in \eqref{def:sigma24surface} and \eqref{def:cubicsurface}.

\subsubsection{Rational points and degeneration of $S$-integral points}\label{sec:degeneration+rational}

To complement our discussion of Diophantine aspects of the Clebsch--Klein surfaces, we now mention in addition two important Diophantine results. They both directly follow from general theorems for smooth cubic surfaces based on completely different methods.

\paragraph{Rational points.}The varieties $Y$ and $U$ both have infinitely many $\QQ$-points. To give a stronger statement, we define $Y^0=Y\setminus \cup_i V_+(x_i)$ for $x_i$ the $i$-th coordinate function on $\mathbb P^4_\ZZ=\textnormal{Proj}(\ZZ[x_i])$ and we take $W\in \{U,Y^0\}$. Then for any $b\in \RR$ with $b\to \infty$ it holds
\begin{equation}\label{eq:rationalpointslowerbound}
|\{P\in W(\QQ), h(P)\leq \log b\}|\gg b^\nu(\log b)^{\rho-1},\quad \nu=\begin{cases}
1  & \textnormal{if } W=U,\\
1/4 & \textnormal{if } W=Y^0.
\end{cases}
\end{equation}
Here $\rho$ is the Picard rank of the smooth projective surface $V_\QQ$ and the implied constant depends only on $V_\QQ$. In $\mathsection$\ref{sec:rationalpoints} we shall deduce \eqref{eq:rationalpointslowerbound} by combining Lemma~\ref{lem:openimmersion} with a result of Slater--Swinnerton-Dyer~\cite{slsw:cubicmanin} in which they established for certain smooth cubic surfaces the lower bound predicted by Manin's conjecture  (\cite{frmats:maninconjecture}). In fact the lower bound in \eqref{eq:rationalpointslowerbound}  is optimal for $W=U$ if  Manin's conjecture holds, see $\mathsection$\ref{sec:rationalpoints}.

\paragraph{Degeneration of $S$-integral points.} Let $K$ be a number field and let $S$ be a finite set of places of $K$. On using Diophantine approximation techniques, Corvaja, Levin and Zannier~(\cite{coza:intpointssurfaces,zannier:ramificationdivisor,coza:intpointscertainsurfaces,levin:siegelpicard,coleza:integralpoints,coza:intpointsdivisibility}) established fundamental Diophantine results proving in particular the degeneration of $S$-integral points for large classes of varieties over $K$. For example, consider the divisor $D=V_+(z_i)\cup V_+(z_j)$ of $V_\QQ$ with $i\neq j$. Then, as explained in $\mathsection$\ref{sec:degeneration}, the general result of Corvaja--Zannier~\cite[Thm 1]{coza:intpointsdivisibility} for smooth cubic surfaces proves that the set of $S$-integral points of 
\begin{equation}\label{degeneration}
V_\QQ\setminus D
\end{equation}
is not Zariski dense. Furthermore, it follows from Corvaja--Zannier~\cite[Thm 9]{coza:intpointsdivisibility} that this degeneration result is essentially optimal; see $\mathsection$\ref{sec:degeneration}. We point out that even in the special case when $K=\QQ$, proving the degeneration of $S$-integral points of $V_\QQ\setminus D$ is out of reach for our approach which can only deal with the much smaller variety $U_\QQ\subset V_\QQ\setminus D$.  Here the main issue is that the moduli interpretation of the points in $V\setminus U$ is substantially more complicated and it is not clear to us how to exploit it. 

One could try to prove new degeneration results for $S$-integral points on surfaces by combining the moduli formalism with the powerful techniques developed by Corvaja, Levin and Zannier.  To this end we now discuss an open problem for  Hilbert modular surfaces.
Let $L$ be a totally real quadratic field, let $\Gamma$ be a congruence subgroup of $\sl2(\OL_L)$, and let  $Y$ be the canonical model over $\bar{\QQ}$ of the complex Hilbert modular surface $\mathbb H^2/\Gamma$. We denote by $X$  the Baily--Borel compactification of $Y$.  The singular locus  $X^s$ of $X$ is finite and $X\setminus Y\subseteq X^s$.   Let  $\pi:\tilde{X}\to X$ be a minimal desingularization.  We write $E$ for the exceptional divisor of $\pi$, and we let $D$ be a divisor of $\tilde{X}$ with $\textnormal{supp}(D)\subseteq E$. 
 
\vspace{0.3cm}
\noindent{\bf Problem.}
\emph{Let $S$ be a finite set of places of a number field $K$ over which everything is defined. Prove or disprove that the set of $S$-integral points of $\tilde{X}\setminus D$ is Zariski dense.}
\vspace{0.3cm}

In the simplest case when $\textnormal{supp}(D)=E$,  we expect that the set of $S$-integral points of $\tilde{X}\setminus D$ is finite and thus degenerate. This follows for example from Theorem~\ref{thm:mainint}~(ii) for $K=\QQ$  and  it should follow from the discussion below \eqref{eq:finitenessviafalt} for all $K$, since $\tilde{X}\setminus E\cong Y\setminus Y^s$  and $Y$ is a coarse Hilbert moduli scheme whose branch locus is precisely the singular locus $Y^s$ of $Y$. 
On the other hand, the problem seems to be more complicated when $D\neq E$. 
A difficulty arises from the fact that the divisor $D$ is never big, since $\textnormal{codim}(X,\pi(D))\geq 2$  and $\pi:\tilde{X}\to X$ is birational.  However $\tilde{X}$ and $D$ have other useful geometric properties which one could try to combine with the  techniques of Corvaja, Levin and Zannier in order to study the problem. For instance Hirzebruch~\cite{hirzebruch:hmsenseignement} explicitly determined the divisor $E$, while Hirzebruch--van de Ven--Zagier~\cite{hirzebruch:hmsenseignement,hiva:hmsclass,hiza:hmsclass} and others classified the surfaces $\tilde{X}$ in the sense of Enriques--Kodaira; see van der Geer~\cite[VII]{vandergeer:hilbertmodular}.







\subsubsection{Outline of the section} In Section~\ref{sec:moduliinterpretation} we introduce some notation and we give in Proposition~\ref{prop:moduliinter} a moduli interpretation of $X$ and $Y$ over $\ZZ[\tfrac{1}{30}]$. Then we prove Proposition~\ref{prop:moduliinter} in Sections~\ref{sec:hodgebundlepositive} and \ref{sec:proofmoduliinter}. More precisely, in Section~\ref{sec:hodgebundlepositive} we study the Hodge bundle $\omega$ on the toroidal compactification of Rapoport~\cite{rapoport:hilbertmodular}  and we show in Proposition~\ref{prop:fingen} that $\omega^{\otimes 2}$ is generated by its global sections. For this we first collect various results: We prove in $\mathsection$\ref{sec:descendomega2} that $\omega^{\otimes 2}$ descends to a line bundle on the coarse moduli space, we define and study in $\mathsection$\ref{sec:modularcurves} integral models of modular curves, and we determine in $\mathsection$\ref{sec:eisenstein} the Fourier expansion of the Eisenstein series $E_i$ corresponding to the cusps. This allows us to compute in $\mathsection$\ref{sec:eisensteindivisor} the divisors of the Eisenstein series $E_i$ in terms of integral models of modular curves and then we complete in $\mathsection$\ref{sec:proofpropfingen} the proof of Proposition~\ref{prop:fingen}. In Section~\ref{sec:proofmoduliinter} we deduce Proposition~\ref{prop:moduliinter}  by combining Proposition~\ref{prop:fingen} with Hirzebruch's work~\cite{hirzebruch:ck} and the construction of the minimal compactification due to Faltings--Chai~\cite{fach:deg,chai:hilbmod}.

In Section~\ref{sec:heightcomparison} we study heights and we prove the height bound in Theorem~\ref{thm:cksurfaces}~(ii). Here an important part of the proof  is the construction of a Theta height on $Y$ which we relate in Lemma~\ref{lem:thetaheight} to the Weil height on $Y\subset \mathbb P^4_\ZZ$. In Section~\ref{sec:proofofthmck} we complete the proof of Theorem~\ref{thm:cksurfaces} and we deduce the Corollary from the introduction.

In Section~\ref{sec:diophaeq} we first collect basic results for the Clebsch--Klein surfaces and the Diophantine equations \eqref{eq:ck} and \eqref{eq:sigma24}. Then we deduce in $\mathsection$\ref{sec:proofofdiopheqcorollaries} the Corollaries~\ref{cor:sigma24} and~\ref{cor:ck} which explicitly bound the height and the number of solutions of \eqref{eq:ck} and \eqref{eq:sigma24}. We conclude by discussing a lower bound for the rational points of bounded height ($\mathsection$\ref{sec:rationalpoints}) and a degeneration result for $S$-integral points ($\mathsection$\ref{sec:degeneration}).

\subsection{Moduli interpretations for $X$ and $Y$}\label{sec:moduliinterpretation}

Let $X$ and $Y$ be as in \eqref{def:sigma24surface}, and let $\mathcal M$ be the Hilbert moduli stack associated to the real quadratic field $L=\QQ(\sqrt{5})$ with discriminant $\Delta=5$. We denote by $\OL$ the ring of integers of $L$. Throughout Sections~\ref{sec:moduliinterpretation}, \ref{sec:hodgebundlepositive} and \ref{sec:proofmoduliinter}, we shall work over $B=\ZZ[\tfrac{1}{30}]$ and we write $\mathcal M=\mathcal M_B$, $X=X_B$ and $Y=Y_B$ to simplify notation. 

\paragraph{The stack $\mathcal M_2$.} We define $\mathcal M_2=\cmp$ where $\mathcal P$ is the moduli problem on $\mathcal M$ of principal level $2$-structures satisfying the condition \eqref{def:symplecticlvl} of \cite[1.21]{rapoport:hilbertmodular}. Then $\mathcal M_2$ is a separated finite type DM-stack over $B$  
with a coarse moduli space $\pi:\mathcal M_2\to M_2$. In fact $M_2$ is a scheme (\cite[Thm 3.1]{conrad:coarse}) since $\mathcal M_2$ admits a finite \'etale scheme cover.  

\paragraph{Moduli interpretation.}Faltings--Chai~\cite{fach:deg,chai:hilbmod} constructed  the minimal compactification $\minm$ of $M_2$. In particular $\minm$ is a normal scheme containing $M_2$ as an open subscheme. The following result gives a moduli interpretation for the points of $X$ and $Y.$

\begin{proposition}\label{prop:moduliinter}
There is an isomorphism $M_2^*\isomto X$ which restricts to $M_2\isomto Y.$
\end{proposition}

Hirzebruch~\cite{hirzebruch:ck} constructed the isomorphism $M_2^*\isomto X$ over $\CC$ by  using the five Eisenstein series $E_i$ corresponding to the five cusps of the Hilbert modular variety $M_2(\CC)$. Proposition~\ref{prop:moduliinter} shows in particular that $Y$ is a coarse moduli scheme over $B$ of the arithmetic moduli problem $\mathcal P$ on $\mathcal M$; we shall explicitly construct in \eqref{def:initialmorphismoverB} an initial morphism $\pi:\mathcal M_2\to Y$ using the five Eisenstein series $E_i$.  In the following Sections~\ref{sec:hodgebundlepositive} and \ref{sec:proofmoduliinter}, we give a proof of Proposition~\ref{prop:moduliinter} via the strategy outlined in $\mathsection$\ref{sec:outlineproofofmoduliint}.

\subsection{The square of the Hodge bundle}\label{sec:hodgebundlepositive}
We continue our notation and terminology. Rapoport~\cite{rapoport:hilbertmodular} constructed a smooth toroidal compactification $\torm$ of $\mathcal M_2$.  In this section we study the square $\omega^{\otimes 2}$ of the Hodge bundle $\omega$ on $\torm$ defined by Chai in~\cite[4.1]{chai:hilbmod}. We denote by $s_i$ the five global sections of $\omega^2$ corresponding to the five Eisenstein series $E_i$ defined in $\mathsection$\ref{sec:eisenstein}. 

\begin{proposition}\label{prop:fingen}
The invertible sheaf $\omega^{\otimes 2}$ on $\torm$ is globally generated by the $s_i$.
\end{proposition}
Hirzebruch~\cite{hirzebruch:ck} essentially proved this statement over the generic fiber, while the method of Faltings--Chai~\cite{fach:deg,chai:hilbmod} gives that a positive power of $\omega$ is generated by its global sections. The proof of Proposition~\ref{prop:fingen} given in $\mathsection$\ref{sec:proofpropfingen} crucially relies on the results and arguments of Hirzebruch and Faltings--Chai, see $\mathsection$\ref{sec:outlinefingenproof} for the principal ideas.  

\paragraph{Geometric properties.}We now introduce some notation and we review geometric properties of $\mathcal M_2$, $M_2$, $\torm$ and $\minm$ which we shall use in the proofs below. The stack $\mathcal M_2$ is smooth over $B$ since the smooth stack in \cite[1.22]{rapoport:hilbertmodular} identifies over $B$ with $\mathcal M_2$.  Then Corollary~\ref{cor:y2emptybranchlocus} implies that $M_2$ is also smooth over $B$.    
The following statements are (direct consequences of) fundamental results in \cite[$\mathsection$3.6 and $\mathsection$4.3]{chai:hilbmod}: The minimal compactification $\minm$ is a normal scheme which is projective and finite type over $B$, $\mathcal M_2\subset \torm$ and $M_2\subset \minm$ are dense open,  there exists a canonical surjective morphism
\begin{equation}\label{def:pibar}
\bar{\pi}:\torm\to\minm
\end{equation}
induced by the fact that a positive power of the Hodge bundle $\omega$ on $\torm$ is generated by its global sections, the restriction of $\bar{\pi}$ to $\mathcal M_2\subset \torm$ is the initial morphism $\pi:\mathcal M_2\to M_2$, and $\bar{\pi}(\torm\setminus\mathcal M_2)=\minm\setminus M_2$.     In particular $M_2$ is fiberwise dense in $\minm$ since $\mathcal M_2$ is fiberwise dense in $\torm$ by construction.  As $\pi:\mathcal M_2\to M_2$ is a universal homeomorphism and $\mathcal M_2$ has irreducible fibers by \cite[6.2]{rapoport:hilbertmodular}, the fibers of $M_2$, and thus of $\minm$, are irreducible.  This implies that $M_2$, and hence $\minm$, is irreducible since $M_2$ is flat and of finite presentation over $B$.  Further $M_2$ and $\minm$ have both dimension 3 with all fibers of dimension 2, since the generic fiber of $M_2$ has dimension 2 by \cite[p.269]{rapoport:hilbertmodular}.  We also mention that Lan's book~\cite{lan:compactifications} on compactifications of PEL-type Shimura varieties contains a detailed construction of the toroidal and minimal compactifications of $\mathcal M_2$.


\paragraph{The moduli problem $\mathcal P$.} For several computations in this section it will be useful to recall the definition of the moduli problem $\mathcal P$ of principal level 2-structures satisfying the condition in \cite[1.21]{rapoport:hilbertmodular}. Let $x=(A,\iota,\varphi)$ be in $\mathcal M(S)$ for some $B$-scheme $S$. Then $\mathcal P(x)$ is the set of all principal level 2-structures $\alpha:(\OL/2\OL)^2_S\to A_2$  such that the diagram 
\begin{equation}\label{def:symplecticlvl}
\xymatrix@R=4em@C=4em{
\Hom_\OL(A,A^\vee)^\textnormal{sym}\otimes_{\OL} \wedge_{\OL/2\OL}^2A_2 \ar[r] \ar[d]^{\varphi^{-1}\otimes \alpha^{-1}} & \mathfrak d^{-1}\otimes_\ZZ \mu_2 \ar[d] \\
\mathfrak d^{-1}\otimes_\OL \wedge_{\OL/2\OL}^2(\OL/2\OL)_S^2 \ar[r] &  \mathfrak d^{-1}\otimes_\ZZ (\ZZ/2\ZZ)_S} 
\end{equation}
commutes. Here $\mathfrak d^{-1}$ is the $\Z$-dual of $\OL$ and we view $\Hom_\OL(A,A^\vee)^\textnormal{sym}$ as a sheaf for the \'etale topology on $(\textnormal{Sch}/S)$. As explained in \cite[p.266]{rapoport:hilbertmodular} the top horizontal morphism is an isomorphism induced by the Weil pairing of $A$ which is $\OL$-compatible. The bottom horizontal morphism comes from the isomorphism $\wedge_{\OL/2\OL}^2(\OL/2\OL)^2\isomto \OL/2\OL$ given by the determinant   and the right vertical morphism comes from $\mu_2\isomto (\ZZ/2\ZZ)_S$.

\subsubsection{Descending powers of $\omega$}\label{sec:descendomega2}
We continue our notation. In what follows we also denote by $\omega$ the restriction of $\omega$ to $\mathcal M_2\subset \torm$. The method of Faltings--Chai~\cite{fach:deg}  gives that a positive power $\omega^{\otimes k}$ of $\omega$  descends along the initial morphism $\pi:\mathcal M_2\to M_2$ to an invertible sheaf on $M_2$, see \cite[Main Thm]{chai:hilbmod}. To show that one can take here $k=2$, we use a different strategy which follows Olsson's proof of~\cite[Prop 6.1]{olsson:gtorsors} and which exploits that we can explicitly compute all the geometric automorphism groups of $\mathcal M_2$.

\begin{lemma}\label{lem:descendingomega}
The canonical (adjunction) morphism $\pi^*\pi_* \omega^{\otimes 2}\isomto \omega^{\otimes 2}$ is an isomorphism of sheaves on $\mathcal M_2$, and $\pi_*\omega^{\otimes 2}$ is an invertible sheaf on $M_2$.
\end{lemma}
\begin{proof}
On using the arguments in the proof of Corollary~\ref{cor:y2emptybranchlocus} which combine Serre's lemma with Lemma~\ref{lem:autofield},  we compute that the geometric automorphism groups of $\mathcal M_2$ are all given by $\{\pm 1\}$. In particular the order of each geometric automorphism group of $\mathcal M_2$ is invertible in our base $B=\spec(\ZZ[\tfrac{1}{30}])$. Hence $\mathcal M_2$ is a tame stack. Then \cite[Prop 6.1]{olsson:gtorsors} gives that the pullback functor $\pi^*$ induces an equivalence of categories between invertible sheaves on $M_2$ and invertible sheaves $\mathcal L$ on $\mathcal M_2$ with the following property: For every geometric point $x:\spec(k)\to \mathcal M_2$ the action of the automorphism group scheme $G_x$ on $x^*\mathcal L$ is trivial. We claim that $\omega^{\otimes 2}$ has this property. Any non-trivial automorphism  $g$ in  $G_x$ has order two,  
and $x^*\omega$ is a rank 1 representation of $G_x$.  
Thus we obtain that $gv\otimes gw=g^2v\otimes w=v\otimes w$ for sections $v,w$ of  $x^*\omega$.  This shows that $G_x$ acts trivially on $x^*\omega\otimes x^*\omega\cong x^*\omega^{\otimes 2}$,  proving the claim.  Now we can apply with $\mathcal L=\omega^{\otimes 2}$ the arguments of the part of the proof of \cite[Prop 6.1]{olsson:gtorsors} showing that $\pi^*$ is essentially surjective. These arguments directly give that $\pi_*\omega^{\otimes 2}$ is an invertible sheaf on $M_2$ and that the adjunction morphism $\pi^*\pi_* \omega^{\otimes 2}\to \omega^{\otimes 2}$ is an isomorphism. This completes the proof of Lemma~\ref{lem:descendingomega}.
\end{proof}

The method of Faltings--Chai also gives that a positive power of  $\omega$ on the toroidal compactification $\torm$ descends along the morphism $\bar{\pi}$ in \eqref{def:pibar} to an invertible sheaf on the minimal compactification $\minm$. More precisely, for any $k\in \ZZ_{\geq 1}$ such that $\omega^{\otimes k}$ is generated by its global sections we obtain that $\bar{\pi}_*\omega^{\otimes k}$ is an invertible sheaf on $\minm$ with 
\begin{equation}\label{eq:adjunctionisopibar}
\bar{\pi}^*\bar{\pi}_*\omega^{\otimes k}\cong \omega^{\otimes k}.
\end{equation}
To see this one can proceed as follows: By construction $\bar{\pi}:\torm\to \minm$ is  the canonical morphism which comes from $\torm\to \mathbb P^m_B$ via Stein factorization $\torm\to^{\bar{\pi}}\minm\to^{\varphi}\mathbb P^m_B$, where $\torm\to \mathbb P^m_B$ is induced by the assumption that $\omega^{\otimes k}$ is generated by its global sections. Then \eqref{eq:adjunctionisopibar} follows from  the isomorphism $\mathcal O_{\minm}\cong\bar{\pi}_*\mathcal O_{\torm}$ and the projection formula which identifies $\mathcal L\cong (\mathcal L\otimes\bar{\pi}_*\OL_{\torm})$ with $\bar{\pi}_*\bar{\pi}^*\mathcal L\cong \bar{\pi}_*\omega^{\otimes k}$ for $\mathcal L=\varphi^*\mathcal O(1)$.  

As the line bundle $\omega^{\otimes k}$ is generated by its global sections for some $k\in \ZZ_{\geq 1}$ by \cite[Main Thm]{chai:hilbmod}, we obtain \eqref{eq:adjunctionisopibar} for this $k$. This is used in the proof of Proposition~\ref{prop:fingen} given below. After completing the proof of Proposition~\ref{prop:fingen}, we can come back and deduce that \eqref{eq:adjunctionisopibar}  holds in fact with the explicit $k=2$ by Proposition~\ref{prop:fingen}. 

\subsubsection{Modular curves inside $M_2$}\label{sec:modularcurves}
We continue our notation. In this section we define and study modular curves inside $M_2$. After giving the construction, we establish some geometric properties which we will use in $\mathsection$\ref{sec:eisensteindivisor} below to compute the divisors associated to certain  Eisenstein series.

\paragraph{Modular morphism.} To define the modular curves $gC\subset M_2$, we consider the separated finite type DM-stack  $\mathcal E_2$ over $B$ whose objects $(E,\alpha)$ over any $B$-scheme $S$ are elliptic curves $E$ over $S$ with a symplectic level 2-structure $\alpha:(\ZZ/2\ZZ)^2_S\isomto E_2$ as in \cite[p.121]{fach:deg}.  By tensoring with $\OL$, we then define a morphism of stacks over $B$,
\begin{equation}\label{def:stackimmersion}
\phi:\mathcal E_2\to \mathcal M_2, \quad (E,\alpha)\mapsto (E,\alpha)\otimes \mathcal O=(A,\iota,\varphi,\beta),
\end{equation}
which lies over the modular morphism $\mathcal E\to\mathcal M$, $E\mapsto (A,\iota,\varphi)$, constructed by Bruinier--Burgos--K\"uhn~\cite[Prop 5.12]{brbuku:hilbmod} for the stack $\mathcal E$ over $B$ of elliptic curves. In \eqref{def:stackimmersion} the abelian scheme $A=E\otimes \OL$ over $S$ represents the presheaf $T\mapsto E(T)\otimes_\ZZ \OL$ on $(\textnormal{Sch}/S)$,  and $\iota:\OL\hookrightarrow \End(A)$ is induced by $\OL$-multiplication on $\OL$.  
Further, as $A^{\vee}=E\otimes \mathfrak d^{-1}$ represents the presheaf $T\mapsto E(T)\otimes \mathfrak d^{-1}$ on $(\textnormal{Sch}/S)$,  the polarization $\varphi$ is given by:  $$\varphi:\mathfrak d^{-1}\to \Hom_{\OL}(A,A^{\vee})^{\textnormal{sym}}, \quad \varphi(\lambda)(e\otimes \alpha)=e\otimes \lambda \alpha,$$  for $e\otimes \alpha$ in $A(T)=E(T)\otimes_\ZZ \OL$ and $\mathfrak d^{-1}$ the $\ZZ$-dual of $\OL$. 
The symplectic $\alpha:(\ZZ/2\ZZ)^2_S\isomto E_2$  is  compatible with the determinant $(\ZZ/2\ZZ)^2\times (\ZZ/2\ZZ)^2\to \ZZ/2\ZZ$ and the Weil pairing composed with $\mu_2\isomto (\ZZ/2\ZZ)_S$. Thus, on using results in \cite[$\mathsection$1]{amir:serretensor}, we  compute that  $$\beta=\alpha\otimes \OL:(\OL/2\OL)^2_S\cong (\ZZ/2\ZZ)^2_S\otimes\OL\isomto E_2\otimes\OL=A_2$$ is a principal level 2-structure on $(A,\iota,\varphi)$ with \eqref{def:symplecticlvl}. Furthermore $(E,\alpha)\mapsto (A,\iota,\varphi,\beta)$ is functorial and $\phi:\mathcal E_2\to\mathcal M_2$ is a morphism of stacks over $B$. Indeed $E\mapsto (A,\iota,\varphi)$ is functorial since it defines the functor $\mathcal E\to\mathcal M$, and if $f:(E,\alpha)\to (E',\alpha')$ is a morphism in $\mathcal E_2(S)$ then $\alpha'=f|_{E_2}\alpha$ and thus $\phi(f)=f\otimes\OL$ satisfies $\beta'=\phi(f)|_{A_2}\beta$.  

\paragraph{Modular curves $gC$.}Let $g\in \textnormal{GL}_2(\OL/2\OL)$. We denote by $gC\subset M_2$ the scheme-theoretic image of the morphism $\pi\tau_g\phi:\mathcal E_2\to M_2$, where the automorphism $\tau_g$ of $\mathcal M_2$ is as in \eqref{def:gactionmp} and $\pi:\mathcal M_2\to M_2$ is an initial morphism. Write $C$ for $gC$ if $g=\textnormal{id}$. We next use results of Bruinier--Burgos--K\"uhn~\cite{brbuku:hilbmod} and Yang~\cite{yang:hmsurfaceamerican} to establish some geometric properties of $gC\subset M_2$ which are applied in Lemma~\ref{lem:divisorintersection} below.

\begin{lemma}\label{lem:geompropimmersion}
The following statements hold.
\begin{itemize}
\item[(i)] The morphism $\phi: \mathcal E_2\to \mathcal M_2$ is proper.
\item[(ii)] The curve $C(\CC)$ is an irreducible component of $F_1(\CC)$.  
\item[(iii)]Let $g,h$ in $\gl2(\OL/2\OL)$. If $gC\cap hC$ is nonempty, then $h^{-1}g$ lies in $\gl2(\ZZ/2\ZZ)$. 
\end{itemize}
\end{lemma}
Here we view $\gl2(\ZZ/2\ZZ)$ as a subgroup of $\gl2(\OL/2\OL)$ via $\ZZ/2\ZZ\otimes \OL\cong \OL/2\OL$,  and $F_1(\CC)$ is the usual modular curve inside $M_{2}(\CC)$. Such modular curves were studied among others by Hirzebruch, Zagier and van der Geer (see \cite{hirzebruch:hmsenseignement,hiza:hmsintersection,geza:hms13}); notice that $F_1(\CC)$ can be  defined analogously as in \cite[Def V.1.3]{vandergeer:hilbertmodular} with $M_2(\CC)$ in place of $X_\Gamma$. In what follows we shall make identifications via GAGA without mentioning it explicitly.

\begin{proof}[Proof of Lemma~\ref{lem:geompropimmersion}]
We first prove (i). As already observed by Yang in his proof of \cite[Lem 2.2]{yang:hmsurfaceamerican}, it follows from Bruinier--Burgos--K\"uhn~\cite[Prop 5.14]{brbuku:hilbmod} that the modular morphism $\mathcal E\to\mathcal M$ is proper. This morphism fits into the commutative diagram
$$
\xymatrix@R=4em@C=4em{
\mathcal E_2 \ar[r]^{\phi} \ar[d] & \mathcal M_2 \ar[d] \\
\mathcal E \ar[r] &  \mathcal M} 
$$
of algebraic stacks in which the vertical morphisms are the forgetful morphisms. As the forgetful morphisms are finite  and $\mathcal E\to\mathcal M$ is proper, the composition $\mathcal E_2\to \mathcal E\to \mathcal M$ is proper and $\mathcal M_2\to \mathcal M$ is separated. Thus $\phi:\mathcal E_2\to \mathcal M_2$ is proper as claimed in (i).

We next show (ii). To relate $C(\CC)$ with $F_1(\CC)$, choose a coarse moduli space $\mathcal M\to M$ and  consider the image $V\subset M(\CC)$ of the curve $F_1(\CC)\subset M_2(\CC)$ under the quotient map 
\begin{equation}\label{eq:quotientmapdef}
p_\CC:M_2(\CC)\cong\mathbb H^2/\Gamma(2)\to \mathbb H^2/\sl2(\OL)\cong M(\CC)
\end{equation}
where $\Gamma(2)$ denotes the kernel of the natural projection $\sl2(\OL)\to\sl2(\OL/2\OL)$.  Here the algebraic space $M$ is in fact a scheme over $B$ and the quotient map $p_\CC$ comes from the base change to $\CC$ of a morphism $p:M_2\to M$ of $B$-schemes,  which is induced by the initial property of $\pi$ and which fits into a commutative diagram
$$
\xymatrix@R=4em@C=4em{
\mathcal M_2 \ar[r]^{\pi} \ar[d] & M_2 \ar[d]^{p} \\
\mathcal M \ar[r] &  M.}
$$
We now combine the displayed diagrams: Bruinier--Burgos--K\"uhn~\cite[Prop 5.12]{brbuku:hilbmod} showed that $V$ is the image of $\mathcal E(\CC)$ under the composition $\mathcal E(\CC)\to \mathcal M(\CC)\to M(\CC)$ and hence $C=\pi\phi(\mathcal E_2)$ satisfies $C(\CC)\subset p_\CC^{-1}(V)$. On using that $F_1(\CC)\cong W/\Gamma(2)$ and $V\cong W/\sl2(\OL)$ for some $W\subset \mathbb H^2$ which is stable under the $\sl2(\OL)$-action, we obtain 
\begin{equation}\label{eq:f1cpreimage}
p_\CC^{-1}(V)=F_1(\CC).
\end{equation}
As the morphism $\pi\phi:\mathcal E_2\to M_2$ is universally closed by (i) and $\mathcal E_2(\CC)$ is irreducible,  we see that $C(\CC)$ is irreducible and closed inside $F_1(\CC)$. Notice that $C(\CC)$ can not be of dimension zero. Thus $C(\CC)$ is an irreducible component of the curve $F_1(\CC)$ as claimed in (ii).  

To prove (iii) we take $g,h$ in $\gl2(\OL/2\OL)$ and we suppose that $gC\cap hC$ is nonempty. After possibly replacing $g$ by $h^{-1}g$, we may and do assume that $h=\textnormal{id}$.  Then $gC(k)\cap C(k)$ is nonempty for some algebraically closed field $k$. We next construct an elliptic curve $E$ over $k$ with symplectic level 2-structures $\alpha,\alpha^*$ such that
\begin{equation}\label{eq:relation-lvl-str}
g\cdot (\alpha\otimes \OL)=\alpha^*\otimes\OL.
\end{equation}
As $gC(k)\cap C(k)$ is nonempty and $gC$ is the scheme theoretic image of the morphism $\pi\tau_g\phi:\mathcal E_2\to M_2$ which is closed by (i),   we  obtain $(E,\alpha)$ and $(E',\alpha')$ in $\mathcal E_2(k)$ such that $\tau_g\phi(E,\alpha)$ and $\phi(E',\alpha')$ are isomorphic in $\mathcal M_2(k)$. Write $(x,g\cdot(\alpha\otimes \OL))$ for $\tau_g\phi(E,\alpha)$ and $(x',\alpha'\otimes \OL)$ for $\phi(E',\alpha')$.  Yang showed in \cite[Lem 2.2]{yang:hmsurfaceamerican} that $\mathcal E\to \mathcal M$ induces a bijection $$\Hom_{\mathcal E(k)}(E,E')\isomto \Hom_{\mathcal M(k)}(x,x').$$
Thus, after applying isomorphisms in $\mathcal E_2(k)$ and $\mathcal M_2(k)$, we obtain a symplectic level 2-structure $\alpha^*$ of $E$ satisfying \eqref{eq:relation-lvl-str}.  Now, it follows from \eqref{eq:relation-lvl-str} that $g=u\otimes \OL$ with $u=(\alpha^*)^{-1}\alpha$. Hence $g$ lies in the image of $\gl2(\ZZ/2\ZZ)$ inside $\gl2(\ZZ/2\ZZ\otimes \OL)\cong \gl2(\OL/2\OL)$ as claimed in (iii). This completes the proof of the lemma.\end{proof}

In view of \eqref{eq:f1cpreimage}, the group $\sl2(\OL/2\OL)\cong \sl2(\OL)/\Gamma(2)$ acts on $F_1(\CC)$  and then on the set of irreducible components of $F_1(\CC)$. In what follows, the action of $\sl2(\OL/2\OL)$ on (the set of irreducible components of) $F_1(\CC)$ will refer to this action.

\paragraph{Components of $F_1(\CC)$.}We showed in Lemma~\ref{lem:geompropimmersion}~(ii) that $C(\CC)$ is an irreducible component of $F_1(\CC)$. To describe all irreducible components of $F_1(\CC)$ in terms of $C$,  write $G$ for $\sl2(\OL/2\OL)$ and let $H\subset G$ be the stabilizer of $C(\CC)$ under the action of $G$  on the set of irreducible components of $F_1(\CC)$.  It is known (\cite[p.189]{vandergeer:hilbertmodular}) that  
\begin{equation}\label{def:f1ccomponents}
F_1(\CC)=\cup F^{ij}_\CC
\end{equation}
has 10 irreducible components $F^{ij}_\CC$ which are labelled so that the closure of $F^{ij}_\CC$ in $\minm(\CC)$ does not intersect after desingularization the two cusps $i$ and $j$ of the minimal desingularization of $\minm(\CC)$. Here we label the five cusps of $\minm(\CC)$ by $0,\dotsc,4$ following Hirzebruch~\cite{hirzebruch:ck}. The next result computes $F^{ij}_\CC$ in terms of $C$. 

\begin{lemma}\label{lem:fijcomputation}
For each $ij$ there exists $g_{ij}\in G$ such that $F^{ij}_\CC=(g_{ij}C)_\CC$. In particular, $F_1(\CC)=\cup (gC)_\CC$ with the disjoint union taken over all left cosets $g\in G/H$.
\end{lemma} 
\begin{proof}
Let $p_\CC:M_2(\CC)\to M(\CC)$ be the quotient map induced by  $\Gamma(2)\subset \sl2(\OL)$ as in \eqref{eq:quotientmapdef}.  The Hirzebruch--Zagier divisor $V=p_\CC(F_1(\CC))$ in $M(\CC)$ is irreducible, and the arguments before \eqref{eq:f1cpreimage} show that the irreducible component $C(\CC)$ of $F_1(\CC)$ satisfies $V=p_\CC(C(\CC))$. Then we deduce that $G$ acts transitively on the set of irreducible components of $F_1(\CC)$.  Hence for each $ij$ there exists $g_{ij}\in G$ such that $$F^{ij}_\CC=g_{ij}(C(\CC))=(g_{ij}C)_\CC.$$ Here the second equality uses that the identification $M_2(\CC)\cong \mathbb H^2/\Gamma(2)$ is compatible with the actions of $G\cong \sl2(\OL)/\Gamma(2)$, where $G\subset\gl2(\OL/2\OL)$ acts on $M_2(\CC)$ via $\pi:[\mathcal M_2(\CC)]\isomto M_2(\CC)$ and \eqref{def:gactionpn}  while $\sl2(\OL)/\Gamma(2)$ acts on $\mathbb H^2/\Gamma(2)$ as usual. This proves the first statement, which implies the second statement since $H$ is the stabilizer of $C(\CC)$.  \end{proof}
This description of the connected components of $F_1(\CC)$ is crucial for our computation of the divisors of certain Eisenstein series in Lemma~\ref{lem:divisorintersection} below.

\subsubsection{Eisenstein series}\label{sec:eisenstein}

We continue our notation. Let $\Gamma(2)$ be the kernel of the projection $\sl2(\OL)\to\sl2(\OL/2\OL)$. In this section we discuss the Fourier expansion of the Eisenstein series of weight two for $\Gamma(2)$ corresponding to the cusps of $\mathbb H^2/\Gamma(2)$. In particular we compute the constant terms.

\paragraph{Definition of $E_i$.}Following Hirzebruch~\cite{hirzebruch:ck} we label the five cusps of $\mathbb H^2/\Gamma(2)$ by $0,\dotsc,4$  with the cusp $\infty$ labelled by 0. The group $\sl2(\OL)$ acts transitively on these five cusps. For any $i\in\{0,\dotsc,4\}$ choose $A_i\in \sl2(\OL)$ which sends the cusp $i$ to $\infty$. Hirzebruch works in \cite[p.164]{hirzebruch:ck} with five Eisenstein series $E_i$ of weight two for $\Gamma(2)$, defined by  \begin{equation}\label{def:eisenstein}
E_0=\sum\phantom{}^{'}\tfrac{1}{N(m\tau+n)^2} , \quad E_i=E_0|_{A_i},
\end{equation}
such that $E_i$ vanishes at all cusps except the $i$-th.  
Here the restricted sum is taken via Hecke's summation over all coprime $(m,n)\in \OL^2$ reducing to $(0,1)$ in $(\OL/2\OL)^2$, where the restriction $\phantom{}^{'}$ means we sum  modulo the subgroup of $\OL^\times$ reducing to  $1$ in $\OL/2\OL$. Further we denote by $|_{A_i}$ the usual slash operator and by $N(\cdot)$ the usual norm, see \eqref{def:slashnorm}. 

\paragraph{Fourier expansion.}Klingen~\cite{klingen:eisensteinfourier} used analytic methods to establish formulas for the Fourier coefficients of Eisenstein series of weight $k\geq 2$ for $\Gamma(2)$. Therein the constant term $a_0$ is determined by the value at $k$ of certain zeta functions, and for $E_i$ one can reduce the computation of $a_0$ to the result $\zeta_{\QQ(\sqrt{5})}(-1)=\tfrac{1}{30}$ of Siegel~\cite{siegel:zetavalues} and Zagier~\cite{zagier:zetavalues}.  Then computing the quantities in Klingen's formulas leads to the following:

\begin{lemma}\label{lem:eisensteinfourierexpansion}
The coefficients of the Fourier expansion of $E_i$ at the cusp $j$  lie in $\Z$ and are coprime. The constant term $a_0$ satisfies $a_0=3$ when $j=i$ and $a_0=0$ otherwise.
\end{lemma} 
The computations required to deduce this lemma from \cite{klingen:eisensteinfourier} are all basic and straight forward. However, we include them here since we are not aware of precise references. We begin by recalling the slash operator $|_A$ and the norm $N(\cdot)$. They are defined by  \begin{equation}\label{def:slashnorm}
f|_{A}(\tau)=\tfrac{1}{N(c\tau+d)^2}f(A\tau) \quad \textnormal{and} \quad N(u\tau+v)=(u\tau_1+v)(u'\tau_2+v')
\end{equation}
for any Hilbert modular form $f$ of weight two for $\Gamma(2)$, any $A\in \sl2(\OL)$ given by {\tiny$\begin{pmatrix}
a \ b\\
c \ d
\end{pmatrix}$}, any $(u,v)\in \OL^2$ and any $\tau\in \mathbb H^2$. Here $x'$ denotes the conjugate of $x\in \QQ(\sqrt{5})$.

\paragraph{Relation to $G_2$ and $G_2^*$.}To apply Klingen's results, we first relate $E_0$ to the Eisenstein series $G_2$ and $G_2^*$ defined in \cite[(12) and $\mathsection$3]{klingen:eisensteinfourier} with $k=2$, $\mathfrak a=\OL$, $\mathfrak n=2\OL$ and $(a_1,a_2)=(0,1)$. The ray class group  modulo $2\OL$ is trivial.  Hence \cite[(20)]{klingen:eisensteinfourier} leads to
\begin{equation}\label{eq:relationEig2}
wE_0=G_2^*=\tfrac{1}{\zeta_2(2)}G_2, \quad w=4.
\end{equation} 
Here $\zeta_2(2)=\sum \tfrac{1}{|\OL/I|^2}$ with the sum taken over all ideals $I\subseteq \OL$ coprime to $2\OL$, and  $w$ is the index of $\{x\in \OL^{\times}_{+}; \, \bar{x}=1\}$ inside   $\{x\in \OL^{\times}; \, \bar{x}=1\}$ where we denote by $\OL^{\times}_+$ the totally positive units of $\OL$ and we write $\bar{x}$ for the image of $x\in \OL$ inside $\OL/2\OL$.  
\paragraph{Constant term and $\zeta_2(2)$.}We next compute the constant term $a_0$ of the Fourier expansion of $E_0$ at $\infty$ and the value $\zeta_2(2)$. The index of $\{x\in \OL^\times_+;\, \bar{x}=1\}$ inside $\OL^\times$ is given by $|(\OL/2\OL)^\times|w$.  Then, on using the formula \cite[(14)]{klingen:eisensteinfourier} for the constant term $\Sigma_1$ of the Fourier expansion of $G_2$ at $\infty$, we find that $\Sigma_1=3w\zeta_2(2)$.  Thus \eqref{eq:relationEig2} proves
\begin{equation}\label{eq:compuconstant+zetavalue}
a_0=3, \quad \textnormal{while} \quad \zeta_2(2)=\tfrac{\pi^4}{\sqrt{5}}\cdot \tfrac{1}{40}
\end{equation}
follows from $\zeta_L(-1)=\tfrac{1}{30}$ in \cite{siegel:zetavalues,zagier:zetavalues} where $\zeta_L$ is the Dedekind zeta function of $L=\QQ(\sqrt{5})$.  Here we used the Euler product formula which gives $\zeta_2(2)=\tfrac{15}{16}\zeta_L(2)$ and we applied the functional equation which relates $\zeta_L(-1)$ with $\zeta_L(2)$.  
\begin{proof}[Proof of Lemma~\ref{lem:eisensteinfourierexpansion}]
Let $i\in\{0,\dotsc,4\}$ be a cusp. By construction the Eisenstein series $E_i$ is of the form $E_i=E_0|_{A_i}$ where $A_i\in \sl2(\OL)$ sends the cusp $i$ to $\infty$. Thus the constant term $a_0$ of the Fourier expansion of $E_i$ at the cusp $j$ has the following properties: It equals the constant term of the Fourier expansion of $E_0$ at $\infty$ if $j=i$, and it satisfies $a_0=0$ otherwise. Hence \eqref{eq:compuconstant+zetavalue} proves the statements about the constant term. 

We now deal with the higher Fourier coefficients of $E_i$. Denote by $G_{2,i}$ the Eisenstein series defined by Klingen in \cite[(12)]{klingen:eisensteinfourier} with $k=2$, $\mathfrak a=\OL$, $\mathfrak n=2\OL$ and $(a_1,a_2)=A_i^t(0,1)$ where $A_i^t$ is the transpose of $A_i$. On using \eqref{eq:relationEig2} and \eqref{eq:compuconstant+zetavalue} we compute  
\begin{equation}\label{eq:comp-Ei-G2i}
E_i=E_0|_{A_i}=\tfrac{\sqrt{5}}{\pi^4}\cdot 10G_{2,i}.
\end{equation}
In \cite[p.182]{klingen:eisensteinfourier}, Klingen obtains a formula for the higher Fourier expansion $\Sigma_2$ of $G_{2,i}$ at the cusp $\infty$ and he observes that $e^{2\pi i(a_2\nu+a'_2\nu')}$ is a 4-th root of unity for all $\nu\in \tfrac{1}{2\sqrt{5}}\OL$. These results and \eqref{eq:comp-Ei-G2i} imply that all higher Fourier coefficients of $E_i$ lie in $\ZZ[\sqrt{-1}]$, and then they all lie in $\ZZ$ since  the terms involving $\sqrt{-1}$ cancel out by symmetry. 
 
Finally, we show that the Fourier coefficients are coprime. For this it suffices to find a Fourier coefficient $a_\nu$ with $3\nmid a_\nu$, since $a_0=3$ and $E_i=E_0|_{A_i}$.  A set of representatives of $\OL^\times$ modulo $\{x\in \OL^\times_+; \, \bar{x}=1\}$ is given by $\{\pm \epsilon^{i}; \, i=0,\dotsc,5\}$.  Then,  on using again Klingen's formula for $\Sigma_2$, we compute that the totally positive element $\nu=\epsilon/\sqrt{5}$ of the module $\tfrac{1}{2}\mathfrak d^{-1}$ satisfies $a_{\nu}=-8$ which is not divisible by 3.  This completes the proof.
\end{proof}
In particular, Lemma~\ref{lem:eisensteinfourierexpansion} assures that the Hilbert modular form $E_i$ of weight two for $\Gamma(2)$ has all Fourier coefficients in $B=\ZZ[\tfrac{1}{30}]$. Thus the $q$-expansion principle of Rapoport~\cite{rapoport:hilbertmodular} and Faltings--Chai~\cite{fach:deg,chai:hilbmod} gives that  $E_i$ corresponds to a global section of the square of the Hodge bundle $\omega^{\otimes 2}$ on the toroidal compactification $\torm$.

\subsubsection{Divisors of the Eisenstein series}\label{sec:eisensteindivisor} 
We continue our notation. Lemma~\ref{lem:descendingomega} gives that $\pi_*\omega^{\otimes 2}$ is an invertible sheaf on the regular scheme $M_2$, where $\pi:\mathcal M_2\to M_2$ is an initial morphism and $\omega$ is the Hodge bundle.  In this section we study the divisors $D_i$ on $M_2$ associated to the five global sections of $\pi_*\omega^{\otimes 2}$  corresponding to the five Eisenstein series $E_i$ defined in $\mathsection$\ref{sec:eisenstein}.  

To compute the divisors $D_i$, we first show that they are horizontal. The following lemma relies on the computation of the constant term of $E_i$ in Lemma~\ref{lem:eisensteinfourierexpansion} and fundamental results for $\torm$ and $\minm$ proven by Rapoport~\cite{rapoport:hilbertmodular} and Faltings--Chai~\cite{fach:deg,chai:hilbmod}.

\begin{lemma}\label{lem:divhorizontal}
Each prime divisor $Z$ of $D_i$ is the closure of its generic fiber $Z_\QQ$.
\end{lemma}
\begin{proof}
Let $Z$ be a prime divisor of $D_i$. The idea of the proof is as follows: We first  reduce the problem to showing that the closure $Z^*$ of $Z$ in $\minm$ is a horizontal divisor. Then we prove that $Z^*$ is a horizontal divisor by combining the $q$-expansion principle with the computation of the constant term of the Eisenstein series $E_i$ in Lemma~\ref{lem:eisensteinfourierexpansion}. 

We first introduce some notation and make a reduction. Let $s_i$ be the global section of $\omega^{\otimes 2}$ on $\torm$ corresponding to the Eisenstein series $E_i$. We also denote by $s_i$ its restriction to $\mathcal M_2\subset \torm$. Write $s=s_i$ and $D=D_i$. For each $k\in \ZZ_{\geq 1}$ we obtain that $kD$ is the divisor on $M_2$ associated to the global section of $\pi_*\omega^{\otimes k}$ defined by $s^{\otimes k}$; indeed  $\pi_*\omega^{\otimes 2k}\cong (\pi_*\omega^{\otimes 2})^{\otimes k}$ by Lemma~\ref{lem:descendingomega} and the projection formula.   
To prove Lemma~\ref{lem:divhorizontal}, we thus may replace $s$ by $s^{\otimes k}$ and $D$ by $kD$ for any $k\in \ZZ_{\geq 1}$. Moreover \cite[Main Thm~(iii)]{chai:hilbmod} and \eqref{eq:adjunctionisopibar} show that we may and do this for some $k$ such that $\bar{\pi}_*\omega^{\otimes 2k}$ is an invertible sheaf on $\minm$ with $\bar{\pi}^*\bar{\pi}_*\omega^{\otimes 2k}\cong\omega^{\otimes 2k}$ for $\bar{\pi}:\torm\to\minm$ as in \eqref{def:pibar}. We now can consider the divisor $D^*$ on the normal scheme $\minm$ associated to the global section of $\bar{\pi}_*\omega^{\otimes 2k}$ defined by $s$.

Let $Z^*$ be the closure in $\minm$ of our given prime divisor $Z$ of $D$. As $\pi$ is the restriction of $\bar{\pi}$ to $\mathcal M_2\subset \torm$ and $\bar{\pi}(\torm\setminus \mathcal M_2)=\minm\setminus M_2$, we see that $\pi_*\omega^{\otimes 2k}$ is the restriction of $\bar{\pi}_*\omega^{\otimes 2k}$ to $M_2\subset\minm$.  It follows that $D$ is the restriction of $D^*$ to $M_2\subset \minm$. In particular $Z^*$ is a prime divisor of $D^*$.  Moreover the claim (1) below gives that $Z^*$ is a horizontal divisor. Thus the generic point $z$ of $Z^*$ lies in the generic fiber $Z^*_\QQ$ of $Z^*$ and $Z^*$ is the closure of $Z^*_\QQ$ in $\minm$.  In fact $z$ also lies in $Z$ since $Z=M_2\cap Z^*$ is nonempty open in the irreducible $Z^*$.  Hence $Z_\QQ=Z\cap Z^*_\QQ$ contains $z$ and thus the closure of $Z_\QQ$ in $M_2$ is $Z^*\cap M_2=Z$.  This proves that $Z$ is the closure of its generic fiber  as claimed in Lemma~\ref{lem:divhorizontal}.

It remains to show the claim (1) that any prime divisor of $D^*$ is horizontal. Each prime divisor of $D^*$ is either horizontal or vertical, since $\minm\to B$ is closed and $B$ is a connected Dedekind scheme.  To show that $D^*$ has no vertical prime divisors, we take a closed point $p\in B$. The fiber $M^*_{2,p}$ of $\minm$ over $p$ is irreducible. Hence any vertical prime divisor $Z^*$ of $D^*$ over $p$  equals $M^*_{2,p}$, since $Z^*\subseteq M^*_{2,p}$ is irreducible closed with $\dim Z^*=\dim \minm-1$ equal to $2=\dim M^*_{2,p}.$
This implies that the global section of $\bar{\pi}_*\omega^{\otimes 2k}$ defining $D^*$ vanishes on $Z^*=M^*_{2,p}$. As $s$ identifies with the pullback of this section under $\bar{\pi}$, we deduce that $s$ vanishes on the whole fiber $\overline{\mathcal M}_{2,p}$ of $\torm$ over $p$.  Then the $q$-expansion principle assures that the $q$-expansion of $s$ at each cusp of $\mathcal M_2$ is zero modulo $p$.  But at the cusp $i$ of $\mathcal M_2$ corresponding to $E_i$, Lemma~\ref{lem:eisensteinfourierexpansion} gives that the constant term of the $q$-expansion of $s=s_i^{\otimes k}$ is $3^k$ which is nonzero modulo $p$.  This shows that  $D^*$ can indeed not have vertical prime divisors, which completes the proof of the lemma.
\end{proof}

Next, we combine the above result with Lemmas~\ref{lem:geompropimmersion}, \ref{lem:fijcomputation} and \ref{lem:eisensteinfourierexpansion} to compute the divisors $D_i$ in terms of the modular curves $gC$ constructed in $\mathsection$\ref{sec:modularcurves}.

\begin{lemma}\label{lem:divisorintersection}
The following statements hold.
\begin{itemize}
\item [(i)] The support of $D_i$ is the union of four disjoint $gC$ with $g\in \sl2(\OL/2\OL)$.

\item[(ii)] The intersection $\cap_i D_i$ is empty.
\end{itemize}
\end{lemma}
\begin{proof}
To prove (i)  we first determine the support of $D_i$ over $\QQ$. In Lemma~\ref{lem:eisensteinfourierexpansion} we computed the constant term of the Fourier expansion of our Eisenstein series $E_i$ in each cusp $j$ of $M_2^*(\CC)$: It is nonzero if and only if $j=i$. Thus our $E_i$ is up to a constant factor the Eisenstein series $E_{i+1}$ in \cite[p.190]{vandergeer:hilbertmodular}.  Then \cite[p.190]{vandergeer:hilbertmodular} implies  \begin{equation}\label{eq:divisorcompoverC}
\textnormal{supp} D_{i,\CC}=\cup F^{ij}_{\CC}
\end{equation}
since by construction the Eisenstein series $E_i$ corresponds to the global section of $\pi_*\omega^{\otimes 2}$ defining the divisor $D_i$. Here the union is taken over the four cusps $j\neq i$ and we recall from \eqref{def:f1ccomponents} that the curve 
$F_1(\CC)=\cup F^{ij}_\CC$
has 10 irreducible components $F^{ij}_\CC$. Lemma~\ref{lem:fijcomputation} gives $g_{ij}\in \sl2(\OL/2\OL)$ such that $F^{ij}_\CC=(g_{ij}C)_\CC$. Then \eqref{eq:divisorcompoverC} leads to  
\begin{equation}\label{eq:divisorcompoverQ}
\textnormal{supp} D_{i,\QQ}=\cup F^{ij}_{\QQ}, \quad F^{ij}=g_{ij}C
\end{equation}
with the union  taken over the four cusps $j\neq i$. We now compute the closure inside $M_2$. As the moduli stack $\mathcal E_2$ of elliptic curves with symplectic level 2-structures is irreducible  and $F^{ij}=g_{ij}C$ is the scheme-theoretic image of $\pi\tau_{g_{ij}}\phi:\mathcal E_2\to M_2$, we obtain that $F^{ij}=g_{ij}C$ is an integral closed subscheme of $M_2$ with generic fiber $F^{ij}_\QQ$ nonempty. Hence $F^{ij}$ is the closure of $F^{ij}_\QQ$ in $M_2$.   Further, the closure of $\textnormal{supp} D_{i,\QQ}$ in $M_2$ is given by $\textnormal{supp} D_i$ since each prime divisor of $D_i$ is the closure of its generic fiber by Lemma~\ref{lem:divhorizontal}.  Thus \eqref{eq:divisorcompoverQ} shows
\begin{equation}\label{eq:suppcompdi}
\textnormal{supp} D_i=\cup F^{ij}
\end{equation} 
with the union taken over the four cusps $j\neq i$. Next, we determine the intersection $F^{ij}\cap F^{kl}$ for arbitrary cusps $i,j,k,l$. We claim that either $F^{ij}\cap F^{kl}$ is empty or $F^{ij}=F^{kl}$. To prove this claim, we assume that $F^{ij}\cap F^{kl}$ is nonempty.  Then Lemma~\ref{lem:geompropimmersion} gives that $g=g_{kl}^{-1}g_{ij}$ lies in the image of $\gl2(\ZZ/2\ZZ)$ in $\gl2(\OL/2\OL)$ via $\ZZ/2\ZZ\otimes \OL\cong \OL/2\OL$. It follows that $gC=C$,  since $C$ is the scheme-theoretic image of the morphism $\pi\phi:\mathcal E_2\to M_2$ and $\phi$ is induced by tensoring with $\OL$. This implies that $F^{ij}=g_{ij}C=g_{kl}C=F^{kl}$ as claimed. We now consider the four subschemes $F^{ij}\subset M_2$ appearing in \eqref{eq:suppcompdi}. They are pairwise distinct since their base changes $F^{ij}_\CC$ are pairwise distinct by \eqref{def:f1ccomponents}, and thus the claim below \eqref{eq:suppcompdi} assures that these four $F^{ij}$ are disjoint. This completes the proof of (i).


It remains to prove (ii). Suppose that (ii) does not hold, that is $\cap D_i$ is nonempty. Then \eqref{eq:suppcompdi} gives that $\cap_i \cup F^{ij}$ is nonempty and hence there are $j_i\neq i$ such that $\cap_i F^{ij_i}$ is nonempty. But, as the five $F^{ij_i}$ can not all be equal by \eqref{def:f1ccomponents}, the claim below \eqref{eq:suppcompdi} shows that there exist at least two disjoint $F^{ij_i}\subset M_2$ which implies that $\cap_i F^{ij_i}$ is empty. This contradiction proves (ii) and completes the proof of Lemma~\ref{lem:divisorintersection}.
\end{proof}
In the proof of Lemma~\ref{lem:divisorintersection}~(i) we obtain in addition a description in \eqref{eq:suppcompdi} of those four $g\in \sl2(\OL/2\OL)$ with $gC$ a prime divisor of $D_i$. The description is in terms of the intersection behaviour of the cusp resolutions in the minimal desingularization of $M_2^*(\CC)$. This is crucial for the proof of (ii) and might be of independent interest.

\subsubsection{Proof of Proposition~\ref{prop:fingen}}\label{sec:proofpropfingen}
We continue our notation. Recall that $\bar{\pi}:\torm\to\minm$ is the canonical morphism from \eqref{def:pibar} and that $s_i$ is the global section of $\omega^{\otimes 2}$ corresponding to the Eisenstein series $E_i$ defined in $\mathsection$\ref{sec:eisenstein} where $\omega$ denotes the Hodge bundle on $\torm$. In the following proof of Proposition~\ref{prop:fingen}, we combine Lemmas~\ref{lem:eisensteinfourierexpansion} and \ref{lem:divisorintersection} with fundamental results for $\minm$ obtained by Rapoport~\cite{rapoport:hilbertmodular} and Faltings--Chai~\cite{fach:deg,chai:hilbmod}.

\begin{proof}[Proof of Proposition~\ref{prop:fingen}]
We first show that the sections $s_i$ have no common zeroes on the boundary $\torm\setminus \mathcal M_2$. It follows for example from \cite[Main Thm]{chai:hilbmod} that $\torm\setminus \mathcal M_2$ has five connected components which we label such that $\bar{\pi}:\torm\to\minm$ sends the component $i$ to the cusp $i$ of $M_2^*$,  where the five cusps of $\minm$ are labelled by $0,\dotsc, 4$.  The section $s_i$  corresponds to the Eisenstein series $E_i$, and Lemma~\ref{lem:eisensteinfourierexpansion} gives that the constant term of the Fourier expansion of $E_i$ at the cusp $i$ of $M_2^*$ is invertible in $B$. Then the $q$-expansion principle  implies that $s_i$ does not vanish in the connected component $i$ of $\torm\setminus \mathcal M_2$.  Thus the sections $s_i$ have no common zeroes on the boundary $\torm\setminus \mathcal M_2$ as desired.

Further, it is a rather direct consequence of Lemma~\ref{lem:divisorintersection}~(ii) that the sections $s_i$ have no common zeroes in $\mathcal M_2$. The details are as follows. Let $p:X\to \mathcal M_2$ be an \'etale scheme cover of $\mathcal M_2$. 
To show that the sections $s_i$ generate $\omega^{\otimes 2}$ over $\mathcal M_2$, it suffices to prove that the intersection
$
\cap_i \textnormal{div}(p^*s_i)
$ is empty.  Corollary~\ref{cor:y2emptybranchlocus} implies that the initial morphism $\pi:\mathcal M_2\to M_2$ is \'etale and hence $q=p\pi$ is \'etale. Then we compute  $$\textnormal{div}(p^*s_i)=\textnormal{div}(q^*t_i)=q^*D_i$$ where $D_i$ is the divisor associated to the global section $t_i$ of $\pi_*\omega^{\otimes 2}$ defined by $s_i$. Here we identified $\pi^*t_i$ with $s_i$ via the adjunction isomorphism $\pi^*\pi_* \omega^{\otimes 2}\isomto \omega^{\otimes 2}$ in Lemma~\ref{lem:descendingomega}. Now, as the intersection $\cap_i D_i$ is empty by Lemma~\ref{lem:divisorintersection}~(ii), we deduce that $
\cap_i \textnormal{div}(p^*s_i)=\cap_i q^*D_i
$ is empty as desired.  This completes the proof of Proposition~\ref{prop:fingen}.
\end{proof}

\subsection{Proof of Proposition~\ref{prop:moduliinter}}\label{sec:proofmoduliinter}
We continue the notation and terminology of $\mathsection$\ref{sec:moduliinterpretation} and $\mathsection$\ref{sec:hodgebundlepositive}. Recall that $\minm$ denotes the minimal compactification of the coarse moduli scheme $M_2$. In this section we use the above results to construct the isomorphisms $\minm\isomto X$ and $M_2\isomto Y$ in Proposition~\ref{prop:moduliinter}.

\begin{proof}[Proof of Proposition~\ref{prop:moduliinter}]
As in the statement of Proposition~\ref{prop:moduliinter}, we consider the objects over the base scheme $B=\spec(\ZZ[\tfrac{1}{30}])$. The following proof consists of four steps.  

\paragraph{1.}We first construct a morphism $\minm\to \mathbb P^4_B$ using the Hodge bundle $\omega$ on the toroidal compactification $\torm$ of $\mathcal M_2$. Proposition~\ref{prop:fingen} gives that $\omega^{\otimes 2}$ is globally generated by the five global sections $s_i$ corresponding to the five Eisenstein series $E_i$ defined in $\mathsection$\ref{sec:eisenstein}. Hence these sections $s_i$ define a morphism $\torm\to \mathbb P^4_B$  
which is automatically proper since $\torm$ is proper over $B$.  Then Stein factorization decomposes $\torm\to \mathbb P^4_B$ into two morphisms: 
$$\torm \to^{\bar{\pi}} \minm\to^f \mathbb P^4_B,$$
where the canonical morphism $\mathcal O_{\minm}\isomto\bar{\pi}_*\mathcal O_{\torm}$ is an isomorphism and the morphism $f$ is finite.  Here we used the construction of $\minm$ which gives moreover that the restriction of $\bar{\pi}$ to $\mathcal M_2$ is an initial morphism $\mathcal M_2\to M_2$, see \cite[Main Thm (vi)]{chai:hilbmod}.

\paragraph{2.}Next, we prove that the scheme theoretic image $X'$ of $f$ equals $X$. The projection formula together with $\omega^{\otimes 2}\cong \bar{\pi}^*f^*\mathcal O(1)$ and $\bar{\pi}_*\mathcal O_{\torm}\cong \mathcal O_{\minm}$ implies that $\bar{\pi}_*\omega^{\otimes 2}\cong f^*\mathcal O(1)$.  In particular $\bar{\pi}_*\omega^{\otimes 2}$ is an invertible sheaf on $\minm$ which is globally generated by the five sections defined by the $s_i$.  The section $s_i$  corresponds to the Eisenstein series $E_i$ which coincides up to a constant factor with the Eisenstein series $E_{i+1}$ in van der Geer~\cite[p.191]{vandergeer:hilbertmodular}. Thus the arguments in \cite[p.190-191]{vandergeer:hilbertmodular}, which can be modified so that they work over $\QQ$ as explained by  Shepherd-Barron--Taylor~\cite[p.289]{sbta:ck},  
show that $$f_\QQ:\minmq\hookrightarrow \mathbb P^4_\QQ$$ is a closed immersion with scheme theoretic image $X_\QQ\subset \mathbb P^4_\QQ$.  Hence $X'_\QQ$ equals $X_\QQ$ inside $\mathbb P^4_\QQ$.  The schemes $X'$ and $X$ are both integral closed subschemes of $\mathbb P^4_B$. Indeed the scheme-theoretic image $X'$ of the finite morphism $f:\minm\to \mathbb P^4_B$ has this property since $\minm$ is integral,  
while a computation (using Macaulay2) with the defining equations of $X$ gives that $X$ is integral as well. Furthermore  $X'\subset \mathbb P^4_B$ and $X\subset \mathbb P^4_B$ have the same generic point, since $X'_\QQ=X_\QQ$ inside $\mathbb P^4_B$ and since the generic points of $X'$ and $X$ are the generic points of $X'_\QQ$ and $X_\QQ$ respectively.  It follows that $X'=X$ as desired.  

\paragraph{3.}We now apply Zariski's main theorem to deduce that $f:\minm\to\mathbb P^4_B$ induces an isomorphism on its image. The noetherian schemes $\minm$ and $X$ are both integral, and a computation (using Macaulay2) with the defining equations of $X$ shows that $X$ is normal. As $X$ is the scheme-theoretic image of $f$ by step 2, we obtain that $f$ induces a morphism   
$$f:\minm\to X.$$
This is a birational morphism of integral schemes, since it extends $f_\QQ:M^*_{2,\QQ}\to X_\QQ$ which is an isomorphism of $\QQ$-schemes as explained in step 2.  Then, as $X$ is normal, we can apply Zariski's main theorem to deduce that $\mathcal O_X\isomto f_*\mathcal O_{\minm}$.  This proves that the finite morphism $f:\minm\to X$ has to be an isomorphism.    

\paragraph{4.}It remains to show that $f:\minm\isomto X$ restricts to an isomorphism $M_2\isomto Y$. Write $V$ for the boundary $\minm\setminus M_2$, and recall that $Z$ is the union of the images of the five $B$-points of $X$ obtained by permuting the coordinates of $(1,0,0,0,0)$. We now prove that
$$f(V)=Z.$$
It follows from \cite[p.190-191]{vandergeer:hilbertmodular} and \cite[p.289]{sbta:ck} that $V_\QQ$ is given by five $\QQ$-points of $\minm$ and $f(V_\QQ)=Z_\QQ$.  
Further, $Z$ is the closure of $Z_\QQ$ in $X$ and \cite[Main Thm (vi)]{chai:hilbmod} implies that $V$ is the closure of $V_\QQ$ in $\minm$.  
As $f$ is a homeomorphism, we then deduce that $f(V)=\overline{f(V_\QQ)}=Z$.  
Thus the isomorphism $f:\minm\isomto X$ restricts to an isomorphism $M_2=\minm\setminus V\isomto X\setminus Z=Y$  as desired. This completes the proof of Proposition~\ref{prop:moduliinter}. 
\end{proof}

Proposition \ref{prop:moduliinter} implies that $Y$ is a coarse Hilbert moduli scheme over $B$ of the moduli problem $\mathcal P$ on $\mathcal M_2$ defined in \eqref{def:symplecticlvl}.  For what follows it will be crucial to specify an initial morphism $\mathcal M_2\to Y$. In the above proof of Proposition~\ref{prop:moduliinter}, we constructed in step~1 the morphisms $\torm\to^{\bar{\pi}} \minm\to^{f} \mathbb P^4_B$ and we showed in step~4  that the restriction to $\mathcal M_2\subset \torm$ of the composition $f\bar{\pi}:\torm\to \mathbb P^4_B$  defines an initial morphism
\begin{equation}\label{def:initialmorphismoverB}
\pi:\mathcal M_2\to Y.
\end{equation}
Indeed $f$ restricts to an isomorphism $M_2\isomto Y$ by step 4 and the restriction of $\bar{\pi}$ to $\mathcal M_2\subset \torm$ is an initial morphism $\mathcal M_2\to M_2$ by \cite[Main Thm (vi)]{chai:hilbmod}. By construction $\pi:\mathcal M_2\to Y$ is given by the five global sections of $\omega^{\otimes 2}$ on $\mathcal M_2$ corresponding via the $q$-expansion principle to the five Eisenstein series $E_i$ for  $\Gamma(2)$ of weight 2 defined in $\mathsection$\ref{sec:eisenstein}.

\subsection{Heights}\label{sec:heightcomparison}

Let $Y$ be as in \eqref{def:sigma24surface} and define $B=\ZZ[\tfrac{1}{30}]$. Proposition~\ref{prop:moduliinter} gives that $Y_B$ is a coarse moduli scheme over $B$ of the arithmetic moduli problem $\mathcal P$ on $\mathcal M_B$ of principal level 2-structures with \eqref{def:symplecticlvl}.  The goal of this section is to prove the following result which explicitly bounds the logarithmic Weil height $h$ on $Y\subset \mathbb P^4_\ZZ$ in terms of the height $h_\phi$.

\begin{proposition}\label{prop:heightcomp}
Any $P\in Y(\bar{\QQ})$ satisfies $h(P)\leq 2h_\phi(P)+8^8\log(h_\phi+8)$.
\end{proposition}
Here $h_\phi$ is the height on $Y(\bar{\QQ})\cong Y_B(\bar{\QQ})$ defined with respect to the initial morphism $\pi:~\cmp\to Y_B$ in \eqref{def:initialmorphismoverB}, and recall that the logarithmic Weil height $h$ on projective space is normalized as in \cite[p.16]{bogu:diophantinegeometry}. We shall discuss the bound at the end of Section~\ref{sec:heightcomparison}. 

In the remaining of this section we prove Proposition~\ref{prop:heightcomp} as follows. We first relate the height $h$ to a Theta height $h_\theta$ on $Y$ by using the work of Hirzebruch~\cite{hirzebruch:ck} and computations for Hilbert theta functions due to G\"otzky~\cite{gotzki:theta}, Gundlach~\cite{gundlach:theta} and Lauter--Naehrig--Yang~\cite{lanaya:theta}. Then we  bound $h_\theta$ in terms of $h_\phi$ via the explicit height comparison of Pazuki~\cite{pazuki:heights} which is based on ideas of Bost--David. Throughout this section we work over $\bar{\QQ}$ and we write $Y$ for $Y_{\bar{\QQ}}$ and $\mathbb P^n$ for $\mathbb P^n_{\bar{\QQ}}$ where $n\in \ZZ_{\geq 1}$.


\paragraph{Theta height on $Y$.}We can define a Theta height on $Y$ by pulling back a Theta height on the coarse moduli space $A_{2,2}$ over $\bar{\QQ}$ of principally polarized abelian varieties with a symplectic level two structure \cite[p.19]{fach:deg}. More precisely, we define the Theta height $$h_\theta=\theta^*h: \ Y(\bar{\QQ})\to\mathbb R$$
where $\theta:Y\to \mathbb P^9$ is the composition of the following morphisms: The morphism $Y\to A_{2,2}$ induced on coarse moduli schemes by forgetting $\iota:\OL\to \End(A)$ and the morphism $A_{2,2}\to \mathbb P^9$ defined by the fourth powers of the even Theta functions $\theta_m=\theta_m(\tau,0)$ as in \cite[p.217]{vandergeer:hilbertmodular}.  The heights $h_\theta$ and $h$ on $Y$ are closely related. 
\begin{lemma}\label{lem:thetaheight}
Any $P\in Y(\bar{\QQ})$ satisfies $h(P)\leq h_\theta(P).$
\end{lemma}
\begin{proof}Let $x_i$ be the pullback along $Y\hookrightarrow \mathbb P^4$ of the $i$-th coordinate function on $\mathbb P^4$, and let $z_j$ be the pullback along $\theta:Y\to \mathbb P^9$ of the $j$-th coordinate function on $\mathbb P^9$.  Lemma~\ref{lem:thetaheight} is a direct consequence of the following claim: On $Y(\bar{\QQ})$ it holds 
\begin{equation}\label{claim:coordinatefunctions}
x_i^4=\pm\prod_{j\in J_i} z_j
\end{equation}
where $J_i\subset \{0,\dotsc,9\}$ with $|J_i|=4$. To prove this claim, we need to introduce some notation. Let $\mathbb H_2$ be the Siegel upper half space of genus 2 on which the symplectic group $\textnormal{Sp}_2(\ZZ)$ acts as usual and write $\mathbb H=\{z\in \CC, \textnormal{im}(z)>0\}$. We denote by $\Gamma_2$  the kernel of the projection $\textnormal{Sp}_2(\ZZ)\to \textnormal{Sp}_2(\ZZ/2\ZZ)$ and we denote by  $\Gamma(2)$ the kernel of the projection $\textnormal{SL}_2(\OL)\to \textnormal{SL}_2(\OL/2\OL)$. Then we consider the commutative diagram  \begin{equation}\label{diag:thetacomp}
\xymatrix@R=4em@C=4em{
Y(\bar{\QQ}) \ar@{^{(}->}[r] \ar[d] & Y(\CC)\ar^{\sim}[r] \ar[d] & [\mathcal M_2(\bar{\CC})]\ar^{\sim}[r] \ar[d] & \mathbb H^2/\Gamma(2)\ar^{\phi_\epsilon}[d] \\
A_{2,2}(\bar{\QQ}) \ar@{^{(}->}[r] &  A_{2,2}(\CC)\ar^{\sim}[r]& [\mathcal A_{2,2}(\bar{\CC})]\ar^{\sim}[r] &\mathbb H_2/\Gamma_2} 
\end{equation}
which is obtained as follows: 
The map $\phi_\epsilon$ is the modular morphism of Lauter--Yang~\cite[p.153]{lanaya:theta} with $\epsilon=\tfrac{1+\sqrt{5}}{2}$. The two bijections in the middle come from the coarse moduli spaces $\pi_{\bar{\QQ}}:\mathcal M_2\to Y$ and $\mathcal A_{2,2}\to A_{2,2}$,  where  $\mathcal M_2=(\mathcal M_{\mathcal P})_{\bar{\QQ}}$  and $\mathcal A_{2,2}$ is the moduli stack over $\bar{\QQ}$ of principally polarized abelian schemes with a symplectic level two structure as in \cite[p.19]{fach:deg}. The two bijections on the right come from the usual uniformization  of complex abelian varieties (with real multiplications),  where $[\mathcal M_2(\CC)]\isomto \mathbb H^2/\Gamma(2)$ factors through the inverse of the map $\phi_0:\mathbb H^2\to \mathbb H^2$ defined in \cite[p.153]{lanaya:theta}. 

We now use the diagram to compute $x_i$ in terms of the $z_j$. The discussion right after \eqref{def:initialmorphismoverB} provides that $\pi_{\bar{\QQ}}:\mathcal M_2\to Y$ is given by the five global sections of the square $\omega^{\otimes 2}$ of the Hodge bundle on $\mathcal M_2$ corresponding via the $q$-expansion principle to the five Eisenstein series $E_i$ for  $\Gamma(2)$ of weight 2 which are defined in $\mathsection$\ref{sec:eisenstein}. Then the inverse $\mathbb H^2/\Gamma(2)\isomto Y(\CC)$ of the composition of the isomorphisms $Y(\CC)\isomto [\mathcal M_2(\CC)]\isomto \mathbb H^2/\Gamma(2)$ in the diagram \eqref{diag:thetacomp} is given by the five Eisenstein series $E_i$. Moreover, it holds
\begin{equation}\label{eq:hirzebrucheisensteincomp}
Y_\CC\cong \proj\bigl(\CC[E_i]/(\sigma_2,\sigma_4)\bigl)\setminus Z_\CC, \quad{\textnormal{and} } \quad E_0=\theta_{01}\theta_{02}\theta_{03}\theta_{04}
\end{equation}
by Hirzebruch's computation in \cite[p.166]{hirzebruch:ck}. Here $Z_\CC$ is given by the five permutations of $(1,0,\dotsc,0)$ and $\theta_{ij}$ is the theta function (viewed on $\mathbb H^2$) of G\"otzky~\cite[(19)]{gotzki:theta} and Gundlach~\cite[p.233]{gundlach:theta}.  We next express the $\theta_{ij}$, and then all Eisenstein series $E_i$, in terms of the Hilbert theta functions $\theta_{a,b}$ for some $(a,b)\in (\OL/2\OL)^2$, where we write $\theta_{a,b}$ for the function $\theta_{a,b}^\epsilon$ of Lauter--Naehrig--Yang~\cite[(2.5)]{lanaya:theta}   defined with respect to the fundamental unit $\epsilon=\tfrac{1+\sqrt{5}}{2}$ and where $\OL$ denotes the ring of integers of $\QQ(\sqrt{5})$. In the table displayed in \cite[p.232]{gundlach:theta}, Gundlach computed all  pairs $(\alpha,\beta)$ which are relevant for the definition of the functions $\theta_{ij}$. For any pair $(\alpha,\beta)$ in this table, it holds that $\alpha\beta\in \{0,1,-1\}$ and thus the field trace $t=\textnormal{Tr}(\tfrac{\epsilon}{\sqrt{5}}\tfrac{ab}{4})$ is zero for $(a,b)=(\alpha,\beta/\epsilon)$.
  Then unwinding the definitions of the functions $\theta_{ij}$ and $\theta_{a,b}$ on $\mathbb H^2$ leads to  \begin{equation}\label{eq:thetaijcomp}
\theta_{ij}=e^{-2\pi i t}\theta_{a,b}=\theta_{a,b}
\end{equation}
for some $(a,b)\in (\OL/2\OL)^2$. The group $\sl2(\OL/2\OL)$ acts transitively on the set of cusps of $\mathbb H^2/\Gamma(2)$ and recall from \eqref{def:eisenstein} that the Eisenstein series $E_i=E_0|_{A_i}$ is obtained from $E_0$ via the slash operator for some $A_i\in \sl2(\OL)$. Thus, on combining \eqref{eq:hirzebrucheisensteincomp} and \eqref{eq:thetaijcomp} with the transformation formula for $\theta_{ij}^8$ established by G\"otzky~\cite[Lem 6]{gotzki:theta}, we obtain
\begin{equation}\label{eq:eisensteinthetacomp}
E_i^4=\pm  \prod \theta_{a,b}^{4}.
\end{equation}
Here the product is taken over four distinct pairs $(a,b)\in (\OL/2\OL)^2$. These four pairs are all even in the sense of \cite[p.154]{lanaya:theta} since $E_i\neq 0$.  For any even $(a,b)\in (\OL/2\OL)^2$,  Lauter--Naehrig--Yang~\cite[$\mathsection$2.5,\,$\mathsection$3]{lanaya:theta} showed that each Hilbert theta function  $\theta_{a,b}$ is (up to a sign) the pullback along $\phi_\epsilon:\mathbb H^2/\Gamma(2)\to \mathbb H_2/\Gamma_2$ of an even Theta function $\theta_m$.  Thus the commutative diagram \eqref{diag:thetacomp} gives that $\theta_{a,b}^4$ identifies on $ Y(\CC)\cong\mathbb H^2/\Gamma(2)$ with the pullback $z_j=\theta^*t_j$ of a coordinate function $t_j$ on $\mathbb P^9$ along the composition $\theta:Y\to A_{2,2}\to \mathbb P^9$.  
Then \eqref{diag:thetacomp} and \eqref{eq:eisensteinthetacomp} imply the claim \eqref{claim:coordinatefunctions} since $x_i$ identifies on $Y(\CC)\cong \mathbb H^2/\Gamma(2)$ with $E_i$ by \eqref{eq:hirzebrucheisensteincomp}. This completes the proof of the lemma.
\end{proof}

One can determine the sign of $E_i^4=\pm\prod \theta_{a,b}^4$ in \eqref{eq:eisensteinthetacomp}, and thus of $x_i^4=\pm \prod_{j\in J_i} z_j$ in \eqref{claim:coordinatefunctions}, by going into the proof of \cite[Lem 6]{gotzki:theta} and computing therein the sign in the transformation formula for five explicit $A_i\in \SL_2(\OL)$ defining the five $E_i=E_0|_{A_i}$.

\paragraph{Proof of Proposition~\ref{prop:heightcomp}.}To prove an explicit height bound, we need to compare our normalizations of the heights with other normalizations. Recall from Section~\ref{sec:heightdef} that $h_\phi:Y(\bar{\QQ})\to \RR$ is the pullback of $h_F:\underline{A}_2(\bar{\QQ})\to \mathbb R$ along the forgetful map
\begin{equation}\label{eq:phifactorheight}
\phi_{\bar{\QQ}}:Y(\bar{\QQ})\isomto^{\pi^{-1}} [\cmp(\bar{\QQ})]\to  \underline{A}_2(\bar{\QQ})
\end{equation}
of the coarse Hilbert moduli scheme $Y$  defined right after \eqref{def:cheight}, where $h_F$ is the stable Faltings height with the metric normalized as in Faltings~\cite[p.354]{faltings:finiteness}. Further, we recall that the logarithmic Weil height $h$ on  $\mathbb P^n$ is defined as usual in Diophantine approximation  with respect to the infinity norm at the infinite places (\cite[p.16]{bogu:diophantinegeometry}). It holds  \begin{equation}\label{eq:normalizations}
h_F=h'_F-\log \pi, \quad h\leq h'\leq h+\tfrac{1}{2}\log (n+1),
\end{equation} 
where $h_F'$ denotes the stable Faltings height with the metric normalized as in Deligne~\cite[p.27]{deligne:faltings} and where $h'$ denotes the logarithmic Weil height on $\mathbb P^n$ defined with respect to the usual Euclidean or Hermitian $l^2$ norm at the infinite places. 

We now deduce Proposition~\ref{prop:heightcomp} by combining Lemma~\ref{lem:thetaheight} with the explicit height comparison of Pazuki~\cite[Cor 1.3.2]{pazuki:heights} which is based on ideas of Bost--David. 

\begin{proof}[Proof of Proposition~\ref{prop:heightcomp}]
To prove the statement, it suffices to bound the Theta height $h_\theta$ on $Y$ since $h_\theta\geq h$ by Lemma~\ref{lem:thetaheight}. We first relate $h_\theta$ to another Theta height $h'_\theta$ on $Y$ which is the pullback of the Theta height defined in Pazuki~\cite{pazuki:heights}. More precisely, applying \cite[Def 2.6]{pazuki:heights} with $g=2$ and $r=2$ gives a height $h_\Theta:A_2(\bar{\QQ})\to \RR$ on the coarse moduli scheme $A_2$ over $\bar{\QQ}$ of principally polarized abelian surfaces.  Then we define
$$h'_{\theta}=f^*h_\Theta: Y(\bar{\QQ})\to \mathbb R$$
where $f:Y\to A_2$ is the composition of the morphism $Y\to A_{2,2}$ induced on coarse moduli schemes by forgetting $\iota:\mathcal O\to \End(A)$  and the morphism  $A_{2,2}\to A_2$ induced on coarse moduli schemes by forgetting the symplectic level 2-structure. The Theta height $h_\Theta$ is determined by the 10 even Theta functions $\theta_m$, since any Theta function $\theta_m$ is nonzero if and only if the Theta characteristic $m$ is even.  Thus (the proof of) \eqref{eq:normalizations} leads to  \begin{equation}\label{eq:thetaheightsrelation}
h_\theta\leq 4h'_{\theta}\leq 4h_\theta+2\log 10.
\end{equation}
We next compare the heights $h'_{\theta}$ and $h_\phi$ on $Y$. It follows from \eqref{eq:phifactorheight} that the map $\phi_{\bar{\QQ}}:Y(\bar{\QQ})\to \underline{A}_2(\bar{\QQ})$ factors as $Y(\bar{\QQ})\to A_{2}(\bar{\QQ})\to \underline{A}_2(\bar{\QQ})$ where $A_{2}(\bar{\QQ})\to \underline{A}_2(\bar{\QQ})$ is induced by forgetting the polarization.  Then, on combining \eqref{eq:normalizations} and \eqref{eq:thetaheightsrelation} with the bound for $|h_\Theta-\tfrac{1}{2}h_F'|$ in Pazuki~\cite[Cor 1.3.2]{pazuki:heights},  we deduce 
\begin{equation}\label{eq:thetaheightprimebound}
4h'_\theta\leq 2h_\phi+8^8\log(h_\phi+8).
\end{equation}
To simplify the shape of this estimate, we used here in addition Bost's explicit lower bound $h'_F\geq -\log(2\pi)$ in \cite{bost:lowerbound}. Finally for any $P\in Y(\bar{\QQ})$ the inequalities  \eqref{eq:thetaheightsrelation} and \eqref{eq:thetaheightprimebound} give an upper bound for $h_\theta(P)\geq h(P)$ as claimed in Proposition~\ref{prop:heightcomp}. 
\end{proof}
We now discuss the bounds. The constants $8^8$ and $8$ appearing in \eqref{eq:thetaheightprimebound} and Proposition~\ref{prop:heightcomp} crucially depend on the value at $g=2$ and $r=2$ of the function $C_2(g,r)$ in Pazuki's general result~\cite[Cor 1.3]{pazuki:heights}. We simplified both constants $8^8$ and $8$ and thus they can be improved up to a certain extent.  However, the factor 2 in front of $h_\phi$ is optimal in \eqref{eq:thetaheightprimebound} by the height comparison \cite[Cor 1.3]{pazuki:heights} and hence the factor 2 might also be optimal in Proposition~\ref{prop:heightcomp} in view of the relation $x_i^4=\pm\prod z_j$ in \eqref{claim:coordinatefunctions}.  

The bound in Proposition~\ref{prop:heightcomp} is linear in $h_\phi$ which is best possible in the sense that for any exponent $e<1$ the bound $h\ll h^{e}_\phi$ can not hold on $Y(\bar{\QQ})$. Indeed on combining \cite[Cor 1.3]{pazuki:heights} and \eqref{eq:thetaheightsrelation} with the relation $x_i^4=\pm\prod z_j$ in \eqref{claim:coordinatefunctions}, we see that such a bound already fails on the subset $Y^0(\QQ)\subset Y(\bar{\QQ})$  which is an infinite set by \eqref{eq:rationalpointslowerbound}, where $Y^0=Y\setminus \cup V_+(x_i)$ for $x_i$ the $i$-th coordinate function on $\mathbb P^4_\ZZ=\textnormal{Proj}(\ZZ[x_i])$.  
\subsection{Proof of Theorem~\ref{thm:cksurfaces} and the Corollary}\label{sec:proofofthmck}

In this section we first prove Theorem~\ref{thm:cksurfaces} by combining Propositions \ref{prop:moduliinter} and \ref{prop:heightcomp}. Then we deduce from Theorems~\ref{thm:mainint} and \ref{thm:cksurfaces} the Corollary  stated in the introduction. 

We continue the notation of Section~\ref{sec:ckmainresults}. Let $Y$ be the variety over $\ZZ$ defined in \eqref{def:sigma24surface}. We denote by $h$ the logarithmic Weil height on $Y\subset\mathbb P^4_\ZZ$, we write $\mathcal M$ for the Hilbert moduli stack associated to $L=\QQ(\sqrt{5})$ and we define $B=\ZZ[\tfrac{1}{30}]$.

\begin{proof}[Proof of Theorem~\ref{thm:cksurfaces}]
Let $\mathcal P$ be the moduli problem on $\mathcal M_B$ of principal level 2-structures satisfying the condition \eqref{def:symplecticlvl}.  Proposition~\ref{prop:moduliinter} gives that $Y_B$ is a coarse moduli scheme over $B$ of $\mathcal P$. Let  $\pi:(\mathcal M_B)_{\mathcal P}\to Y_B$ be the initial morphism given in \eqref{def:initialmorphismoverB} and let $h_\phi$ be the height on $Y(\bar{\QQ})\cong Y_B(\bar{\QQ})$ defined with respect to $\pi$. We claim that
$$B_{\mathcal P}=\emptyset,  \quad |\mathcal P|_{\bar{\QQ}}\leq 16 \ \textnormal{ and } \ h(P)\leq 2h_\phi(P)+8^8\log (h_\phi(P)+8), \, P\in Y(\bar{\QQ}).$$ 
Indeed on using exactly the same arguments as in the proof of Corollary~\ref{cor:y2emptybranchlocus},  we deduce from Proposition~\ref{prop:auto} that the branch locus $B_{\mathcal P}$ is empty. Further, as in \eqref{eq:proofsdegbound} we find that the maximal number of level $\mathcal P$-structures over $\bar{\QQ}$ is bounded by $|\mathcal P|_{\bar{\QQ}}\leq 16$, while the height bound is provided by Proposition~\ref{prop:heightcomp}. Now, the claim implies Theorem~\ref{thm:cksurfaces}  after extending $Y_B$ to a coarse moduli scheme over $\ZZ$. The details are as follows: We extend $Y_B$ to a coarse moduli scheme over $\ZZ$ via the naive extension $\mathcal P'$ on $\mathcal M$ defined right above \eqref{eq:naiveextension} and we use the
initial morphism $\pi':\mathcal M_{\mathcal P'}\to Y_B$ given by the composition 
\begin{equation}\label{def:morphismckthm}
\mathcal M_{\mathcal P'}\isomto (\mathcal M_B)_{\mathcal P}\to^{\pi} Y_B.
\end{equation} Here $\mathcal M_{\mathcal P'}\isomto (\mathcal M_B)_{\mathcal P}$ is induced by the identity. Thus the branch locus $B_{\mathcal P'}$ equals $B_{\mathcal P}$, the height defined with respect to $\pi':\mathcal M_{\mathcal P'}\to Y_B$ coincides on $Y(\bar{\QQ})\cong Y_B(\bar{\QQ})$ with the height $h_\phi$ defined with respect to $\pi:(\mathcal M_B)_{\mathcal P}\to Y_B$, and  $\mathcal P'$ is an arithmetic moduli problem on $\mathcal M$ since $\mathcal P$ is an arithmetic moduli problem on $\mathcal M_B$.  We also notice that $|\mathcal P'|_{\bar{\QQ}}=|\mathcal P|_{\bar{\QQ}}$.  Hence the claim proves that Theorem~\ref{thm:cksurfaces} holds with $\mathcal P=\mathcal P'$ and $\pi=\pi'$. 
\end{proof}
We now deduce from Theorems~\ref{thm:mainint} and \ref{thm:cksurfaces} the Corollary stated in the introduction.
\begin{proof}[Proof of Corollary]
Let $S$ be a finite set of rational primes. Theorem~\ref{thm:cksurfaces} gives that $Y_B$ is a coarse moduli scheme of an arithmetic moduli problem $\mathcal P$ on $\mathcal M$ with $|\mathcal P|_{\bar{\QQ}}\leq 16$ such that the branch locus $B_{\mathcal P}$ is empty and such that each $P\in Y(\bar{\QQ})$ satisfies $$h(P)\leq 2h_\phi(P)+8^8\log (h_\phi(P)+8).$$
Here $h_\phi$ is the height on $Y(\bar{\QQ})\cong Y_B(\bar{\QQ})$ defined in \eqref{def:cheight} with respect to the initial morphism $\pi:\cmp\to Y_B$ given in \eqref{def:initialmorphismoverB}. Base change properties of the height $h_\phi$ assure that $h_\phi$ coincides on $Y(\ZZ_S)\subset Y(\bar{\QQ})$ with the height $h_\phi$ on $Y(\ZZ_S)$ obtained by applying the definition at the end of Section~\ref{sec:heightdef} with $\pi:\cmp\to Y_B$; notice that $\spec(\ZZ)-S=\spec(\ZZ_S)$ is a nonempty open subscheme of $\spec(\ZZ)$.   Then an application of Theorem~\ref{thm:mainint} with $T=B$, $S=\spec(\ZZ_S)$ and $Z=\emptyset$ gives that any $P\in Y(\ZZ_S)$ satisfies
$$h_\phi(P)\leq c_1N_{U}^{24}, \quad \textnormal{and} \quad |Y(\ZZ_S)|\leq c_2|\mathcal P|_{\bar{\QQ}} N_{U}^{e_2}\Delta\log(3\Delta)^3.$$
Here the constants are $c_1=7^{7^{14}}$, $c_2=9^{9^{18}}$ and $e_2=4^{16}$, while $\Delta=5$ and $N_U=\prod_{p\in S'} p$ for $S'=S\cup\{2,3,5\}$. As $|\mathcal P|_{\bar{\QQ}}\leq 16$ and $N_U\leq 30N_S$, the displayed bounds show that the Corollary holds with  $e=10^{12}$ and $c=10^e$ as claimed. This completes the proof.  \end{proof}

\subsection{Diophantine equations and the Clebsch--Klein surfaces}\label{sec:diophaeq}

In this section we first collect basic results for the Clebsch--Klein surfaces and the Diophantine equations \eqref{eq:ck} and \eqref{eq:sigma24}. Then  we combine these results with the Corollary stated in the introduction to deduce Corollaries~\ref{cor:sigma24} and \ref{cor:ck} giving explicit bounds for the solutions of \eqref{eq:ck} and \eqref{eq:sigma24}. Finally, we discuss lower bounds for rational points of bounded height ($\mathsection$\ref{sec:rationalpoints}) and the degeneration of $S$-integral points ($\mathsection$\ref{sec:degeneration}).  

We continue the notation introduced in Section~\ref{sec:ckmainresults}. Let $Y$ be the variety over $\ZZ$ defined in \eqref{def:sigma24surface} and let $U$ be the variety over $\ZZ$ defined in \eqref{def:cubicsurface}. For any finite set $S$ of rational primes, we recall that $\ZZ_S=\ZZ[1/N_S]$ where $N_S=\prod_{p\in S} p$. We first state basic lemmas relating the solutions of \eqref{eq:ck} and \eqref{eq:sigma24} with the $\ZZ_S$-points of $Y$ and $U$ respectively. Sending $x$ to $-x$ defines an action of the group $\{\pm 1\}$ on the set of solutions of \eqref{eq:sigma24}.  

\begin{lemma}\label{lem:solutionssigma24}
The set of solutions of \eqref{eq:sigma24} modulo $\{\pm 1\}$ identifies  with $Y(\ZZ_S)$.
\end{lemma}
Here the identification comes from the usual description of $\ZZ_S$-points on projective space.  The identification in the following lemma also comes from this description.

\begin{lemma}\label{lem:solutionsck}
The set of solutions of \eqref{eq:ck} identifies with $U(\ZZ_S)$.
\end{lemma}
We write $\mathbb P^4$ for $\mathbb P^4_\ZZ$. In Section~\ref{sec:ckmainresults} we defined the logarithmic Weil height $h$ on the sets of solutions of \eqref{eq:ck} and \eqref{eq:sigma24} and also on the subsets $Y(\ZZ_S)$ and $U(\ZZ_S)$ of $\mathbb P^4(\bar{\QQ})$. 

\begin{lemma}\label{lem:identificationsheightcompatible}
The identifications in Lemmas~\ref{lem:solutionssigma24} and \ref{lem:solutionsck} are compatible with $h$.
\end{lemma}

In the remaining of this section we prove inter alia the above lemmas and we explicitly construct the open immersion $U\hookrightarrow Y$ over $\ZZ[1/3]$ appearing in Lemma~\ref{lem:openimmersion}. The proofs use standard computations. In particular we shall use the usual identification 
\begin{equation}\label{eq:pointsp4}
\mathbb P^4(\ZZ_S)\isomto\Sigma_{\mathbb P^4}/\ZZ_S^\times.
\end{equation}
Here the group $\ZZ_S^\times$ acts diagonally on the set $\Sigma_{\mathbb P^4}$  of $x\in \ZZ_S^5$ with $\gcd(x_i)\in \ZZ_S^\times$  
 and the bijection \eqref{eq:pointsp4} comes from the standard description of $\mathbb P^4(\ZZ_S)$ in terms of isomorphism classes of surjections $\ZZ_S^5\to \ZZ_S$ using that $\textnormal{Pic}(\ZZ_S)=0$.  We first prove Lemma~\ref{lem:solutionssigma24}.

\begin{proof}[Proof of Lemma~\ref{lem:solutionssigma24}]
Let $P$ a $\ZZ_S$-point of $\mathbb P^4$ corresponding via \eqref{eq:pointsp4} to some $[x]$ in $\Sigma_{\mathbb P^4}/\ZZ_S^\times$, and let $Q$ be the $\ZZ_S$-point of $\mathbb P^4$ given by $e_1=(1,0,\dotsc,0)$. For any closed point $p\in \spec(\ZZ_S)$, the images in $\mathbb P^4$ of the sections $P$ and $Q$  have empty intersection in the fiber $\mathbb P^4_p$ if and only if the image of $x$ in $\mathbb F_p^5$ does not lie in the line $L_p=\mathbb F_p\cdot e_1$ where $\mathbb F_p=\ZZ_S/p\ZZ_S$.  It follows that the identification \eqref{eq:pointsp4} restricts to a bijection  $$
Y(\ZZ_S)\isomto \Sigma_Y/\ZZ_S^\times
$$
for $\Sigma_{Y}$ the set of $x\in \ZZ_S^{5}$ with $\gcd(x_i)\in \ZZ_S^\times$ such that $\sigma_2(x)=0=\sigma_4(x)$  and such that for all rational primes $p\notin S$ the image of $x$ in $\mathbb F_p^5$ does not lie in any of the five permutations of the line $L_p$. Let $\Sigma_\sigma$ be the set of $x\in \Sigma$ with $\sigma_2(x)=0=\sigma_4(x)$ for $\Sigma$ as in \eqref{eq:sigma24}. Multiplying with the denominators of the $x_i$ and dividing by the $\gcd$ defines a bijection
\begin{equation}\label{eq:sigmabijection}
\Sigma_Y/\ZZ_S^\times\isomto \Sigma_\sigma/\{\pm 1\}
\end{equation}
which composed  with $Y(\ZZ_S)\isomto \Sigma_Y/\ZZ_S^\times$  gives an identification $Y(\ZZ_S)\isomto\Sigma_\sigma/\{\pm 1\}$ as desired. The remaining of this proof contains the details of the definition of the bijection \eqref{eq:sigmabijection}. Suppose that $x=(x_i)$ lies in $\Sigma_Y$. As the $x_i$ are in $\ZZ_S$ with $\gcd(x_i)\in \ZZ_S^\times$,  there exists $u\in \ZZ_S^\times$ which is unique up to $\pm 1$ such that $y=ux$ lies in $\ZZ^5$ with $\gcd(y_i)=1$.  The scalar multiple $y=ux$ of $x$ satisfies again the homogeneous equations $\sigma_2=0=\sigma_4$. Moreover, if $p\notin S$ is a rational prime then $\textnormal{ord}_p(u)=0$ and thus our assumption $x\in \Sigma_Y$ assures that the image of $y$ in $(\ZZ/p\ZZ)^5\cong\mathbb F_p^5$ is not a scalar multiple of a standard basis vector of $(\ZZ/p\ZZ)^5$.  Hence $y$ lies in $\Sigma_\sigma$. On using that $u$ is unique up to $\pm 1$, we see that $x\mapsto y$ defines an injective map $\Sigma_Y/\ZZ_S^\times\to \Sigma_\sigma/\{\pm 1\}$ which is surjective since $\Sigma_\sigma\subset \Sigma_Y$.  Now, we define \eqref{eq:sigmabijection} by $x\mapsto y$. This completes the proof of Lemma~\ref{lem:solutionssigma24}.
\end{proof} 

We now deduce Lemma~\ref{lem:solutionsck} from the identification $\mathbb P^4(\ZZ_S)\isomto\Sigma_{\mathbb P^4}/\ZZ_S^\times$ in \eqref{eq:pointsp4}.

\begin{proof}[Proof of Lemma~\ref{lem:solutionsck}]Recall that $U=V\setminus \cup V_+(z_i)$ for $V\subset\mathbb P^4$ as in \eqref{def:cubicsurface}.
We first compute the image of $U(\ZZ_S)\subset \mathbb P^4(\ZZ_S)$ under the identification \eqref{eq:pointsp4}.  Let $P$ be a $\ZZ_S$-point of $\mathbb P^4=\textnormal{Proj}(\ZZ[z_i])$ corresponding via \eqref{eq:pointsp4} to $[x]$ in $\Sigma_{\mathbb P^4}/\ZZ_S^\times$. Then $P$ is a $\ZZ_S$-point of $\mathbb P^4\setminus V_+(z_i)$ if and only if $x_i$ lies in $\ZZ_S^\times$. This implies that \eqref{eq:pointsp4} restricts to a bijection  $$U(\ZZ_S)\isomto \Sigma_U/\ZZ_S^\times$$
where $\ZZ_S^\times$ acts diagonally on the set $\Sigma_U$ of $x\in (\ZZ^\times_S)^5$ satisfying $\sum x_i=0=\sum x_i^3$.
Let $\Sigma'$ be the set of solutions of \eqref{eq:ck}. To identify $\Sigma_U/\ZZ_S^\times$ with $\Sigma'$, we take $x\in \Sigma_U$ and we define $y_{i}=-x_i/x_4$ for $i\leq 3$. As $\sum y_i=1=\sum y_i^3$, we see that $y=(y_i)$ lies in $\Sigma'$. It follows that $x\mapsto y$ defines a bijection $\Sigma_U/\ZZ_S^\times\isomto \Sigma'$  which composed with $U(\ZZ_S)\isomto \Sigma_U/\ZZ_S^\times$ gives an identification $U(\ZZ_S)\isomto \Sigma'$ as desired.  This completes the proof of Lemma~\ref{lem:solutionsck}.
\end{proof}
Next, we show that the two identifications constructed in the above proofs are compatible with $h$. For each $[x]\in \Sigma_{\mathbb P^4}/\ZZ_S^\times$ we define $h([x])$ as the logarithmic Weil height of $x\in\mathbb P^4(\bar{\QQ})$. Then the identification $\mathbb P^4(\ZZ_S)\isomto\Sigma_{\mathbb P^4}/\ZZ_S^\times$ in \eqref{eq:pointsp4} is compatible with $h$.
\begin{proof}[Proof of Lemma~\ref{lem:identificationsheightcompatible}]
In the proof of Lemma~\ref{lem:solutionssigma24} we identified (modulo $\pm 1$) the set of solutions $\Sigma_\sigma$ of \eqref{eq:sigma24} with $Y(\ZZ_S)$ via the inverse $\Sigma_\sigma/\{\pm 1\}\isomto Y(\ZZ_S)$ of the composition  $$Y(\ZZ_S)\isomto \Sigma_Y/\ZZ_S^\times\isomto \Sigma_\sigma/\{\pm 1\}.$$  Here $Y(\ZZ_S)\isomto \Sigma_Y/\ZZ_S^\times$ is compatible with $h$ since it is the restriction of \eqref{eq:pointsp4}, while $\Sigma_Y/\ZZ_S^\times\isomto \Sigma_\sigma/\{\pm 1\}$ is induced by $x\mapsto y$ which is compatible with $h$ since $x$ and $y=ux$ are equal in $\mathbb P^4(\bar{\QQ})$ and since $h(x)=\max_i\log |x_i|$ for any $x\in \mathbb P^4(\bar{\QQ})$ with $x_i\in \ZZ$ and $\gcd(x_i)=1$. Thus the identification $\Sigma_\sigma/\{\pm 1\}\isomto Y(\ZZ_S)$ is compatible with $h$ as desired.

We now go into the proof of Lemma~\ref{lem:solutionsck}. Therein we identified the set $\Sigma'$ of solutions of \eqref{eq:ck} with $U(\ZZ_S)$ via the inverse $\Sigma'\isomto U(\ZZ_S)$ of the composition  $$U(\ZZ_S)\isomto \Sigma_U/\ZZ_S^\times\isomto \Sigma'.$$ Here again, the bijection $U(\ZZ_S)\isomto \Sigma_U/\ZZ_S^\times$ is compatible with $h$ since it is the restriction of \eqref{eq:pointsp4}. The bijection $\Sigma_U/\ZZ_S^\times\isomto \Sigma'$ is induced by $x\mapsto y$ which is compatible with $h$ by the following computation: As $x$ and $y'=(y,-1)$ are equal in $\mathbb P^4(\bar{\QQ})$, we obtain that $h(x)=h(y')=h((y,1))=h(y)$.  Hence the identification $\Sigma'\isomto U(\ZZ_S)$ is compatible with $h$ as desired. This completes the proof of Lemma~\ref{lem:identificationsheightcompatible}.
\end{proof}

Finally we construct the open immersion $\varphi:U\hookrightarrow Y$ over $\ZZ[1/3]$ of Lemma~\ref{lem:openimmersion}.

\begin{proof}[Proof of Lemma~\ref{lem:openimmersion}]
Consider the open subscheme $Y^0=Y\setminus \cup V_+(x_i)$ of $Y$. Any isomorphism $U\isomto Y^0$ induces an open immersion $U\hookrightarrow Y$. To construct an isomorphism $U\isomto Y^0$ over $\ZZ[1/3]$, we first explicitly describe the schemes $U=V\setminus \cup V_+(z_i)$ and $Y^0$ inside $\mathbb A^4_\ZZ$. On using that  $U=D_+(\prod z_i)\subset V$,  
we see that $t_i=z_i/z_0$ satisfies   
\begin{equation}\label{eq:Ucomp}
U=\spec\bigl(\ZZ[t_i,t_i^{-1}]/(u,v)\bigl)\subset \mathbb A^4_\ZZ.
\end{equation}
Here $u=\sigma_1(t_i)=\sum t_i$ and $v=\sum t_i^3$ with both sums taken over all $i\in \{0,\dotsc,4\}$. Analogously, on using that $Y^0=D_+(\prod x_i)\subset Y$, we deduce that $s_i=x_i/x_0$ satisfies
$$Y^0=\spec\left(\ZZ[s_i,s_i^{-1}]/\bigl(\sigma_2(s_i),\sigma_4(s_i)\bigl)\right)\subset \mathbb A^4_\ZZ.$$
 In the remaining of this proof, we work over $R=\ZZ[1/3]$ and we use the descriptions of $U\subset \mathbb A^4_\ZZ$ and $Y^0\subset \mathbb A^4_\ZZ$ to show that $s_i\mapsto 1/t_i$ defines an isomorphism of $R$-schemes:
$$
\varphi:U_R\isomto Y^0_R.
$$ 
Let $I$ be the ideal of $A=R[s_i,s_i^{-1}]$ generated by $\sigma_2(s_i),\sigma_4(s_i)$ and let $J$ be the ideal of $B=R[t_i,t_i^{-1}]$ generated by $u,v$. To see that $s_i\mapsto 1/t_i$ defines an isomorphism $A/I\isomto B/J$ of $R$-algebras, we recall the Newton identity $3\sigma_3=\sigma_1\sigma_2-\sigma_1p_2+p_3$ for $p_n=\sum y_i^n$ and we observe that $\tau \sigma_2(s_i)=\sigma_3(t_i)$ and $\tau \sigma_4(s_i)=u$ where $s_i=1/t_i$ and $\tau=\prod t_i$.  Thus $s_i\mapsto 1/t_i$ sends $I$ inside $J$ since $3$ is invertible in $R$, and $t_i\mapsto 1/s_i$ sends $J$ inside $I$.  It follows that $s_i\mapsto 1/t_i$ indeed defines an isomorphism $A/I\isomto B/J$ of $R$-algebras, and hence an isomorphism $\varphi:U_R\isomto Y^0_R$ of $R$-schemes, whose inverse is given by $t_i\mapsto 1/s_i$.

It remains to prove the height inequalities for any point $P\in U(\QQ)$. We view  $P$ inside $\mathbb P^4(\QQ)$ via $U\hookrightarrow V\hookrightarrow \mathbb P^4$. Then the description of $U$ in \eqref{eq:Ucomp} shows that $P\in \mathbb P^4(\QQ)$ is defined by some $x\in \QQ^5$ with $x_1=1$ and $x_i\neq 0$.  There exists $a\in \ZZ^5$ with $\gcd(a_i)=1$ and $a_i\neq 0$ such that $x_i=a_i/a_0$. It holds that $h(P)=\log \max_i |a_i|$ and then we compute  $$h(\varphi(P))\leq \max_i \log |a_i^{-1}\prod a_j|\leq 4h(P).$$ To verify the converse height inequality, we now use the displayed description of $Y^0$. Then an application of the above arguments with $\varphi^{-1}:Y^0_R\isomto U_R$ and $Q=\varphi(P)$ shows that $h(P)=h(\varphi^{-1}(Q))$ is at most $4h(Q)=4h(\varphi(P))$. Thus the height inequalities hold for any $P\in U(\QQ)$ as desired. This completes the proof of Lemma~\ref{lem:openimmersion}.
\end{proof}

\subsubsection{Proof of Corollaries~\ref{cor:sigma24} and \ref{cor:ck}}\label{sec:proofofdiopheqcorollaries}

We continue the notation of the previous section.  For any finite set $S$ of rational primes, the (proof of the) Corollary stated in the introduction gives that each $P\in Y(\ZZ_S)$ satisfies  
\begin{equation}\label{eq:statecorbounds}
4h(P)\leq c(\tfrac{1}{3}N_{S})^{24} \quad \textnormal{ and } \quad |Y(\ZZ_S)|\leq (\tfrac{c}{3}N_S)^e
\end{equation}
for  $e=10^{12}$ and $c=10^e$. We now deduce Corollaries~\ref{cor:sigma24} and \ref{cor:ck} by combining the results in \eqref{eq:statecorbounds} with the basic lemmas collected in the previous section.

\begin{proof}[Proof of Corollary~\ref{cor:sigma24}]
In Lemma~\ref{lem:solutionssigma24} we identified (modulo $\pm 1$) the set of  solutions of \eqref{eq:sigma24} with the set $Y(\ZZ_S)$, and Lemma~\ref{lem:identificationsheightcompatible} assures that this identification is compatible with the Weil height $h$. Thus \eqref{eq:statecorbounds} proves Corollary~\ref{cor:sigma24}.
\end{proof}

\begin{proof}[Proof of Corollary~\ref{cor:ck}]
Recall that $\Sigma'$ denotes the set of solutions of \eqref{eq:ck} and  define $T=S\cup \{3\}$. Lemma~\ref{lem:identificationsheightcompatible} gives that the identification $\Sigma'\isomto U(\ZZ_S)$ constructed in Lemma~\ref{lem:solutionsck}  is compatible with the Weil height  $h$. Then we deduce
$$|\Sigma'|=|U(\ZZ_S)|\leq |Y(\ZZ_{T})| \quad  \textnormal{and} \quad \sup_{x\in \Sigma'}h(x)=\sup_{P\in U(\ZZ_S)}h(P)\leq 4\sup_{Q\in Y(\ZZ_{T})}h(Q).$$
Here we exploited that $U(\ZZ_S)\subseteq U(\ZZ_T)=U_{\ZZ[1/3]}(\ZZ_T)$ which follows since $U$ is separated over the terminal object $\ZZ$ and $3$ is invertible in the dense open $\spec(\ZZ_T)\subseteq\spec(\ZZ_S)$,  and  we used the open immersion $\varphi:U_{\ZZ[1/3]}\to Y_{\ZZ[1/3]}$ of Lemma~\ref{lem:openimmersion} with $h(P)\leq 4h(\varphi(P))$ for all $P\in U(\QQ)$.  An application of \eqref{eq:statecorbounds} with $S=T$ gives explicit bounds for $|Y(\ZZ_T)|$ and $\sup_{Q\in Y(\ZZ_T)}h(Q)$ in terms of $N_T$. As $N_T\leq 3N_S$, these explicit bounds combined with the displayed inequalities show that Corollary~\ref{cor:ck} holds with $e=10^{12}$ and $c=10^e$. 
\end{proof}

\subsubsection{Lower bounds for the number of rational points of bounded height}\label{sec:rationalpoints}

On using a secant-tangent process, Slater--Swinnerton-Dyer~\cite{slsw:cubicmanin} established for certain smooth cubic surfaces the lower bound predicted by Manin's conjecture (\cite{frmats:maninconjecture}) on rational points of bounded height. In this section we give the basic computations which show that their result directly implies the lower bounds in \eqref{eq:rationalpointslowerbound}. 

We continue our notation. Let $U\subset V$ be as in \eqref{def:cubicsurface} and  let $\rho$ be the Picard rank of the smooth projective surface $V_\QQ$. We first consider the $\QQ$-points of $U$ of bounded height.

\vspace{0.3cm}
\noindent\emph{\ \, Claim: For any $b\in \RR$ with $b\to \infty$ it holds} $|\{P\in U(\QQ), h(P)\leq \log b\}|\gg b(\log b)^{\rho-1}.$
\vspace{0.3cm}

\noindent Here the implied constant depends only on $V_\QQ$. The claim has the following two consequences. Firstly, it proves the lower bounds in \eqref{eq:rationalpointslowerbound} when combined with the isomorphism $U_\QQ\isomto Y^0_\QQ$ constructed in the proof of Lemma~\ref{lem:openimmersion} for $Y^0=Y\setminus \cup_iV_+(x_i)$.  
Secondly, the claim and (the proofs of) Lemmas~\ref{lem:solutionsck} and \ref{lem:identificationsheightcompatible} show for each $b\in \RR$ with $b\to \infty$ that
\begin{equation}\label{eq:rationallowerboundequation}
|\{x\in \mathcal S, h(x)\leq \log b\}|\gg b(\log b)^{\rho-1}, \ \mathcal S=\{x\in (\QQ^\times)^4, \sum x_i^3=1=\sum x_i\}.
\end{equation}
Here for each $x\in \mathcal S$ we denote by $h(x)$ the logarithmic Weil height of $x\in \mathbb A^4_\ZZ(\bar{\QQ})$ and the implied constant depends again only on $V_\QQ$. The following proof shows that the lower bounds in the claim and in \eqref{eq:rationallowerboundequation} are both optimal if  Manin's conjecture holds for the $\QQ$-points of the smooth cubic surface defined by $\sum x_i^3-(\sum x_i)^3$ inside $\mathbb P^3_\QQ$.

\begin{proof}[Proof of the Claim]
To identify $V_\QQ$ with a smooth cubic surface in $\mathbb P^3_\QQ$, we consider $f=\sum x_i^3-(\sum x_i)^3$ in $A=\ZZ[x_1,\dotsc,x_4]$ and $p_3=\sum z_i^3$ in $B=\ZZ[z_0,\dotsc,z_4]$. Then $x_i\mapsto z_i$ defines an isomorphism of rings $A/(f)\isomto B/(\sigma_1,p_3)$ with inverse given by $z_i\mapsto x_i$ for $i\geq 1$ and $z_0\mapsto -\sum x_i$.   This induces an isomorphism of varieties over $\ZZ$:
\begin{equation}\label{def:v'}
\tau: V'=\textnormal{Proj}\bigl(B/(f)\bigl)\isomto V, \quad h(P)\leq h(\tau(P))\leq h(P)+\log 8,\ P\in V'(\bar{\QQ}),
\end{equation}
where $h$ denotes the logarithmic Weil height on $V'\subset \mathbb P^3_\ZZ$ and $V\subset \mathbb P^4_\ZZ$. We now work over $\QQ$. To simplify notation we write $V$ for $V_\QQ$ and $V'$ for $V'_\QQ$.  As $V'\subset \mathbb P^3_\QQ$ is a smooth cubic surface with two rational skew lines,  it satisfies all assumptions in \cite[Thm 1]{slsw:cubicmanin} of Slater--Swinnerton-Dyer. Thus for each $b\in \RR$ with $b\to \infty$ their result gives
\begin{equation}\label{eq:lowerboundclaimproof}
|\{P\in U'(\QQ), h(P)\leq \log b\}|\gg b(\log b)^{\rho-1},
\end{equation}
where $U'\subset V'$ is obtained by removing all the lines from $V'$ and where the implied constant depends only on $V\cong V'$. To relate $U'$ with $U$, we consider the 15 lines $\Delta_n$ on $V$ given by the permutations of the line $z_0+z_1=0$, $z_2+z_3=0$, $z_4=0$ and we consider for $\epsilon=\tfrac{1+\sqrt{5}}{2}$ the 12 lines $L_i$ on $V$ given by the permutations (fixing $z_0$) of the line $$z_1+\epsilon z_2+z_3=0, \ \epsilon z_1+z_2+z_4=0, \ -\epsilon(z_1+z_2)+z_0=0.$$
Notice that $\tau^{-1}(L_i)$ and $\tau^{-1}(\Delta_n)$ are both lines inside $V'$,  the lines $L_i$ have no $\QQ$-points since $\epsilon$ does not lie in $\QQ$, and the variety $V'$ contains exactly 27 lines since it is a smooth cubic surface. Then, on using the Newton identity for $p_3$, we compute  \begin{equation}\label{comp:uv}
U=V\setminus (\cup \Delta_n)\quad \textnormal{and}\quad \tau: U'(\QQ)\isomto (U\setminus \cup L_i)(\QQ)=U(\QQ).
\end{equation}
Thus \eqref{def:v'} and \eqref{eq:lowerboundclaimproof} prove the claim. Furthermore, if Manin's conjecture holds for $V'(\QQ)$, then the lower bound \eqref{eq:lowerboundclaimproof} is optimal and hence the claim is optimal as well; indeed \eqref{def:v'} assures that the bijection $\tau^{-1}:U(\QQ)=(U\setminus \cup L_i)(\QQ)\isomto U'(\QQ)$ injects the set $\{P\in U(\QQ), h(P)\leq \log b\}$  into $\{P\in U'(\QQ), h(P)\leq \log b\}$ for all $b\in \RR_{>0}$.
\end{proof}

\subsubsection{Degeneration of $S$-integral points}\label{sec:degeneration}

Corvaja--Zannier~\cite{coza:intpointssurfaces,coza:intpointscertainsurfaces,coza:intpointsdivisibility} proved in particular the degeneration of $S$-integral points on smooth cubic surfaces minus two completely reducible hyperplane sections.  In this section we give the basic computations which show that their result directly implies the degeneration of $S$-integral points on the affine variety defined in \eqref{degeneration}.  

We continue our notation. Denote by $V=V_\QQ$ the projective variety over $\QQ$ defined by \eqref{def:cubicsurface}, and for  $i\neq j$ consider the divisor $D=V_+(z_i)\cup V_+(z_j)$ of $V$. We again  identify  $V$ with the smooth cubic surface $V'\subset\mathbb P^3_\QQ$ defined by \eqref{def:v'}. After restricting the inverse   $\tau^{-1}:V\isomto V'$ of the isomorphism in \eqref{def:v'} to the open subvariety $V\setminus D$, we obtain
$$\tau^{-1}: V\setminus D\isomto V'\setminus D',  \quad D'=\tau^{-1}(D).$$
To determine the divisor $D'$ of $V'$, we denote by $\{\Delta_n\}$ the 15 lines defined in $\mathsection$\ref{sec:rationalpoints}.  On using the same arguments as in \eqref{comp:uv}, we compute that $V_+(z_i)$ is the union of the three lines in $\{\Delta_n\}$ with $z_i=0$.  Then $D=\cup \Delta_n$ is the union of six lines $\Delta_n$ and we find that  $$D'=\cup \tau^{-1}(\Delta_n)$$ consists of six lines $\tau^{-1}(\Delta_n)$ lying in two planes such that no three of them intersect.  Hence an application of Corvaja--Zannier~\cite[Thm 1]{coza:intpointsdivisibility} with  $V'$ and $D'$ gives the following:  For any number field $K$ and each finite set $S$ of places of $K$, the set of $S$-integral points of $V'\setminus D'$, and thus of $V\setminus D$, is not Zariski dense. Moreover \cite[Thm 9]{coza:intpointsdivisibility} of Corvaja--Zannier  shows inter alia that this statement is optimal in the following sense: For any line $\Delta_n$  inside $D$ define $D_0=D\setminus \Delta_n$ and $D'_0=D'\setminus \tau^{-1}(\Delta_n)$. Then there exist $K$ and $S$ such that the set of $S$-integral points of $V'\setminus D'_0$, and thus of $V\setminus D_0$, is in fact Zariski dense.

\newpage

{\scriptsize
\bibliographystyle{amsalpha}
\bibliography{../literature}
}

\vspace{5cm}

\newpage

\vspace{0.1cm}
\noindent Rafael von K\"anel, IAS Tsinghua University, Beijing

\noindent E-mail address: {\sf rafaelvonkanel@gmail.com}

\vspace{1cm}

\noindent Arno Kret, University of Amsterdam, Amsterdam

\noindent E-mail address: {\sf  arnokret@gmail.com}

\end{document}